\documentclass[a4paper,10pt]{smfart}
\usepackage[utf8]{inputenc}
\usepackage{amssymb}
\usepackage{amsmath}
\usepackage{bbm}
\usepackage{mathrsfs}
\usepackage[all]{xy}
\usepackage{manfnt}
\usepackage{pigpen}

\usepackage{graphicx}
\usepackage{xcolor}
\usepackage[nocfg,notocbasic]{nomencl}
\usepackage{bm}
\usepackage[backref=page]{hyperref}
\title[Higher Hochschild homology]{Higher Hochschild homology and exponential functors}

\author{Geoffrey Powell}
\address{Univ Angers, CNRS, LAREMA, SFR MATHSTIC, F-49000 Angers, France}
\email{Geoffrey.Powell@math.cnrs.fr}
\urladdr{https://math.univ-angers.fr/~powell/}
\author{Christine Vespa}
\address{Aix-Marseille Université/Centrale Marseille, Institut de Mathématiques de Marseille, Marseille, France.}
\email{christine.vespa@univ-amu.fr}
\urladdr{https://www.i2m.univ-amu.fr/perso/christine.vespa/}


\newtheorem{THM}{Theorem}
\newtheorem{PROP}[THM]{Proposition}

\newtheorem{PROB}[THM]{Problem}
\newtheorem{thm}{Theorem}[section]
\newtheorem{prop}[thm]{Proposition}

\newtheorem{cor}[thm]{Corollary}
\newtheorem{lem}[thm]{Lemma}
\theoremstyle{definition}
\newtheorem{defn}[thm]{Definition}
\newtheorem{exam}[thm]{Example}

\theoremstyle{remark}
\newtheorem{rem}[thm]{Remark}

\newtheorem{nota}[thm]{Notation}
\newtheorem{hyp}[thm]{Hypothesis}

\renewcommand{\phi}{\varphi}
\renewcommand{\epsilon}{\varepsilon}
\renewcommand{\hom}{\mathrm{Hom}}
\newcommand{\cre}{\mathrm{cr}}
\newcommand{\f}{\mathcal{F}}
\newcommand{\cala}{\mathscr{A}}
\newcommand{\calb}{\mathscr{B}}
\newcommand{\calc}{\mathscr{C}}
\newcommand{\cald}{\mathscr{D}}
\newcommand{\cale}{\mathscr{E}}
\newcommand{\calm}{\mathscr{M}}
\newcommand{\fr}[1][r]{\zed^{\star #1}}
\newcommand{\fcatk}[1][\calc]{\f (#1; \kring)}
\newcommand{\fcat}[2]{\f (#1 ; #2)}
\newcommand{\fcatkfd}[1][\calc]{\f^{\mathsf{fd}}(#1; \kring)}
\newcommand{\foutk}[1][\calc]{\f^{\out} (#1; \kring)}
\newcommand{\fout}[2]{\f^{\out}(#1; #2)}
\newcommand{\fpoly}[2]{\f_{#1}(#2; \kring)}

\newcommand{\gr}{\mathbf{gr}}

\newcommand{\nat}{\mathbb{N}}
\newcommand{\ab}{\mathbf{ab}}
\newcommand{\zed}{\mathbb{Z}}
\newcommand{\A}{\mathfrak{a}}
\newcommand{\ak}{\A_\kring}
\newcommand{\ext}{\mathrm{Ext}}

\newcommand{\rat}{\mathbb{Q}}
\newcommand{\out}{\mathrm{Out}}
\newcommand{\op}{^\mathrm{op}}
\newcommand{\ob}{\mathsf{Ob}\hspace{3pt}}
\newcommand{\kring}{\mathbbm{k}}
\newcommand{\End}{\mathrm{End}}
\newcommand{\aut}{\mathrm{Aut}}
\newcommand{\ad}{\mathrm{ad}}
\newcommand{\inn}{\mathrm{Inn}}
\newcommand{\taubar}{\overline{\tau_\zed}}
\newcommand{\tauabbar}{\overline{\tauab}}
\newcommand{\tauab}{\tau^{\ab}_\zed}
\newcommand{\fin}{\mathbf{Fin}}
\newcommand{\sets}{\mathbf{Set}}
\newcommand{\setspt}{\sets_*}
\newcommand{\n}{\mathbf{n}}
\newcommand{\sset}{\Delta\op \sets}
\newcommand{\ssetpt}{\sset_*}
\newcommand{\filt}{\mathfrak{F}}

\newcommand{\dash}{\hspace{-1pt}-\hspace{-1pt}}
\newcommand{\modules}{\mathrm{mod}}

\newcommand{\inj}{\mathrm{Inj}}
\newcommand{\eval}{\mathrm{ev}}
\newcommand{\tor}{\mathrm{Tor}}
\newcommand{\loday}{\mathcal{L}}

\newcommand{\fb}{{\bm{\Sigma}}}
\newcommand{\smod}[1][\kring]{\fcatk[\fb]} 
\newcommand{\tenfb}{\odot}
\newcommand{\orient}{\mathrm{Or}}
\newcommand{\lie}{\mathbb{L}\mathrm{ie}}
\newcommand{\liemod}{\mathrm{Lie}}
\newcommand{\lieschur}{\mathbf{Lie}}

\newcommand{\gsym}{\mathbb{S}}
\newcommand{\schur}{\mathbf{S}}

\newcommand{\hq}{/\hspace{-3pt}/}

\newcommand{\talg}{\mathbb{T}}

\newcommand{\palg}{{\mathbb{P}^\fb}}
\newcommand{\pcoalg}{{\mathbb{P}^\fb_{\mathrm{coalg}}}}
\newcommand{\neck}{{P^\fb}}
\newcommand{\prodshuff}{\mu_{\mathrm{shuff}}}
\newcommand{\prodconcat}{\mu_{\mathrm{concat}}}
\newcommand{\coprodshuff}{\Delta_{\mathrm{shuff}}}
\newcommand{\coproddecon}{\Delta_{\mathrm{deconcat}}}
\newcommand{\fmod}{\mathbf{mod}_\kring}

\newcommand{\boldbeta}{\bm{\beta}}

\newcommand{\unit}{\mathbbm{1}}
\newcommand{\cmon}{\mathbf{CMon}}
\newcommand{\ccomon}{\mathbf{CComon}}

\newcommand{\lsoc}{\mathrm{length}_{\mathrm{soc}}}
\newcommand{\qhat}[1]{\widehat{q_{#1}^\gr}}
\newcommand{\soc}{\mathrm{soc}}
\renewcommand{\theta}{\vartheta}

\newcommand{\adbar}{\overline{\mathrm{ad}}}
\newcommand{\fexp}{\f^{\mathrm{exp}}}
\newcommand{\hopf}{\mathbf{Hopf}}
\newcommand{\hcocom}{{\hopf^{\mathrm{cocom}}}}
\newcommand{\hcom}{{\hopf^{\mathrm{com}}}}
\newcommand{\hbicom}{{\hopf^{\mathrm{bicom}}}}
\newcommand{\fmodq}{\mathbf{mod}_\rat}
\newcommand{\shuff}{\mathrm{shuff}}
\newcommand{\coshuff}{\mathrm{coshuff}}
\newcommand{\id}{\mathrm{Id}}
\newcommand{\fbcr}{\mathsf{grad}^\gr_\fb}
\newcommand{\foutQ}[1][\calc]{\f^{\out} (#1; \rat)}
\newcommand{\fpoutgrQ}[1][d]{{\f^{\out}_{ #1} (\gr ;  \rat)}}
\newcommand{\fcatQ}[1][\calc]{\f (#1; \rat)}
\newcommand{\fpolyQ}[2]{\f_{#1}(#2; \rat)}
\newcommand{\aQ}{\A_\rat}
\newcommand{\fpiQ}[1]{\fpolyQ{<\infty}{#1}}

\newcommand{\coad}{\mathrm{coad}}
\newcommand{\coadbar}{\overline{\coad}}
\newcommand{\coker}{\mathrm{coker\ }}
\newcommand{\fs}{\bm{\Omega}}
\newcommand{\cog}{\mathrm{coalg}}
\newcommand{\sym}{\mathfrak{S}}

\numberwithin{equation}{section}
\makenomenclature

\begin{document}

\begin{altabstract}
We study higher Hochschild homology evaluated on wedges of circles, viewed as a functor on the category of free groups. The main  results use   coefficients arising from square-zero extensions; this is motivated by work of Turchin and Willwacher in relation to hairy graph cohomology. 

The functorial point of view allows us to exploit  tools such as the theory of polynomial functors and exponential functors. We also introduce and make essential use of the category of outer functors, the full subcategory of functors on free groups on which inner automorphisms act trivially. 

We give a description of higher Hochschild homology in terms of intrinsically defined  polynomial outer functors;   we also obtain several explicit computations of these outer functors, working over a field of characteristic zero. In particular, higher Hochschild homology gives a natural source of non-trivial polynomial outer functors.\\
\end{altabstract}

\begin{abstract}
\textbf{(Homologie de Hochschild supérieure et foncteurs exponentiels)}
On étudie l’homologie de Hochschild supérieure évaluée sur des bouquets de cercles en tant que foncteur sur la catégorie des groupes libres. Les résultats principaux utilisent des coefficients provenant des extensions à carré nul. Ceci est motivé par le travail de Turchin et Willwacher en lien avec la cohomologie des graphes chevelus.

Le point de vue fonctoriel nous permet d’exploiter des outils tels que la théorie des foncteurs polynomiaux et celle des foncteurs exponentiels. On introduit et on utilise de manière essentielle la catégorie des outre-foncteurs, qui est la sous-catégorie pleine des foncteurs sur les groupes libres sur lesquels les automorphismes intérieurs agissent trivialement.

Nous donnons une description de l’homologie de Hochschild supérieure en termes d’outre-foncteurs polynomiaux définis intrinsèquement. Nous obtenons également plusieurs calculs explicites de ces outre-foncteurs lorsqu’on travaille sur un corps de caractéristique nulle. En particulier, l’homologie de Hochschild supérieure est une source naturelle d’outre-foncteurs polynomiaux non-triviaux.
\end{abstract}

\maketitle


\section{Introduction}
\label{sec:intro}

Higher Hochschild homology, defined by Pirashvili \cite{Phh}, is a generalization of the classical Hochschild homology for commutative rings.  It is a (non-additive) homology theory for pointed topological spaces. We denote by $HH_* (X; L)$ higher Hochschild homology of a  pointed space $X$ with coefficients in a  $\Gamma$-module $L$, where $\Gamma$ is the category of finite pointed sets. A fundamental example of a $\Gamma$-module is given by the Loday construction $\loday (A, M)$ for a commutative algebra $A$ and an $A$-module $M$. Taking $X$ to be the circle $S^1$, $HH_* (S^1; \loday (A,M))$ identifies with the classical Hochschild homology $HH_* (A,M)$.

In this paper, we focus on the case where $X$ is a finite wedge  of circles $(S^1)^{\bigvee r}$, for $r \in \nat$. This is inspired by the work of Turchin and Willwacher \cite{TW} in connection with the calculation of hairy graph homology \cite{TWtop}.  The homology $HH_* ((S^1)^{\bigvee r}; \loday (A,M))$ is particularly interesting because it can be seen as a natural representation of the group $\aut(\zed^{\star r})$, where $\zed^{\star r}$ is the free group. Turchin and Willwacher observed in \cite{TW} that when the coefficients are of the specific type $\loday (A,A)$, $HH_* ((S^1)^{\bigvee r}; \loday (A,A))$ is a representation of the outer automorphism group $\out (\zed^{\star r})$. Since the representation theory of  $\out (\zed^{\star r})$ is not yet well-understood, having concrete examples of such representations is particularly valuable. 

Instead of examining these representations individually for a fixed value of $r$, we introduce a novel approach by studying them through a functorial perspective. Specifically, by interpreting the finite wedge $(S^1)^{\bigvee r}$ as the classifying space $B \zed^{\star r}$ of the free group $\zed^{\star r}$, we obtain a functor from the category $\gr$ of finitely-generated free groups to the pointed homotopy category. This allows us to view higher Hochschild homology $\zed^{\star r} \mapsto HH_* (B \zed^{\star r}; L)$  as an  $\nat$-graded object of $\fcatQ[\gr]$, the category of functors from $\gr$  to $\rat$-vector spaces. 

Our key strategy in this paper is to take into account the full structure as a functor on $\gr$, rather than just the underlying family of representations of the automorphism groups $\aut (\zed^{\star r})$, for $r \in \nat$. Functors on the category of finitely-generated free groups $\gr$ are relatively rigid objects; this  simplifies their analysis compared to that of arbitrary  representations of  $\aut(\zed^{\star r})$. Considering functors on $\gr$ allows us to highlight properties and study relationships in a more organized and systematic manner. 

The first functorial tool used in this paper is the notion of  {\em polynomial functors}. This notion, initially introduced by Eilenberg and Mac Lane \cite{EM} for functors on categories of modules, has been extended to a wider context in \cite{HPV}, which includes functors on the category of finitely-generated free groups $\gr$. Polynomiality  gives  a  measure of the complexity of a functor, namely its polynomial degree.  A functor $M$ on $\gr$ being polynomial of degree $d$ roughly means that all values of $M$ are determined by the restrictions $M(\zed^{\star r})$ for $r \leq d$.

The significance of polynomiality for higher Hochschild homology  is exhibited by the following, proved as part of Theorem \ref{thm:Hodge_filt_groups}:

\begin{THM}
\label{THM:poly}
For a $\Gamma$-module $L$ and $d \in \nat$, the functor $HH_d (B(-); L)$ on $\gr$ has polynomial degree $d$.
\end{THM}

A polynomial functor comes  equipped with a natural polynomial filtration, the form of which is completely understood for functors on $\gr$ (see Proposition \ref{prop:subquotients}); this plays a crucial role.

In this paper we introduce a new  functorial tool,  the notion of an {\em outer functor} on free groups. The category of  outer functors on $\gr$ is the full subcategory $\foutQ[\gr] \subset \fcatQ[\gr]$ of functors on which the subgroup of inner automorphisms $\mathrm{Inn}(\zed^{\star r})\subseteq \aut(\zed^{\star r})$ acts trivially, for each $r \in \nat$. Higher Hochschild homology (for certain coefficients) gives a source of highly non-trivial  outer functors; for example, as a particular case of Proposition \ref{prop:HH_group_case} we obtain:

\begin{PROP} \label{Prop-Intro}
For a commutative $\rat$-algebra $A$ and $d \in \nat$,  $HH_d (B(-); \loday (A,A))$  is an outer functor on $\gr$.
\end{PROP}

Indeed, whenever the coefficients $L$ arise from a functor defined on {\em unpointed}  finite sets (for example, $L = \loday (A,A)$),  higher Hochschild homology $X \mapsto HH_* (X; L)$ is a functor on the unpointed homotopy category. In this situation, it follows that  the action of the automorphism group $\aut(\zed^{\star r}) $ on $HH_* (B \zed^{\star r}; L)$  factors  across the outer automorphism group $\out (\zed^{\star r})$, as observed by Turchin and Willwacher \cite{TW}. Proposition \ref{prop:HH_group_case} is the functorial analogue of this observation. 

As above, higher Hochschild homology gives examples of objects in $\foutQ[\gr]$, but this is not the only context in which such functors arise naturally (see the end of the introduction) and an in-depth study of them is of independent interest.

The category $\foutQ[\gr]$ of outer functors  is crucial in our work, as is the right adjoint 
\[
\omega :
\fcatQ[\gr]
\rightarrow 
\foutQ[\gr] 
\]
 to the inclusion $\foutQ[\gr] \subset \fcatQ[\gr]$. This arises naturally in the study of higher Hochschild homology, as explained  below. 
 
 The third functorial tool used in this paper is provided by {\em exponential functors}; i.e., symmetric monoidal functors on $\gr$. 
 We recall, in Theorem \ref{thm:expo_Hopf_general}, the equivalence of categories between the category of exponential functors on $\gr$, $\fexp (\gr ; \rat)$,  and the category of commutative Hopf algebras on $\rat$,  $ \hcom (\rat\dash \modules),$ induced by the functor 
  $\Psi :  \hcom (\rat\dash \modules) \rightarrow \fexp(\gr ; \rat)$ given on objects by $\Psi H (\zed ^{\star n}) = H^{\otimes n}$. (The result is more general: the category $\rat\dash\modules$ of $\rat$-vector spaces can be replaced by any suitable symmetric monoidal category.)
 
 \bigskip
To state our results on higher Hochschild homology, we must first introduce the appropriate coefficients. As in the work of Turchin and Willwacher \cite{TWtop}, we consider the square-zero extension $A_V:= \rat \oplus V$ of $\rat$ by a finite-dimensional $\rat$-vector space $V$. For $V=\rat$, this corresponds to the dual numbers $\rat [\epsilon]$. 

The main problem  that we address is:

\begin{PROB}
\label{Pb-intro}
Identify the following functors on $\gr$:
\begin{eqnarray*}
& & HH_* (B(-) ; \loday (A_V, \rat) ) \\
& & HH_* (B(-) ; \loday (A_V, A_V) ).
\end{eqnarray*}
\end{PROB}

A  key observation is that the association $V \mapsto A_V$ is a functor of $V \in \ob \fmodq$, where $\fmodq$ is the category of finite-dimensional $\rat$-vector spaces. This naturality is a crucial ingredient in studying these structures; it also allows us to bring into play the Schur correspondence between representations of the symmetric groups and functors on $\fmodq$. This leads us to study the functors $HH_* (B(-) ; \loday (A_V, \rat) )$ and $HH_* (B(-) ; \loday (A_V, A_V) )$ on $\gr$ naturally with respect to $V$. 

The functor $HH_* (B(-) ; \loday (A_V, \rat) )$ is exponential (see Proposition \ref{prop:shuffle_vee}) and we identify the corresponding commutative Hopf algebra as follows. Denoting by $sV$ the homological suspension of $V\in \ob \fmodq$,  we consider the tensor coalgebra $T_\mathrm{coalg}(sV)$ as a graded-commutative Hopf algebra with respect to the  shuffle product. This then yields the exponential functor $\Psi(T_\mathrm{coalg}(sV))$. In Theorem \ref{thm:hhh_Loday_AV_k}, we obtain the following description:

\begin{THM}
For  $V \in \ob \fmodq$, there is a natural isomorphism of functors with values in graded-commutative $\rat$-algebras:
\[
HH_* (B (-); \loday (A_{V} , \rat) ) 
\cong 
\Psi (T_\mathrm{coalg}(sV)),
\]
where $HH_* (B (-); \loday (A_{V} , \rat) ) $ is equipped with the shuffle product. This is natural with respect to $V$.
\end{THM}

The functor $HH_* (B(-) ; \loday (A_V, A_V) )$ is then studied by using the long exact sequence associated to the short exact sequence of $\Gamma$-module coefficients:
\[
0 
\rightarrow 
\loday (A_V, \rat) \otimes V 
\rightarrow 
\loday (A_V, A_V) 
\rightarrow 
\loday (A_V, \rat)
\rightarrow 
0,
\]
in which the surjection $\loday (A_V, A_V) 
\rightarrow 
\loday (A_V, \rat)$ is induced by the augmentation $A_V \twoheadrightarrow \rat$.
The connecting morphism of this long exact sequence is identified in Proposition \ref{prop:connecting} as being the coadjoint coaction
\[
\coadbar : \Psi (T_\mathrm{coalg}(sV)) \rightarrow sV \otimes \Psi ( T_\mathrm{coalg}(sV))
\]
which fits into the following exact sequence:
\begin{equation}
\label{es}
0
\rightarrow 
\omega 
 \Psi T_\cog (sV)
 \rightarrow 
 \Psi T_\cog (sV) \stackrel{\coadbar}{\longrightarrow}\Psi T_\cog (sV) \otimes sV
 \rightarrow 
 \coker (\coadbar)
 \rightarrow 
 0.
\end{equation}

 This leads to Theorem \ref{thm:HH_lodayAVAV_car0}, in which the $\nat$ accounts for the homological grading:

\begin{THM} 
\label{THM4-intro}
For  $V \in \ob \fmodq$, there is 
a natural isomorphism in $\fcatQ[\nat \times \gr]$:
\begin{eqnarray*}
HH_* \big(B (-); \loday (A_V, A_V))
\cong 
\omega \Psi (T_{\mathrm{coalg}}(sV) )
\oplus 
\coker( \coadbar_{T_{\mathrm{coalg}}(sV)})[-1],
\end{eqnarray*}
where $[-1]$ denotes the shift in homological degree. These identifications are natural with respect to $V$.
\end{THM}

It follows that, to understand these higher Hochschild homology functors, it suffices to understand the functor $V \mapsto  \Psi T_\cog (sV)$ and the outer functor $V \mapsto \omega \Psi T_\cog (sV)$, since the remaining term can be derived from these using the exact sequence (\ref{es}).

For $V=\rat$, Theorem \ref{THM4-intro} extends and gives a functorial interpretation of results of \cite{TW} for the dual numbers (cf.  \cite[Theorem 1]{TW} and the direct sum decomposition given in \cite[Section 2.2]{TW}).  See Remark \ref{rem:TW_question} for more details.

\subsection{Combinatorial coefficients}
To prove Theorem \ref{THM4-intro}, we use the functoriality on $V$ evoked above. The Schur correspondence gives
\begin{eqnarray*}
\inj ^\Gamma & \leftrightarrow & \loday (A_V; \rat) \\
\theta^* \inj ^\fin & \leftrightarrow & \loday (A_V, A_V) 
\end{eqnarray*}
where $\inj ^\Gamma$ and $\theta^* \inj ^\fin$ are $\Gamma$-modules introduced in Section \ref{sec:loday}, 
 taking values in $\fb$-modules, where $\fb$ is the category of finite sets and bijections (which serves to encode the family of symmetric groups $\{ \sym_d \ | \ d \in \nat \}$).

Problem \ref{Pb-intro} is then equivalent to that of calculating the following functors on $\gr$ 
 \begin{eqnarray*}
& & HH_* (B(-) ; \inj ^\Gamma ) \\
& & HH_* (B(-) ; \theta^* \inj ^\fin )
\end{eqnarray*}
taking values in $\fb$-modules.

This reformulation  has two advantages.
\begin{enumerate}
\item Since the functors $HH_* (B(-) ; \inj ^\Gamma ) $ and $HH_* (B(-) ; \theta^* \inj ^\fin )$ take values in $\fb$-modules, this allows us to consider the isotypical components directly. Explicitly, for a partition $\lambda\vdash n$, we consider the associated simple $\rat [\sym_n]$-module, $S_\lambda$, viewed as an $\fb$-module supported on $\mathbf{n}:= \{1, \ldots, n \}$. Applying the functor $- \otimes _\fb S_\lambda$ gives the respective isotypical components indexed by $\lambda$:
\begin{eqnarray*}
&& HH_* \big(B (-);\inj^\Gamma)\otimes_\fb S_\lambda
\\
&& HH_* \big(B (-);\theta^* \inj^\fin) \otimes_\fb S_\lambda.
\end{eqnarray*}
These  are identified in Corollary \ref{cor:isotypical_components}.
\item The functors $HH_* (B(-) ; \inj ^\Gamma ) $ and $HH_* (B(-) ; \theta^* \inj ^\fin )$ on $\gr$ are equivalent to \textit{intrinsic} polynomial functors obtained from the functors $\beta_d$ defined below.

To explain this, we first need to give more details on polynomial functors, notably introducing the functors $\beta_d$.
\end{enumerate}

\subsection{The functors $\beta_d$}
The full subcategory of functors of polynomial degree $d$ is written $\f_d (\gr; \rat)$. Using  abelianization,  one has the functor $\aQ \in \ob \f (\gr; \rat)$ given by $G \mapsto G/[G,G]\otimes \rat$. This is the basic example of a polynomial functor of degree one. More generally, for $d \in \nat$, the $d$-fold tensor product $\aQ^{\otimes d}$ is an example of a   polynomial functor of degree $d$ in  $\f (\gr; \rat)$. 

The theory of polynomial functors provides the cross-effect functor $\cre_d$, which is an exact functor $\cre_d: \f_d (\gr; \rat) \rightarrow \rat [\sym_d] \dash\modules$. This admits a left adjoint $\alpha_d$ that has a simple description: for a $\sym_d$-module $M$,  
\[
\alpha _d M = (\aQ^{\otimes d}) \otimes_{\sym_d} M,
\]
where $\sym_d$ acts on $\aQ^{\otimes d}$ by place permutations. 

These allow us to describe the simple polynomial functors on $\gr$. Namely, for any such simple functor $F$, there exists a unique pair $(d, \lambda)$, where $\lambda$ is a partition of $d$, such that $F \cong \alpha_d (S_\lambda)$. Here $d$ is the (precise) polynomial degree of $F$.

Also of importance here is the {\em right} adjoint $\beta_d : \rat [\sym_d ]\dash\modules \rightarrow \f_d (\gr; \rat)$  to the cross-effect functor; this arises explicitly in our study of higher Hochschild homology. Unfortunately, unlike the case of $\alpha_d$, the functor $\beta_d$ does not admit an elementary description. 
 
In Theorem \ref{thm:beta_hhh} we describe  the  isotypical components of $HH_* \big(B (-);\theta^* \inj^\fin)$ in terms of the functors $\beta_d$ and their associated outer functors, $\omega \beta_d $:

\begin{THM}
\label{Thm5-Intro}
For $d \in \nat^*$ and $\lambda \vdash d$, there  are isomorphisms in $\foutQ[\gr]$:
\[
HH_* \big(B (-);\theta^* \inj^\fin) \otimes_\fb S_\lambda
\cong 
\left\{
\begin{array}{ll}
\omega \beta_d S_{\lambda^\dagger} & *=d \\
\coker (\coadbar_{\lambda^\dagger}) & *= d-1\\
0  & \mbox{otherwise,}
\end{array}
\right.
\]
where $\lambda^\dagger$ is the conjugate of the partition $\lambda$ and $\coadbar_{\lambda^\dagger}$ is shorthand for $\coadbar \otimes_\fb S_{\lambda^\dagger}$. 
\end{THM}

This Theorem is one of our motivations for studying the functors $\beta_d$ and  $\omega \beta_d$; as we will see below, these functors also have intrinsic interest.

First, we note that the functors $\beta_d$  are highly non-trivial. For example, taking $d=2$ and  the signature representation $\mathrm{sgn}_2$ of $\sym_2$, $\beta_2 \mathrm{sgn}_2$ occurs in a non-split short exact sequence:
\[
0 
\rightarrow 
\Lambda^2 \circ \aQ 
\rightarrow 
\beta_2 \mathrm{sgn}_2
\rightarrow 
\aQ
\rightarrow 
0,
\]
where $\Lambda^2 \circ \aQ$ identifies with $\alpha_2 \mathrm{sgn}_2 $. This shows that $\f_2 (\gr; \rat)$ is not semi-simple. However, $\beta_2 \mathrm{sgn}_2$ is not an outer functor, and one has $\omega \beta_2 \mathrm{sgn}_2 \cong \alpha_2 \mathrm{sgn}_2 \cong \Lambda^2 \circ \aQ  $.

The significance of the functors $\beta_d$ is established by the following (see Corollary \ref{cor:poly_enough_injectives} and Corollary \ref{cor:omega_boldbeta_inj_cogen}), where $\f_{< \infty} (\gr; \rat)=\bigcup_{d \in \nat} \f _d (\gr; \rat)$  is the full  category of all polynomial functors and $\fpoutgrQ[<\infty]$ the corresponding category of  polynomial outer functors:

\begin{THM}
The families $\{ \beta_d \rat [\sym_d ] \ | \ d \in \nat \} $ and 
$ \{ \omega \beta_d \rat [\sym_d]  \ | \ d \in \nat \}$ 
are  sets of injective cogenerators of $
\fpiQ{\gr}$ and $\fpoutgrQ[<\infty]$ respectively.
\end{THM}

One of the keys to proving Theorem \ref{Thm5-Intro} is to work {\em globally}, by  assembling the family of functors $\{ \beta_d \rat [\sym_d] \ | \ d \in \nat \} $ to give the functor $\boldbeta$ from $\gr$ to $\f (\fb; \rat)$, the category of $\fb$-modules (see Section \ref{subsect:boldbeta}).  This approach also applies to describe the functors $\omega \beta_d$, which assemble to the functor $\omega \boldbeta$ with values in $\fb$-modules.

\subsection{Explicit computations and application}
Motivated by Theorem \ref{Thm5-Intro}, we give the complete structure of the functors  $\beta_d S_{\lambda}$ and  $\omega \beta_d S_{\lambda}$ for several explicit partitions $\lambda$:  
\begin{itemize}
\item for $\lambda = ((n-1),1)$ the result is given in Section \ref{subsect:(n1)_omega};
\item for $\lambda=(1^n)$ the result is given in Section \ref{subsect:omega_beta_1^n};
\item for $\lambda=(211)$ the result is given in Example \ref{exam:S211}.
\end{itemize}
(See Section \ref{sec:omega_psi} for further computations.)
More precisely we obtain the socle filtration of these functors. We stress that the category $
\fpiQ{\gr}$ is not semi-simple and, apart from exceptional cases, the functors $\beta_d S_\lambda$ are not semi-simple, so this socle filtration is non-trivial.

Our strategy to obtain the complete structure of the functors  $\beta_d S_{\lambda}$  has the following steps:
 \begin{enumerate}
\item 
We first identify the composition factors of the functor $\beta_d S_{\lambda}$ in Corollary \ref{cor:Grothendieck_ring_Psi_pcoalg}. As noted above, since $\beta_dS_\lambda$ is usually not semi-simple, this is not sufficient to  describe its structure.
\item  
In Proposition \ref{prop:beta_socle}, we identify the socle of the functor $ \beta_d M$:
\begin{PROP}
For $M \in \ob \rat[\sym_d]\dash\modules$, the functor $\alpha_d M $ is the socle of $\beta_d M$.
\end{PROP}

\item 
To analyse the socle filtration, we use the main result of \cite{V_ext} which gives, in particular,  that
$\ext^1_{\f (\gr; \rat)} (\A^{\otimes s}, \A^{\otimes t})=0$ if $t-s \neq 1$ (see Corollary \ref{cor:d-1_beta_S_mu} for a more precise statement). This allows us to compare the socle filtration and the polynomial filtration (see  Corollary \ref{exam:beta_M_socle_filt}):

\begin{PROP}
For any $d \in \nat$ and $M \neq 0$ a  $\rat [\sym_d]$-module, up to reindexing, the socle filtration of $\beta_d M$ coincides with the polynomial filtration of $\beta_d M$.
\end{PROP}
\item 
Where possible, we seek to identify the extensions occurring between the layers of the socle filtration. This is feasible in some cases, notably when the functors are multiplicity free (i.e., each composition factor has multiplicity at most one) or the structure is particularly simple. 

For instance, for $n \geq 1$ we show that the functors $\omega \beta_{(1^n)}$ are uniserial (see Theorem \ref{thm:omega_omega1_sign}) and that it has socle length $[\frac{n+1}{2}]$.
\end{enumerate}

Then, to obtain the structure of  the functors $\omega \beta_d S_{\lambda}$, we apply  Proposition \ref{prop:coadbar_es_isotypical}:

\begin{PROP}
For $\lambda \vdash d$, we have the exact sequence:
\begin{eqnarray*}
0
\rightarrow 
\omega \beta_d S_\lambda 
\rightarrow 
\beta_d S_\lambda
\stackrel{\coadbar_\lambda}{\longrightarrow} 
\bigoplus_{\substack{\mu \preceq \lambda \\ |\mu | = d-1} }
\beta_{d-1} S_\mu
\rightarrow 
\coker \big (\coadbar_\lambda \big)
\rightarrow 
0.
\end{eqnarray*}
\end{PROP}

Computing $\omega \beta_d S_\lambda $ is thus equivalent to determining the kernel of the map $\coadbar_\lambda$. This is a difficult question in general.  However, in Theorem \ref{thm:coadbar} we show that the morphism $\coadbar_\lambda$ is almost always non-zero. More precisely we show: 

\begin{THM}
For $\lambda \vdash d$, the morphism 
$$
\beta_d S_\lambda
\stackrel{\coadbar_\lambda}{\longrightarrow} 
\bigoplus_{\substack{\mu \preceq \lambda \\ |\mu | = d-1} }
\beta_{d-1} S_\mu
$$
 is zero if and only if $\lambda = (d)$.
 \end{THM}
 
Moreover, Theorem \ref{thm:coadbar} explains the behaviour of $\coadbar_\lambda$ on the layer of polynomial degree $d-1$. In the examples given above, these results are enough to deduce the  structure of the functors $\omega \beta_d S_\lambda $.

Beyond explicit computations, we give below another consequence of the understanding of $HH_* \big(B (-); \loday (A_V, A_V))$ as a functor on $\gr$. By Proposition \ref{prop:abelianization_dual_numbers}, we obtain the following result, where $\A: \gr \to \ab$ is the functor with values in the category of finitely-generated free abelian groups provided by the abelianization:

\begin{PROP}
For $n \in \nat^*$, the kernel of the natural transformation of functors on $\gr$
\[
HH_n (B \zed^{\star r}; \loday (\rat [\epsilon], \rat [\epsilon]))
\rightarrow 
HH_n (B \A (\zed^{\star r}); \loday (\rat [\epsilon], \rat [\epsilon]))
\]
is non-trivial. 
\end{PROP}

The proof of this result relies upon  our calculations and the non-semisimplicity of the category of polynomial functors on $\gr$, giving an illustration of the interest of taking into account the full functorial structure.  

As a Corollary, we recover the Theorem  of Dundas and Tenti \cite{DT} that higher Hochschild homology $X \mapsto HH_* (X; \loday (\rat [\epsilon], \rat [\epsilon]))$ is not a stable invariant of $X$ (see Remark \ref{rem:Dundas_Tenti}).

\subsection{The Hodge filtration}
In Section \ref{sect:hodge_filt},  we explain how Pirashvili's Hodge filtration of higher Hochschild homology \cite{Phh} is related to the functorial picture  considered here. 

For $X$ a pointed simplicial set and $L\in \ob \fcatQ[\Gamma]$, Pirashvili \cite[Theorem 2.4]{Phh} constructed the Hodge filtration spectral sequence 
\begin{eqnarray*}
E_{pq}^2 
= 
\tor^\Gamma _p (\mathcal{T}_q H_* (X; \rat)  , L) 
\Rightarrow 
HH_* (X;L). 
\end{eqnarray*}
The associated filtration of $HH_* (X;L)$ is Pirashvili's Hodge filtration.

We take $X$ to be the classifying space $B \zed^{\star r}$, considered as a functor of the free group $\zed^{\star r} $.
Then, for any $\Gamma$-module, Pirashvili's Hodge filtration for $HH_* (B(-); L)$ gives an increasing filtration. In a fixed homological degree $d$, we reindex this as:
\[
\ldots \subset  
 \filt^{q+1} HH_d (B(-); L)
\subset  
 \filt^q HH_d (B(-); L)
 \subset  
 \ldots 
 \subset
 \filt^0 HH_d (B(-); L)
\]
where $
\filt^0 HH_d (B(-); L)
 =HH_d(B(-); L)$ and  $\filt^{d+1} HH_d (B(-); L)=0$.

The filtration quotients  are described explicitly in Theorem \ref{thm:Hodge_filt_groups}, which shows that (up to indexing) the Hodge filtration coincides with the polynomial filtration. In particular, this yields Theorem \ref{THM:poly} above.

These results also give access to calculations of $\ext$ groups over $\f(\Gamma; \rat)$ (see Theorem \ref{thm:Koszul_duality}).

\subsection{Cohomological properties of functors}
This paper contains a number of intrinsic results on functors and outer functors on $\gr$ which should have further applications.

By construction, there are inclusions of abelian categories:
\[
\fcatQ[\ab]
\subsetneq
\foutQ[\gr]
\subsetneq
\fcatQ[\gr]
\] 
which restrict to the categories of polynomial functors:
 \[
\fpolyQ{d}{\ab}
\subseteq
\fpoutgrQ[d]
\subseteq
\fpolyQ{d}{\gr}
\] 
 for $d\in \nat \cup \{< \infty\}$.  In each case, for $d$ finite, $\mathcal{F}_d$ denotes the  full subcategory of polynomial functors of degree at most $d$;   each of these categories has enough projectives and enough injectives. The category $\mathcal{F}_{<\infty}$ denotes  that of all polynomial functors, $\bigcup _{d\in \nat} \mathcal{F}_d$.

 As a consequence of Theorem \ref{thm:decomp_Psi_pcoalg} one gets, in Corollary \ref{cor:d-1_beta_S_mu}, the  explicit calculation of extensions between simple objects of the category $\fcatQ[\gr]$. 

\begin{THM}
\label{Intro-THM11}
For $\mu \vdash d$, 
\[
\dim 
\ext^1_{\fcatQ[\gr]}
(\alpha_{|\rho|} S_\rho, \alpha_d S_\mu)
= 
\left\{
\begin{array}{ll}
0& |\rho |\neq d-1\\
\sum_{\nu \vdash d-2 } c^\rho_{\nu, 1}c^{\mu}_{\nu,11}
&|\rho |= d -1.
\end{array}
\right.
\]
where $c^*_{*,*}$ denote the Littlewood-Richardson coefficients.
\end{THM}

Several calculations of $\ext$ groups were obtained in \cite{V_ext}, for example when $\rho$ and $\nu$ are the respective signature representations; Theorem \ref{Intro-THM11} gives a complete description of $\ext^1$.

The category $\fpiQ{\ab}$ is  semi-simple, whereas, for $d>1$, $\fpolyQ{d}{\gr}$ has global dimension $d-1$ by results of  \cite{DPV}.  For polynomial outer functors,  we show that the category $\f_d^{\mathrm{Out}} (\gr; \rat)$ differs essentially from $\fpolyQ{d}{\ab}$ and is closer to the category $\fpolyQ{d}{\gr}$. In particular, as part of  Corollary \ref{cor:fpout_semisimple_dleq2}, we show:

\begin{THM}
The category  $\f_d^{\mathrm{Out}} (\gr; \rat)$ is semi-simple if and only if $d \leq 2$. 
\end{THM}

This contrasts with the fact that $\fpolyQ{d}{\ab}$ is semi-simple. Moreover, we show in Corollary \ref{cor:gl_dim_fpoutgr}:

\begin{THM}
For $0<d \in \nat$, the category $\f_d^{\mathrm{Out}} (\gr; \rat)$ has global dimension at most $d-2$, with equality in the cases $d\in\{ 2,3\}$.

 In particular, for $d \geq 2$, 
 \[
 \mathrm{gl.dim} \f_d^{\mathrm{Out}} (\gr; \rat)< \mathrm{gl.dim} \fpolyQ{d}{\gr} = d-1.
 \]
and  $\mathrm{gl.dim} \f_d^{\mathrm{Out}} (\gr; \rat)>0$ for $d >2$.
\end{THM}

In  Corollary \ref{cor:ext1_gr_fout} we describe  $\ext^1_{\foutQ[\gr]}$ between simple functors:

\begin{THM}
\label{COR:9}
For $\lambda \vdash d$, where $d>0$, $\ext^1_{\foutQ[\gr]}
(\alpha_{|\rho|} S_\rho, \alpha_d  S_\lambda)=0$ if $|\rho | \neq d-1$. 

If $\rho \vdash d-1$ and $\lambda \neq (d)$:
\[
\dim 
\ext^1_{\foutQ[\gr]}
(\alpha_{d-1} S_\rho, \alpha_d  S_\lambda)
= 
\left\{
\begin{array}{ll}
\sum_{\nu \vdash d-2 } c^\rho_{\nu, 1}c^{\lambda}_{\nu,11} 
&\rho \not \preceq \lambda 
\\ 
\big(\sum_{\nu \vdash d-2 } c^\rho_{\nu, 1}c^{\lambda}_{\nu,11} \big) - 1 
& \rho \preceq \lambda. 
\end{array}
\right.
\]
\end{THM}

For a simple representation $S_\lambda$ of $\sym_{d}$, the polynomial functor $\beta_d S_\lambda$ is injective in $\fpiQ{\gr}$, by Theorem \ref{thm:injectivity_beta}. This leads to the following (see Corollary \ref{cor:inj_envelope_omega}), which gives an intrinsic characterization of $\omega \beta_d S_\lambda$:

\begin{THM}
\label{THM:10}
For $\lambda \vdash d$, $\omega \beta_d S_\lambda$ is the injective envelope in $\f_{<\infty}^{\mathrm{Out}} (\gr; \rat)$ of the simple functor $\alpha_d S_\lambda$. 
\end{THM}

This gives a conceptual characterization of $\omega \beta_d S_\lambda$ in  $\f_{<\infty}^{\mathrm{Out}} (\gr; \rat)$.
 The difference between $\f_{<\infty}^{\mathrm{Out}} (\gr; \rat)$ and $\f_{<\infty} (\gr; \rat)$ is also underlined by:

\begin{THM}
For $\lambda \vdash d$,  $\beta_{d} S_\lambda$ is an object of $\foutQ[\gr]$ if and only if $\lambda = (d)$.

For $\lambda = (d)$, one has the identifications:
\begin{eqnarray*}
\omega \beta_d S_{(d)} & \cong & \alpha_d S_{(d)} \\
\coker (\coadbar_{(d)}) & \cong &  \alpha_{d-1} S_{(d-1)},
\end{eqnarray*}
where $S_{(d-1)}$ is taken to be zero for $d=0$. In particular, $\alpha_d S_{(d)}$ is injective in $\fpoutgrQ[<\infty]$.
\end{THM}

We observe that there are other examples of injectives in $\fpoutgrQ[<\infty]$ that are simple (hence arise from $\fpiQ{\ab}$). For example, by Remark \ref{rem-12.22} we have:
$$\omega \beta_n S_{((n-1),1)}=\alpha_n S_{((n-1)1)}$$
which is a simple functor, that is injective in $\fpoutgrQ[<\infty]$  (but not in $\fpiQ{\gr}
$).

\subsection{Other references}
Since the first version of this paper was uploaded to {\tt arXiv} in 2018, several related works have appeared. For example:
\begin{enumerate}
\item 
In \cite{2021arXiv211001934P}, the first author relates polynomial functors on $\gr$ to representations of the PROP associated to the Lie operad. In \cite{MR4696223}, he then identifies outer polynomial functors within this framework.  
\item
In \cite{MR4613613,MR4699865}, Katada studies certain functors on $\gr\op$ associated with the category of Jacobi diagrams in handlebodies introduced by Habiro and Massuyeau in \cite{Habiro-Massuyeau}. She proves that these are outer functors and are polynomial. In \cite{2022arXiv220210907V}, the second author extends this study to a family of functors indexed by $n \in \mathbb{N}$; the functors considered by Katada correspond to the case $n=0$.
\item 
In 
\cite{2022arXiv220212494G}, Gadish and Hainaut study the compactly supported rational cohomology of configuration spaces of
points on wedges of spheres. They relate this to higher Hochschild homology and show that their calculations for higher dimension spheres are related to those appearing in our study of outer functors on free groups. 
\item 
In \cite{doi:10.1080/10586458.2023.2209749}, Bibby, Chan, Gadish and Yun study the top weight rational cohomology of the moduli space of curves $\mathcal{M}_{2,n}$. They reduce to calculating the homology of compactified configuration spaces of graphs, 
 which in turn relates their calculations to \cite{2022arXiv220212494G} and the higher Hochschild homology considered here. 
 
Their subsequent work \cite{2023arXiv230701960B} studies the weight zero, compactly supported rational cohomology of the moduli space of curves $\mathcal{M}_{g,n}$ for $g \geq 2$,  using a spectral sequence. As an application, the authors obtain new calculations for $g=3$ and small $n$.
\item 
In \cite{2023arXiv231116881H}, Hainaut studies  the $\ext$-groups $\ext^*_{\fpoutgrQ[< \infty]} (\A^{\otimes s}, \A^{\otimes t})$ calculated in the category of polynomial outer functors. Using calculations from  \cite{2022arXiv220212494G}, he has shown that these groups can be non-zero for $*\neq t-s$. This contrasts with Vespa's result \cite{V_ext} calculating 
$\ext^*_{\f (\gr; \rat)} (\A^{\otimes s}, \A^{\otimes t})$, which is isomorphic to $\ext^*_{\f_{< \infty} (\gr; \rat)} (\A^{\otimes s}, \A^{\otimes t})$ by \cite{DPV}. Calculation of $\ext^1_{\fpoutgrQ[< \infty]} (\A^{\otimes s}, \A^{\otimes s+1})$ can be deduced from the computations obtained in \cite{PV2}.
\end{enumerate}

\bigskip
\noindent 
{\bf Organization of the paper:}

The paper is split into parts corresponding to the major themes:

\begin{enumerate}
\item 
Part \ref{part:background} presents background on functors on $\gr$ and reviews the theory of exponential functors. 
\item 
Part \ref{part:poly} develops the theory of polynomial functors on $\gr$ together with the tools that are required in analysing this, notably the Schur correspondence; the main objective is to describe the functors $\beta_d$ by studying the assembled functor $\boldbeta$. 
\item 
Part \ref{part:outer} introduces outer functors on free groups together with the functor $\omega$. This is one of the innovations of the paper and is of independent interest.
\item 
Part \ref{part:psi_omega_psi} shows how to calculate the functors $\boldbeta$ and $\omega \boldbeta$; this is illustrated by concrete examples.
\item 
Part \ref{part:HHH} introduces higher Hochschild homology and the associated functors on free groups. It is then shown how these functors are described using the functors $\boldbeta$ and $\omega \boldbeta$.
\item 
Part \ref{part:app} provides appendices reviewing basic material on functor categories and on the representation theory of the symmetric groups.
\end{enumerate}

\subsection{Notation and Conventions}

\begin{enumerate}
\item
$\nat$ denotes the non-negative integers and $\nat^*$ the positive integers;
\item 
when used as a category, $\nat$ is given the discrete structure; 
\item 
$\kring$  denotes a  commutative ring, which will always be assumed to be unital; 
\item 
the category of $\kring$-modules is sometimes written $\kring\dash\modules$ and the full subcategory of free, finite-rank $\kring$-modules is denoted $\fmod$;
\item 
if $\kring$ is a field, $^\sharp$ denotes vector space duality;
\item
for $\calc$ a small category, $\fcatk[\calc]$ denotes the category of functors from $\calc$ to $\kring$-modules;
\item
$\gr$ is the category of finitely-generated free groups and $\ab$ that of finitely-generated free abelian groups;
\item 
the symmetric group on $n$ letters is denoted $\sym_n$ and its signature representation by $\mathrm{sgn}_n$;   
\item 
$\sets$ is the category of sets and $\setspt$ that of pointed sets;  
\item 
in diagrams representing adjunctions, the convention used is that, for a pair of functors represented by  $\leftrightarrows$ or $\rightleftarrows$, the top arrow is the left adjoint. 
\end{enumerate}

\begin{nota}
\label{nota:groth_gp}
If the isomorphism classes of simple objects of an abelian category $\calc$ form a set, the Grothendieck group $G_0 (\calc)$ is the free abelian group generated by this set. 
If $S$ is a simple object of $\calc$, the associated element of $G_0 (\calc)$ is written $[S]$.

More generally, for $F$ a finite object of $\calc$ (that is, $F$ has a finite composition series), 
$[F] \in G_0 (\calc)$ is defined recursively by the additivity property for short exact sequences of finite  objects in $\calc$: 
\[ 
0
\rightarrow 
F_1 
\rightarrow 
F_2 
\rightarrow 
F_3
\rightarrow 
0
\]
$[F_2] = [F_1]+ [F_3]$ in $G_0 (\calc)$.
\end{nota}

\subsection{Acknowledgements}

\ 
\medskip

The authors are grateful to  Louis Hainaut for his interest and for his comments on a previous version of the paper.

\medskip
{\bf Financial support:}
This work was partially supported by the ANR Project {\em ChroK}, {\tt ANR-16-CE40-0003}. The first author was also supported by the project {\em Nouvelle Équipe},  convention {\tt 2013-10203/10204} between the Région des Pays de la Loire and the Université d'Angers. The second author was also suppported by the ANR projects {\em AlMaRe} {\tt ANR-19-CE40-0001-01},  {\em HighAGT} {\tt ANR-20-CE40-0016}, and {\em  SHoCoS} {\tt ANR-22-CE40-0008}.

\part{Background}
\label{part:background}
The aim of this part is to fix notation and recall results for later reference. We introduce the basic notions used in the paper, such as the categories $\gr$ and $\ab$ and functors on these categories.  We also introduce exponential functors, which are an important tool, and recall their relationship with Hopf algebras. 

This part contains no new results and the reader familiar with them may  skip it.

\section{The categories $\gr$ and $\ab$ and  functors on them}
\label{sect:fgr}

This section first introduces the categories $\gr$ and $\ab$ of finitely-generated free (respectively free abelian) groups before considering the basic structure for functors on groups that is required throughout the paper.

\subsection{Introducing the categories $\gr$ and $\ab$}
\label{subsect:gr_ab}

\begin{nota}
\label{nota:ab_gr}
\nomenclature{$\gr$}{finitely-generated free groups\nomrefpage}
\nomenclature{$\ab$}{finitely-generated free abelian groups\nomrefpage}
 Let   $\gr$ denote the category of 
finitely-generated free groups and  $\ab $ that of 
finitely-generated free abelian groups, considered as full subcategories of the category of groups. (The categories $\gr$ and $\ab$ are essentially small, with  skeleta  with objects $\zed ^{\star n}$ and $\zed^n$, $n \in \nat$, respectively.)
\end{nota}

\nomenclature{$\A $}{abelianization functor $\gr\rightarrow \ab$\nomrefpage}
The free product of groups $\star$ equips $\gr$ with the symmetric monoidal structure $(\gr, \star, \{e\})$. Likewise, the direct product $\oplus$ of abelian groups equips $\ab$ with the symmetric monoidal structure $(\ab, \oplus, 0)$. 
 Abelianization of groups, $G \mapsto G/[G,G]$ induces a functor $\A : \gr \rightarrow \ab$.
With respect to the above, $ \A : \gr \rightarrow \ab$ is symmetric monoidal. 

For usage in Section \ref{sec:expo}, we analyse morphisms in  $\gr$ by exploiting the symmetric monoidal structure. 

\begin{nota}
\label{nota:generators_Z}
 Write $x$ for a chosen generator of the group $\zed$ and $x_i$ ($1 \leq i \leq 
n$) for the corresponding 
 generators of $\zed^{\star n}$, $n \in \nat$.
\end{nota} 
 
Since $\star$ is the coproduct in $\gr$ (respectively $\oplus$ for $\ab$), one can consider cogroup objects in $\gr$ (resp. $\ab$).  The following is well-known (writing the group structure in $\ab$ additively): 
 
\begin{lem}
\label{lem:zed_cogroup}
\ 
\begin{enumerate}
\item 
The group $\zed$ is a cogroup object in $\gr$, with 
 counit $\epsilon : \zed \rightarrow \{e\}$,  coproduct  $\Delta : \zed \rightarrow 
\zed \star \zed$  given by $x \mapsto x_1 x_2$, and inverse $\chi : \zed 
\rightarrow \zed$ given by  $x \mapsto x^{-1}$. 
 \item 
 The group $\zed$ is an abelian cogroup object in $\ab$, with 
 counit $\epsilon_\ab : \zed \rightarrow 0$,  coproduct  $\Delta_\ab : \zed \rightarrow 
\zed \oplus \zed$ given by $x \mapsto (x,x)$, and inverse $\chi_\ab : \zed 
\rightarrow \zed$ given by 
 $x \mapsto -x$. 
 \end{enumerate}
\end{lem}

\begin{rem}
The cogroup structure on $\zed$ endows $\hom _\gr (\zed, G)$ with a natural group structure, for all $G \in \ob \gr$, by definition of a cogroup object. This coincides with the group structure of $G$, via the natural isomorphism $\hom_\gr (\zed, G) \cong G$.
\end{rem}

Restricting to the skeleton of $\gr$ with objects $\zed^{\star n}$, for $n \in \nat$, the category $\gr$ can be viewed as a PROP with respect to $\star$. This PROP structure has been exploited in \cite{MR3765469}, \cite{P},  \cite{Hab}, for example.  In particular, as a PROP, one has the following  generators for the morphisms of $\gr$:

\begin{prop}
\label{prop:generate_mor_gr}
Considered as a PROP, the morphisms of $\gr$ are generated by the following:
\begin{enumerate}
\item 
the unit $\eta :\{e \} \rightarrow \zed$; 
\item 
the fold map $\nabla : \zed \star \zed \rightarrow \zed$; 
\item 
the cogroup structure morphisms $\epsilon : \zed \rightarrow \{e\}$,   $\Delta : \zed \rightarrow \zed \star \zed$,  and  $\chi : \zed 
\rightarrow \zed$.
\end{enumerate}
\end{prop}

\begin{rem}
\ 
\begin{enumerate}
\item 
The PROP structure includes the information from the symmetric monoidal structure, in particular the action, for each $n \in \nat$, of $\sym_n$ on $\zed^{\star n}$. 
\item 
There is an analogous statement for $\ab$. 
\end{enumerate}
\end{rem}

\subsection{Functors on $\gr$ and $\ab$}
For background on functor categories, the reader is referred to Appendix \ref{sec:background}. In particular,  for $\calc$ an essentially small category, $\f(\calc; \kring)$ denotes the category of functors from $\calc$ to $\kring$-modules, where 
 $\kring$ is a  unital, commutative ring.

\begin{prop}
\label{prop:abelianization}
 The abelianization functor $\A : \gr \rightarrow \ab$ induces  exact functors 
 \begin{eqnarray*}
  \circ \A &:& \fcatk[\ab] \rightarrow \fcatk[\gr]
  \\
    \circ \A &:& \fcatk[\ab\op] \rightarrow \fcatk[\gr\op]
 \end{eqnarray*}
which are symmetric monoidal and fully faithful. The images are strictly full and closed under the formation of  sums and sub-quotients.

Taking $\kring = \rat$, the subcategory $\fcatQ[\ab] \subset \fcatQ[\gr]$ (respectively $\fcatQ[\ab\op]\subset \fcatQ[\gr\op]$) is not thick.
\end{prop}

\begin{proof}
This is an application of a standard result for functor categories (see 
\cite[Proposition C.1.4]{D} or \cite[Proposition A.2]{MR2391661}). That 
$\circ \A$ is fully faithful is a consequence of the fact that $\A$ is 
essentially surjective and full. 

That the subcategories are not thick is exhibited by Example \ref{exam:Passi_not_out}.
\end{proof}

The following Lemma shows that, for the category $\ab$, the contravariant and covariant functor categories are equivalent. This is far from the case for $\gr$, whence the necessity to consider both variances.

\begin{lem}
 \label{lem:duality_ab}
 The  duality functor 
  $
  \hom_{\ab} (-, \zed) : \ab \op \rightarrow \ab
 $ 
 is an equivalence of categories.  By precomposition, this induces an equivalence of categories 
 $
  \fcatk [\ab] \stackrel{\cong}{\rightarrow } \fcatk[\ab \op].
 $
\end{lem}

\begin{nota}
\nomenclature{$\ak $}{abelianization as an object of $\fcatk[\gr]$\nomrefpage}
Denote by $\ak \in \ob \fcatk[\gr]$ the composite functor $\A(-) \otimes_\zed \kring$, where $\A : \gr \rightarrow \ab$ is the abelianization functor.
\end{nota}

The standard projective functors in $\fcatk[\gr]$ considered in the following example are fundamental objects and play an important role  in the current work; for example, their structure is exploited in Section \ref{sec:beta}.
 
\begin{exam}
\label{exam:standards_proj_gr}
\nomenclature{$P^\gr_G$}{standard projective in $\fcatk[\gr]$\nomrefpage}
For $G$ in $\gr$, the standard projective functor  $P^\gr_G$ is the functor $\kring [\hom_\gr(G, -)]$ in $\fcatk[\gr]$ (see Appendix \ref{subsect:functor}). 

For $G= \zed$, the associated standard projective $P^\gr _\zed$ identifies with the group ring functor $H \mapsto \kring [H]$, for $H \in \ob \gr$,  considered as taking values in $\kring$-modules (i.e., forgetting structure). 
 This decomposes as $P^\gr _\zed  \cong \kring \oplus \overline{P^\gr _\zed}$, where $\overline{P^\gr _\zed}$ is the reduced part. Then $\overline{P_\zed^\gr} (H)$ 
 identifies with the $\kring$-module underlying the augmentation ideal $\mathcal{I} (H) \subset \kring 
[H]$. In particular, as a submodule of  $\kring [H]$, it is generated by the elements $[h]-[e]$, where $e \in H$ is the unit and $h \in H$.

There is an important surjective map $\overline{P^\gr_\zed} \twoheadrightarrow \ak$ in $\fcatk[\gr]$; evaluated on $H \in \ob \gr$, this is the $\kring$-linear map sending $[h]-[e]$ to $\underline{h} \otimes 1 \in \A(H) \otimes \kring$, where $\underline{h}$ denotes the image of $h$ in the abelianization $\A(H)$.
\end{exam}

We now introduce the shift functors:

\begin{nota}
\label{nota:shift_functors}
\nomenclature{$\tau_\zed$}{shift functor on $\fcatk[\gr]$\nomrefpage}
\nomenclature{$\taubar$}{reduced shift functor on $\fcatk[\gr]$\nomrefpage}
 Denote by 
\begin{enumerate}
 \item 
 $\tau _\zed : \fcatk[\gr] \rightarrow \fcatk[\gr]$ the 
shift functor  defined on $F \in \ob \fcatk[\gr]$ by $(\tau_\zed F) (G) := F (G 
\star \zed)$, equipped with 
the canonical 
 splitting 
 \[
  1_{\fcatk[\gr]} \hookrightarrow \tau_\zed \twoheadrightarrow 1_{\fcatk[\gr]},
 \]
and let $\taubar : \fcatk[\gr] \rightarrow \fcatk[\gr]$ denote the reduced shift functor, given by the cokernel of 
this split monomorphism, 
 so that there is a natural decomposition $\tau_\zed F \cong F \oplus \taubar F$; 
\item 
$\tauab : \fcatk[\ab] \rightarrow \fcatk[\ab]$ the  shift functor  defined by $(\tauab F) (A) := F (A \oplus \zed)$; the 
reduced functor $\tauabbar$ is defined similarly. 
 \end{enumerate}
 Analogously, denote by $\tau_\zed, \taubar : \fcatk[\gr\op] \rightarrow \fcatk[\gr\op]$ (respectively  $\tauab, \tauabbar : \fcatk[\ab\op] \rightarrow 
\fcatk[\ab\op]$)  the shift and reduced shift functors in the contravariant setting.
 \end{nota}

\begin{prop}
\label{prop:properties_tau}
The functors $\tau_\zed$ satisfy the following properties:
\begin{enumerate}
 \item 
$\tau_\zed : \fcatk[\gr] \rightarrow \fcatk[\gr]$ is exact and  is 
right adjoint to the exact functor 
 $P^\gr_{\zed} \otimes - : \fcatk[\gr] \rightarrow \fcatk[\gr]$;  in particular, 
$\tau_\zed$ preserves injectives;
\item 
$\taubar : \fcatk[\gr] \rightarrow \fcatk[\gr]$ is exact and  is 
right adjoint to the exact functor 
 $\overline{P^\gr_{\zed}} \otimes - : \fcatk[\gr] \rightarrow \fcatk[\gr]$; 
 \item 
$\tau_\zed : \fcatk[\gr\op] \rightarrow \fcatk[\gr\op]$ is exact 
and preserves projectives. 
\end{enumerate}
\end{prop}

\begin{proof}
These are standard results for functor categories (cf. \cite[Appendice C]{D}). 
The proof is sketched for $\tau_\zed$; the case of $\taubar$ 
is similar. 

It is clear that $\tau_\zed$ is exact. Consider $\fcatk[\gr]$; the functor $P^\gr_\zed$ takes values in 
free $\kring$-modules, hence $P^\gr_\zed \otimes - $ is exact. 
Moreover, there is a canonical isomorphism $P^\gr _\zed \otimes P^\gr _G \cong 
P^\gr_{\zed \star G}$, since $\star$ is the coproduct in $\gr$, which induces 
the natural isomorphism 
\[
 \hom_{\fcatk[\gr]} (P^\gr _\zed \otimes P^\gr _G, F) \cong  \hom_{\fcatk[\gr]} 
(P^\gr _G, \tau_\zed F)
\]
on projective generators. Since $P^\gr_\zed \otimes - $ is exact, this 
establishes the adjunction. It follows formally that $\tau_\zed$ preserves 
injectives.

In the category $\fcatk[\gr\op]$, one has $$\tau_\zed P^{\gr\op}_G (-) \cong 
\kring [\hom_{\gr} (- \star \zed, G)] \cong \kring [G] \otimes P^{\gr \op} _G 
(-).$$ 
The right hand side is projective, as required.
\end{proof}

 \begin{prop}
 \label{prop:shift_abelianization}
  The shift functors $\tau_\zed$ and $\tauab$ are compatible with respect to 
abelianization: namely the following diagrams commute up to natural isomorphism
  \[
   \xymatrix{
   \fcatk[\ab] \ar[r]^{\tauab} 
   \ar[d]_{\circ \A}
   &
   \fcatk[\ab]
   \ar[d]^{\circ \A}
   &
    \fcatk[\ab \op] \ar[r]^{\tauab} 
   \ar[d]_{\circ \A}
   &
   \fcatk[\ab \op]
   \ar[d]^{\circ \A}
   \\
   \fcatk[\gr]
   \ar[r]_{\tau_\zed}
   &
   \fcatk[\gr]
   &
   \fcatk[\gr \op]
   \ar[r]_{\tau_\zed}
   &
   \fcatk[\gr \op].
   }
  \]
 \end{prop}

\begin{proof}
 This is a  formal consequence of the natural isomorphism $\A (G \star \zed) \cong \A 
(G) \oplus \zed$. 
\end{proof}

\section{Exponential functors and Hopf algebras}
\label{sec:expo}

In this section we review the theory of exponential functors defined on $\gr$ and on $\ab$ respectively, together with their contravariant counterparts. Exponential functors are defined as symmetric monoidal functors. This implies that the complete structure of an exponential functor can be encoded in a Hopf algebra (see Theorem \ref{thm:expo_Hopf_general}). 

In particular, we introduce the functors $\Phi$ and $\Psi$ that we exploit throughout this work: for $H$ a commutative Hopf algebra, $\Psi H$ is given by the associated exponential functor on $\gr$; for $H$ cocommutative, $\Phi H$ is given by the associated exponential functor on $\gr\op$.

\subsection{Exponential functors on $\gr$ and $\ab$}
\label{subsect:expo_functors}

Throughout this section, $\calm$ is an abelian category that is equipped with a symmetric monoidal structure  $(\calm, \otimes,\unit)$. We shall always assume that $\otimes$ is additive with respect to each variable. 

The following are the key examples that occur in the paper: 
\begin{enumerate}
\item 
The category of graded $\kring$-modules over a commutative ring $\kring$, equipped with the graded tensor product $\otimes_\kring$; there are two standard choices for symmetry, either invoking Koszul signs or not.
\item 
The category $\f (\fmodq ; \rat)$ of functors from finite-dimensional $\rat$-vector spaces to $\rat$-vector spaces, equipped with the objectwise tensor product $\otimes_\rat$.
\item 
The category $\f(\fb; \rat)$ of $\fb$-modules (reviewed in Section \ref{sec:fb}), equipped with the convolution product $\tenfb$; again there are  two possible choices for symmetry, either invoking Koszul signs or not.
\end{enumerate}

\begin{defn}
\label{defn:expo_functor}
For $\calc$ one of $\gr, \gr\op, \ab, \ab\op$, the category $\fexp (\calc; \calm)$  of exponential functors from $\calc$ to $\calm$ is the category of symmetric monoidal functors from $\calc$ to $\calm$; this is equipped with the forgetful functor  $\fexp (\calc; \calm) \rightarrow \f (\calc; \calm)$.
\end{defn}

Thus, an exponential functor from $\gr$ to $\calm$ is a functor $F \in \ob \f (\gr; \calm)$ such that, for  $G_1$, $G_2$ in $\gr$ there is a natural isomorphism 
$$F (G_1 \star G_2) \cong F(G_1) \otimes F(G_2)$$
 and these isomorphisms are compatible with the  symmetric monoidal structure of $\calm$.

Since $\A : \gr \rightarrow \ab$ is symmetric monoidal, one has:

\begin{prop}
\label{prop:abelianize_expo}
The abelianization functor $\gr \rightarrow \ab$ induces natural transformations:
\begin{eqnarray*}
\circ \A : \fexp (\ab; \calm) &\hookrightarrow & \fexp (\gr; \calm) 
\\
\circ \A : \fexp (\ab\op; \calm) &\hookrightarrow & \fexp (\gr\op; \calm) 
\end{eqnarray*}
that are fully faithful.

Moreover, the equivalence of categories $\hom_\ab (-, \zed) : \ab\op \stackrel{\cong}{\rightarrow} \ab$ induces an equivalence of categories $\fexp (\ab; \calm) \cong \fexp (\ab\op; \calm)$.
\end{prop}

\begin{exam}
\label{exam:expo_group_ring}
The following examples are fundamental in the applications considered here.
\begin{enumerate}
\item 
The group ring functor $G \mapsto \rat [G]$ is {\em not} exponential on $\gr$ with respect to $(\gr,\star, \{e\})$.
This can be seen  as follows. Evaluated on $G= \zed$, $\rat[\zed]$ is a {\em bicommutative} Hopf algebra; were $\rat[-]$ exponential, by  Theorem \ref{thm:expo_Hopf_general} below, it would have to factor across abelianization, which is clearly false.
\item 
The group ring functor $A \mapsto \rat [A]$  on $\ab$ {\em is} exponential with respect to $(\ab, \oplus, 0)$, since there are natural isomorphisms $\rat [A_1 \oplus A_2]\cong \rat[A_1] \otimes \rat[A_2]$ for $A_1$, $A_2$ in $\ab$,  exhibiting $\rat[-]$ as  a symmetric monoidal functor.
\item 
For a fixed group $H\in \ob \gr$, the contravariant functor on $\gr$ given by $G \mapsto \rat [\hom_\gr (G, H)]$ is exponential via the natural isomorphisms $\rat [\hom_\gr (G_1 \star G_2, H)]
\cong 
\rat [\hom_\gr (G_1, H)] \otimes \rat [\hom_\gr (G_2, H)]$, using that $\star$ is the coproduct in $\gr$. 
 Moreover, the exponential structure is natural with respect to the group $H$. 
\end{enumerate}
Further basic examples can be given by applying Theorem \ref{thm:expo_Hopf_general} below.
\end{exam}

\subsection{Introducing the functors $\Psi$ and $\Phi$}

The symmetric monoidal structure $(\calm , \otimes , \unit)$ allows us to consider Hopf algebras in $\calm$, by the obvious abstraction of the usual definition in $\kring$-modules. Explicitly, a Hopf algebra structure on $H$ is given by the morphisms $(\epsilon_H, \eta_H, \Delta_H, \mu_H, \chi_H)$, where $\unit \stackrel{\eta_H}{\rightarrow} H \stackrel{\epsilon_H}{\rightarrow} \unit$ are the unit and augmentation respectively, $\mu_H: H\otimes H \rightarrow H$ is the product, $\Delta_H : H \rightarrow H \otimes H$ the coproduct, and $\chi_H : H \rightarrow H$ the conjugation (or antipode); these satisfy the appropriate axioms.

A Hopf algebra $H$ is commutative if $\mu_H$ is commutative (i.e., $\mu_H = \mu_H \tau_{H\otimes H}$, $\tau_{H\otimes H}$ denoting the symmetry); it is cocommutative if $\Delta_H$ is cocommutative (i.e., $\Delta _H = \tau_{H\otimes H} \Delta_H$); it is bicommutative if it is both commutative and cocommutative.

\begin{nota}
\label{nota:hopf}
Write $\hcom(\calm)$, $\hcocom(\calm)$,  $\hbicom(\calm)$  for the categories of Hopf algebras in $\calm$ that are respectively commutative, cocommutative, bicommutative.
\end{nota}

The following Theorem gives the relationship between exponential functors and Hopf algebras. 
 In particular, the functors $\Phi$ and $\Psi$ correspond to the construction of an exponential functor from an appropriate Hopf algebra. 
 
 \nomenclature{$\Phi$}{exponential functor construction $\hcocom(\calm)  \rightarrow \fcat{\gr\op}{\calm}$\nomrefpage}
 \nomenclature{$\Psi$}{exponential functor construction $\hcom(\calm)  \rightarrow \fcat{\gr}{\calm}$\nomrefpage}

\begin{thm}
\label{thm:expo_Hopf_general}
Evaluation on $\zed$ induces equivalences of categories 
 \begin{eqnarray*}
  \eval_\zed &:& \fexp (\gr ; \calm) \stackrel{\cong}{\longrightarrow} 
 \hcom(\calm) \\
   \eval_\zed &:& \fexp (\gr\op ; \calm) 
\stackrel{\cong}{\longrightarrow}
\hcocom(\calm).
 \end{eqnarray*}
These restrict to give the equivalences 
\[
\fexp (\ab ; \calm) \cong \fexp (\ab \op; \calm)  
\stackrel{\cong}{\longrightarrow} \hbicom(\calm).
\]

In particular, there are faithful embeddings:
 \begin{enumerate}
 \item  
  $\Psi :  \hcom (\calm) \rightarrow \fcat{\gr}{\calm}$ such that $\Psi H (\zed ^{\star n}) = H^{\otimes n}$
  \item 
  $\Phi : \hcocom(\calm)  \rightarrow \fcat{\gr\op}{\calm}$ such that $\Phi H (\zed ^{\star n}) = H^{\otimes n}$
\end{enumerate}
and these take values in the respective categories of exponential functors.

Restricted to $\hbicom(\calm)$, up to natural isomorphism, these functors factor across $\fcat{\ab\op}{\calm}$ and $\fcat{\ab}{\calm}$ 
respectively:
\[
 \xymatrix{
\hbicom(\calm) \ar@{^(->}[d]
 \ar[r]^{\Phi}
 &
 \fcat{\ab\op}{\calm}
 \ar[d]^{\circ \A}
&
 \hbicom(\calm) \ar@{^(->}[d]
 \ar[r]^\Psi
 & 
 \fcat{\ab}{\calm}
 \ar[d]^{\circ \A}  
 \\
\hcocom(\calm) 
 \ar[r]_\Phi 
 &
 \fcat{\gr\op}{\calm}
 &
\hcom(\calm) 
 \ar[r]_\Psi 
 &
\fcat{\gr}{\calm}.
 }
\]

\end{thm}

\begin{proof}
For $\calm = \kring\dash\modules$, Theorem \ref{thm:expo_Hopf_general} can be proved as in \cite[Remark 2.1, Theorem 2.2]{MR3765469}, by identifying the PROP associated to the relevant category of Hopf algebras, as in \cite{P}, \cite{Hab}. This argument extends to the general case. For the convenience of the reader, we outline the argument for exponential functors on $\gr$.

Consider $\eval_\zed$ on $\fexp (\gr ; \calm)$. For an exponential functor $F$ we may assume that $F(\{e\})=\unit$ and $F(\zed \star \zed)\cong F(\zed) \otimes F(\zed)$; then the structure morphisms $(\eta, \nabla, \Delta, \epsilon, \chi)$ of $\gr$ given in Proposition \ref{prop:generate_mor_gr} make $F(\zed)$ into a commutative Hopf algebra in $\calm$ and this structure is natural. This gives the evaluation functor $\eval_\zed : \fexp (\gr ; \calm) \rightarrow \hcom(\calm)$. 

To show that this is an equivalence, it suffices to exhibit the quasi-inverse  $\Psi :  \hcom (\calm) \rightarrow \fexp(\gr;\calm) \subset \f(\gr; \calm)$. Consider $H \in \ob \hcom (\calm)$. The functor $\Psi H$ is defined on objects by  $\Psi H (\zed ^{\star n}) = H^{\otimes n}$. 
Using Proposition \ref{prop:generate_mor_gr}, the action of morphisms of $\gr$ is given in terms of  the Hopf algebra structure of $H$. This construction is natural with respect to $H$ and, by construction, yields an exponential functor. Finally one checks that this is a quasi-inverse, as required. 
\end{proof}

\begin{exam}
\label{exam:expo_group_ring+hopf}
As in Example \ref{exam:expo_group_ring}, $\rat [\hom_\gr(-,H)]$ is an exponential functor on $\gr\op$. The underlying cocommutative Hopf algebra is the group ring $\rat[H]$, so that there is an isomorphism of 
functors on $\gr\op$:
\[
\rat [\hom_\gr(-,H)] \cong \Phi \big( \rat[H] \big).
\]
This isomorphism is natural with respect to $H$ in $\gr$, using the naturality of the group ring as a cocommutative Hopf algebra.
\end{exam}

\begin{rem}
 \ 
 \begin{enumerate} 
\item 
There are variants of Theorem \ref{thm:expo_Hopf_general}: for example, Touz\'e has given a classification of the exponential functors from the category of finitely generated projective $R$-modules to $\kring$-modules \cite[Theorem 5.3]{T}; taking $R=\zed$, this corresponds to the case $\ab$  stated in Theorem \ref{thm:expo_Hopf_general}.
  \item 
For $H$ a cocommutative Hopf algebra in $\kring$-modules, evaluating $\Phi H$ on $\zed^{\star n}$ gives a right $\End (\zed^{\star n})$-action on $H^{\otimes n} \cong \Phi H (\zed^{\star n})$ and, by restriction, a right $\aut (\zed^{\star n})$-action.
This recovers the action constructed (in a special case) by Turchin and 
Willwacher in \cite[Section 1]{TW}. Similarly one recovers the right $\aut 
(\zed^{\star n})$-action on $H^{\otimes n}$ given by  Conant and Kassabov \cite[Section 
4]{CK}. 
 \end{enumerate}
\end{rem}

\subsection{(Co)multiplicative structures on exponential functors}
\label{subsect:mult_expo}

An exponential functor in $\fexp (\gr ; \calm)$ does {\em not} in general give a functor from $\gr$ to $\hcom(\calm)$. Nevertheless, it does carry a certain multiplicative structure; this is useful in applications for analysing the underlying functor. Similarly, an exponential functor on $\gr\op$ carries a certain comultiplicative structure. 

To explain this, recall that a unital commutative monoid in $\calm$ is given by a tuple $(M, \epsilon_M, \mu_M)$ where $\epsilon _M : \unit  \rightarrow M$, $\mu_M : M \otimes M \rightarrow M$ satisfy the associativity, commutativity and unit axioms; morphisms are defined in the obvious way.  Similarly, a counital cocommutative comonoid in $\calm$ is given by a tuple $(M, \eta_M, \Delta_M )$, where $\eta_M :M \rightarrow \unit$ and $\Delta_M : M \rightarrow M\otimes M$, satisfying the appropriate axioms; morphisms are again defined in the obvious way.

\begin{nota}
Denote by
\begin{enumerate}
\item 
 $\cmon (\calm)$ the category of unital, commutative monoids in $\calm$, equipped with the forgetful functor $\cmon (\calm) \rightarrow \calm$;
\item 
$\ccomon (\calm)$ the category of counital, cocommutative comonoids in $\calm$, equipped with the forgetful functor $\ccomon (\calm) \rightarrow \calm$.
\end{enumerate}
\end{nota}

The following  generalizes the fact that, working over a unital commutative ring $\kring$, the tensor product defines the coproduct in the category of unital, commutative $\kring$-algebras, with initial object $\kring$.

\begin{prop}
\label{prop:mon_comon_(co)cartesian}
Let $(\calm, \otimes, \unit)$ be a symmetric monoidal category, then 
\begin{enumerate}
 \item 
$\cmon (\calm)$ has finite coproducts given by $\otimes$, with 
initial object $\unit$; 
 \item
$\ccomon(\calm)$ has finite products given by $\otimes$, with final 
object $\unit$.
\end{enumerate}
\end{prop}

Thus one can consider the category  $\mathrm{Func} ( \gr; \cmon (\calm)) $ of functors from $\gr$ to $\cmon (\calm)$,  equipped with the induced forgetful functor $\mathrm{Func} ( \gr; \cmon (\calm)) \rightarrow \f(\gr;\calm)$. Similarly, one considers $\mathrm{Func} ( \gr; \ccomon (\calm)) $.

\begin{prop}
\label{prop:Phi_Psi_(co)mult}
\ 
\begin{enumerate}
\item 
The functor  $\Psi :  \hcom (\calm) \rightarrow \fcat{\gr}{\calm}$ factors across  $$\mathrm{Func} ( \gr; \cmon (\calm))\rightarrow \f(\gr;\calm). $$
\item 
The functor  $\Phi :  \hcocom (\calm) \rightarrow \fcat{\gr\op}{\calm}$ factors across $$\mathrm{Func} ( \gr; \ccomon (\calm)) \rightarrow \f(\gr\op;\calm).$$
\end{enumerate}
\end{prop}

\begin{proof}
The first statement follows from the fact that $\hcom (\calm)$ is equivalent to the category of cogroup objects in $\cmon (\calm)$; in particular, all structure morphisms involved lie in $\cmon(\calm)$. It follows that a functor in the image of $\Psi$ takes values naturally in $\cmon (\calm)$.

The second statement is categorically dual.
\end{proof}

\begin{rem}
\label{rem:comm_exp}
For $H$ a commutative Hopf algebra in $\calm$, one has the associated exponential functor $\Psi H \in \ob \fexp (\gr; \calm) \subset \ob \f(\gr; \calm)$. Proposition \ref{prop:Phi_Psi_(co)mult} gives that, for any $n \in \nat$, $\Psi H (\zed^{\star n})$ is naturally a unital commutative monoid in $\calm$. In particular, there is a commutative multiplication 
\[
\Psi H  \otimes \Psi H \rightarrow \Psi H
\]
in $\f (\gr; \calm)$. This construction is natural with respect to $H$ in $\hcom (\calm)$.
\end{rem}

\subsection{Naturality}

The functors $\Phi$ and $\Psi$ depend naturally upon   $(\calm, \otimes ,\unit)$, as stated in the following Proposition. The relevant symmetric monoidal category is indicated by a subscript for the respective functors $\Phi$, $\Psi$.

\begin{prop}
\label{prop:naturality_Phi_Psi}
Let $(\calm , \otimes_\calm, \unit_\calm) \rightarrow (\cale , \otimes_\cale, \unit_\cale)$ be a symmetric monoidal functor.
 The following diagrams commute up to natural isomorphism:
\[
\xymatrix{
\mathbf{Hopf}^{\mathrm{cocom}}(\calm) 
\ar[d]
\ar[r]^{\Phi_\calm}
&
 \fcat{\gr\op}{\calm}
 \ar[d]
 &
 \mathbf{Hopf}^{\mathrm{com}}(\calm)
 \ar[r]^{\Psi_\calm}
 \ar[d]
 &
 \fcat{\gr}{\calm}
\ar[d]
 \\
\mathbf{Hopf}^{\mathrm{cocom}}(\cale) 
\ar[r]_{\Phi_\cale}
&
 \fcat{\gr\op}{\cale}
&
\mathbf{Hopf}^{\mathrm{com}}(\cale)
\ar[r]_{\Psi_\cale}
&
\fcat{\gr}{\cale},
}
\]
where the vertical arrows are induced by the symmetric monoidal functor $\calm \rightarrow \cale$.
\end{prop}

\begin{proof}
The naturality follows directly from the proof of Theorem \ref{thm:expo_Hopf_general}.
\end{proof}

\begin{rem}
This is a fundamental tool, which is applied in several different ways in the paper.
For example: 
\begin{enumerate}
\item 
In Proposition  \ref{prop:duality_Psi_palg_Phi_pcoalg}, it is used to relate $\Phi$ and $\Psi$ through duality (in a particular case).
\item 
In Section \ref{sec:beta}, it is crucial in analysing the structure of the injective cogenerators of the category of polynomial functors on $\gr$. 
\item 
The Schur functor construction relates $\fb$-modules with $\f(\fmodq;\rat)$ and is symmetric monoidal, where $\fb$ is the category of finite sets and bijections (see Section \ref{subsect:schur_correspond}). This relates exponential functors with values in $\fb$-modules with the more familiar case of functors to $\rat$-vector spaces, i.e., relating 
\[
\fexp (\gr; \f(\fb;\rat)) \mbox{ and } \fexp (\gr; \f(\fmodq; \rat)).
\]
\item 
It allows the introduction of Koszul signs using the symmetric monoidal functor $(-)^\dagger$ of Theorem \ref{thm:equiv_dagger}. This  is exploited in Proposition \ref{prop:omega_pcoalg_dagger}, for example.
\end{enumerate}
\end{rem}

\part{Polynomial functors on free groups}
\label{part:poly}

This part provides information on polynomial functors on free groups, along the way developing some of the central tools of this paper.

Section \ref{sec:recoll} presents the theory of polynomial functors on free groups. Many of the results in Section \ref{sec:recoll} are consequences of the recollement setup for polynomial functors on free groups given in \cite{DV} and recalled here in Theorem \ref{thm:recollement_poly_d}. In particular, in Definition \ref{defn:poly_filt} we introduce the polynomial filtration of a functor (as considered in \cite{DV}) and, for $d \in \nat$,  we introduce the functor $\beta_d$ that associates  a polynomial functor of degree $d$ on $\gr$ to a $\rat [\sym_d]$-module. The functor $\beta_d$ occurs as a {\em right} adjoint in the recollement diagram of Theorem \ref{thm:recollement_poly_d}.

The relationship between the polynomial filtration and the socle filtration given in Section \ref{subsect:poly_filt_soc} is new and fundamental for the explicit description of the structure of the functors appearing in Section \ref{sec:psi}, for example. In Corollary \ref{exam:beta_M_socle_filt} we show that, for $M \neq 0$ a  $\rat [\sym_d]$-module, the socle filtration and the polynomial filtration of $\beta_d M$ coincide (up to reindexing).

The functors $\beta_d$ play a central role in our study of higher Hochschild homology in Part \ref{part:HHH}, explaining their importance.
One of the key techniques that we use to study them is to consider the family of functors $\beta_d$, for $d\in \nat$, as a whole. Namely, we use the category $\fb$ of finite sets and bijections and  assemble the  $\beta_d$ to give a functor $\boldbeta$ from $\fb$ to the category of functors on free groups. The main goal of this part is to describe the functor $\boldbeta$.

This motivates Section \ref{sec:fb}, in which  we recall the relationship between the category of $\fb$-modules and functors on finite-dimensional $\rat$-vector spaces given by the Schur correspondence. (The contents of Section \ref{sec:fb}  is mostly well-known and the reader familiar with Schur functors may wish to skip this section.) We also include  a careful treatment of Koszul signs, since these  appear in our study of higher Hochschild homology in Part \ref{part:HHH}. 

In Section \ref{sec:pcoalg}, we introduce the Hopf  $\fb$-modules associated with the Hopf algebras given by the tensor algebra and cotensor algebra via the Schur correspondence. Using the relationship between Hopf algebras and exponential functors of free groups recalled in Theorem \ref{thm:expo_Hopf_general}, this gives rise to exponential functors on free groups.

In Section \ref{sec:beta} we study the functor $\boldbeta$ in detail. The main result of this part is Theorem \ref{thm:beta_description}, giving the description of $\boldbeta$ using the functors introduced in Section \ref{sec:pcoalg}.

\section{Polynomial functors on groups}
\label{sec:recoll}

This section reviews the theory of polynomial functors on the category $\gr$ as required in the applications.
The notion of polynomial functor initially introduced by Eilenberg and Mac Lane \cite{EM} for functors on categories of modules has been extended to a wider context in \cite{HPV}, which includes functors on the category of finitely-generated free groups $\gr$. Throughout, we restrict to functors taking values in $\rat$-vector spaces, which ensures that the categories of 
$\rat[\sym_d]$ modules, for $d\in \nat$, are all semi-simple. As a result, the theory is both very elegant and powerful. 

In Section \ref{subsect:recoll_poly}, we recall the recollement diagram given in \cite{DV}  (restricted here to functors taking values in $\rat$-vector spaces). In particular, the functors $\beta_d$, $d \in \nat$,  that play an important role here in relation to higher Hochschild homology, are recalled in Theorem \ref{thm:recollement_poly_d}. These are investigated in detail in Section \ref{sec:beta}. 
 
In Section  \ref{subsect:poly_filt}, we recall the polynomial filtration of polynomial functors and show that it satisfies a unicity property (see Proposition \ref{prop:canon_poly_filt}), using the fact that we are working over $\rat$. 

In Section \ref{subsect:poly_filt_soc}, we relate  the polynomial filtration to the socle filtration, again dependent on working over $\rat$, and deduce several consequences.

Finally, in Section \ref{subsect:injectivity_beta_char_0} we recall some homological properties of the categories of polynomial functors.

Most of the content of Sections \ref{subsect:recoll_poly}, \ref{subsect:poly_filt} and \ref{subsect:injectivity_beta_char_0} is not new; the results are recollections or refinements of results given in  \cite{DV} and \cite{DPV}. The relation between the polynomial filtration and the socle filtration given in Section \ref{subsect:poly_filt_soc} is new.

\subsection{Background}

Using the structure introduced in Section \ref{subsect:gr_ab}, by the general theory of Section \ref{sec:poly},  one has the notion of polynomial functor  in both $\fcatQ[\gr]$ and $\fcatQ[\ab]$ (likewise for $\fcatQ[\gr\op]$ and  $\fcatQ[\ab\op]$), defined using the respective coproducts. 

\begin{rem}
We concentrate here on the covariant case, $\fcatQ[\gr]$ and $\fcatQ[\ab]$. There are analogous results  for contravariant functors; these are related to the covariant case by  duality (see Section \ref{sec:poly}).
\end{rem}

\nomenclature{$\fpolyQ{d}{\gr}$}{polynomial functors of degree $d$ on $\gr$\nomrefpage}
\nomenclature{$\fpolyQ{d}{\ab}$}{polynomial functors of degree $d$ on $\ab$\nomrefpage}
\nomenclature{$\fpiQ{\gr}$}{polynomial functors on $\gr$\nomrefpage}
\nomenclature{$\fpiQ{\ab}$}{polynomial functors on $\ab$\nomrefpage}

The full subcategories of polynomial functors of degree at most $d \in \nat$ are denoted $\fpolyQ{d}{\gr}$ and $\fpolyQ{d}{\ab}$. The full subcategories of polynomial functors are denoted  $\fpiQ{\gr}$ and $\fpiQ{\ab}$ so that 
$
\fpiQ{\gr}
= 
\bigcup_{d\in \nat}
\fpolyQ{d}{\gr}
$ and  likewise for $\ab$.

Since $\A : \gr \rightarrow \ab$ preserves coproducts, abelianization restricts to polynomial functors (cf. Proposition \ref{prop:shift_abelianization}):

\begin{prop}
\label{prop:abelianization_poly}
 The exact functor $\circ 
\A
: \fcatQ[\ab]\rightarrow \fcatQ[\gr]$ restricts for $d\in \nat$ to 
$
 \circ \A : \fpolyQ{d}{\ab}\rightarrow \fpolyQ{d}{\gr}.
$
\end{prop}

Below we exploit some basic facts about the representation theory of the symmetric groups (see Section \ref{app:rep_sym_car0} for some recollections). In particular, for $d \in \nat$, the category of $\rat [\sym_d]$-modules is semi-simple. The regular representation decomposes as 
\begin{equation} 
\label{221005-1356}
\rat [\sym_d] \cong \bigoplus_{\lambda \vdash d} S_\lambda^{\oplus \dim (S_\lambda)},
\end{equation}
where $S_\lambda$ is the simple module indexed by the partition $\lambda \vdash d$. Moreover, a simple $\rat[\sym_d]$-module is isomorphic to $S_\lambda$ for a unique partition $\lambda \vdash d$. 

\subsection{Recollement for polynomial functors}
\label{subsect:recoll_poly}

In \cite{DV}, the authors give the  fundamental relationship between polynomial functors and representations of the symmetric groups in terms of a {\em recollement diagram} (we refer the reader to \cite{DV} for the definition of a recollement diagram). In this section, we recall this result for functors taking values in $\rat$-vector spaces.

This allows us to prove a number of results about $\fpiQ{\gr}$, in particular describing the simple functors and their injective envelopes in the following sections.

\begin{nota}
\label{nota:alpha_beta_cr_p_q}
\nomenclature{$\cre_d$}{$d$th cross-effect functor\nomrefpage}
\nomenclature{$\alpha_d$}{left adjoint to $\cre_d$\nomrefpage}
\nomenclature{$\beta_d$}{right adjoint to $\cre_d$\nomrefpage}
\nomenclature{$p_d^\gr$}{right adjoint to $\f_d (\gr; \rat) \subset \f(\gr; \rat)$\nomrefpage}
\nomenclature{$q_d^\gr$}{left adjoint to $\f_d (\gr; \rat) \subset \f(\gr; \rat)$\nomrefpage}
For $d \in \nat$, denote by
\begin{enumerate}
\item 
$p_d^\gr : \f (\gr; \rat) \rightarrow \f_d (\gr; \rat) $ the right adjoint to the inclusion $\f_d (\gr; \rat) \subset \f(\gr; \rat)$ and $q_d^\gr : \f (\gr; \rat) \rightarrow \f_d (\gr; \rat) $ the left adjoint (cf. Notation \ref{nota:p_q_poly}); 
\item 
${\cre_d}: \f_d (\gr; \rat) \rightarrow \rat [\sym_d]\dash\modules$ the $d$th cross-effect functor; 
\item 
$\alpha_d :  \rat [\sym_d] \dash\modules \rightarrow \f_d (\gr; \rat)$ the left adjoint to $\cre_d$ and $\beta_d :  \rat [\sym_d] \dash\modules \rightarrow \f_d (\gr; \rat)$ its right adjoint.
\end{enumerate}
\end{nota}

\begin{thm}
\label{thm:recollement_poly_d}
For $d \in \nat$, there is a  recollement diagram
\[
 \xymatrix{
\fpolyQ{d-1}{\gr} 
\ar@{^(->}[r]
&
\fpolyQ{d}{\gr}
\ar@<-.5ex>@/_1pc/[l]|{q_{d-1}^\gr}
\ar@<.5ex>@/^1pc/[l]|{p_{d-1}^\gr}
\ar[r]|(.4){\cre_d}
&
\rat [\sym_d]\dash \modules
\ar@<-.5ex>@/_1pc/[l]|{\alpha_d}
\ar@<.5ex>@/^1pc/[l]|{\beta_d}. 
}
\]
\begin{enumerate}
\item \label{thm:recollement_poly_d-1}
The $d$th cross effect functor $\cre_d$ is exact
and is naturally 
equivalent to $\hom_{\fpolyQ{d}{\gr}}( \aQ ^{\otimes d}, - )$. 
\item \label{thm:recollement_poly_d-3}
The left adjoint to $\cre_d$,  $\alpha_d$, is given by $\alpha_d : M \mapsto \aQ ^{\otimes d} \otimes_{\sym_d} M$. In particular, $\alpha_d(M)$ is in the image of the functor   $\circ \A : \fpolyQ{d}{\ab} \rightarrow \fpolyQ{d}{\gr}$. The functor $\alpha_d$ preserves projectives and  is exact.
\item \label{thm:recollement_poly_d-4}
The right adjoint $\beta_d$ to $\cre_d$ is exact and preserves injectives. 
\end{enumerate}
\end{thm}

\begin{proof}
Most of this is proved in \cite[Théorème 3.2]{DV}, including the explicit form of $\cre_d$ of (\ref{thm:recollement_poly_d-1}) and, for (\ref{thm:recollement_poly_d-3}), the definition of $\alpha_d$.
These results rely on the fact that $\aQ^{\otimes d}$ is projective in $\f_d (\gr; \rat)$, which follows from the isomorphism $q^\gr_d (\overline{P^\gr_\zed}^{\otimes d}) \cong \aQ^{\otimes d}$.  The exactitude of $\alpha_d$ follows from that of $\otimes_{\sym_d}$ working over $\rat$.

Finally, (\ref{thm:recollement_poly_d-4}) is proved in \cite[Proposition 4.4]{DPV}.
\end{proof}

\begin{cor}
\label{cor:alpha_simple}
For $d \in \nat$ and $F \in \ob \fpolyQ{d}{\gr}$, the following are equivalent:
\begin{enumerate}
\item 
$F$ is simple and has polynomial degree exactly $d$; 
\item 
$\cre_d F$ is a simple $\rat[\sym_d]$-module and $F \cong \alpha_d \cre_d F$.
\end{enumerate}

In particular, for $\lambda \vdash d$, $\alpha_d S_\lambda$ is simple. 
\end{cor}

\begin{proof}
It is straightforward to see that, if $F$ is simple with polynomial degree exactly $d$, then $\cre_d F$ is non-zero and is a simple $\rat[\sym_d]$-module. From this, the result reduces easily to showing that, for $\lambda \vdash d$, $\alpha_d S_\lambda$ is simple.

By construction, $\alpha_d S_\lambda$ is simple modulo composition factors of polynomial degree $<d$, since $\cre_d (\alpha_d S_\lambda)  \cong S_\lambda$ is simple. Now, if $G$ has polynomial degree less than $d$, $\hom_{\fpolyQ{d}{\gr}} (\alpha_d S_\lambda, G)=0$ by adjunction, since $\cre_d G=0$, by hypothesis.

It follows that there is an exact sequence 
\[
0
\rightarrow 
p_{d-1}^\gr \alpha_d S_\lambda 
\rightarrow 
\alpha_d S_\lambda 
\rightarrow 
\overline{\alpha_d S_\lambda }
\rightarrow 
0,
\]
where $\overline{\alpha_d S_\lambda }$ is simple. It remains to show that $p_{d-1}^\gr \alpha_d S_\lambda  =0$. We provide a direct, elementary argument.

Suppose otherwise, then there exists $t<d$ such that $p_{d-1}^\gr \alpha_d S_\lambda$ has polynomial degree $t$ and 
 $\cre_t (p_{d-1}^\gr \alpha_d S_\lambda) \neq 0$. Hence, by adjunction,  there exists a non-trivial map $\aQ^{\otimes t} \rightarrow p_{d-1}^\gr \alpha_d S_\lambda$. Composing with the inclusions gives the non-trivial composite:
\[
\aQ^{\otimes t} \rightarrow p_{d-1}^\gr \alpha_d S_\lambda \hookrightarrow \alpha_d S_\lambda \hookrightarrow \aQ^{\otimes d}.
\]
This is a contradiction, since $\hom_{\fcatQ[\gr]} (\aQ^{\otimes t}, \aQ^{\otimes d})$ is zero unless $t =d$ (for example, this follows from  \cite[Theorem 1]{V_ext}).
\end{proof}

\begin{exam}
\label{exam:aQ_otimes_d}
By the description of $\alpha_d$ given in Theorem \ref{thm:recollement_poly_d}, one has the isomorphism $\alpha_d \rat [\sym_d] \cong \aQ ^{\otimes d}$. Since the functor $\alpha_d$ is exact, applying it to the decomposition (\ref{221005-1356}) gives:
\begin{equation} 
\label{221009-1731}
\aQ ^{\otimes d} \cong \bigoplus_{\lambda \vdash d} (\alpha_d S_\lambda)^{\oplus \dim (S_\lambda)}
\end{equation}
where $\alpha_d S_\lambda$ is simple of polynomial degree exactly $d$. 
Hence, $\aQ ^{\otimes d}$ is semisimple with all composition factors of polynomial degree exactly $d$. 

More generally, for $M$ a $\sym_d$-module, $\alpha_d M$ is semisimple with all composition factors of polynomial degree exactly $d$. 
\end{exam}

Working over $\rat$, using the above arguments,  we can refine Theorem \ref{thm:recollement_poly_d} and related results of 
\cite{DV} as follows.

\begin{nota}
\label{nota:qhat}
\nomenclature{$\qhat{d}F$}{$d$th subquotient of the polynomial filtration of $F$\nomrefpage}
For $F \in \ob \fcatQ[\gr]$ and $d \in \nat$, denote by $\qhat{d}F$  the kernel of the natural surjection $q_d^\gr  F
\twoheadrightarrow 
q_{d-1}^\gr F$.
\end{nota}

\begin{prop}
\label{prop:exact_sequences_recollement} 
For $d \in \nat$ and $F \in \ob \fcatQ[\gr]$, there  is a natural exact sequence:
$$
0
\rightarrow
\alpha_d \cre_d q_d^\gr F
\rightarrow 
q_d^\gr  F
\rightarrow 
q_{d-1}^\gr F
\rightarrow 
0.
$$
Hence
\begin{enumerate}
\item 
there is a natural isomorphism $\qhat{d} F \cong \alpha_d \cre_d q_d^\gr F$;
\item \label{prop:exact_sequences_recollement-2} 
$\qhat{d} F$ is semi-simple, consisting of the composition factors of $q_d^\gr F$ of  polynomial degree exactly $d$;
\item 
\label{item:q_qhat}
the inclusion $\qhat{d} F\subset q_d^\gr F$ induces an isomorphism
$\cre_d  \qhat{d} F\cong  \cre_d q_d^\gr F$.
\end{enumerate}
\end{prop}

\begin{proof}
This proof is closely related to that of \cite[Corollaire 3.7]{DV}. That reference provides the exact sequence 
\[
\alpha_d \cre_d q_d^\gr F
\rightarrow 
q_d^\gr  F
\rightarrow 
q_{d-1}^\gr F
\rightarrow 
0.
\]
Here, the functor $\alpha_d \cre_d q_d^\gr F$ is semi-simple, with all composition factors of polynomial degree exactly $d$, as a consequence of Corollary \ref{cor:alpha_simple}.

The injectivity of the first map follows since $\cre_d (\alpha_d \cre_d q_d^\gr F
\rightarrow 
q_d^\gr  F)$ identifies as the identity morphism of $\cre_d q_d^\gr  F$ and $\cre_d$ detects composition factors of degree exactly $d$, by the recollement framework of Theorem \ref{thm:recollement_poly_d}.

The remaining statements follow easily.
\end{proof}

Proposition \ref{prop:exact_sequences_recollement} implies  that, when restricted to polynomial functors, the functors $q^\gr_d$ are exact:

\begin{prop}
\label{prop:unimodularity_char_0}
For $d \in \nat$, the functor 
$q^\gr_d : \fpiQ{\gr}\rightarrow \fpolyQ{d}{\gr}$ is exact, hence so is the functor 
$\qhat{d} = \alpha_d \cre_d q_d^\gr$.
\end{prop}

\begin{proof}
By a straightforward induction, it suffices to show that, if $i : F \hookrightarrow G$ is a monomorphism in $\fpolyQ{D}{\gr}$, for some $D \in \nat$, then $q_{D-1}^\gr F \rightarrow q_{D-1}^\gr G$ is injective.

Consider the morphism (induced by $i$) of the short exact sequences provided by Proposition \ref{prop:exact_sequences_recollement}:
\[
\xymatrix{
0
\ar[r]
&
\alpha_D \cre_D F 
\ar[r]
\ar@{^(->}[d]
&
F
\ar@{^(->}[d]^i
\ar[r]
&
q_{D-1}^\gr F 
\ar[r]
\ar[d]
&
0
\\
0 
\ar[r]
&
\alpha_D \cre_D G 
\ar[r]
&
G
\ar[r]
&
q_{D-1}^\gr G 
\ar[r]
&
0.
}
\]

The left hand square is Cartesian (i.e., $\alpha_D \cre_D F = F \cap (\alpha_D \cre_D G)$), since $\alpha_D \cre_D F$ (respectively $\alpha_D \cre_D G$) is the largest subfunctor of $F$ (resp. $G$) that is semi-simple with composition factors of degree exactly $D$, by Proposition \ref{prop:exact_sequences_recollement}. It follows that the right hand vertical morphism is injective, as required. 
\end{proof}

\subsection{The polynomial filtration}
\label{subsect:poly_filt}
In this section, we recall the polynomial filtration of a functor on $\gr$ as introduced in \cite{DV}. This is given an elegant characterization in Proposition \ref{prop:canon_poly_filt}.

\begin{defn}
\label{defn:poly_filt}
For $F$ in $\f (\gr; \rat)$, the polynomial filtration of $F$ is the descending filtration associated to the  tower under $F$:
\[
\ldots 
\twoheadrightarrow 
 q_t^\gr F 
\twoheadrightarrow 
q_{t-1}^\gr F 
\twoheadrightarrow 
q_{t-2}^\gr F
\twoheadrightarrow 
\ldots 
\twoheadrightarrow 
q_0^\gr F
\twoheadrightarrow 
q_{-1}^\gr F =0. 
\]

If  $F$ has  polynomial degree $d$, this tower reduces to
\[
F= q_d^\gr F 
\twoheadrightarrow 
q_{d-1}^\gr F 
\twoheadrightarrow 
q_{d-2}^\gr F
\twoheadrightarrow 
\ldots 
\twoheadrightarrow 
q_0^\gr F
\twoheadrightarrow 
q_{-1}^\gr F =0. 
\]  
and the polynomial filtration is taken to be the associated increasing filtration $0=F_{-1} \subset F_0 \subset F_1 \subset \ldots \subset F_d = F$ given by $F_t:= \ker \{ F\twoheadrightarrow q_{d-t-1}^\gr F\}$.
\end{defn}

The subquotients of the polynomial filtration are described by:

\begin{prop}
\label{prop:subquotients}
For $F$ a polynomial functor of degree $d$, the subquotient $F_i/F_{i-1}$ of the polynomial filtration is isomorphic to $\qhat{d-i} F\cong \aQ ^{\otimes d-i} \otimes_{\sym_{d-i}} \cre_{d-i} q_{d-i}^\gr F$. In particular, $F_i/F_{i-1} $ is semi-simple with all composition factors having polynomial degree exactly $d-i$. 
\end{prop}

\begin{proof}
The isomorphism between $F_i/F_{i-1}$ and $\qhat{d-i} F$ follows directly from the definitions of the polynomial filtration and  of the functors $\qhat{t}$. The explicit description of $\qhat{d-i} F$ is obtained by combining   Theorem \ref{thm:recollement_poly_d} (\ref{thm:recollement_poly_d-3}) and Proposition \ref{prop:exact_sequences_recollement}. The second statement is then a consequence of Corollary \ref{cor:alpha_simple}.
\end{proof}

The polynomial filtration has the following characterization, that follows from Proposition \ref{prop:exact_sequences_recollement}.

\begin{prop}
\label{prop:canon_poly_filt}
The polynomial filtration of $F \in \ob \fpolyQ{d}{\gr}$, for $d \in \nat$, is the unique increasing filtration 
\[
0=F_{-1} \subset F_0 \subset F_1 \subset \ldots \subset F_d = F,
\]
such that, for each $t\geq 0$,   all composition factors of $F_t/ F_{t-1}$ have polynomial degree exactly  $d-t$.
\end{prop}

\begin{proof}
Suppose given a filtration of $F$ satisfying the given hypotheses. 

By hypothesis, all the composition factors of $F_0$ have polynomial degree exactly $d$. So Proposition  \ref{prop:exact_sequences_recollement} 
(\ref{prop:exact_sequences_recollement-2}) gives the morphism of short exact sequences 
\[
\xymatrix{
0
\ar[r]
&
F_0 
\ar[r]
\ar@{^(->}[d]
&
F
\ar@{=}[d]
\ar[r]
&
F/F_0
\ar[r]
\ar@{->>}[d]
&
0
\\
0
\ar[r]
&
\qhat{d}F
\ar[r]
&
F
\ar[r]
&
q_{d-1}^\gr F
\ar[r]
&
0.
}
\]
The hypotheses on the $F_i/F_{i-1}$,  $i>0$, imply that $F/F_0$ has polynomial degree $\leq d-1$, hence the surjection $F\twoheadrightarrow F/F_0$ factorizes across $q_{d-1}^\gr F \twoheadrightarrow F/F_0$. 

It follows that  the right hand vertical arrow is an isomorphism $F/F_0 \stackrel{\cong}{\rightarrow} q_{d-1}^\gr F$, hence the left hand arrow also, by the five lemma. By an obvious induction upon $d$, one concludes.
\end{proof}

\subsection{The socle filtration and the polynomial filtration}
\label{subsect:poly_filt_soc}

In this section, we prove that the polynomial filtration for functors taking values in $\rat$-vector spaces is closely related to the socle filtration of the functor.

Recall the socle and the socle filtration:

\begin{defn}
\label{defn:socle_length}
\nomenclature{$\soc(X)$}{socle of $X$\nomrefpage}
\nomenclature{$\soc_n(X)$}{$n$th term of the socle filtration of $X$\nomrefpage}
\nomenclature{$\lsoc(X)$}{socle length of $X$\nomrefpage}
For $X$ an object of an abelian category,
\begin{enumerate}
\item
the socle of $X$, $\soc(X)$, is the largest semi-simple subobject of $X$; 
\item 
the (increasing) socle filtration $\soc_n (X)$ of $X$ is defined recursively by 
$\soc_0 X=0$ and $\soc_{n+1} (X)$ is the preimage of $\soc (X/ \soc_n (X))$; 
\item 
the socle length of $X$,  $\lsoc(X)\in \nat \cup \{\infty\}$,  is 
$\inf \{ n | \soc _n  (X) = X  \}$. 
\end{enumerate}
\end{defn}

From the  recollement situation of Theorem \ref{thm:recollement_poly_d}, the counit $\cre_d \beta_d \to \id$ is an isomorphism. Forming  the adjoint to $\id \rightarrow \cre_d \beta_d$ gives the natural transformation $\alpha_d \rightarrow \beta_d$, denoted here by $N$. The recollement framework implies that the natural transformation $\cre_d N : \cre_d \alpha_d \rightarrow \cre_d \beta_d$ identifies as the identity morphism. 

\begin{prop}
\label{prop:beta_socle}
For $M \in \ob \rat[\sym_d]\dash\modules$ and $F \in \ob \fpolyQ{d-1}{\gr}$, where $d \in \nat$,
\[
\hom_{\fcatQ[\gr]}(F, \beta_d M) =0.
\]
Moreover, the morphism $N_M: \alpha_d M \rightarrow \beta_d M$ is injective and identifies $\alpha_d M $ as the socle of $\beta_d M$.
\end{prop}

\begin{proof}
The first statement follows by adjunction, since $\cre_d F=0$, by hypothesis.

The kernel of $N_M : \alpha_d M \rightarrow \beta_d M$ is a subfunctor of $\alpha_d M$, hence is semi-simple with all composition factors of polynomial degree exactly $d$, using Corollary \ref{cor:alpha_simple} for the structure of $\alpha_d M$. As recalled above, applying the exact functor  $\cre_d$ to the morphism $N_M: \alpha_d M \to \beta_d M$ gives the identity morphism of $M$. It follows that $\cre_d (\ker N_M) =0$. Since $\ker (N_M)$ is semi-simple with all composition factors of degree exactly $d$, this implies that $\ker (N_M)=0$.

Consider the short exact sequence $0\rightarrow \alpha_d M \rightarrow \beta_d M \rightarrow C \rightarrow 0$, thus defining $C$. Applying the exact functor $\cre_d$ gives $\cre_d C=0$, hence  $C$ has polynomial degree at most $d-1$. Applying the socle functor $\soc(-)$, which is left exact,  gives the exact sequence
\[
0\rightarrow \soc(\alpha_d M) = \alpha_d M \rightarrow \soc(\beta_d M)  \rightarrow  \soc(C).
\]

Now, by the first statement, if $S$ is simple of polynomial degree strictly less than $d$, $\hom _{\fcatQ[\gr]}(S, \soc (\beta_d M)) = \hom _{\fcatQ[\gr]}(S, \beta_d M)=0$. Since $ \soc(C)$ has polynomial degree at most $d-1$, it follows that $\soc ({\beta_d M})$ has no contribution from $\soc(C)$, whence the result.
\end{proof}

We deduce the following result describing the injective envelope of a simple polynomial functor:

\begin{cor}
\label{cor:inj_envelope}
For $\lambda \vdash d$, the morphism
$N_{S_\lambda} :\alpha_d S_\lambda \rightarrow \beta_d S_\lambda$ exhibits 
$\beta_d S_\lambda$ as the injective envelope of the simple $\alpha_d S_\lambda$ in $\f_d (\gr;\rat)$.
\end{cor} 

\begin{proof}
The functor $\alpha_d S_\lambda$ is simple by Corollary \ref{cor:alpha_simple} and the functor $\beta_d S_\lambda$ is injective in $\f_d (\gr;\rat)$, by Theorem \ref{thm:recollement_poly_d}. Proposition \ref{prop:beta_socle} shows that the morphism $N_{S_\lambda}: \alpha_d S_\lambda \rightarrow \beta_d S_\lambda$ exhibits $\alpha_d S_\lambda$ as the socle of $\beta_d S_\lambda$. It follows that  $\beta_d S_\lambda$ is the injective envelope of $\alpha_d S_\lambda$.
\end{proof}

\begin{rem}
Corollary \ref{cor:inj_envelope} can be strengthened by Theorem \ref{thm:injectivity_beta} below, showing that $\beta_d S_\lambda$ is the injective envelope of $\alpha_d S_\lambda$ in $\f_{< \infty} (\gr; \rat)$.
\end{rem}

The understanding of the polynomial filtration provided by Propositions \ref{prop:canon_poly_filt} can be made even more precise using: 

\begin{prop}
\label{prop:Ext1_poly_char0}
For $S, T$ simple polynomial functors of $\fcatQ[\gr]$ of polynomial degrees $|S|$, $|T|$ respectively, 
\[
\ext^* _{\fcatQ[\gr]} (S, T) =0 \mbox{\ if \ } * \neq |T|-|S|.
\]
In particular $\ext^1 _{\fcatQ[\gr]} (S, T) =0$ if $|S|\neq |T|-1$.
\end{prop}

\begin{proof}
The result follows from the main Theorem of \cite{V_ext} which implies that 
$
\ext^* _{\fcatQ[\gr]} (\aQ^{\otimes n}, \aQ ^{\otimes m}) =0
$ 
if $* \neq m-n$ and using (\ref{221009-1731}).
\end{proof}

We first note the following immediate consequence of the form of the polynomial filtration given by Proposition \ref{prop:canon_poly_filt}.

\begin{lem}
\label{lem:poly_deg_F_socF}
For $F$ a polynomial functor, the exact polynomial degree of $F$ is equal to the exact polynomial degree of $\soc (F)$.
\end{lem}

\begin{proof}
The polynomial degree of $\soc (F)$ is clearly at most that of $F$. 

To establish the reverse inequality,  suppose that $F$ has polynomial degree exactly $d$. The polynomial filtration shows that $\qhat{d}F \neq 0$ and is a subfunctor of $F$. Since $\qhat{d}F$ is semi-simple, it is contained in $\soc (F)$, which therefore has polynomial degree at least $d$.
\end{proof}

Moreover, one has:

\begin{lem}
\label{lem:soc_degree}
For $0 \neq F \in \ob \f_d (\gr; \rat)$, for $d \in \nat$. The following conditions are equivalent:
\begin{enumerate}
\item 
$\soc (F) \neq 0$ and each composition factor of $\soc (F)$ has polynomial degree exactly $d$; 
\item 
$\cre_d F \neq 0$ and $\soc (F) \cong \alpha_d \cre_d F $; 
\item 
the adjunction unit $F \hookrightarrow \beta_d \cre_d F$ is injective.
\end{enumerate}
In particular, $F$ has polynomial degree exactly $d$.
\end{lem}

\begin{proof}
The equivalence of the first two conditions is clear. 

By Proposition \ref{prop:beta_socle},  $N_{\cre_d F} : \alpha_d \cre_d F \rightarrow \beta_d \cre_d F$ is injective. Hence, if $\soc (F) \cong \alpha_d \cre_d F$ one deduces that the composite $\soc (F) \subset F \rightarrow \beta_d \cre_d F$ is injective. This implies  the third condition. 

Now suppose the third condition holds; this implies that   $ \soc (F) \rightarrow \soc (\beta_d \cre_d F)$ is  injective. By Proposition \ref{prop:beta_socle}, $\soc (\beta_d \cre_d F)$ is equal to $\alpha_d \cre_d F$, which is semi-simple with each composition factor of degree exactly $d$. It follows that $\soc (F)$ satisfies the first condition.
\end{proof}

\begin{prop}
\label{prop:socle_filtration_char_0}
Suppose that  $0 \neq F \in \ob \f_d(\gr; \rat)$ satisfies the equivalent conditions of Lemma \ref{lem:soc_degree}. Then 
$$
\soc _n (F)
= 
\ker \big(
F \rightarrow q_{d-n}^\gr F 
\big)
$$
 for each $n \in \nat$. Thus, up to reindexing, the socle filtration of $F$  coincides with the polynomial filtration of $F$. 
\end{prop}

\begin{proof}
By hypothesis,  $F$ has polynomial degree exactly $d$ and $\soc (F) \cong \alpha_d \cre_d F$. Proposition \ref{prop:exact_sequences_recollement}  therefore implies that $F/\soc (F)$ is isomorphic to $q_{d-1}^\gr F$. If this is $0$, there is nothing to prove, so we suppose otherwise.

Then, by an evident induction upon $d$, it suffices to show that $q_{d-1}F$ satisfies the equivalent conditions of Lemma \ref{lem:soc_degree} with respect to $d-1$. In particular, this will imply that $\soc (q_{d-1}^\gr F) = \qhat{d-1} (F)$ consists of simple functors of polynomial degree exactly $d-1$.

Suppose otherwise, i.e., that $\soc (q_{d-1}^\gr F)$ has a composition factor $S$  of polynomial degree $< d-1$. This would imply the existence of a non-split short exact sequence of the form
\[
0
\rightarrow \soc (F)
\rightarrow \mathscr{E}
\rightarrow S 
\rightarrow 0.
\]
This  contradicts  Proposition \ref{prop:Ext1_poly_char0}, thus  completing the proof of the inductive step.
\end{proof}

\begin{cor}
\label{exam:beta_M_socle_filt}
For any $d \in \nat$ and $M \neq 0$ a  $\rat [\sym_d]$-module, up to reindexing, the socle filtration of $\beta_d M$ coincides with the polynomial filtration of $\beta_d M$.
\end{cor}

\begin{proof}
By Proposition \ref{prop:beta_socle}, $\beta_d M$ satisfies the second condition of  Lemma \ref{lem:soc_degree} with respect to $d \in \nat$, hence the result follows from Proposition \ref{prop:socle_filtration_char_0}.
\end{proof}

\subsection{(Co)homological properties for polynomial functors}
\label{subsect:injectivity_beta_char_0}

Here we recall some cohomological properties of the category of polynomial functors. 

By Theorem \ref{thm:recollement_poly_d} and since $\rat[\sym_d]$ is semi-simple, for any $M \in \ob \rat [\sym_d]\dash\modules$, $\beta_dM$ is  injective in $\fpolyQ{d}{\gr}$. Much more is true: 

\begin{thm}
\label{thm:injectivity_beta}
For any $d \in \nat$ and $M \in \ob \rat [\sym_d]\dash\modules$, $\beta_d M$ is injective in $\fpiQ{\gr}$. 
\end{thm}

\begin{proof}
By \cite[Théorème 1]{DPV}, if $F, G \in \ob \fpolyQ{D}{\gr}$ for some $D \in \nat$, then 
the canonical morphism $\ext^* _{\fpolyQ{D}{\gr}}(F, G) \rightarrow \ext^* _{\fcatQ[\gr]}(F, G)$ is an isomorphism. 

Hence, to prove the result,  it suffices to prove that 
$
\ext^1_{\fcatQ[\gr]} (S, \beta_d M) = 0
$ 
for any simple polynomial functor $S \in \fpiQ{\gr}$. Since, $\beta_d M$ is injective in $\fpoly{d}{\gr}$, we may suppose that $S$ has polynomial degree strictly greater than $d$. The result now follows from Proposition \ref{prop:Ext1_poly_char0}.
\end{proof}

The following is then clear:

\begin{cor}
\label{cor:poly_enough_injectives}
The family $\{\beta_d \rat[\sym_d] \ | \ d\in \nat\}$ is a family of injective cogenerators for 
 $\fpiQ{\gr}$.
\end{cor}

Recall that the global (injective) dimension of an abelian category with enough injectives is the supremum of the injective dimensions of its objects.

The following result underlines the essential homological difference between polynomial functors on $\ab$ and on $\gr$.

\begin{thm}
\cite[Section 4]{DPV}
\label{thm:fcatk_ab_semisimple}
\begin{enumerate}
\item 
The category $\fpiQ{\ab}$ is semi-simple; 
\item 
For $0<d\in \nat$, the category $\fpolyQ{d}{\gr}$ has global dimension $d-1$.
\item 
The category $\fpiQ{\gr}$ does not have finite global dimension.
\end{enumerate}
\end{thm}

Thus, although the categories $\fpiQ{\gr}$ and $\fpiQ{\ab}$ have the same simple objects, the structure of  $\fpiQ{\gr}$  is much richer than that of $\fpiQ{\ab}$. A similar statement holds comparing $\fpolyQ{d}{\gr}$ and
$\fpolyQ{d}{\ab}$, for $d>1$.

\section{$\fb$-modules and Schur functors}
\label{sec:fb}

\nomenclature{$\fmodq$}{finite-dimensional $\rat$-vector spaces\nomrefpage}
The purpose of this section is to review the category of $\fb$-modules and the relationship through Schur functors with the functor category $\f(\fmodq;\rat)$, where $\fmodq$ is the category of finite-dimensional $\rat$-vector spaces. No claim to originality is made.
 
This theory is necessary so as to introduce the Hopf algebra structures (see Section \ref{sec:pcoalg}) that underlie the exponential functors that are used in Section \ref{sec:beta} to describe the functors $\beta_d$ introduced in Section \ref{sec:recoll}.
 
A further important ingredient is the functor $(-)^\dagger$ that allows Koszul signs to be treated 
 
\subsection{$\fb$-modules and their basic structure}

\nomenclature{$\fb$}{finite sets and bijections\nomrefpage}
The category of finite sets and bijections is denoted $\fb$. This has a small skeleton with objects $\mathbf{n}= \{1, \ldots , n\}$, $n \in \nat$ (by convention $\mathbf{0} = \emptyset$). 
Hence there is an equivalence of categories 
\begin{eqnarray}
\label{eqn:equiv_fb}
\f(\fb;\rat)
\cong 
 \prod _{d \geq 0} \rat [\sym_d]\dash\modules. 
\end{eqnarray}
In particular, a $\rat [\sym_n]$-module can be considered as an object of 
$\fcatk[\fb]$ that is supported on $\mathbf{n}$.

Since $\fb$ is a groupoid, the inverse induces an isomorphism of categories $\fb \op \cong \fb$, whence 
\begin{eqnarray}
\f(\fb; \rat)\cong \f(\fb \op; \rat).
\end{eqnarray}
 
The category $\fb$ is symmetric monoidal for  $\amalg$ the disjoint union of 
sets.  

\begin{rem}
\label{rem:dist_iso_fb}
For $m, n \in \nat$, there is the following standard choice of isomorphism
 $ 
  \mathbf{m} \amalg \mathbf{n} \rightarrow \mathbf{m+n},
 $ 
sending $i \in \mathbf{m}$ to $i \in \mathbf{m+n}$ and $j \in \mathbf{n}$ to 
$m+j \in \mathbf{m+n}$. The object $\mathbf{0}$ plays the role of a left and right unit for this construction, in the obvious sense.

This construction is associative in the following sense: for $m,n,q \in \nat$ the composites 
$(\mathbf{m} \amalg \mathbf{n}) \amalg \mathbf{q}  \rightarrow \mathbf{m+n} \amalg \mathbf{q} 
\rightarrow \mathbf{m+n+q}$ and $\mathbf{m} \amalg (\mathbf{n} \amalg \mathbf{q})  \rightarrow \mathbf{m} \amalg  \mathbf{n+q} \rightarrow \mathbf{m+n+q}$ identify. However, it is clearly not commutative.
\end{rem}

\begin{exam}
 For $n \in \nat$, the associated standard projective is $P^\fb _{\mathbf{n}} : S 
\mapsto \rat [\hom_\fb (\mathbf{n}, S)]$.
 Hence $P^\fb _{\mathbf{n}}$ is non-zero only if $|S|=n$, when the 
corresponding 
representation of the symmetric group is 
 the regular representation.
\end{exam}

\begin{nota}
\label{nota:unit}
\nomenclature{$\unit$}{unit in $\f (\fb; \rat)$\nomrefpage}
Write $\unit $ for $P^\fb _{\mathbf{0}}$; this is the functor $\mathbf{0} \mapsto \rat$ and $\mathbf{n} \mapsto 0$, $n >0$. 
\end{nota}

The category $\f(\fb; \rat)$ is tensor abelian for  the Day convolution product:

\begin{defn}
\label{defn:tenfb}
\nomenclature{$\tenfb$}{Day convolution product for $\f (\fb;\rat)$\nomrefpage}
 Let $(\f(\fb;\rat) , \tenfb,\unit)$ denote the monoidal
structure given, for $V , W \in \ob \f(\fb;\rat)$, by
\[
 (V \tenfb W) (S) := \bigoplus_{S = S_1 \amalg S_2} V(S_1) \otimes_\rat W 
(S_2),
\]
for $S \in \ob \fb$, where the sum is over ordered decompositions of $S$ into two 
subsets (possibly empty).
\end{defn}

\begin{rem}
\label{rem:smod_reps}
 Under the equivalence of categories (\ref{eqn:equiv_fb}), an object $V$ of $\f(\fb;\rat)$ is considered as the sequence 
$V(d):= V(\mathbf{d})$, $d \in \nat$, of representations of the symmetric groups. 
Under this interpretation, one has 
\[
 (V \tenfb W) (n) 
 = 
 \bigoplus _{n_1 + n_2 = n}
 V (n_1) \otimes W (n_2) \uparrow _{\sym_{n_1} \times \sym_{n_2}}^{\sym_n},
\]
where the induction uses the inclusion $\sym_{n_1} \times \sym_{n_2} \subset 
\sym_n$
induced by $\mathbf{n_1} \amalg \mathbf{n_2} \cong \mathbf{n}$.
 \end{rem}
 
The Day convolution product can be interpreted as being induced by $\amalg$  on $\fb$, by the following:

\begin{prop}
\label{prop:proj_smod_tenfb}
For finite sets $T_1, T_2$, $\amalg$ induces an isomorphism
 \[
 P^\fb _{T_1}
 \tenfb
 P^\fb _{T_2}
 \stackrel{\cong}{\rightarrow}
 P^\fb _{T_1 \amalg T_2}.
 \]
This makes $T \mapsto P^\fb_T$ a monoidal functor from $(\fb \op , \amalg , \mathbf{0})$ to $\f(\fb; \rat)$.  
\end{prop}

\begin{proof}
Giving a bijection $T_1 \amalg T_2 \cong S$ is equivalent to giving a pair of subsets $S_1, S_2 \subset S$ and two bijections $T_1 \cong S_1$ and $T_2 \cong S_2$. One checks that this correspondence induces the given natural isomorphism. 

The monoidal property is then established by a direct verification.
\end{proof}

\begin{exam}\label{exam:tenfb_P1}
The functor $-\tenfb P^\fb_{\mathbf{1}}$ is an exact endofunctor of $\f(\fb;\rat)$. Under the equivalence (\ref{eqn:equiv_fb}), this has components given by the induction functors 
\[
(-)\uparrow_{\sym_d}^{\sym_{d+1}}
 : 
 \rat[\sym_d]\dash\modules 
 \rightarrow 
 \rat[\sym_{d+1}]\dash\modules.
 \] 
\end{exam}

\subsection{Symmetries and twisting}
\label{subsect:dagger}

The category $(\f(\fb;\rat) , \tenfb, \unit)$ has two possible 
symmetries, according to whether the symmetry introduces Koszul signs. These exploit  
the usual symmetric monoidal  structure of $\fmodq$, which provides the 
isomorphism
\begin{eqnarray}
\label{eqn:twist_vs}
 V(S_1) \otimes W (S_2) \stackrel{\cong}{\rightarrow} W(S_2) \otimes V(S_1).
\end{eqnarray}

\begin{defn}
\nomenclature{$\tau$}{symmetry for $\tenfb$\nomrefpage}
\nomenclature{$\sigma$}{Koszul-signed symmetry for $\tenfb$\nomrefpage}
\label{defn:tenfb_symms}
 Let
 \begin{enumerate}
  \item 
 $(\f(\fb;\rat), \tenfb,\unit, \tau)$ be the symmetric monoidal structure, with $\tau$ 
induced by  (\ref{eqn:twist_vs}); 
 \item 
 $(\f(\fb;\rat), \tenfb,\unit, \sigma)$ be the symmetric monoidal structure, with $\sigma$ 
induced by  (\ref{eqn:twist_vs}) multiplied by the Koszul sign $(-1)^{|S_1||S_2|}$.  
 \end{enumerate}
\end{defn}

These symmetries are related using the following orientation $\fb$-module:

\begin{defn}
\label{defn:orientation_module}
\nomenclature{$\orient$}{orientation $\fb$-module\nomrefpage}
 Let $\orient \in \ob \f(\fb;\rat)$ denote the orientation $\fb$-module  
 $\orient (S) := \Lambda^{|S|} (\rat [S])$, 
where $\rat [S]$ is the free $\rat$-vector space on the finite set $S$ and $\Lambda^{|S|}$ is the top exterior power.
\end{defn}

When $S$ is equipped with a total order, $\orient (S) \cong \rat$ has a canonical generator:

\begin{nota}
\label{nota:iota}
For $n \in \nat$,  denote by $\iota_n$ the generator of $\orient (\mathbf{n}) \cong \Lambda^n(\rat [\mathbf{n}]) \cong \rat$ given by $[1] \wedge [2] \wedge \ldots \wedge [n]$, where $[i]$ is the  generator of $\rat [\mathbf{n}]$ corresponding to $i \in \mathbf{n}$.
\end{nota}

We will require compatibility between the orientation modules $\orient (\mathbf{n})$. The required information is encoded in  a bicommutative Hopf algebra structure on $\orient$. In the statement,  we use the  distinguished isomorphism $\mathbf{m}\amalg \mathbf{n} \cong \mathbf{m+n}$ of Remark \ref{rem:dist_iso_fb}:

\begin{prop}
\label{prop:orient_(co)mult}
The $\fb$-module $\orient$ has the structure of a bicommutative Hopf algebra in $(\f(\fb;\rat),\tenfb, \unit, \sigma)$ that is uniquely determined by the following conditions for each $m,n  \in \nat^2$:
\begin{enumerate}
\item 
the map $\orient (\mathbf{m} ) \otimes \orient(\mathbf{n}) \rightarrow \orient (\mathbf{m}\amalg \mathbf{n}) \cong \orient (\mathbf{m+n})$ given by the product $\orient \tenfb \orient \rightarrow \orient $ sends $\iota_m \otimes \iota_n \mapsto \iota_{m+n}$;
\item 
the map $ \orient (\mathbf{m+n}) \cong  \orient (\mathbf{m}\amalg \mathbf{n}) \rightarrow \orient (\mathbf{m} ) \otimes \orient(\mathbf{n})$ given by the coproduct $\orient \rightarrow \orient \tenfb \orient$  sends $\iota_{m+n}\mapsto \iota_m \otimes \iota_n $. 
\end{enumerate}
\end{prop}

\begin{rem}
This result can be understood using the Schur functor associated to  $\orient$ (see Section \ref{subsect:schur_correspond} for Schur functors). The above Hopf algebra induces  the usual bicommutative Hopf algebra structure on the exterior algebra $\Lambda^* (V)$, considered as a functor of $V$. 
\end{rem}

The desired relation between the two symmetric monoidal structures on $\f (\fb;\rat)$  is given by using the objectwise tensor product (as in Section \ref{subsect:functor}) with $\orient$:

\begin{thm}
\label{thm:equiv_dagger}
\cite[Proposition 7.4.3]{SS}
\nomenclature{$(-)^\dagger$}{twisting functor for $\f(\fb;\rat)$\nomrefpage}
 The functor $- \otimes \orient : \f(\fb; \rat) \rightarrow \f(\fb; \rat)$ induces an 
equivalence of symmetric monoidal 
 categories 
 \[
 (-)^\dagger :  (\f(\fb; \rat) , \tenfb,\unit, \tau ) \stackrel{\cong}{\rightarrow} (\f(\fb; \rat) , 
\tenfb,\unit, \sigma).
 \]
 \end{thm}

\begin{exam}
\label{exam:dagger_const}
Consider the constant functor $\rat$ in $\f(\fb; \rat)$. This has an obvious bicommutative Hopf algebra structure in $(\f(\fb;\rat),\tenfb, \unit, \tau)$. On applying  $(-)^\dagger$ to $\rat$, one recovers $\orient$  as a bicommutative Hopf algebra in $(\f(\fb;\rat),\tenfb, \unit, \sigma)$, as in Proposition \ref{prop:orient_(co)mult}.
\end{exam}

 The following Example explains the interpretation of the functor $(-)^\dagger$ in terms of representations of the symmetric groups and also motivates the notation.

\begin{exam}
\label{exam:dagger_partitions}
Let $S_\lambda$ denote the 
simple $\rat [\sym_n]$-module indexed by the 
partition $\lambda$, considered as an object of $\f(\fb; \rat)$. Then 
$(S_\lambda)^\dagger \cong S_{\lambda^\dagger}$, where $\lambda^\dagger$ 
denotes the conjugate partition.  (See Appendix  \ref{app:rep_sym_car0} for background.)
\end{exam}

\subsection{Duality for $\fb$-modules}

Our constructions use Hopf algebras in  the symmetric monoidal categories 
$(\f(\fb; \rat),\tenfb , \unit, \tau)$ 
(respectively with $\sigma $ in place of $\tau$), in particular the categories $\hcom(\f(\fb;\rat))$ of commutative Hopf algebras and $\hcocom (\f(\fb;\rat))$ of cocommutative Hopf algebras (with respect to the specified symmetric monoidal structure). Duality for $\fb$-modules allows us to relate these and also to relate the associated exponential functors.

We use the duality functor $D_{\fb}$ of Notation \ref{nota:duality}; by definition, this is a functor 
$ 
\f(\fb;\rat) \op \rightarrow \f(\fb \op;\rat)$.
However, using the isomorphism of categories $\f(\fb;\rat) \cong \f(\fb\op;\rat)$, this is considered here as a functor 
$$
D_{\fb}
:
\f(\fb;\rat) \op \rightarrow \f(\fb;\rat)
.
$$
Explicitly, for $F$ a $\fb$-module, the dual module $D_{\fb}F$ is given on $\mathbf{n}$ by 
$ D_{\fb}F (\mathbf{n}) = \hom_{\fmodq} (F(\mathbf{n}), \rat)$, considered as a {\em left} $\sym_n$-module.

The following Proposition follows from the properties of vector space duality with respect to the tensor product.

\begin{prop}
\label{prop:duality_smod_tenfb}
Suppose that $V, W \in \ob \f(\fb;\rat)$ take finite-dimensional values. Then the duality adjunction of Proposition \ref{prop:duality} induces a natural isomorphism 
\[
(D_{\fb} V) \tenfb (D_{\fb} W) 
\cong 
D_{\fb} (V \tenfb W) 
\]
that is compatible with the symmetry $\tau$ of $(\f(\fb:\rat), \tenfb, \unit)$. 

Moreover, the functor $(-)^\dagger$ commutes with $D_{\fb}$.
\end{prop}

\subsection{The Schur functor construction}
\label{subsect:schur_correspond}

This section reviews the Schur functor construction that underlies the classical Schur-Weyl correspondence. Recall that $\fmodq$ denotes the category of finite-dimensional $\rat$-vector spaces.  

 The Schur construction is the functor
$
\f(\fb; \rat) \rightarrow \f(\fmodq; \rat)$ that sends a $\fb$-module $F$  to 
\[
V \mapsto \bigoplus _{n \in \nat} V^{\otimes n} \otimes_{\sym_n}
 F(\mathbf{n}).
\]
A functor in the image of this construction is termed a {\em Schur functor} here.

\begin{defn}
\label{defn:homog_poly_functor}
A Schur functor associated to a non-zero $\fb$-module supported on $\mathbf{n}$ will be termed a {\em homogeneous polynomial functor of degree $n$}.
\end{defn}

\begin{rem}
\label{rem:homog_poly_fmodq}
Using the general theory of Section \ref{sec:poly}, one has the notion of polynomial functor on $\fmodq$. A homogeneous polynomial functor of degree $n$ is polynomial of degree exactly $n$,  justifying the terminology. Moreover, polynomial functors in $\fcatQ[\fmodq]$ admit a particularly simple description given by \cite[Appendix I.A]{MacD} and \cite[Chapitre 4]{MR927763}: a polynomial functor of degree $n$ is a direct sum of homogeneous polynomial  functors of degree $\leq n$.

The analogous statement for $\fpiQ{\gr}$ is not true: Example \ref{exam:Passi_not_out} gives an example of a polynomial functor which is not the direct sum of homogeneous functors. Proposition \ref{prop:canon_poly_filt} gives the correct general statement for $\fpiQ{\gr}$.
\end{rem}

The Schur functor construction can be encoded using  the functor $\otimes _\fb$ (see Section \ref{subsect:tensor_over_C}) and the following bifunctor:

\begin{nota}
\label{nota:talg*_bifunctor}
\nomenclature{$\talg$}{tensor bifunctor\nomrefpage}
Let  $ \talg \in \ob \fcatQ[\fmodq \times \fb\op]$ denote the bifunctor defined by 
\[
(V, \mathbf{d}) \mapsto \talg^d (V):= V^{\otimes d},
\]
equipped with the place permutation action of $\sym_d$ on the right.
\end{nota}

\nomenclature{$\talg\otimes_\fb -$}{Schur functor construction\nomrefpage}
The Schur functor construction then identifies with the functor 
 $$\talg \otimes_\fb - : \f(\fb; \rat) \rightarrow \f(\fmodq; \rat).$$
One also has the functor $\hom_{\f(\fmodq;\rat)} (\talg, -) : \f (\fmodq; \rat ) \rightarrow \f(\fb;\rat)$. 

Recall that the objectwise tensor product provides the symmetric monoidal structure $(\f(\fmodq;\rat) , \otimes , \rat)$. Then the main properties of these constructions are resumed in:

\begin{prop}
\label{prop:schur_functor}
The functor $\talg \otimes_\fb - : \f(\fb; \rat) \rightarrow \f(\fmodq; \rat)$ is:
\begin{enumerate}
\item 
exact;
\item 
a fully-faithful embedding; 
\item 
symmetric monoidal with respect to $(\f(\fb; \rat) , \tenfb, \unit, \tau)$.
\end{enumerate}

The functor $\hom_{\f(\fmodq;\rat)} (\talg, -)$ is a retract of $\talg \otimes_\fb -$. In particular, for $d\in \nat$ and a $\sym_d$-module $M$, there is a natural isomorphism of $\sym_d$-modules
\[
M \cong \hom_{\f(\fmodq;\rat)} (\talg^d, \talg \otimes_\fb M).
\]
\end{prop}
 
 \begin{proof}
 These standard properties are established in \cite[Appendix I.A]{MacD} and \cite[Chapitre 4]{MR927763}, for example.
 \end{proof}

The endofunctor $(-)^\dagger$ of the category $\f(\fb;\rat)$ can be transported to Schur functors, as explained below. This gives an alternative (and possibly more illuminating) way of understanding $(-)^\dagger$ and the equivalence of symmetric monoidal categories given in Theorem \ref{thm:equiv_dagger}. In particular, it explains why $(-)^\dagger$ can be viewed as introducing Koszul signs.
 
\begin{nota}
\label{nota:dagger_schur}
For $G \in \ob \f(\fmodq; \rat)$ in the image of $\talg \otimes_\fb -$, say $G = \talg \otimes_\fb F$ for a $\fb$-module $F$, write $G^\dagger$ for the functor $\talg \otimes_\fb (F^\dagger)$.
\end{nota}

\begin{rem}
\label{rem:heuristic_dagger}
Heuristically, the operation $(-)^\dagger$ on Schur functors can be understood by introducing a grading and imposing Koszul signs.  This uses the fact that any Schur functor extends naturally to a functor defined on graded vector spaces.

For example, consider the functor  $F(W) := W^{\otimes n} \otimes_{\sym_n} M$, for a $\sym_n$-module $M$, $n \in \nat$. Here, $W$ can be taken to be a graded vector space. In particular, considering $sV$ the vector space $V$ placed in degree one, one has the functor of $V$
\[
V \mapsto F(sV) = (sV)^{\otimes n} \otimes_{\sym_n} M
\]
which takes values concentrated in degree $n$.

Now $(sV)^{\otimes n}$ identifies as $V^{\otimes n} \otimes \mathrm{sgn}_n$ placed in degree $n$, equipped with the diagonal action of $\sym_n$. Hence the functor $V \mapsto F(sV)$  identifies as the Schur functor associated to the $\sym_n$-module $M^\dagger:= M \otimes \mathrm{sgn}_n$, placed in degree $n$. 
\end{rem}

\begin{exam}
The Schur functor associated to the constant $\fb$-module $\kring$ is the functor $V \mapsto \bigoplus_{n\in 
\nat} S^n (V)$, where $S^n$ is the $n$th symmetric power functor. The $\fb$-module $\rat ^\dagger$ identifies with the orientation $\fb$-module $\orient$ and this has associated Schur functor $V \mapsto \bigoplus_{n \in \nat} \Lambda^n (V)$, where $\Lambda^n$ is the $n$th exterior power functor. 

This gives the equalities in $\f(\fmodq;\rat)$, for $n \in \nat$:
\begin{eqnarray*}
(S^n)^\dagger = \Lambda^n \mbox{ and } (\Lambda^n)^\dagger = S^n .
\end{eqnarray*}
In terms of the heuristic interpretation above, this reflects the natural identifications $S^n(sV) \cong \Lambda^n (V) [n] $ and $\Lambda^n (sV) \cong S^n (V)[n]$, where $[n]$ denotes the appropriate degree shift. 
\end{exam}

For $F, G \in \ob \f(\fmodq;\rat)$, if $G$ takes finite-dimensional values, then one can form the composition $F \circ G$ in $\f(\fmodq; \rat)$. For suitable Schur functors, one has:

\begin{prop}
\label{prop:dagger_compose}
\cite[Section 7.4.8]{SS}
For $F$ a  Schur functor and $G$ a homogeneous polynomial functor of degree $d$   that takes 
finite-dimensional values,
\[
(F \circ G)^\dagger \cong
\left\{
\begin{array}{ll}
F 
\circ (G^\dagger) & \mbox{$d$ even}
\\
(F^\dagger) \circ (G^\dagger) &  \mbox{$d$ odd.}
\end{array}
\right.
\]
\end{prop}

\section{Fundamental Hopf algebras in $\fb$-modules and their exponential functors} 
\label{sec:pcoalg}

This section introduces the Hopf algebras $\palg$ and $\pcoalg$ in $\f(\fb;\rat)$. 
Under the Schur correspondence, we have the dictionary:
\begin{eqnarray*}
\palg &\leftrightarrow & \Big( V \mapsto T (V) \Big)\\
\pcoalg & \leftrightarrow & \Big( V \mapsto T_\cog (V) \Big),
\end{eqnarray*}
where $T(V)$ is the cocommutative Hopf algebra given by the tensor algebra with shuffle coproduct and $T_\cog (V)$ is the commutative Hopf algebra given by the tensor coalgebra with shuffle product.

By Theorem \ref{thm:expo_Hopf_general}, we can form the respective exponential functors, $\Phi \palg$ in $\f(\gr\op; \f(\fb; \rat))$ and $\Psi \pcoalg$ in $\f(\gr; \f(\fb;\rat))$. 
 The functors $\Psi \pcoalg$ and $\Phi \palg$  play an essential role in Section \ref{sec:beta}, where the structure of the functors $\beta_d$  of Section \ref{sec:recoll}, for $d \in \nat$, is analysed.

\subsection{The Hopf algebras $\palg$ and $\pcoalg$}

Recall from Notation \ref{nota:unit} that $\unit$ denotes $P^\fb_{\mathbf{0}}$.

 \begin{nota}
 \label{nota:neck}
 \nomenclature{$\neck$}{standard projectives in $\f(\fb; \rat)$\nomrefpage}
Let  $\neck$ denote the $\fb$-module 
$$\neck : = \bigoplus_{n \in \nat} P^\fb_{\mathbf{n}} =   \bigoplus_{n \in \nat} \rat [\hom_\fb (\mathbf{n}, -)]$$
 with
 the associated split inclusion
  \begin{eqnarray}
\label{eqn:eta_epsilon}
\unit \stackrel{\eta}{\hookrightarrow} \neck \stackrel{\epsilon}{\twoheadrightarrow} \unit
\end{eqnarray}
that defines the morphisms $\eta$ and $\epsilon$.
\end{nota}

We now proceed to construct a cocommutative (respectively commutative) Hopf algebra structure on $\neck$, for which $\eta$ and $\epsilon$ provide the (co)augmentation.

The following Lemma is clear; it provides the finiteness property that is required when considering duality.

\begin{lem}
\label{lem:neck_eval_fd}
For $n\in \nat$, $\neck (\mathbf{n})$ is isomorphic to $\rat[\sym_n]$ as a $\sym_n$-module. In particular, $\neck$ takes finite-dimensional values.
\end{lem}

Proposition \ref{prop:proj_smod_tenfb} implies:

\begin{prop}
\label{prop:iso_neck_tenfb}
There is an isomorphism in $\f(\fb; \rat)$:
\[
\neck  \tenfb \neck \cong 
\bigoplus_{(n_1, n_2 ) \in \nat^{\times 2}} P^\fb _{\mathbf{n_1} \amalg \mathbf{n_2}}
= 
\bigoplus_{(n_1, n_2 ) \in \nat^{\times 2}}  \rat  [\hom_\fb (\mathbf{n_1} \amalg \mathbf{n_2}, -)]. 
\]

With respect to this isomorphism, 
\begin{enumerate}
\item 
 a morphism $\neck \tenfb \neck \rightarrow \neck$ is uniquely determined by its components
 $P^\fb _{\mathbf{n_1} \amalg \mathbf{n_2}} \rightarrow P^\fb _{\mathbf{n_1+n_2}}$, for $(n_1,n_2)\in \nat^{ 2}$, all other components being zero;
 \item 
 a morphism $\neck \rightarrow \neck \tenfb \neck $ is uniquely determined by its components
 $ P^\fb _{\mathbf{n_1+n_2}} \rightarrow P^\fb _{\mathbf{n_1} \amalg \mathbf{n_2}} $,  for $(n_1,n_2)\in \nat^{ 2}$, all other components being zero.
\end{enumerate}
\end{prop}

For $(n_1, n_2 ) \in \nat^{\times 2}$, we use the  distinguished isomorphism $ \mathbf{n_1}\amalg  \mathbf{n_2} \cong \mathbf{n_1+n_2}$, as in Remark \ref{rem:dist_iso_fb}.

\begin{defn}
\label{defn:concat_deconcat}
Using Proposition \ref{prop:iso_neck_tenfb}, define
\begin{enumerate}
\item
the concatenation product  $\prodconcat : \neck \tenfb \neck \rightarrow \neck$ to be the morphism 
 with component $P^\fb_{\mathbf{n_1}\amalg \mathbf{n_2}}
\rightarrow  
 P^\fb_{\mathbf{n_1+n_2}}$, for $(n_1, n_2) \in \nat^2$, induced by the distinguished isomorphism $\mathbf{n_1+n_2} \cong \mathbf{n_1} \amalg \mathbf{n_2}$;
\item 
the deconcatenation coproduct $\coproddecon : \neck \rightarrow \neck \tenfb 
\neck$ to be the morphism with component $P^\fb _{\mathbf{n_1+n_2}} \rightarrow P^\fb _{\mathbf{n_1} \amalg \mathbf{n_2}}$ for $(n_1, n_2) \in \nat$, induced by the distinguished isomorphism $\mathbf{n_1} \amalg \mathbf{n_2} \cong \mathbf{n_1  +n_2}$.
\end{enumerate}
\end{defn}

 Via the Schur functor construction (using the properties given in Proposition \ref{prop:schur_functor}), these correspond to forming the natural tensor algebra $V \mapsto T(V)$ and the natural tensor coalgebra $V \mapsto T_{\mathrm{coalg}}(V)$, with the concatenation product and deconcatenation coproduct respectively.

\begin{lem}
\label{lem:monoid_comonoid_neck}
In the monoidal category $(\f(\fb; \rat), \tenfb, \unit)$:
\begin{enumerate}
\item 
$(\neck, \prodconcat, \eta)$ is an associative, unital monoid;
\item 
$(\neck, \coproddecon, \epsilon)$ is a coassociative, counital comonoid.
\end{enumerate}
Moreover, these structures are dual under  $D_{\fb}$, considered as a functor $\f(\fb; \rat)\op \rightarrow \f(\fb; \rat)$.
\end{lem}

\begin{proof}
The result is proved by exploiting the associativity property in Proposition \ref{prop:proj_smod_tenfb} and using Proposition \ref{prop:duality_smod_tenfb} in the analysis of the duality,  Lemma \ref{lem:neck_eval_fd} providing  the requisite finiteness hypothesis. 
\end{proof}

These structures extend to Hopf algebra structures on $\neck$. For this we use the following structure morphisms, where, for $m  \in \nat$,  $\mathbf{m}$ is equipped with the total order inherited from $\nat$. 

\begin{nota}
For $(n_1, n_2) \in \nat$ with $n:=n_1+n_2$, let
\begin{enumerate}
\item
$\shuff_{n_1,n_2} \in \rat [\hom_\fb(\mathbf{n_1} \amalg \mathbf{n_2}, \mathbf{n})]$ be the sum of the maps $f \in \hom_\fb(\mathbf{n_1} \amalg \mathbf{n_2}, \mathbf{n})$ such that $f|_{\mathbf{n_1}}$ and $f|_{\mathbf{n_2}}$ are order-preserving; 
\item 
$\coshuff_{n_1,n_2} \in \rat [\hom_\fb( \mathbf{n},\mathbf{n_1} \amalg \mathbf{n_2})]$ be the sum of  $g \in \hom_\fb( \mathbf{n},\mathbf{n_1} \amalg \mathbf{n_2})$ such that the restrictions $g^{-1} (\mathbf{n_i}) \rightarrow \mathbf{n_i}$ are order-preserving, for $i\in \{1, 2\}$;
\item
$\mathrm{flip}_n :\mathbf{n} \rightarrow \mathbf{n}$ be the map $j \mapsto n+1 -j$. 
\end{enumerate}
\end{nota}

Analogously to Definition \ref{defn:concat_deconcat}, we have:

\begin{defn}
\ 
\begin{enumerate}
\item 
 The shuffle coproduct  $\coprodshuff : \neck \rightarrow \neck \tenfb \neck$ has component $P^\fb_\mathbf{n_1+n_2} 
\rightarrow 
P^\fb_{\mathbf{n_1} \amalg \mathbf{n_2}}$, for  $(n_1, n_2) \in \nat^2$,  
induced by $\shuff_{n_1,n_2}$;
\item 
the shuffle product $\prodshuff : \neck \tenfb \neck\rightarrow \neck $ has component $P^\fb_{\mathbf{n_1} \amalg \mathbf{n_2}} \rightarrow P^\fb_\mathbf{n_1+n_2}$, for  $(n_1, n_2) \in \nat^2$, 
 induced by $\coshuff_{n_1,n_2}$;
\item 
the $n$th component of the involution $\chi : \neck \rightarrow \neck$ is  $(-1)^n P^\fb _{\mathrm{flip}_n} : P^\fb_{\mathbf{n}} \rightarrow P^\fb_{\mathbf{n}}$.
\end{enumerate}
\end{defn}

\begin{prop}
\label{prop:hopf_alg_fb_modules}
\nomenclature{$\palg$}{cocommutative Hopf algebra in $\f(\fb;\rat)$\nomrefpage}
\nomenclature{$\pcoalg$}{commutative Hopf algebra in $\f (\fb;\rat)$\nomrefpage}
With respect to $(\f(\fb;\rat), \tenfb, \unit , \tau)$,
\begin{enumerate}
\item 
$\palg := (\neck, \prodconcat, \coprodshuff, \chi, \eta, \epsilon)$ 
is a  cocommutative Hopf algebra;
\item 
$\pcoalg := (\neck, \prodshuff, \coproddecon, \chi, \eta, \epsilon)$ is a commutative Hopf algebra. 
\end{enumerate}
These structures are dual under $D_{\fb}$, considered as a functor $\f(\fb; \rat)\op \rightarrow \f(\fb; \rat)$.
\end{prop}

\begin{proof}
The result can be checked directly from the definitions.

Since we are working over $\rat$, an alternative proof is to check that these structures induce Hopf algebra structures on the Schur functor associated to $\neck$, exploiting the properties of the Schur functor construction stated in Proposition \ref{prop:schur_functor}. In both cases, the underlying Schur functor is
\[
V \mapsto  \bigoplus_{n\in \nat} V^{\otimes n}.
\]

One checks that the structure $\palg$ equips $T (V)$ with the tensor algebra Hopf algebra structure, for which the generators are primitive. The result follows in this case. 

For $\pcoalg$, one checks that the corresponding Hopf algebra is $T_{\mathrm{coalg}} (V)$, equipped with the Hopf algebra structure with deconcatenation coproduct and shuffle product. 

These Hopf algebras are dual (for graded duality with respect to the length grading), from which the duality statement follows.
\end{proof}

\subsection{Twisting the Hopf algebras using $(-)^\dagger$}

Applying the functor $(-)^\dagger$ of Theorem \ref{thm:equiv_dagger}, one has the immediate Corollary to Proposition \ref{prop:hopf_alg_fb_modules}:

\begin{cor}
\label{cor:palg_pcoalg_dagger}
With respect to  $(\f(\fb;\rat), \tenfb, \unit , \sigma)$,
\begin{enumerate}
\item 
${\palg}^\dagger := (\neck, \prodconcat, \coprodshuff, \chi, \eta, \epsilon)^{\dagger}$ 
is a  cocommutative Hopf algebra;
\item 
${\pcoalg}^\dagger := (\neck, \prodshuff, \coproddecon, \chi, \eta, \epsilon)^\dagger$ is a commutative Hopf algebra. 
\end{enumerate}
These structures are dual under $D_{\fb}$.
\end{cor}

To further analyse these structures, one uses the isomorphism between ${\neck}^\dagger$ and $\neck$ that is constructed below. First, observe that, by definition:
\[
{\neck}^\dagger \cong \bigoplus_{n\in \nat} \rat [\hom_\fb (\mathbf{n}, X)] \otimes \orient(X).
\]

Recall from Notation \ref{nota:iota} that $\iota_n$ denotes the canonical generator of $\orient (\mathbf{n})$. Then, by Yoneda's lemma, the  element $[\id_{\mathbf{n}}] \otimes \iota_n \in  \rat [\hom_\fb (\mathbf{n}, \mathbf{n})] \otimes \orient(\mathbf{n})$ induces a morphism:
\begin{eqnarray}
\label{eqn:twist}
P^\fb _\mathbf{n} \rightarrow P^\fb_{\mathbf{n}} \otimes \orient (-).
\end{eqnarray}

\begin{lem}
\label{lem:neck_dagger}
The morphisms (\ref{eqn:twist}) assemble to  an isomorphism $
\neck 
\stackrel{\cong}{\rightarrow}
{\neck}^\dagger
$.
\end{lem}

\begin{proof}
It suffices to show that each map $P^\fb _\mathbf{n} \stackrel{\cong}{\rightarrow} P^\fb_{\mathbf{n}} \otimes \orient (-)$ is an isomorphism. This follows from the fact that 
the morphism of $\sym_n$-modules $\rat \sym_n \rightarrow \rat \sym_n \otimes \mathrm{sgn}_n$ sending $[e]$ to $[e] \otimes 1$ (where $e\in \sym_n$ is the identity and $1$ is the generator of $\kring \cong \mathrm{sgn}_n$) is an isomorphism.
\end{proof}

\begin{prop}
\label{prop:palg_pcoalg_dagger_identify}
Via the isomorphism $\neck \cong {\neck}^\dagger$ of Lemma \ref{lem:neck_dagger},
\begin{enumerate}
\item 
${\palg}^\dagger$ is isomorphic as a Hopf algebra to $(\neck, \prodconcat, \coprodshuff^\dagger, \chi, \eta, \epsilon)$, where $\coprodshuff^\dagger$ is derived from $\coprodshuff$ by transport of structure; 
\item 
${\pcoalg}^\dagger$ is isomorphic as a Hopf algebra to $(\neck, \prodshuff^\dagger, \coproddecon, \chi, \eta, \epsilon)$, where $\prodshuff^\dagger$ is derived from $\prodshuff$ by transport of structure. 
\end{enumerate}
\end{prop}

\begin{proof}
The result follows from the fact that the isomorphisms of Proposition \ref{prop:proj_smod_tenfb} are compatible via the functor $(-)^\dagger$ with the isomorphism of Lemma \ref{lem:neck_dagger}. This uses  the compatibility of the family of   generators $(\iota _n | n\in \nat)$ that is encoded in the bicommutative Hopf algebra structure of $\orient$ (see Proposition \ref{prop:orient_(co)mult}). 

An alternative heuristic argument in the case of $\palg$ is as follows, based upon Remark \ref{rem:heuristic_dagger} ($\pcoalg$ can be treated similarly). As in the proof of Proposition \ref{prop:hopf_alg_fb_modules}, applying the Schur functor construction yields the natural Hopf algebra $T (V)$, equipped with the concatenation product and shuffle coproduct. The functor $(-)^\dagger$ is interpreted at the level of Schur functors as replacing $V$ by $sV$ and taking into account the Koszul signs. 

Now $T (sV) \cong \bigoplus_{n\in \nat} (sV)^{\otimes n} \cong \bigoplus_{n\in \nat} s^n V^{\otimes n}$, which, forgetting the grading, is isomorphic to $ \bigoplus_{n\in \nat} V^{\otimes n}$. The only structure morphism which introduces Koszul signs is the shuffle coproduct. This is immediate for $\prodconcat$, $\eta$, and $\epsilon$; for $\chi$, one checks that it is also the case.
\end{proof}

\subsection{The associated exponential functors}
\label{subsect:expo_Phi_palg_Psi_pcoalg}

In this section, we work with the symmetric monoidal structure $(\f(\fb; \rat), \tenfb, \unit, \tau)$. By Proposition \ref{prop:hopf_alg_fb_modules}, with respect to this structure, $\palg$ is a cocommutative Hopf algebra and $\pcoalg$ is a commutative Hopf algebra. Hence, by Theorem \ref{thm:expo_Hopf_general}, we have the corresponding exponential functors, by the constructions of Section \ref{sec:expo}:
\begin{eqnarray*}
\Psi \pcoalg & \in &  \ob \f (\gr; \f(\fb;\rat))\\
\Phi \palg & \in & \ob \f (\gr\op ; \f (\fb; \rat)).
\end{eqnarray*}
 
The values of these as $\fb$-modules are identified by the following:

\begin{prop}
\label{prop:underlying_fb_module_Phi_palg_Psi_pcoalg}
For $t\in \nat$, there are isomorphisms of $\fb$-modules:
\begin{eqnarray*}
\Psi \pcoalg (\zed^{\star t}) 
\cong 
\Phi \palg (\zed^{\star t}) 
&\cong& 
(P^\fb) ^{\tenfb t}
\\
&\cong &
\bigoplus_{(n_i \ | \ 1 \leq i \leq t) \in \nat^t} \rat [\hom_\fb (\amalg_{i=1}^t \mathbf{n_i}, -)], 
\end{eqnarray*}
where $\amalg_{i=1}^t \mathbf{n_i} = \mathbf{n_1} \amalg \ldots \amalg \mathbf{n_t}$.

In particular, for each $n \in \nat$, $\Psi \pcoalg (\zed^{\star t})(\mathbf{n})$ and $  
\Phi \palg (\zed^{\star t}) (\mathbf{n})$ are free $\sym_n$-modules of finite rank.
\end{prop}

\begin{proof}
The first statements follow from the construction of the functors $\Psi$ and $\Phi$ (see Theorem \ref{thm:expo_Hopf_general}). The explicit identification of the $\fb$-module $(P^\fb) ^{\tenfb t}$ follows from Proposition \ref{prop:proj_smod_tenfb}.

The final statement generalizes Lemma \ref{lem:neck_eval_fd} and follows directly from the above identification.
\end{proof}

The functors $\Psi \pcoalg$ and $\Phi \palg$ are related by duality. To express this, it is convenient to use the natural equivalences of categories
\begin{eqnarray*}
\f (\gr; \f(\fb;\rat)) & \cong & \f (\gr \times \fb ; \rat) \\
\f (\gr\op; \f(\fb;\rat)) & \cong & \f (\gr\op \times \fb ; \rat).
\end{eqnarray*}
Then, using the isomorphism of categories $\fb \cong \fb\op$, the duality adjunction of Proposition \ref{prop:duality} becomes:
\[
D_{\gr \times \fb} : \f (\gr \times \fb ; \rat)\op \rightleftarrows \f(\gr\op \times \fb; \rat) : D_{\gr\op \times \fb}\op.
\]

\begin{prop}
\label{prop:duality_Psi_palg_Phi_pcoalg}
There are natural isomorphisms:
\begin{eqnarray*}
\Psi \pcoalg & \cong & D_{\gr\op \times \fb} \big( \Phi \palg \big)  \mbox{\quad  in $\f (\gr \times \fb; \rat)$} \\
\Phi \palg & \cong & D_{\gr \times \fb}\big( \Psi \pcoalg \big) \mbox{\quad in $\f (\gr\op \times \fb; \rat)$.}
\end{eqnarray*}
\end{prop}

\begin{proof}
By Proposition \ref{prop:hopf_alg_fb_modules}, the Hopf algebras $\pcoalg$ and $\palg$ are dual under $D_{\fb}$. In particular, considered as bivariant functors with values in $\rat$-vector spaces, both take finite-dimensional values.

This duality passes to the associated exponential functors by using the naturality given by Proposition \ref{prop:naturality_Phi_Psi}. For this one has to consider the exponential functors as taking values in $\f(\fb; \rat)$; the above discussion implies that they take values in $\fb$-modules taking finite-dimensional values. 

The duality adjunction 
\[
D_\fb : \f (\fb ; \rat) \op \rightleftarrows \f(\fb; \rat) : D_{\fb}\op
\]
(again using the isomorphism $\fb \cong \fb\op$) 
restricts to an equivalence of symmetric monoidal categories between the full subcategories of functors taking finite-dimensional values. 

The result thus follows from Proposition \ref{prop:naturality_Phi_Psi} using $D_\fb$ as the functor $\calm \rightarrow \cale$.
\end{proof}

\begin{rem}
The argument used above can be paraphrased as stating that, under the appropriate finiteness hypotheses, duality  transposes $\Phi$ and $\Psi$, in that $D_\fb \Phi = \Psi D_\fb$ and $D_\fb \Psi = \Phi D_\fb$.
\end{rem}

\section{The functors $\beta_d$}
\label{sec:beta}

The goal of this section is to describe  the fundamental functors 
$$
\beta_d : \rat [\sym_d]\dash\modules \rightarrow \f_d(\gr; \rat) \subset \f (\gr; \rat),
$$ 
for $d\in \nat$,  by using exponential functors. The functors $\beta_d$ were recalled in Section \ref{sec:recoll}, where it is shown that they give rise to a family of injective cogenerators of $\fpiQ{\gr}$, the category of polynomial functors on $\gr$, namely the family $\{  \beta_d \rat [\sym_d] \ | \ d \in \nat \}$.

The main result of the section is Theorem \ref{thm:beta_d}, which gives a description of the functors $\beta_d$ in terms of the exponential functor $\Psi \pcoalg$. This Theorem follows directly from Theorem \ref{thm:beta_description} which identifies the functor encoded by the family $\beta_d \rat [\sym_d]$ as $\Psi \pcoalg$. 

The proof of Theorem \ref{thm:beta_description} in Section \ref{subsect:boldbeta} relies upon analysing the associated graded of the polynomial filtration by passing to $\fb$-modules; the required material is developed in Section \ref{subsect:poly_filt_revisit} building upon the material of Section \ref{sec:recoll}.

\subsection{The polynomial filtration revisited}
\label{subsect:poly_filt_revisit}

Recall the functor  $\qhat{d}$ of Notation \ref{nota:qhat}, defined for $d\in \nat$. By Proposition \ref{prop:exact_sequences_recollement}, the inclusion $\qhat{d} \subset q_d^\gr $ induces  a natural isomorphism for $F \in \ob \f(\gr ; \rat)$:
\begin{eqnarray*}
 \cre_d \qhat{d} F 
 \cong 
 \cre_d q_d ^\gr F.
\end{eqnarray*}

The polynomial filtration of $F$ is defined as in Definition \ref{defn:poly_filt},  so that the associated graded  is:
\[
\bigoplus_{t\in \nat} \qhat{t} F.
\]

By Proposition \ref{prop:exact_sequences_recollement}, there is a natural isomorphism 
$\qhat{t} F \cong \alpha_t \cre_t q_t^\gr F$, hence this associated graded is determined by the sequence of $\sym_t$-modules $\cre_t q_t^\gr F \cong \cre_t \qhat{t} F$, for $t\in \nat$. This motivates the introduction of the following:

\begin{defn}
\label{defn:fbcr}
\nomenclature{$\fbcr$}{associated graded to polynomial filtration as $\fb$-module\nomrefpage}
Let $\fbcr : \f(\gr;\rat) \rightarrow \f (\fb ; \rat)$ be the functor defined by 
\[
\fbcr F (\mathbf{t}) := \cre_t q_t^\gr F \cong  \cre_t \qhat{t} F .
\]
\end{defn}

\begin{rem}
The associated graded to the polynomial filtration of $F$ is determined by $\fbcr F$. 
\end{rem}

In the following statement, we restrict to  $\fpiQ{\gr}\subset \f(\gr; \rat)$, the subcategory of polynomial functors on $\gr$. The category $\fpiQ{\gr}$  is considered as symmetric monoidal with respect to the objectwise tensor product $\otimes$ and the category $\f (\fb;\rat)$  with respect to $\tenfb$ with symmetry $\tau$.

\begin{prop}
\label{prop:fbcr}
The functor $\fbcr : \fpiQ{\gr} \rightarrow \f (\fb ; \rat)$ is:
\begin{enumerate}
\item 
exact;
\item 
symmetric monoidal.
\end{enumerate}
\end{prop}

\begin{proof}
For $t \in \nat$, the functor $\cre_t$ is exact and $q^\gr_t$ is exact by  Proposition \ref{prop:unimodularity_char_0}. The exactness of $\fbcr$ follows. 

To show that $\fbcr$ is symmetric monoidal, one uses the characterization of the polynomial filtration given in Proposition \ref{prop:canon_poly_filt}. From this one deduces that, for polynomial functors $F$ and $G$, there is a natural isomorphism
\[
\qhat{t} (F \otimes G) 
\cong 
\bigoplus_{t_1 +t_2 = t} 
\qhat{t_1} (F) \otimes \qhat{t_2} (G) 
\]
and this is compatible with the symmetric monoidal structure. 
\end{proof}

We will need to apply this working with tensor products of $P^\gr_\zed$, which are not polynomial. 
To do this, we  reduce to the polynomial case, using Proposition \ref{prop:poly_filt_proj} below.
Recall that, for $t \in \nat$, $q^\gr_t : \f (\gr; \rat) \rightarrow \f_t (\gr; \rat)$ is the left adjoint to the inclusion of the full subcategory of polynomial functors of degree $t$.

\begin{prop}
\label{prop:poly_filt_proj}
For $t, n \in \nat$ the $n$-fold tensor product $(P^\gr_\zed) ^{\otimes n} 
\twoheadrightarrow (q^\gr_t P^\gr_\zed) ^{\otimes n}$ of the canonical surjection $ P^\gr_\zed
\twoheadrightarrow q^\gr_t P^\gr_\zed$ induces an isomorphism:
\[
q^\gr_t \big((P^\gr_\zed) ^{\otimes n} \big)
\stackrel{\cong}{\rightarrow}  
q_t^\gr\big( (q^\gr_t P^\gr_\zed) ^{\otimes n}\big).
\]
\end{prop}

\begin{proof}
This result follows from the analysis of the polynomial filtration of the standard projectives given in \cite{DPV}. 
For the convenience of the reader, a direct proof is given here.

Write $K_t$ for the kernel of the natural projection $P^\gr_\zed \twoheadrightarrow q^\gr_t P_\zed^\gr$. Thus there is an exact sequence 
\[
\bigoplus_{i+j = n-1} (P^\gr_\zed)^{\otimes i} \otimes K_t \otimes (P^\gr_\zed) ^{\otimes j}
\rightarrow 
(P^\gr_\zed) ^{\otimes n}
\rightarrow 
(q^\gr_t P^\gr_\zed) ^{\otimes n}
\rightarrow 
0.
\]
To prove the result, since $q_t^\gr$ is right exact, it suffices to show that applying $q_t^\gr$ to the left hand term gives zero. One reduces to showing that $q_t ^\gr (K_t \otimes (P^\gr_\zed) ^{\otimes n-1})$ is zero. This is equivalent to showing that $\hom_{\f(\gr; \rat)} (K_t \otimes (P^\gr_\zed) ^{\otimes n-1}, G)$ is zero, for any functor $G$ of polynomial degree $t$.   

Now, by Proposition \ref{prop:properties_tau}, 
$$\hom_{\f(\gr; \rat)} (K_t \otimes (P^\gr_\zed) ^{\otimes n-1}, G) \cong \hom_{\f(\gr; \rat)} (K_t, \tau_\zed^{n-1} G).$$
 Since $G$ has polynomial degree $t$, $\tau_\zed^{n-1} G$ also has polynomial degree $t$ (this standard fact follows readily from the definition of polynomial degree). Therefore, $\hom_{\f(\gr; \rat)} (K_t, \tau_\zed^{n-1} G)$
 is isomorphic to $\hom_{\f_t(\gr; \rat)} (q_t^\gr K_t, \tau_\zed^{n-1} G)$. 
 
To conclude, it suffices to show that $q_t ^\gr K_t=0$. Now, by construction of the functor $q_t^\gr$ (see \cite[Proposition 3.7]{DPV}, for example), $K_t$ identifies as the image of the iterated product $(\overline{P_\zed^\gr}) ^{\otimes t+1} \rightarrow P_\zed^\gr$. Hence it suffices to show that  $\hom_{\f(\gr; \rat)} ((\overline{P_\zed^\gr}) ^{\otimes t+1} , G')=0$ for any functor $G'$ of polynomial degree $t$. Again, by Proposition \ref{prop:properties_tau}, it suffices to show that $\taubar^{t+1} G' =0$. Once again, this standard fact follows from the definition of polynomial degree. 
\end{proof}

\begin{exam}
\label{exam:fbcr_truncation}
For $t, n \in \nat$, Proposition \ref{prop:poly_filt_proj} implies that 
$$
\fbcr \big((P^\gr_\zed) ^{\otimes n} \big) (\mathbf{d}) =   \fbcr \big((q^\gr_t P^\gr_\zed) ^{\otimes n} \big) (\mathbf{d})$$
 for all $d \leq t$.  Now $(q^\gr_t P^\gr_\zed) ^{\otimes n}$ is polynomial (explicitly, it has polynomial degree $nt$).   Hence, this allows 
$\fbcr \big((P^\gr_\zed) ^{\otimes n} \big) $ to be studied using polynomial functors. In particular,  Proposition \ref{prop:fbcr} can be applied in this context. 

One deduces that there is a $\sym_n$-equivariant isomorphism
\[
\fbcr \big((P^\gr_\zed) ^{\otimes n} \big)
\cong 
\big(\fbcr P^\gr_\zed) \big)^{\odot n}.
\]
\end{exam}

\subsection{The functor $\boldbeta$}
\label{subsect:boldbeta}

For each $d\in \nat$, one has the functor $\beta_d \rat [\sym_d]$  which belongs to $\f_d (\gr; \rat)$ and which is equipped with the $\sym_d$-action induced by the right regular action on $\rat [\sym_d]$. These functors assemble to a functor on $\fb$:

\begin{nota}
\label{nota:boldbeta}
\nomenclature{$\boldbeta$}{the assembled functors $\beta_d$\nomrefpage}
Denote by  $\boldbeta \in \ob \f(\gr \times \fb;\rat)$ the functor 
 $
 \boldbeta : (-,   \mathbf{d}) \mapsto \beta_d \rat [\sym_d](-).
 $
\end{nota}

The following fundamental result identifies this  in terms of $\pcoalg$ (see Section \ref{subsect:expo_Phi_palg_Psi_pcoalg}).

\begin{thm}
\label{thm:beta_description}
There is an isomorphism
 $\boldbeta
\cong
 \Psi \pcoalg$
 in $\f(\gr\times \fb;\rat)$.
\end{thm}

The first ingredient in the proof is the following, in which $(-)^\sharp$ denotes vector space duality.

\begin{lem}
\label{lem:bold_beta_cre_q}
For $d \in \nat$ and $G \in \ob \gr$, there is an isomorphism of $\sym_d$-modules:
\[
\boldbeta (G, \mathbf{d}) \cong 
\Big( \cre_d q_d^\gr P^\gr_G \Big)^\sharp
 = \Big( \fbcr P^\gr_G (\mathbf{d}) \Big)^\sharp
. 
\]
Moreover, this is natural with respect to $G \in \ob \gr$.
\end{lem}

\begin{proof}
There are natural isomorphisms, using the adjunctions of Theorem \ref{thm:recollement_poly_d}:
\begin{eqnarray*}
\boldbeta (G, \mathbf{d}) 
\cong
\hom_{\f (\gr, \rat)} (P^\gr _G , \boldbeta (-, \mathbf{d}) )
&\cong & 
\hom_{\f_d (\gr, \rat)} (q_d^\gr P^\gr _G , \beta_d\rat [\sym_d]  ( -) )
\\
&\cong&
\hom_{\sym_d} (\cre_d q_d^\gr P^\gr _G , \rat [\sym_d]  ),
\end{eqnarray*}
where the first isomorphism is given by Yoneda, the second uses the definition of $\boldbeta$, the fact that $\beta_d\rat [\sym_d]  ( -)$ has polynomial degree $d$, and the definition of $q_d^\gr$ as  a left adjoint. The final isomorphism uses that $\cre_d$ is left adjoint to $\beta_d$.  

Now, for any (left) $\sym_d$-module $M$, there is a natural isomorphism of right $\sym_d$-modules $\hom_{\sym_d} (M, \rat [\sym_d] ) \cong M^\sharp$. (This isomorphism is induced by composing with the $\rat$-linear map $\rat[\sym_d] \twoheadrightarrow \rat$ sending a generator $[g]$ to zero unless $g=e$, and $[e]\mapsto 1$.) This gives the isomorphism of $\sym_d$-modules $\boldbeta (G, \mathbf{d}) \cong 
\Big( \cre_d q_d^\gr P^\gr_G \Big)^\sharp$. The final equality is given by the definition of $\fbcr$.

 Finally, it is clear that this is natural with respect to $G$.
\end{proof}

Lemma \ref{lem:bold_beta_cre_q} shows that $\boldbeta$ is dual to $G \mapsto \fbcr P^\gr_G$. In particular, $\boldbeta$ only depends upon the associated graded to the polynomial filtration of $P^\gr_G$, considered as a functor of $G$.

Hence, we proceed to analyse the structure of the associated graded of the polynomial filtration of $P^\gr_G$. For this we consider the functors $P^\gr_G$, for $G \in \ob \gr\op$, as forming the bivariant functor 
$\rat [\hom_\gr (-, -)]$ in $\f( \gr\op \times \gr; \rat) \cong \f(\gr\op; \f(\gr; \rat))$.  

The following essentially restates  Examples \ref{exam:expo_group_ring} and \ref{exam:expo_group_ring+hopf}:

\begin{prop}
\label{prop:standard_proj_expo}
Considered in $\f(\gr\op; \f(\gr; \rat))$, the functor $\rat [\hom_\gr (-,-)]$  is exponential. In particular, 
\[
\rat[\hom_\gr (-,-)] \cong \Phi P^\gr _{\zed}.
\]
Here, $P^\gr_{\zed}$ is the cocommutative Hopf algebra in $\f(\gr; \rat)$ given by the group ring functor $H\mapsto \rat [H]$.
\end{prop}

The Hopf algebra structure of $P^\gr _\zed$ passes to the associated graded of the polynomial filtration. In the following, write $T(\aQ)$ for the tensor Hopf algebra in $\f (\gr; \rat)$,
\[
T(\aQ) = \bigoplus_{t\in \nat} \aQ^{\otimes t}
\]
with concatenation product and  the shuffle coproduct, interpreted functorially. Thus the coproduct is determined by its restriction to $\aQ$, on which it identifies as the map $\aQ \rightarrow (\rat \otimes \aQ) \oplus (\aQ \otimes \rat) \subset T(\aQ) \otimes T(\aQ)$, given by the diagonal $\aQ \hookrightarrow \aQ^{\oplus 2}$  (after absorbing the respective tensor with $\rat$).

\begin{lem}
\label{lem:assoc_grad_P_zed}
The cocommutative Hopf algebra structure of $P^\gr _\zed$ in $\f (\gr; \rat)$ induces a cocommutative Hopf algebra structure on its associated graded, this identifies as Hopf algebras:
\[
\bigoplus_{t\in \nat} \qhat{t} P^\gr _\zed
\cong 
T(\aQ).
\]
\end{lem}

\begin{proof}
This analysis for the filtration of the group ring functor $H \mapsto \rat [H]$ by powers of the augmentation ideal follows from \cite[Chapter VIII]{Passi}. The relationship with the polynomial filtration is established as follows.

 For $H$ free, \cite[Proposition 3.7]{DPV} shows that the kernel of the natural
surjection
\[
 P_\zed^\gr(H) \twoheadrightarrow q_d P_\zed^\gr (H)
\]
 is isomorphic to $\mathcal{I}^{d+1} (H)$, the $(d+1)$st power of the 
augmentation ideal $\mathcal{I}(H)$. This gives 
an explicit description of the polynomial filtration of $P_\zed^\gr$ (the 
quotients $q_d P_\zed ^\gr$ 
are known as {\em Passi functors}). It follows, as in \cite[Remarque 2.2]{DPV}, that the associated graded of $ P_\zed^\gr $ for this filtration identifies with the tensor algebra $\bigoplus_{t \geq 0}  \aQ^{\otimes t}$ on the functor $\aQ$. 

Finally, the Hopf algebra structure is primitively-generated, as stated, by the results of \cite[Chapter VIII]{Passi}. 
\end{proof}

The following proposition  reformulates this using the functor $\fbcr : \f(\gr; \rat) \rightarrow \f(\fb; \rat)$:

\begin{prop}
\label{prop:fbcr_P_zed}
The cocommutative Hopf algebra structure of $P^\gr _\zed$ in $\f (\gr; \rat)$ induces a cocommutative Hopf algebra structure on $\fbcr P^\gr_\zed$ in $\f (\fb;\rat)$ and 
there is an isomorphism of Hopf algebras:
\[
\fbcr P^\gr_\zed \cong \palg.
\]
\end{prop}

\begin{proof}
The fact that one obtains a Hopf algebra structure follows from the symmetric monoidal property given in Proposition \ref{prop:fbcr}, refined to deal with tensor products of $P^\gr_\zed$ by using Proposition \ref{prop:poly_filt_proj}, as in Example \ref{exam:fbcr_truncation}.

The identification of the Hopf algebra structure then follows from Lemma \ref{lem:assoc_grad_P_zed}.
\end{proof}

\begin{proof}[Proof of Theorem \ref{thm:beta_description}]
Lemma \ref{lem:bold_beta_cre_q} in conjunction with Proposition \ref{prop:standard_proj_expo} implies that there is an isomorphism:
\[
\boldbeta \cong D_{\gr\op \times \fb} \big( \fbcr (\Phi P_\zed^\gr)\big)
\]
in $\f (\gr \times \fb; \rat)$.

Now, extending the symmetric monoidal property of $\fbcr$ given by Proposition \ref{prop:fbcr} as in the proof of Proposition \ref{prop:fbcr_P_zed}, there is an isomorphism of functors in $\f (\gr\op \times \fb; \rat)$:
\[
\fbcr (\Phi P_\zed^\gr)
\cong 
\Phi (\fbcr P_\zed^\gr),
\]
by the naturality of $\Phi$ with respect to symmetric monoidal functors given by Proposition \ref{prop:naturality_Phi_Psi}.

Hence, by Proposition \ref{prop:fbcr_P_zed}, $\fbcr (\Phi P_\zed^\gr) \cong \Phi \palg$. 
 The result follows by the duality isomorphism given in Proposition \ref{prop:duality_Psi_palg_Phi_pcoalg}.
\end{proof}

\subsection{Multiplicative  structure}

By Theorem \ref{thm:beta_description}, we can identify the functor $\boldbeta$ with the exponential functor $\Psi \pcoalg$ in $\f (\gr; \f(\fb;\rat))$. Thus Proposition \ref{prop:Phi_Psi_(co)mult} yields the following, which succinctly encodes the `multiplicative structure' formed by the $\beta_d \rat [\sym_d]$, for $d \in \nat$. 

\begin{cor}
\label{cor:mult_boldbeta}
Considered as a functor in $\f (\gr; \f(\fb;\rat))$, $\boldbeta$ takes values in the category of twisted commutative algebras (aka. unital, commutative monoids in $\f (\fb; \rat)$). 

Explicitly, for each $d_1, d_2 \in \nat$, there is a structure morphism:
\[
\boldbeta (-, \mathbf{d_1}) \otimes \boldbeta (-, \mathbf{d_2}) \rightarrow 
\boldbeta (-, \mathbf{d_1} \amalg \mathbf{d_2}) 
\]
and these satisfy the unit, associativity and commutativity constraints.
\end{cor}

\subsection{The functors $\beta_d$} 
Theorem \ref{thm:beta_description} allows the description of the functors $\beta_d$. In the following, we consider a $\sym_d$-module $M$ as a $\fb$-module supported on $\mathbf{d}$ and we use the isomorphism of categories $\fb\op \cong \fb$ to adjust variance so as to apply the coend $\otimes_\fb$.

\begin{thm}
\label{thm:beta_d}
For $d \in \nat$, the functors $\beta_d : \rat [\sym_d]\dash\modules \rightarrow \f(\gr; \rat)$ is naturally isomorphic to the functor 
\[
M \mapsto \boldbeta \otimes_\fb M \cong \big(\Psi \pcoalg\big) \otimes_\fb M.
\]
\end{thm} 

\begin{proof}
As in the proof of Lemma \ref{lem:bold_beta_cre_q}, there are natural isomorphisms:
\begin{eqnarray*}
\beta_d M (G)
\cong
\hom_{\f (\gr, \rat)} (P^\gr _G , \beta_d M )
&\cong & 
\hom_{\f_d(\gr, \rat)} (q_d^\gr P^\gr _G , \beta_d M )
\\
&\cong &
\hom_{\sym_d} (\cre_d q_d^\gr P^\gr _G , M  ).
\end{eqnarray*}
By Lemma \ref{lem:bold_beta_cre_q}, the last term is isomorphic to 
$$\hom_{\sym_d} (\boldbeta (G, \mathbf{d})^\sharp  , M  )
\cong 
\Big(\boldbeta (G, \mathbf{d})  \otimes  M \Big)^{\sym_d},   
$$
for the diagonal action of $\sym_d$, using that the underlying vector space of $\boldbeta (G, \mathbf{d})$ has finite dimension. Since we are working over $\rat$, the $\sym_d$-invariants can be replaced by $\sym_d$-coinvariants, so that this can be rewritten as $\boldbeta \otimes_\fb M$ by considering $M$ as a $\fb$-module supported on $\mathbf{d}$. 

Finally, Theorem \ref{thm:beta_description} gives the isomorphism $\boldbeta \cong \Psi \pcoalg$, from which the result follows. 
\end{proof}

\part{Outer functors on groups}
\label{part:outer}

In this part, we introduce what we call  \textit{outer functors} on free groups. An outer functor is a functor on free groups on which inner automorphisms act trivially: it thus provides a family of representations of the outer automorphism groups of free groups, which are compatible via the functoriality. 

In Section \ref{sec:outer} we introduce outer functors and establish some  first properties. (In the text we consider  both covariant and contravariant functors;  to simplify the exposition, in the rest of this introduction we  only consider the covariant case.)

By construction, outer functors form a full subcategory of the category of functors on free groups. Section \ref{sec:outer_adjoints} considers the adjoints to the associated  inclusion of categories. The right adjoint, $\omega$, is used in Part \ref{part:HHH} to describe the outer functors associated to higher Hochschild homology. In Section \ref{sec:omega_exponential} we describe the action of $\omega$ on exponential functors; this is used in Part \ref{part:psi_omega_psi} to study the structure of the outer functors that appear in Part \ref{part:HHH}.

\section{The category of outer functors}
\label{sec:outer}

This section introduces the category of outer functors: this is the full subcategory of functors on $\gr$ on which inner automorphisms act trivially. Outer functors are the fundamental structures that arise in this paper; studying the full structure of an outer functor $F$ rather than simply the family of representations $F(\zed^{\star n})$ of the outer automorphism groups $\mathrm{Out}(\zed^{\star n})$ provides essential information.

Section \ref{subsect:poly_outer_functors} specializes to  polynomial outer functors.

\subsection{Introducing outer functors}
\label{subsect:outer_functors}

For $G$ a group, conjugation induces the adjoint action $\ad : G 
\rightarrow \aut (G)$ with image the normal subgroup $\inn (G) \lhd \aut (G)$ 
of inner automorphisms. The group of outer automorphisms of $G$ is the cokernel
$\out (G):= \aut (G) / \inn (G)$, which is  equipped with the canonical surjection 
 $
 \aut (G) 
 \twoheadrightarrow 
 \out (G). 
$ 
 An $\aut (G)$-module $M$ arises from an $\out (G)$-module structure via this 
surjection if and only if every inner automorphism acts trivially upon $M$. 

\begin{rem}
The conjugation action is natural, in the 
following sense: consider  $\ad_G : G \times G \rightarrow 
G $, $\ad_G (g, h) := \ad (g) (h) = g h g^{-1}$. 
 Then, for $\phi : G' \rightarrow G$ a group morphism, the following diagram 
commutes:
 \[
  \xymatrix{
  G' \times G' 
  \ar[r]^(.6){\ad_{G'}} 
  \ar[d]_{\phi \times \phi}
  &
  G' 
  \ar[d]^{\phi}
  \\
  G \times G 
  \ar[r]_(.6){\ad_{G}}
  &
  G.
  }
 \]
\end{rem}

In the following, $\cala$ is an arbitrary abelian category. For the remainder of the section, this will be taken to be the category of $\kring$-modules over an arbitrary unital, commutative ring.

\begin{defn}
\label{def:fout}
The category  $\fout{\gr}{\cala}$ (respectively $\fout{\gr\op}{\cala}$) of outer functors is the full subcategory 
of  $\f(\gr;\cala)$ (resp. $\f(\gr\op;\cala)$) of objects $G$ such that, for each 
$n 
\in \nat$, 
the canonical $\aut (\zed^{\star n})$-action on $G (\zed^{\star n})$  arises from an $\out 
(\zed^{\star n})$-action. 
\end{defn}

\begin{rem}
Belonging to  $\fout{\gr}{\cala}$ (respectively $\fout{\gr\op}{\cala}$) is a property and not an additional structure. 
\end{rem}

The following records basic properties of these categories, specializing to $\cala = \kring \dash \modules$ to simplify the exposition.

\begin{prop}
\label{prop:foutk}
\nomenclature{$\foutk[\gr]$}{outer functors on $\gr$\nomrefpage}
\nomenclature{$\foutk[\gr\op]$}{outer functors on $\gr\op$\nomrefpage}
\ 
\begin{enumerate}
\item 
 The category $\foutk[\gr]$ (resp. $\foutk[\gr \op]$) is 
strictly full in  $\fcatk[\gr]$ (resp. $\fcatk[\gr\op]$)  and closed under formation of  sums and sub-quotients; it is not thick. 
\item 
The respective tensor structures of $\fcatk[\gr]$ and $\fcatk[\gr \op]$  restrict  to 
\begin{eqnarray*}
\otimes_\kring &:& \foutk[\gr]\times \foutk[\gr] \rightarrow \foutk[\gr]\\
\otimes_\kring &:& \foutk[\gr\op]\times \foutk[\gr \op] \rightarrow \foutk[\gr\op].
\end{eqnarray*}
\item 
The abelianization functor $\A$ induces fully-faithful, symmetric monoidal,  
exact functors: 
 \begin{eqnarray*}
  \circ \A &:& \fcatk[\ab] \rightarrow \foutk[\gr]  \\
 \circ \A &:& \fcatk[\ab\op] \rightarrow \foutk[\gr\op] .
 \end{eqnarray*}
\item 
If $\kring$ is a field, the duality adjunction of Proposition \ref{prop:duality} 
restricts to 
\[
  D_{\gr} : \foutk[\gr] \op \rightleftarrows \foutk[\gr \op] : D_{\gr \op} \op
 \]
 and induces an equivalence of categories between the full subcategories of 
functors taking  finite-dimensional values.
\end{enumerate}
\end{prop}

\begin{proof}
Most of this is immediate. That the respective subcategories $\foutk[\gr]$ and  $\foutk[\gr \op]$ 
are not thick is exhibited by Example \ref{exam:Passi_not_out}.
\end{proof}

\subsection{The polynomial filtration of outer functors}
\label{subsect:poly_outer_functors}

To simplify the exposition and in accordance with Section \ref{sec:recoll}, we take $\kring=\rat$. 
The polynomial filtrations of $\fcatQ[\gr]$ and $\fcatQ[\gr\op]$ pass to $\foutQ[\gr]$ and $\foutQ[\gr\op]$ respectively:

\begin{nota}
\label{nota:fout_poly}
\nomenclature{$\fpoutgrQ[d]$}{polynomial outer functors  on $\gr$ of degree $d$\nomrefpage}
\nomenclature{$\f^{\out}_d (\gr \op; \rat)$}{polynomial outer functors  on $\gr\op$ of degree $d$\nomrefpage}
\nomenclature{$\fpoutgrQ[< \infty]$}{polynomial outer functors  on $\gr$\nomrefpage}
For $d \in \nat \cup \{< \infty \}$, denote by 
\begin{enumerate}
\item 
$\fpoutgrQ[d]$ the full subcategory $\foutQ[\gr] \cap \fpolyQ{d}{\gr}\subset \foutQ[\gr]$;
\item 
$\f^{\out}_d (\gr \op; \rat)$  the full subcategory $$\foutQ[\gr\op] \cap \fpolyQ{d}{\gr\op}\subset\foutQ[\gr\op]. $$
\end{enumerate}
\end{nota}

\begin{rem}
\ 
\begin{enumerate}
\item 
Proposition \ref{prop:foutk} has an evident analogue in the polynomial context. In particular, the duality adjunction restricts to 
\[
  D_{\gr} : \fpoutgrQ[d] \op \rightleftarrows \f^{\out}_d (\gr \op;\rat)  : D_{\gr \op} \op.
 \]
This restricts to an equivalence of categories between the full subcategories of functors taking finite-dimensional values. Hence we  restrict attention mostly to the covariant setting, namely $\foutQ[\gr]$ and $\fpoutgrQ[d]$, for $d \in \nat$.
\item 
We are most interested in polynomial outer functors, namely $\fpoutgrQ[<\infty] = \bigcup_{d\in \nat} \fpoutgrQ[d]$. There are inclusions:
\[
\fpiQ{\ab}
\subsetneq
\fpoutgrQ[<\infty]
\subsetneq 
\fpiQ{\gr}.
\]
We show that, cohomologically, $\fpoutgrQ[<\infty]$ is much closer to $\fpiQ{\gr}$ than to $\fpiQ{\ab}$. The process of fleshing out this assertion begins in the following section.
\end{enumerate}
\end{rem}

\section{Outer adjoints}
\label{sec:outer_adjoints}

The right adjoint to the inclusion $\foutk[\gr]\hookrightarrow \fcatk[\gr]$ turns out to be of significant interest; for example, it arises in Part \ref{part:HHH} considering higher Hochschild homology. This right adjoint is introduced in this section, as well as the corresponding left adjoint in the contravariant setting. There are also adjoints on the other side, but these do not arise directly in relation to higher Hochschild homology.

For most of this section, we work with $\kring$ an arbitrary unital, commutative ring.

\subsection{The functors $\omega$ and $\Omega$}
\label{subsect:omega_Omega}

We start by giving an alternative characterization of the category $\foutk[\gr]$ (respectively $\foutk[\gr\op]$)  using the natural transformation $\kappa$ (resp. $\kappa^*$) introduced below. This then allows us to give an explicit description of the adjoints that interest us.

\begin{nota}
\label{nota:kappa}
 For $G \in \ob \gr$, denote by $\kappa_G : G \rightarrow G \star \zed$ the 
group morphism $g \mapsto x g x^{-1}$, where $x$ denotes a fixed generator of 
$\zed$, so that $\kappa$ is a natural transformation with respect to $G$. 
\end{nota}

\begin{rem}
\label{rem:univ_conj}
For  $G \in \ob \gr$, 
 \begin{enumerate}
  \item 
  $\kappa_G$ is the composite of the canonical inclusion $G 
\subset G \star \zed$ with $\ad (x) : G \star \zed \rightarrow G \star \zed$. 
  \item 
  $\kappa_G$ is the {\em universal conjugation} in the following sense: for $h 
\in G$, $\ad (h)$ is given by the composite of $\kappa_G$ with the group 
morphism $G \star \zed \rightarrow G$ given by the identity on $G$ and sending 
$x$ to $h$. 
 \end{enumerate}
\end{rem}

Below we use the functors $\tau_\zed$ and $\taubar$ introduced in Notation \ref{nota:shift_functors}.

\begin{nota}
\ 
\begin{enumerate}
\item 
For $F \in \ob \f(\gr; \rat)$, let $\kappa : F \rightarrow \tau_\zed F$ denote the natural transformation induced by the morphisms $\kappa_G$. 
\item
For $F \in \ob \f(\gr\op; \rat)$, let $\kappa^* : \tau_\zed F \rightarrow F$ denote the natural transformation induced by the morphisms $\kappa_G$. 
\end{enumerate}
\end{nota}

\begin{lem}
\label{lem:fout_(co)equalizer}
 \ 
 \begin{enumerate}
  \item 
  The subcategory $\foutk[\gr]$ is the full subcategory of functors $F$ for 
which the natural transformations 
  \[
   \xymatrix{
   F 
   \ar@<.5ex>[r]^(.4){\kappa} 
   \ar@<-.5ex>[r]_(.4){\iota}
   &
   \tau_\zed F
   }
  \]
coincide, where $\iota$ is induced by the 
canonical inclusion $G \hookrightarrow G \star \zed$.
\item 
 The subcategory $\foutk[\gr\op]$ is the full subcategory of functors $F$ for 
which the following natural transformations coincide
  \[
   \xymatrix{
   \tau_\zed F 
   \ar@<.5ex>[r]^{\kappa^*} 
   \ar@<-.5ex>[r]_{\iota^*}
   &
   F
   }
  \] where $\iota^*$ is induced by the 
canonical inclusion $G \hookrightarrow G \star \zed$.
\end{enumerate}
\end{lem}

\begin{proof}
This is a  consequence of the fact that $\kappa_G$ is the universal 
conjugation. Consider the covariant case. 
Since 
$\kappa$ can be written as $\ad (x) \circ 
\iota$, it is clear that, 
 if $F \in \foutk [\gr]$, then the two morphisms coincide. Conversely, for 
given 
$\zed^{\star n} \in \ob \gr$, consider the 
 natural map 
 \[
  \tau_\zed F (\zed ^{\star n}) 
  \rightarrow 
  \prod_{i=1}^n F (\zed^{\star n}), 
 \]
where the $i$th map is induced by sending $x$ to $x_i$, the generator of the 
$i$th factor of $\zed^{\star n}$, and the identity on $\zed^{\star n}$.
 The equalizer of the two morphisms given by composition:
 \[
   \xymatrix{
   F (\zed^{\star n}) 
   \ar@<.5ex>[r]^{\kappa} 
   \ar@<-.5ex>[r]_{\iota}
   &
   \tau_\zed F (\zed^{\star n}) 
   \ar[r]
   & 
   \prod_{i=1}^n F (\zed^{\star n}) 
   }
 \]
is, by construction, $F(\zed^{\star n}) ^{\ad(\zed^{\star n})}$, the invariants 
for 
the conjugation action. Hence, if $\kappa$ and $\iota$ coincide on $F$, then $F 
\in \ob \foutk[\gr]$, as required. 
 
The proof of the contravariant case is categorically dual.
\end{proof}

\begin{rem}
\label{rem:differ_equalizers}
 In general, the map  $ \tau_\zed F (\zed ^{\star 
n}) 
  \rightarrow 
  \prod_{i=1}^n F (\zed^{\star n})$ in the above proof is not injective, so that the 
equalizer of the natural transformations $\kappa $ and $\iota$ evaluated on 
$\zed ^{\star n}$ is  smaller than the invariants
   $F(\zed^{\star n}) ^{\ad(\zed^{\star n})}$ in general.
\end{rem}

\begin{defn}
\label{def:omega}
\nomenclature{$\omega$}{right adjoint to $\foutk[\gr]\subset \fcatk[\gr]$\nomrefpage}
\nomenclature{$\Omega$}{left adjoint to $\foutk[\gr\op]\subset \fcatk[\gr\op]$\nomrefpage}
Define the functors 
 \begin{enumerate}
  \item 
$\omega : \fcatk[\gr] \rightarrow \foutk [\gr]$   by 
  $
   \omega F := \mathrm{equalizer } \big( \xymatrix{
   F 
   \ar@<.5ex>[r]^{\kappa} 
   \ar@<-.5ex>[r]_{\iota}
   &
   \tau_\zed F
   }
   \big);
  $
\item 
$\Omega : \fcatk[\gr\op] \rightarrow \foutk [\gr\op]$  by 
  $
   \Omega F := \mathrm{coequalizer } \big( \xymatrix{
  \tau_\zed F 
   \ar@<.5ex>[r]^{\kappa^*} 
   \ar@<-.5ex>[r]_{\iota^*}
   &
    F
   }
   \big).
  $
 \end{enumerate}
\end{defn}

\begin{rem}
\label{rem:check_omegas}
To see that Definition \ref{def:omega} gives functors as claimed, one has to 
check that  $\omega$, $\Omega$ take values in the 
respective categories $\foutk[\gr]$ and $\foutk[\gr\op]$. 

Consider the covariant case (the contravariant case is categorically dual). It is clear that 
$\omega 
F$  is a functor in $\fcatk[\gr]$; we show that it lies in 
the subcategory $\foutk[\gr]$ by using the criterion provided by Lemma \ref{lem:fout_(co)equalizer}. 

The canonical inclusion $\omega F \hookrightarrow F$ induces a commutative 
diagram 
 \[
  \xymatrix{
  \omega F 
   \ar@<.5ex>[r]^{\kappa} 
   \ar@<-.5ex>[r]_{\iota}
  \ar@{^(->}[d]
  &
   \tau_\zed \omega F
    \ar@{^(->}[d]
    \\
    F 
   \ar@<.5ex>[r]^{\kappa} 
   \ar@<-.5ex>[r]_{\iota}
   &
   \tau_\zed F,
  }
 \]
where the injectivity of the right hand vertical arrow follows from the 
exactness of $\tau_\zed$. It follows that $\kappa$ and $\iota$ coincide on 
$\omega F$.
\end{rem}

One can also  define  the functors $\omega$ and $\Omega$ by using 
 $\taubar$, as follows.

\begin{nota}
\label{nota:kappabar}
 \ 
 \begin{enumerate}
  \item 
  For $F \in \ob\fcatk[\gr]$, denote by $\overline{\kappa} : F \rightarrow 
\taubar 
F$  the composite
  $
   F \stackrel{\kappa}{\rightarrow} \tau_\zed F \twoheadrightarrow \taubar F$,   
where the surjection is the canonical projection. 
\item 
 For $F \in \ob\fcatk[\gr\op]$, denote by $\overline{\kappa^*} :  \taubar F 
\rightarrow F$  the composite
  $ 
 \taubar F \hookrightarrow \tau_\zed F   \stackrel{\kappa^*}{\rightarrow} F $,
 where the monomorphism is the canonical inclusion.  
 \end{enumerate}
\end{nota}

\begin{prop}
 \label{prop:omega_Omega_taubar}
 \ 
 \begin{enumerate}
  \item 
    For $F \in \ob\fcatk[\gr]$, there is a natural isomorphism 
    $
      \omega F \cong \ker \big( F \stackrel{\overline{\kappa}}{\rightarrow}  
\taubar F \big).
    $
\item 
 For $F \in \ob\fcatk[\gr\op]$, there is a natural isomorphism 
 $
    \Omega F \cong  \mathrm{coker } \big( \taubar F 
\stackrel{\overline{\kappa^*}} {\rightarrow} F\big).
 $
 \end{enumerate}
\end{prop}

\begin{proof}
We prove the first statement; the proof of the second is similar. By definition, $\omega F$ is the kernel of $\kappa - \iota$. This morphism factors canonically across $\overline{\kappa}$ via the inclusion $\taubar F \hookrightarrow \tau_\zed F$ that is given as the kernel of the canonical projection $\tau_\zed F \twoheadrightarrow \tau_0 F = F$. 
\end{proof}

The significance of the functors $\omega$ and $\Omega$ is established by:

\begin{prop}
\label{prop:omega_Omega_adjunction}
 \ 
 \begin{enumerate}
  \item 
  The functor $\omega : \fcatk[\gr] \rightarrow \foutk [\gr]$ is right adjoint 
to the inclusion $\foutk[\gr] \hookrightarrow \fcatk[\gr]$.  In particular,   the functor $\omega$ is left exact.
\item 
  The functor $\Omega : \fcatk[\gr\op] \rightarrow \foutk [\gr\op]$ is left 
adjoint to the inclusion $\foutk[\gr\op] \hookrightarrow \fcatk[\gr\op]$.  In particular,  the functor $\Omega$ is right exact.
\end{enumerate}
\end{prop}

\begin{proof}
This follows, using the respective universal 
properties, from the  definitions of $\omega$ and 
$\Omega$.
\end{proof}

Since  $\foutk[\gr]$ and $\foutk[\gr\op]$ are closed under subquotients, one has the formal consequence:

\begin{cor}
 \label{cor:omega_pb}
 \ 
 \begin{enumerate}
  \item 
If 
   $F \hookrightarrow G$ is a monomorphism of $\fcatk[\gr]$, then the following commutative diagram is cartesian:
  \[
   \xymatrix{
   \omega F 
   \ar[r]
      \ar[d]
      &
      \omega G
      \ar@{^(->}[d]
      \\
      F    \ar@{^(->}[r]
      &
      G.   
   }
  \]
  \item 
If  $F \twoheadrightarrow G$ is a surjection of $\fcatk[\gr\op]$, then the following commutative diagram is cocartesian:
  \[
   \xymatrix{
   F 
   \ar@{->>}[r]
   \ar@{->>}[d]
   &
   G 
   \ar[d]
   \\
   \Omega F
   \ar[r]
   &
   \Omega G.
   }
  \]
 \end{enumerate}
\end{cor}

In the next Proposition, we show that $\omega$ always acts non-trivially on polynomial functors (working here over $\rat$). This should be contrasted with the behaviour exhibited in Example \ref{exam:proj_abel_omega} below.

\begin{prop}
\label{prop:non-triviality_omega}
For $0\neq F \in \ob \fpiQ{\gr}$ of polynomial degree exactly $d$, there is a canonical inclusion:
\[
\qhat{d} F \subset \omega F, 
\]
in particular $\omega F \neq 0$.
\end{prop}

\begin{proof}
The hypothesis on $F$ implies, using the polynomial filtration of $F$ (see Section \ref{subsect:poly_filt}), that $\qhat{d}F \neq 0$ and is a subfunctor of $F$. Since $\qhat{d}F$ lies in the image of $\f(\ab;\rat) \hookrightarrow \foutQ[\gr]$, the result follows. 
\end{proof}

\begin{exam}
\label{exam:proj_abel_omega}
Consider the projective functor $P^\gr_\zed \in \ob \fcatQ[\gr]$;  this splits as $P^\gr_\zed \cong \rat \oplus \overline{P^\gr_\zed}$. We claim that $\omega \overline{P^\gr_\zed}=0$,  equivalently  that $
 \omega P^\gr_\zed  = \rat$ (using $\omega \rat = \rat$). 

 That  $
 \omega P^\gr_\zed  = \rat$ is seen as follows. By definition, 
 \[
  \omega P^\gr_\zed  (G) = \mathrm{equalizer } \big( \xymatrix{
   \rat [G] 
   \ar@<.5ex>[r]^(.4){\kappa} 
   \ar@<-.5ex>[r]_(.4){\iota}
   &
   \rat[G \star \zed]
\big)  
 }.
 \]
One has the obvious splitting  $\rat[G \star \zed ] \cong \rat[G] \oplus \rat [(G\star \zed) 
\backslash G]$ of $\rat$-vector spaces, where the inclusion of $\rat[G]$ is given by $\iota$. For $g \not = e \in G$, $\kappa [g]$ lies in $\rat [(G\star \zed) 
\backslash G]$, from which the claimed identification follows.

This is striking, in view of Proposition \ref{prop:non-triviality_omega}. Namely, for any $0<d \in \nat$, one has the canonical  surjection:
\[
\overline{P^\gr_\zed}
\twoheadrightarrow 
q_d ^\gr \overline{P^\gr_\zed}
\] 
and this is non-trivial. By construction, $q_d ^\gr \overline{P^\gr_\zed}$ is polynomial and hence, by Proposition \ref{prop:non-triviality_omega}, $\omega q_d ^\gr \overline{P^\gr_\zed} \neq 0$. However, by the above, on applying $\omega$ to the canonical surjection one has 
\[
0 = \omega \overline{P^\gr_\zed}
\rightarrow 
\omega q_d ^\gr \overline{P^\gr_\zed} \neq 0,
\]
thus exhibiting the dramatic failure of $\omega$ to be exact.
\end{exam}

\begin{rem}
The behaviour exhibited in Example \ref{exam:proj_abel_omega} should be compared with the formation of the functor $P^\gr_\zed / \ad$ that is considered in Section \ref{subsect:inv_coinv_ad}, where $(-)/\ad$ is the {\em left} adjoint to the inclusion of $\foutQ[\gr]$.
\end{rem}

\subsection{Monoidal properties of $\omega$ and $\Omega$}

The behaviour of $\omega$ and $\Omega$ with respect to the  tensor products of $\f(\gr; \kring) $ and $\f(\gr\op; \kring)$ respectively is of interest. 
The proofs (by applying Proposition \ref{prop:foutk}) of the following results are left to the reader.

\begin{prop}
\label{prop:omegas_monoidal}
\ 
\begin{enumerate}
\item 
The functor $\omega$ is lax symmetric monoidal: for $F, G \in \ob \fcatk[\gr]$, 
there is a natural morphism:
\[
(\omega F) \otimes (\omega G) 
\rightarrow 
\omega (F \otimes G) 
\]
in $\foutk[\gr]$; this is a monomorphism if $\omega F$ and $G$   both take values in flat $\kring$-modules (respectively for $\omega G$ and $F$).

Similarly, there is a natural morphism
\[
(\omega F) \otimes (\omega G) 
\rightarrow 
\omega (F \otimes (\omega G)) 
\]
that is an isomorphism
if $\omega G$ takes values in flat $\kring$-modules.
\item 
The functor $\Omega$ is colax symmetric monoidal; for $F, G \in \ob \fcatk[\gr 
\op]$, there is a natural epimorphism:
\[
\Omega (F \otimes G) \twoheadrightarrow 
(\Omega F) \otimes (\Omega G).
\]
Similarly, there is a natural isomorphism
\[
\Omega (F \otimes (\Omega G)) 
\stackrel{\cong}{\rightarrow} 
(\Omega F) \otimes (\Omega G).
\]
\end{enumerate}
\end{prop}

This has the following immediate consequence, in which, for the unit, one exploits the fact that the constant functor $\kring$ is an outer functor (so that $\omega \kring = \kring$): 

\begin{cor}
 \label{cor:algebra_omega}
 Let $A \in \ob \fcatk[\gr]$ be a functor taking values (naturally) in unital, 
 $\kring$-algebras, with product $\mu_A : A\otimes A \rightarrow A$.  Then $\omega A \in \ob \foutk[\gr]$ is naturally a unital sub $\kring$-algebra 
of $A$, with product the composite:
 \[
  (\omega A ) \otimes (\omega A) \rightarrow \omega (A \otimes A) 
\stackrel{\omega(\mu _A)} {\rightarrow} \omega A.
 \]

If $M \in \ob \fcatk[\gr]$ is an $A$-module, then $\omega M$ is naturally an $\omega A$ module.  
\end{cor}

\subsection{Generalizing the target category}
For an arbitrary abelian category $\cala$ we have defined $\fout{\gr}{\cala} \subset \fcat{\gr}{\cala}$. This  inclusion admits a right adjoint:
\[
\omega_\cala : \fcat{\gr}{\cala}
\rightarrow 
\fout{\gr}{\cala} .
\]

Generalizing Proposition \ref{prop:omega_Omega_taubar}, one has the following result, in which $\taubar$ denotes the reduced shift functor, defined for $\f( \gr; \cala)$ as in Notation \ref{nota:shift_functors}.

\begin{prop}
\label{prop:omega_cala}
The functor $\omega_\cala : \fcat{\gr}{\cala}
\rightarrow 
\fout{\gr}{\cala} $ is given on $F \in \ob \fcat{\gr}{\cala}$ by 
   $
      \omega_\cala F :=  \ker \big( F \stackrel{\overline{\kappa}}{\rightarrow}  
\taubar F \big).
    $
\end{prop}

\begin{exam}
\label{exam:omega_calc_fcatkC}
Let $\calc$ be a small category and $\cala$ be the category $\fcatk[\calc]$. Then $\fcat{\gr}{\cala}$ is equivalent to the category of bifunctors $\fcatk[\gr \times \calc]$. The functor 
\[
\omega_\cala : \fcatk[\gr \times \calc] \rightarrow \fout{\gr}{\fcatk[\calc]} \subset \fcatk[\gr \times \calc]
\]
is compatible with $\omega$ on $\fcatk[\gr]$ via evaluation on objects $C$ of $\calc$. Namely, for a bifunctor $F$ and $C \in \ob \calc$, there is an isomorphism in $\fcatk[\gr]$:
\[
\eval_C (\omega_\cala F) \cong \omega (\eval_C F)
\]
and this is natural in $C$. Hence, in this case, $\omega_\cala$ can be understood in terms of $\omega$.
\end{exam}

An exact functor $\Theta : \cala \rightarrow \calb$ between abelian categories induces the exact functor 
$
\Theta_* : \fcat{\gr}{\cala} \rightarrow \fcat{\gr}{\calb} 
$
by post-composition. This clearly restricts to an exact functor 
$$\Theta_* : \fout{\gr}{\cala} \rightarrow \fout{\gr}{\calb}.$$

The following straightforward result gives the compatibility between the respective functors $\omega$:

\begin{prop}
\label{prop:omega_cal_compat}
Let $\Theta : \cala \rightarrow \calb$ be an exact functor between abelian categories. Then there is 
a natural isomorphism
\[
\Theta_* \circ \omega_\cala \cong \omega _\calb \circ \Theta_*
\]
between functors $\fcat{\gr}{\cala} \rightarrow \fout{\gr}{\calb}$.
\end{prop}

\subsection{Recollement for outer polynomial functors}
\label{subsect:recoll_out} 
In this section, we take $\kring=\rat$ so that the results of Section \ref{sec:recoll} can be applied. Theorem \ref{thm:recollement_poly_d} has the following counterpart for  $\foutQ[\gr]$:

\begin{cor}
\label{cor:recollement_pout}
For $d \in \nat$, the recollement diagram of Theorem \ref{thm:recollement_poly_d} induces  
\[
 \xymatrix{
\fpoutgrQ[d-1] 
\ar@{^(->}[r]
&
\fpoutgrQ[d]
\ar@<-.5ex>@/_1pc/[l]|{q_{d-1}^\gr}
\ar@<.5ex>@/^1pc/[l]|{p_{d-1}^\gr}
\ar[r]|(.45){\cre_d}
&
\rat [\sym_d]\dash \modules
\ar@<-.5ex>@/_1pc/[l]|{\alpha_d}
\ar@<.5ex>@/^1pc/[l]|{\omega \beta_d}.
}
\]
\end{cor}

\begin{proof}
By the explicit description given in Theorem \ref{thm:recollement_poly_d}, it is clear that the functor $\alpha_d$ takes values in $\f_d(\ab; \rat)$ and hence in $\fpoutgrQ[d]$, thus gives the left adjoint to the cross-effect functor $\cre_d$. For the right adjoint, it is necessary to compose $\beta_d$ with the functor $\omega$. Since $\foutQ[\gr]$ is closed under subquotients, by Proposition \ref{prop:foutk}, the functors $p_{d-1}^\gr$ and $q_{d-1}^\gr$ restrict as indicated. 
\end{proof}

The following basic example serves both to show that the category $\foutQ[\gr]$ is not thick and that 
 the composite functor $\omega \beta_n$ is, in general, not equal to $\beta_n$.

\begin{exam}
\label{exam:Passi_not_out}
Consider the Passi functor $q_2^\gr P_\zed^\gr$. As in the proof of Lemma \ref{lem:assoc_grad_P_zed} (see also \cite{Passi}), $q_2^\gr P_\zed^\gr$ is the quotient of the group ring functor $G \mapsto P_\zed^\gr (G) = \rat [G]$ by the third power $\mathcal{I}^3 (G)$ of the augmentation ideal; equivalently  this is 
the quotient of $G \mapsto \rat [G]$ by the relation:
\begin{eqnarray*}
&& [ghk] - [e] \equiv \\
&&([g]-[e]) ([h]-[e]) 
+ ([g]-[e]) ([k]-[e]) 
+ ([h]-[e]) ([k]-[e])
\\
&&
+ \big( 
([g]-[e]) + ([h]-[e]) + ([k]-[e])
\big).  
\end{eqnarray*}
Now, direct calculation shows that $([x]-[e])([x^{-1}] -[e]) = - \big ( ([x]-[e]) + 
([x^{-1}] -[e]) \big)$. Hence, taking $h=x$ and $k = x^{-1}$, the above relation reduces to:
\[
([g]-[e]) ([x]- [e]) + ([g]-[e]) ([x^{-1}]- [e]) \equiv 0.
\] 

Combining this with the relation for the triple  $(g=x,h,k= x^{-1})$, one gets:
\[
 ([xhx^{-1}] - [e]) - ([h]-[e] ) 
\equiv 
([x]-[e]) ([h] - [e] ) - ([h]-[e]) ([x]-[e]).  
\]

The left hand side of this expression essentially corresponds to $\overline{\kappa}(h)$, considering 
$x$ as the generator of the copy of $\rat$ corresponding to $\taubar$. In particular, this is non-zero. 
It follows that $q_2^\gr P_\zed^\gr$ does not  lie in $\foutQ[\gr]$.

However, the polynomial filtration shows that 
$q_2^\gr P_\zed^\gr$ is an extension of functors from $\fcatQ[\ab]$:
\[
 0
\rightarrow 
\aQ ^{\otimes 2} 
\rightarrow 
q_2^\gr P_\zed^\gr 
\rightarrow 
\rat \oplus \aQ
\rightarrow 
0. 
\]
The above calculation shows  that $\omega q_2^\gr P_\zed^\gr \cong \rat \oplus \aQ^{\otimes 2}$, showing that $\omega q_2^\gr \neq q_2^\gr$ and that $\foutQ[\gr]$ is not a thick subcategory of $\f(\gr; \rat)$.

The above can be refined using that  $\aQ^{\otimes 2} $ splits as a direct sum of simple functors $\Lambda^2 \circ \aQ \oplus S^2 \circ \aQ$; using the results of Section \ref{subsect:recoll_poly}, this splitting is obtained by applying the functor $\alpha_2$ to the splitting of the regular representation $\rat [\sym_2] \cong \mathrm{sgn}_2 \oplus \mathrm{triv}_2$ as the direct sum of the signature and trivial representations.

It follows from the results of \cite{V_ext}) that $q_2^\gr P_\zed ^\gr \cong \rat \oplus S^2 \circ \aQ \oplus \mathcal{E}$, where 
\[
 0
\rightarrow 
\Lambda^2 \circ \aQ
\rightarrow 
\mathcal{E}
\rightarrow 
 \aQ
\rightarrow 
0
\] 
is the unique non-trivial extension (up to isomorphism). As above, $\mathcal{E}$ does not lie in $\foutQ[\gr]$, whereas $S^2 \circ \aQ$ lies in the image of  $\f(\ab; \rat)$ in $  \f(\gr; \rat)$ and hence in $\foutQ[\gr]$.

Now, $\mathcal{E}$ can be shown to be isomorphic to $\beta_2 \mathrm{sgn}_2$, hence we see that $\omega \beta_2 \mathrm{sgn}_2 \subsetneq \beta_2 \mathrm{sgn}_2$. More generally, for $n>1$, the results of this paper show that 
$
\omega \beta_n \subsetneq \beta_n$.
\end{exam}

The recollement framework for outer polynomial functors leads to a family of injective cogenerators:

\begin{prop}
\label{prop:inj_cogenerators_Out_poly}
\ 
\begin{enumerate}
\item 
For $M$ a $\rat[\sym_d]$-module, $\omega \beta_d M$ is  injective in $\fpoutgrQ[<\infty]$.
\item 
$\big\{ \omega \beta_d \rat[\sym_d] \ | \ d \in \nat\}$ is a family of injective cogenerators of $\fpoutgrQ[<\infty]$.
\end{enumerate}
\end{prop}

\begin{proof}
By Theorem \ref{thm:injectivity_beta}, $\beta_d M$ is injective in $\fpiQ{\gr}$. Since $\omega$ is right adjoint to an exact functor, namely the inclusion $\fpoutgrQ[<\infty] \subset  \fpiQ{\gr}$, it follows that $\omega \beta_d M$ is  injective in $\fpoutgrQ[<\infty]$. The second statement follows directly.
\end{proof}

Hence we have the following consequence of Theorem \ref{thm:beta_description}, which establishes the significance of the functor $\omega \Psi \pcoalg$ when working with polynomial outer functors.

\begin{cor}
\label{cor:omega_boldbeta_inj_cogen}
 There is a natural isomorphism in $\f(\gr \times \fb ; \rat)$:
 $$\omega \boldbeta \cong \omega \Psi \pcoalg.$$
 
In particular, the functor  $\omega \Psi \pcoalg$  encodes a family of injective cogenerators of 
$\fpoutgrQ[<\infty]$.
\end{cor}

Moreover, we have the following counterpart of Corollary \ref{cor:inj_envelope}, which gives an intrinsic characterization of $\omega \beta_d S_\lambda$:

\begin{cor}
\label{cor:inj_envelope_omega}
For $\lambda \vdash d$, 
$\omega \beta_d S_\lambda$ is the injective envelope of the simple $\alpha_d S_\lambda$ in $\fpoutgrQ[<\infty]$.
\end{cor}

\subsection{Invariants and coinvariants for the adjoint action}
\label{subsect:inv_coinv_ad}

The inclusions of the full subcategories of outer functors $\foutk[\gr] \subset \fcatk[\gr]$ and $\foutk[\gr\op]\subset \fcatk[\gr\op]$ also admit a left adjoint and right adjoint respectively:

\begin{prop}
\label{prop:easy_out_adjoints}
\ 
\begin{enumerate}
\item 
The inclusion $\foutk[\gr] \hookrightarrow \fcatk[\gr]$ has left adjoint 
$(-)/\ad
: 
\fcatk[\gr] \rightarrow \foutk[\gr]$ defined by passage to coinvariants under 
the adjoint action; i.e.,  for $F \in \ob \fcatk[\gr]$, 
\[
 (F/\ad) (G) := F (G) /\ad (G).
\]
 \item 
 The inclusion $\foutk[\gr \op] \hookrightarrow \fcatk[\gr \op]$ has right 
adjoint 
 $$(-)^\ad : \fcatk[\gr \op] \rightarrow \foutk [\gr]$$
  given by passage to 
invariants under the adjoint action; i.e., for $F \in \ob \fcatk[\gr\op]$, 
 \[
 F^ \ad (G) := F (G) ^{\ad (G)}. 
\]
\end{enumerate}
\end{prop}

\begin{proof}
 It clearly suffices to prove that passage to invariants (respectively 
coinvariants) is functorial. This follows from the naturality of the adjoint 
action; the proof is given for the contravariant case. 
 
 Consider a group morphism $\phi : G' \rightarrow G $, and the induced morphism 
$F (\phi) : F (G) \rightarrow F (G')$ for $F \in \ob \fcatk[\gr\op]$. By 
definition, $\rho F (G) := F (G) ^{\ad (G)}$; it suffices to show that this maps under $F(\phi)$ to the $\ad 
(G')$-invariants of $F(G')$. Hence, choose $g' \in G'$ and set $g:= \phi (g') 
\in G$. The commutative diagram 
 \[
  \xymatrix{
  G' \ar[r]^{\ad (g')} 
  \ar[d]_\phi 
  &
  G' \ar[d]^\phi 
  \\
  G 
  \ar[r]_{\ad (g)}
  &
  G
  }
 \]
induces the commutative diagram:
\[
  \xymatrix{
 F( G) \ar[r]^{\ad (g)} 
  \ar[d]_{F(\phi)} 
  &
 F( G) \ar[d]^{F(\phi)} 
  \\
  F(G') 
  \ar[r]_{\ad (g')}
  &
  F(G'),
  }
 \]
from which it follows that $\rho F(G)$ maps to the invariants under the 
action of $\ad (g')$ for all such $g ' \in G'$, whence the result.
\end{proof}

The following is an immediate consequence of Propositions \ref{prop:omega_Omega_adjunction} and  
\ref{prop:easy_out_adjoints}:

\begin{cor}
\label{cor:foutk_easy_proj_inj}
The categories   $\foutk[\gr]$ and  $\foutk[\gr\op]$ both have 
 enough injectives and enough projectives.
\end{cor}

\begin{rem}
For the applications here to higher Hochschild homology, we are more interested in  {\em polynomial}  outer functors and hence the injectives in $\f_{<\infty}^{\mathrm{Out}} (\gr; \rat)$  (as considered in Section \ref{subsect:recoll_out}), rather than those provided by Corollary \ref{cor:foutk_easy_proj_inj}.
\end{rem}

The adjoint functors $F \mapsto F /\ad$ (respectively $F \mapsto F^\ad$) can be given a treatment that is analogous to that of $\omega$ and $\Omega$ in Section \ref{subsect:omega_Omega}. To explain this, we concentrate on the covariant case. 

Namely, $\kappa$ is adjoint (via Proposition \ref{prop:properties_tau}) to 
the natural transformation 
$$
\mathrm{Ad} :  P^\gr_\zed \otimes F \rightarrow F 
$$
that corresponds to  the adjoint action
$
\kring[G] \otimes F (G) \rightarrow F(G)
$ 
that sends an element $[g] \otimes f$ (for $g \in G$ and $f \in   F (G$) to $\ad (g) f$. Likewise, $\overline{\kappa}$ is adjoint to the reduced action
$$
\overline{\mathrm{Ad}} : 
\overline{P^\gr_\zed} \otimes F \rightarrow F. 
$$

\begin{prop}
\label{prop:/ad}
For $F \in \ob \fcatk[\gr]$, there are natural isomorphisms:
\begin{eqnarray*}
F/ \ad & \cong & \mathrm{coequalizer} \  \Big(
 P^\gr_\zed \otimes F 
\substack{\mathrm{Ad} \\ \rightrightarrows \\ \epsilon \otimes \id_F} 
F 
\Big) 
\\
 & \cong & 
 \coker \Big( 
\overline{P^\gr_\zed} \otimes F \stackrel{\overline{\mathrm{Ad}} }{\longrightarrow} F
\Big),
\end{eqnarray*} 
where $\epsilon : P^\gr_\zed \rightarrow \kring$ is the projection.
\end{prop}

\begin{proof}
The description of $F/\ad$ as the coequalizer of $\mathrm{Ad}$ and $\epsilon \otimes \id_F$ is just a functorial reinterpretation of the definition of the coinvariants for the adjoint action. The fact that this corresponds to the cokernel of $\overline{\mathrm{Ad}}$ follows analogously to Proposition \ref{prop:omega_Omega_taubar}.
\end{proof}

\begin{rem}
\label{rem:other_adjoints}
For $F \in \ob \fcatk[\gr]$ and $G \in \ob \gr$, it is clear that there is a natural inclusion 
$ 
\omega F (G) \subset F (G)^{\ad(G)} . 
$ 
However, the codomain defines a subfunctor of $F$ if and only if these are equalities, for all $G$. 
 Similarly, if $F \in \ob \fcatk[\gr\op]$, there is a natural surjection $F(G)/ \ad(G) \twoheadrightarrow \Omega F(G)$. The domain is a quotient functor of $F$ if and only if all of these surjections are equalities. 

This is related to the phenomenon observed by Conant and Kassabov \cite[Remark 4.18]{CK}, who focus on the action of the groups $\mathrm{Aut}(\zed^{\star r})$ rather than working with functors. 
\end{rem}

\begin{exam}
\label{exam:P/ad}
Consider the functor $F = P^\gr _\zed$. Then $P^\gr_\zed/ \ad $ identifies as 
\[
G  \mapsto |\kring G | := \kring G  / [\kring G , \kring G  ],
\]
where, for typographical simplicity,  here $\kring G $ denotes the group ring and $[\kring G , \kring G  ]$ is the $\kring$-module generated by the elements $[gh]- [hg]$ for $g, h \in G $. 

We note that $|\kring G |$ can be given a topological interpretation, as follows.
 Choose an oriented surface $\Sigma$ with one boundary component such that  $G$ is isomorphic to the fundamental group $\pi_1 (\Sigma)$. Then $|\kring G|$ is the $\kring$-linear span of the homotopy classes of free loops in $\Sigma$. This has additional structure arising from the surface, such as the Goldman bracket; see  \cite{MR3758425}, for example. The naturality of $|\kring G|$ with respect to the group $G$ (as a functor to $\kring$-modules) may be of interest in this context.

By construction, $G  \mapsto |\kring G |$ defines a functor in $\foutk [\gr]$, and there are surjections (not isomorphisms)
\[
P^\gr _\zed
\twoheadrightarrow 
P^\gr_\zed/ \ad 
\twoheadrightarrow 
\kring [-] \circ \A,
\]
where $\kring [-] \circ \A$ lies in $\fcatk[\ab] \subset \fcatk[\gr]$, corresponding to the natural surjections:
\[
\kring G  
\twoheadrightarrow 
|\kring G |
\twoheadrightarrow 
\kring G _\ab.
\]
Moreover, there is a natural splitting $|\kring G | = \kring \oplus \overline{|\kring G |}$, where $G  \mapsto \overline{|\kring G |}$ is the functor $\overline{P^\gr_\zed} /\ad$.

The product in the group ring restricts to $\mu : \overline{P^\gr_\zed} \otimes \overline{P^\gr_\zed} \rightarrow \overline{P^\gr_\zed}$ and the analysis of the Passi filtration (cf. Example \ref{exam:Passi_not_out}) shows that this fits into the exact sequence 
\begin{eqnarray}
\label{eqn:presentations_ak}
\overline{P^\gr_\zed} \otimes \overline{P^\gr_\zed} \stackrel{\mu}{\rightarrow} \overline{P^\gr_\zed}
\rightarrow 
\ak 
\rightarrow 
0.
\end{eqnarray}
(This corresponds to the beginning of the reduced bar construction for the group ring functor.)

Defining $[-,-] := \mu - \mu \circ \tau$, where $\tau$ transposes the tensor factors, gives 
$[,] : \overline{P^\gr_\zed} \otimes \overline{P^\gr_\zed} \rightarrow \overline{P^\gr_\zed}$. Essentially by construction, one has the exact sequence:
\begin{eqnarray}
\label{eqn:pbar/ad}
\overline{P^\gr_\zed} \otimes \overline{P^\gr_\zed} \stackrel{[-,-]}{\longrightarrow} \overline{P^\gr_\zed}
\rightarrow 
\overline{P^\gr_\zed}/\ad
\rightarrow 
0.
\end{eqnarray}
that gives a presentation of $\overline{P^\gr_\zed}/\ad$.

The natural surjection $\overline{P^\gr_\zed}
\twoheadrightarrow 
\ak $ factorizes to give a natural surjection $\overline{P^\gr_\zed}/\ad \twoheadrightarrow \ak$.
\end{exam}

\begin{exam}
\label{exam:P_higher/ad}
More generally, one can consider $P^\gr_{\zed^{\star r}}$ for any $r \in \nat$. This functor sends $G $ to $(\kring G )^{\otimes r}$ and the adjoint action acts {\em diagonally}.  In this case, for $g_1, \ldots , g_r$ and $h$ in $G $, the basic relation is 
\[
[g_1 h] \otimes \ldots \otimes [g_r h] 
= 
[hg_1] \otimes \ldots \otimes [hg_r].
\]
Hence the functor $P^\gr_{\zed^{\star r}}/\ad$ is less familiar in the case $r >1$.
\end{exam}

We conclude this section by analysing the structure of $\overline{P^\gr_\zed} /\ad$ of Example \ref{exam:P/ad}, taking $\kring = \rat$ for simplicity and so as to be able to appeal to Section  \ref{subsect:recoll_out}.

\begin{prop}
\label{prop:pbar/ad_split}
The natural surjection  $\overline{P^\gr_\zed}/\ad \twoheadrightarrow \aQ$ admits a section, so that $\aQ$ is a direct summand of $\overline{P^\gr_\zed}/\ad$. In particular, $\aQ$ is projective in $\foutQ[\gr]$. 
\end{prop}

\begin{proof}
To construct the section $\aQ \rightarrow  \overline{P^\gr_\zed}/\ad$, we first construct a  map $\overline{P^\gr_\zed} \rightarrow 
\overline{P^\gr_\zed}/\ad$ corresponding to the composite with the projection $\overline{P^\gr_\zed} \twoheadrightarrow \aQ$. As in  Example \ref{exam:Passi_not_out}, $q^\gr_2 \overline{P^\gr_\zed}$ splits as $S^2 \circ \aQ \oplus \mathscr{E}$ where $\mathscr{E}$ is not in $\foutQ[\gr]$
 and lies in the (non-split) short exact sequence 
\[
 0
\rightarrow 
\Lambda^2 \circ \aQ 
\rightarrow 
\mathscr{E} 
\rightarrow 
\aQ
\rightarrow 
0. 
\]
Correspondingly, $q^\gr_2 (\overline{P^\gr_\zed}/\ad) \cong  \aQ \oplus S^2 \circ \aQ$, the composition factor $\Lambda^2 \circ \aQ $ lying in the image of $[-,-]$ (using the notation of Example \ref{exam:P/ad}), via the presentation (\ref{eqn:pbar/ad}).

Consider the composite $\overline{P^\gr_\zed} \twoheadrightarrow \mathscr{E} \hookrightarrow q_2^\gr P_\zed^\gr $. This lifts  
across the projection $\overline{P^\gr_\zed} \twoheadrightarrow q_2^\gr \overline{P_\zed^\gr}$ to give  $\overline{P^\gr_\zed} \stackrel{\tilde{\sigma}}{\rightarrow} \overline{P^\gr_\zed}$ so that one can form the composite $\sigma$ given by:
\[
\overline{P^\gr_\zed} \stackrel{\tilde{\sigma}}{\rightarrow} \overline{P^\gr_\zed} \twoheadrightarrow \overline{P^\gr_\zed}/\ad,
\] 
where the second map is the canonical projection.  

By construction, the composite $\sigma \circ \mu$ (again using the notation of Example \ref{exam:P/ad}) is zero, hence $\sigma$ factors to give the map $\aQ \rightarrow \overline{P^\gr_\zed}/\ad$, using the presentation (\ref{eqn:presentations_ak}). The construction of $\sigma$ also ensures that this gives a section of the natural surjection, as required. 

This shows that $\aQ$ is a direct summand of $\overline{P^\gr_\zed}/\ad$. By construction (cf. Corollary \ref{cor:foutk_easy_proj_inj}), the latter is projective in $\foutQ[\gr]$, thus so is $\aQ$.
\end{proof}

\begin{exam}
By Proposition \ref{prop:pbar/ad_split}, there is a splitting in $\foutQ[\gr]$:
\[
\overline{P^\gr_\zed}/\ad 
\cong 
\aQ \oplus
\big(\overline{P^\gr_\zed}/\ad\big)'. 
\]
Moreover, there  is a surjection  $\big(\overline{P^\gr_\zed}/\ad\big)'\twoheadrightarrow S^2 \circ \aQ$ and $\big(\overline{P^\gr_\zed}/\ad\big)'$ is projective in  $\foutQ[\gr]$.

Using the presentation (\ref{eqn:pbar/ad}), it is not difficult to show that this factorizes across a surjection 
\[
\big(\overline{P^\gr_\zed}/\ad\big)'
\twoheadrightarrow 
\mathscr{F},
\]
where $\mathscr{F}$ is a functor in $\foutQ[\gr]$ that occurs in a non-split extension 
\[
0
\rightarrow 
\Lambda^3 \circ \aQ \rightarrow 
\mathscr{F}
\rightarrow 
S^2 \circ \aQ 
\rightarrow 
0.
\] 
(Indeed, the results of Section \ref{sec:omega_psi} show that there exists a {\em unique} such functor, up to isomorphism, and the surjection exists by projectivity of $\big(\overline{P^\gr_\zed}/\ad\big)'$.)

This gives a first, naturally occurring functor in $\foutQ[\gr]$ that does not arise from $\fcatQ[\ab]$, exhibited as a quotient of $\big(\overline{P^\gr_\zed}/\ad\big)'$.
\end{exam}

\section{The functors $\omega$, $\Omega$ on exponential functors}
\label{sec:omega_exponential}

This section addresses  the calculation of $\omega$ (respectively $\Omega$) on exponential functors. This is motivated by Corollary \ref{cor:omega_boldbeta_inj_cogen}, which identifies $\omega \boldbeta$ as $\omega \Psi \pcoalg$, where $\Psi \pcoalg$ is the exponential functor constructed from $\pcoalg$.

Proposition \ref{prop:Omega_Phi_omega_Psi} gives a general description in terms of the underlying Hopf algebra of an exponential functor, as given by Theorem \ref{thm:expo_Hopf_general}; this uses the adjoint action (respectively coadjoint coaction) that is introduced in Section \ref{subsect:adj_coadj}.

We then specialize to the examples that interest us, applying Proposition \ref{prop:Omega_Phi_omega_Psi}. Proposition \ref{prop:Omega_tensor_algebra} considers the case of the exponential functor $\Phi T(V)$ associated to the tensor Hopf algebra $T(V)$ (respectively $\Psi T_{\mathrm{coalg}}(V)$ associated to the tensor coalgebra Hopf algebra $T_{\mathrm{coalg}}(V)$). Analogously, Proposition \ref{prop:Omega_palg} treats the exponential functors associated to $\palg$ and $\pcoalg$. These results are related by the Schur functor construction in Corollary \ref{cor:schur_expo_relate}.

\subsection{The adjoint action and the coadjoint coaction}
\label{subsect:adj_coadj}

As in Section \ref{subsect:expo_functors}, $\calm$ is an abelian category equipped with a symmetric monoidal structure $(\calm , \otimes , \unit)$.

We consider Hopf algebras within this framework; recall that a Hopf algebra structure on $H$ is specified by the morphisms $(\epsilon, \eta, \Delta, \mu, \chi)$, where $\unit \stackrel{\eta}{\rightarrow} H \stackrel{\epsilon}{\rightarrow} \unit$, $\mu: H\otimes H \rightarrow H$ is the product, $\Delta : H \rightarrow H \otimes H$ the coproduct, and $\chi : H \rightarrow H$ the conjugation (or antipode).  We write $\overline{H}$ for the augmentation ideal of a Hopf algebra $H$, so that there is a splitting $H \cong \unit \oplus \overline{H}$.

Recall that the (left) adjoint action of $H$ on itself is $\ad : H \otimes H \rightarrow H$ given by the composite:
\[
\xymatrix{
H \otimes H
\ar[r]^(.4){\tilde{\Delta} \otimes \id_H} 
&
(H \otimes H) \otimes H
\ar[r]_(.6)\cong ^(.6){\id_H \otimes \tau} 
&
H^{\otimes 3}
\ar[r]^(.6){\mu^{(2)}}
&
H,
} 
\]
where $\tilde{\Delta}$ is the composite $(\id_H \otimes \chi ) \Delta$ and $\mu^{(2)}$ is the iterated product.

Dually, the (right) coadjoint coaction of $H$ on itself is $\coad : H \rightarrow H \otimes H$ given by the composite:
\[
\xymatrix{
H
\ar[r]^{\Delta^{(2)}} 
&
H^{\otimes 3} 
\ar[r]_(.4)\cong^(.4){\tau\otimes \id _H } 
&
H \otimes (H \otimes H) 
\ar[r]^(.6){\tilde{\mu}}
&
H  \otimes H,
}
\]
where $\Delta^{(2)}$ is the iterated coproduct and $\tilde{\mu}$ is the composite $\mu (\chi \otimes \id_H)$.

Hence, for any $N \in \nat$, using the respective diagonal structures, one has the (co)actions:
\begin{eqnarray*}
\ad &:& H \otimes H^{\otimes N} \rightarrow H^{\otimes N} \\
\coad &:& H^{\otimes N} \rightarrow H^{\otimes N} \otimes H.
\end{eqnarray*}

Recall that, if $H$ is cocommutative, Theorem \ref{thm:expo_Hopf_general} yields the exponential functor $\Phi H$ on $\gr\op$;  if $H$ is commutative, it yields $\Psi H$ on $\gr$. 

\begin{prop}
\label{prop:adj_coadj}
\nomenclature{$\adbar$}{reduced adjoint action\nomrefpage}
\nomenclature{$\coadbar$}{reduced coadjoint coaction\nomrefpage}
Let $H$ be a Hopf algebra in $(\calm, \otimes ,\unit)$.
\begin{enumerate}
\item 
If $H$ is cocommutative, the adjoint action induces a left  action of $H$ on $\Phi H$ in $\f(\gr\op; \calm)$:
\[
\ad : H \otimes \Phi H \rightarrow \Phi H 
\]
that restricts to $\adbar : \overline{H} \otimes \Phi H \rightarrow \Phi H$ via $\overline{H}\hookrightarrow H$.
\item 
If $H$ is commutative, the coadjoint coaction induces a right coaction of $H$ on $\Psi H$ in $\f (\gr ; \calm)$:
\[
\coad : \Psi H \rightarrow \Psi H \otimes H
\]
that corestricts to $\coadbar : \Psi H \rightarrow \Psi H \otimes \overline{H}$ via $H \twoheadrightarrow \overline{H}$.
\end{enumerate}
\end{prop}

\begin{proof}
In general, the adjoint action makes $H$ into a $H$-module algebra; 
if $H$ is cocommutative, then the adjoint action is compatible
 with $\Delta : H \rightarrow H\otimes H$ and with $\chi$. From this it is straightforward to show that the adjoint action is compatible with the structure of $\Phi H$ as a functor on $\gr\op$.
 
The proof for the commutative case is categorically dual.
\end{proof}

\begin{rem}
The above adjoint action of $H$ on $\Phi H$ is related to the  action of $H$ on the  $\mathrm{Aut}(\zed^{\star N})$-module $H^{\otimes N}$, as studied by Conant and Kassabov in \cite[Section 4]{CK}, via the evaluation $\Phi H (\zed^{\star N}) = H^{\otimes N}$. In working functorially, we consider all $N \in \nat$ and also do not restrict to automorphisms in the category $\gr$.
\end{rem}

\subsection{Calculating $\Omega$ and $\omega$ on exponential functors}

The following is clear:

\begin{lem}
\label{lem:tau_Phi_Psi}
Let $H$ be a Hopf algebra in $(\calm, \otimes ,\unit)$.
\begin{enumerate}
\item 
If $H$ is cocommutative, there is a natural isomorphism $\tau_\zed \Phi H \cong H \otimes \Phi H$ in $\f(\gr\op; \calm)$; this restricts to $\taubar \Phi H \cong \overline{H} \otimes \Phi H$.
\item 
If $H$ is commutative, there is a natural isomorphism $\tau_\zed \Psi H \cong  \Psi H \otimes H$ in $\f(\gr; \calm)$; this corestricts to $\taubar \Psi H \cong  \Psi H \otimes \overline{H} $.
\end{enumerate}
\end{lem}

This leads to the identification of the functors $\Omega \Phi H$ and $\omega \Psi H$ in the respective cases:

\begin{prop}
\label{prop:Omega_Phi_omega_Psi}
Let $H$ be a Hopf algebra in $(\calm, \otimes ,\unit)$.
\begin{enumerate}
\item 
If $H$ is cocommutative, there is a natural isomorphism in $\f(\gr\op; \calm)$:
\[
\Omega \Phi H \cong \mathrm{coker} \  \big( \adbar : \overline{H} \otimes \Phi H \rightarrow \Phi H\big).
\]
\item 
If $H$ is commutative, there is a natural isomorphism  in $\f(\gr; \calm)$:
\[
\omega \Psi H \cong \ker  \big(\coadbar : \Psi H \rightarrow \Psi H \otimes \overline{H} \big).
\]
\end{enumerate}
\end{prop}

\begin{proof}
Consider the first case. By Proposition \ref{prop:omega_Omega_taubar}, $\Omega \Phi H$ is the cokernel of $\overline{\kappa^*} : \taubar \Phi H \rightarrow \Phi H$ and, by Lemma \ref{lem:tau_Phi_Psi}, $\taubar \Phi H$ is  isomorphic to $\overline{H} \otimes \Phi H$. Hence, to prove the assertion, it suffices to show that $\overline{\kappa^*} $ identifies with $\adbar$. This is proved by a direct verification.

The second case is categorically dual.
\end{proof}

\begin{rem}
\label{rem:restrict_gen_coresrict_cogen}
In the cocommutative case, since the adjoint action $H \otimes \Phi H \rightarrow \Phi H$ is an $H$-module structure, to calculate $\Omega \Phi H$ one can restrict to generators of $H$  (cf.  \cite[Lemma 4.13]{CK}). Similarly, in the commutative case, the coadjoint coaction $\Psi H \rightarrow \Psi H \otimes H$ is an $H$-comodule structure, hence to calculate $\omega \Psi H$ one can project to cogenerators of $H$.
\end{rem}

We now apply Proposition \ref{prop:Omega_Phi_omega_Psi} to our principal examples of interest. 
First take $(\calm , \otimes , \unit)$ to be $(\fmodq, \otimes , \rat)$.

\begin{prop}
 \label{prop:Omega_tensor_algebra}
 For  $V$ a finite-dimensional $\rat$-vector space:
\begin{enumerate}
\item 
There is a natural isomorphism in $\fcatk[\gr\op]$:
 \[
   (\Omega \Phi T(-)) (V) 
 \cong 
 \mathrm{coker }
 \big(
 V \otimes \Phi T (V) \stackrel{\adbar}{\rightarrow}  \Phi T(V)
 \big).
 \]
This is natural in $V \in \ob \fmodq$.

Here, $\adbar$ sends $v \otimes (w_1 \otimes \ldots \otimes w_N)$ to 
\[
\sum_{i=1}^N 
w_1 \otimes \ldots \otimes w_{i-1} \otimes [v,w_i] \otimes w_{i+1} \otimes \ldots \otimes w_N,\]
where $v\in V$ and $w_i \in T(V)$.
\item
There is a natural isomorphism in $\f(\gr;\rat)$:
\[
(\omega \Psi T_{\mathrm{coalg}}(-))(V)
\cong 
\ker 
\big(
\Psi T_{\mathrm{coalg}}(V) 
\stackrel{\coadbar}{\rightarrow}
\Psi T_{\mathrm{coalg}}(V) \otimes V
\big).
\]
This is natural in $V$.

Here, $\coadbar$ sends $(w_1 \otimes \ldots \otimes w_N)$, for $w_i \in T_{\mathrm{coalg}}(V)$, to 
\begin{eqnarray*}
\sum_{i=1}^N \Big(
(w_1 \otimes \ldots \otimes w_{i-1} \otimes \delta_1' (w_i) \otimes w_{i+1} \otimes \ldots \otimes w_N) \otimes \delta_1''(w_i)
-
\\
\quad 
(w_1 \otimes \ldots \otimes w_{i-1} \otimes \delta_2'' (w_i) \otimes w_{i+1} \otimes \ldots \otimes w_N) \otimes \delta_2'(w_i) 
\Big)
,\end{eqnarray*}
 where $\delta_1 : T_{\mathrm{coalg}}(V) \rightarrow T_{\mathrm{coalg}}(V) \otimes V$ and $\delta_2 : T_{\mathrm{coalg}}(V) \rightarrow V \otimes T_{\mathrm{coalg}}(V)$ denote the components of the diagonal with linear terms, and $\delta'$, $\delta''$ are given by Sweedler's convention (i.e., writing $\delta w = \delta' (w) \otimes \delta''(w)$ with implicit summation).
 \end{enumerate}
 \end{prop}

\begin{proof}
The first statement follows from Proposition \ref{prop:Omega_Phi_omega_Psi} by applying Remark \ref{rem:restrict_gen_coresrict_cogen}, using that $T(V)$ is generated as an algebra by $V$. Naturality with respect to $V$ is clear. 

The second statement is proved similarly, using that $T_{\mathrm{coalg}}(V)$ is cogenerated as a coalgebra by $V$. 
\end{proof}

We now take $(\calm , \otimes , \unit)$ to be $(\f (\fb;\rat), \tenfb , \unit , \tau)$ and consider the Hopf algebras $\palg$ and $\pcoalg$ that were introduced in Section \ref{sec:pcoalg}.  Recall (see the proof of Proposition \ref{prop:hopf_alg_fb_modules}) that these are related to the Hopf algebras $T(V)$ and $T_{\mathrm{coalg}}(V)$ respectively by the Schur functor construction.

With respect to the Schur functor correspondence, the generators $V$ of the tensor algebra $T(V)$ correspond to the canonical inclusion $P^\fb_{\mathbf{1}}\hookrightarrow \palg$ and the cogenerators $V$ of the tensor coalgebra $T_{\mathrm{coalg}}(V)$ correspond to the canonical projection $\pcoalg \twoheadrightarrow P^\fb_{\mathbf{1}}$. 

The analogue of Proposition \ref{prop:Omega_tensor_algebra} is then:

\begin{prop}
\label{prop:Omega_palg}
\ 
\begin{enumerate}
\item 
There is a natural isomorphism in $\f (\gr\op; \f (\fb; \rat))$:
\[
\Omega \Phi \palg \cong \mathrm{coker} \Big( P^\fb_{\mathbf{1}} \tenfb \Phi \palg \stackrel{\adbar}{\rightarrow} \Phi \palg 
\Big).
\]
\item 
There is a natural isomorphism in $\f (\gr; \f (\fb ; \rat))$:
\[
\omega \Psi \pcoalg \cong \ker 
\Big( 
\Psi \pcoalg \stackrel{\coadbar}{\rightarrow} \Psi \pcoalg \tenfb P^\fb_{\mathbf{1}}
\Big).
\] 
\end{enumerate}
\end{prop}

Corollary \ref{cor:omega_boldbeta_inj_cogen} thus yields:

\begin{cor}
\label{cor:omega_boldbeta}
 There is a natural isomorphism in $\f(\gr \times \fb; \rat)$:
 $$\omega \boldbeta \cong 
 \ker 
\Big( 
\Psi \pcoalg \stackrel{\coadbar}{\rightarrow} \Psi \pcoalg \tenfb P^\fb_{\mathbf{1}}
\Big).
 $$
\end{cor}

We now make the relationship between Propositions \ref{prop:Omega_tensor_algebra} and \ref{prop:Omega_palg} explicit,  using the Schur functor construction 
 $$\talg \otimes_\fb - : \f(\fb; \rat) \rightarrow \f(\fmodq; \rat)$$
 of Section \ref{subsect:schur_correspond}, which is exact and symmetric monoidal, by Proposition \ref{prop:schur_functor}. 
 
 In particular, this induces the isomorphisms of Hopf algebras:
 \begin{eqnarray*}
\talg \otimes_\fb \palg &\cong & T(-)
\\
\talg \otimes_\fb \pcoalg & \cong & T_{\mathrm{coalg}}(-).
\end{eqnarray*}
From this  one deduces the  comparison result:

\begin{cor}
\label{cor:schur_expo_relate}
There are natural isomorphisms:
\begin{eqnarray*}
\talg \otimes_\fb (\Omega \Phi \palg) &\cong & \Omega \Phi T(-) \mbox{\quad\quad\quad in $\f(\gr\op; \f(\fmodq ;\rat))$} 
\\
\talg \otimes_\fb (\omega \Psi \pcoalg) & \cong &\omega \Psi T_{\mathrm{coalg}}(-) \mbox{\quad \quad in $\f(\gr;(\f(\fmodq ;\rat))$.} 
\end{eqnarray*}
\end{cor}

\subsection{Introducing Koszul signs}

Recall from Section \ref{sec:fb} that there is an equivalence of symmetric monoidal categories
\[
(-)^\dagger : (\smod, \tenfb, \unit, \tau) 
\rightarrow 
(\smod , \tenfb, \unit, \sigma). 
\]
Applying this functor to $\palg$ gives the Hopf algebra $\palg^\dagger$,  which is cocommutative in $(\smod , \tenfb, \unit,\sigma)$; similarly, one obtains the commutative Hopf algebra $\pcoalg^\dagger$. These structures are identified in Proposition \ref{prop:palg_pcoalg_dagger_identify}. This gives the canonical inclusion $P^\fb_\mathbf{1} \hookrightarrow \palg^\dagger$ and the canonical surjection $\pcoalg^\dagger \twoheadrightarrow  P^\fb_\mathbf{1}$.

\begin{prop}
\label{prop:omega_pcoalg_dagger}
\ 
\begin{enumerate}
\item 
There are natural isomorphisms in $\f(\gr\op; \f(\fb; \rat))$:
\begin{eqnarray*}
\Phi (\palg^\dagger) &\cong & \big(\Phi \palg )^\dagger 
\\
\Omega \Phi (\palg^\dagger) &\cong & \big( \Omega \Phi \palg )^\dagger,
\end{eqnarray*}
where the functor $\Phi$ on the left is applied with respect to the symmetry $\sigma$ and, on the right, with respect to $\tau$.
\item 
There are natural isomorphisms in $\f(\gr; \f(\fb; \rat))$:
\begin{eqnarray*}
\Psi (\pcoalg^\dagger) &\cong& \big(\Psi \pcoalg )^\dagger
\\
\omega \Psi (\pcoalg^\dagger) &\cong& \big( \omega \Psi \pcoalg )^\dagger,
\end{eqnarray*}
where the functor $\Psi$ on the left is applied with respect to the symmetry $\sigma$ and, on the right, with respect to $\tau$.
\end{enumerate} 
\end{prop}

\begin{proof}
We consider the case of $\palg$, that of $\pcoalg$ being proved by a similar argument. 

The first statement follows from Proposition \ref{prop:naturality_Phi_Psi}, since $(-)^\dagger$ is a symmetric monoidal equivalence of categories. The statement about $\Omega$ follows similarly, since $(-)^\dagger$ is exact and $\adbar$ is defined in terms of the Hopf algebra structure of $\palg$ (respectively $\palg^\dagger$).
\end{proof}

Corollary \ref{cor:omega_boldbeta_inj_cogen} provides the  isomorphism 
 $\omega \boldbeta \cong \omega \Psi \pcoalg$ in $\f(\gr \times \fb; \rat)$; Proposition \ref{prop:omega_pcoalg_dagger} thus gives:

\begin{cor}
\label{cor:omega_boldbeta_dagger}
There is a natural isomorphism in $\f (\gr \times \fb ; \rat)$
\begin{eqnarray*}
 \omega \Psi (\pcoalg ^\dagger)
  & \cong &
 (\omega \boldbeta)^\dagger.
\end{eqnarray*}
\end{cor}

\begin{rem}
Here, we regard $\omega \boldbeta$ as being the fundamental object, with an intrinsic interpretation in $\fpoutgrQ [<\infty]$. The above Corollary relates $\omega \Psi(\pcoalg ^\dagger)$ to this.
\end{rem}

\part{Calculating $\Psi$ and $\omega \Psi$}
\label{part:psi_omega_psi}

This part analyses the structure of the functors $\beta_d$ and $\omega \beta_d$, for $d \in \nat$. By the results of Section \ref{sec:beta}, this is equivalent to analysing the exponential functor   $\Psi \pcoalg$ (abbreviated in the title to $\Psi$) and the associated functor $\omega\Psi \pcoalg$ (abbreviated to $\omega \Psi$). The latter is of particular interest, due to its relationship with higher Hochschild homology, as explained in Part \ref{part:HHH} (see  Corollary \ref{cor:isotypical_components}, for example).

Section \ref{sec:psi} studies the indecomposable polynomial functors of the form $\beta_{|\mu |} S_\mu$, for $\mu \vdash d$. The multiplicity of the composition factors are completely determined in terms of coefficients from the representation theory of the symmetric groups, namely the Littlewood-Richardson coefficients and certain plethysm coefficients (see Corollary \ref{cor:Grothendieck_ring_Psi_pcoalg}). Moreover, the theory is illustrated further by analysing the structure of such functors for certain infinite families of partitions (see Section \ref{subsect:three_infinite_families}).

Section \ref{sec:omega_psi} then addresses the problem of calculating $\omega\beta_{|\mu |} S_\mu$. Here the story is less complete: at time of writing, there is no known general formula describing the multiplicities of the composition factors of  $\omega\beta_{|\mu |} S_\mu$ in terms of the representation theory of the symmetric groups (although there is a lower bound, which is known to be exact in some cases). We do, however, give some useful initial information which sheds significant light on the structure of these functors.

 \section{Calculating $\Psi$} 
\label{sec:psi}

The main aim of this section is to give information on the exponential functor $\Psi \pcoalg$, which belongs to $\fcatQ[\gr]$.  The analysis commences with  that of $\Psi H$, for $H$ a suitable commutative Hopf algebra in Section \ref{subsect:PsiH_graded}. This uses the basic structure theory of Hopf algebras in characteristic zero, which gives Lemma \ref{lem:Hopf_filtration} (related to the Borel structure theorem); this is the key input into Proposition \ref{prop:filter_Psi_H}.

This theory is then applied to the case $H= T_\cog (V)$, the tensor coalgebra equipped with the shuffle product, calculating $\Psi T_\cog (V)$ as a functor of $V \in \ob \fmodq$. This gives the structure of $\Psi \pcoalg$, by the Schur correspondence, in Theorem \ref{thm:decomp_Psi_pcoalg}.

Now, by Theorem \ref{thm:beta_description}, one has the isomorphism $\boldbeta \cong \Psi \pcoalg$, where the functor $\boldbeta$ encodes the injective cogenerators of  $\fpiQ {\gr}$ (see  Theorem \ref{thm:injectivity_beta}), hence  $\Psi \pcoalg$ determines these  injective cogenerators. Using this, Theorem \ref{thm:decomp_Psi_pcoalg} also gives a description  of the composition factors of the injectives $\beta_{|\mu|} S_\mu$, for any partition $\mu$. 
 As an immediate consequence, Corollary \ref{cor:d-1_beta_S_mu} gives an explicit description of the extension groups $\ext^1_{\fcatQ[\gr]}$ between simple polynomial functors.  

Theorem \ref{thm:beta_description} is then made more concrete, first by treating examples for $|\mu|\leq 4$, and then by giving an explicit expression for the multiplicities of the composition factors of $\beta_{|\mu|} S_\mu$ in Corollary \ref{cor:Grothendieck_ring_Psi_pcoalg}. 
 
To illustrate the methods further, the cases of the families of partitions $(n)$, $(1^n)$ and $(n-1, 1)$, for $n \in \nat$, are treated in Section \ref{subsect:three_infinite_families}.

\subsection{Calculating $\Psi H$ for commutative Hopf algebras}
\label{subsect:PsiH_graded}
In this section we work with $\nat$-graded vector spaces, with symmetry for $\otimes$ invoking Koszul signs. There is an evident counterpart when working without Koszul signs (as indicated in Remark \ref{rem:no_Koszul_signs}).

\begin{nota}
\label{nota:gsym}
\nomenclature{$\gsym^*$}{free graded-commutative $\rat$-algebra\nomrefpage}
Denote by $\gsym^*$  the free graded-commutative $\rat$-algebra functor, so 
that, for $V$ an $\nat^*$-graded 
 vector space, there is an isomorphism 
 \[
  \gsym^* (V) \cong 
  S^* (V^{\mathrm{even}}) \otimes \Lambda^* (V^{\mathrm{odd}})
 \]
of the underlying graded $\rat$-vector spaces, with the induced $\nat$-grading. (The symmetric and exterior power functors are defined in Example \ref{exam:schur_symm_ext}).
\end{nota}

\begin{rem}
\label{rem:no_Koszul_signs}
\nomenclature{$S^*$}{free commutative $\rat$-algebra\nomrefpage}
Working in the graded setting {\em without} Koszul signs, one replaces 
$\gsym^*$ by the usual  symmetric power functors, $S^*$, and  modifies the following accordingly, replacing graded-commutative by commutative.
\end{rem}

\begin{nota}
\label{nota:Q_P_Hopf} \cite{MM}
For $A $ (respectively $C$) a graded, connected algebra (respectively coalgebra) over $\rat$, denote by
\begin{enumerate}
 \item 
 $QA$ the graded vector space of indecomposables; 
 \item 
 $PC$  the graded vector space of primitives. 
\end{enumerate}
In particular, for $H$ a graded, connected Hopf algebra over $\rat$, one has the indecomposables  $QH$ and  the primitives $PH$. 
\end{nota}

Recall (see \cite{MS1} for example) that every monomorphism of connected, 
graded-commutative $\rat$-Hopf
algebras is normal. Moreover, if $H_1 \hookrightarrow H_2$ is such an 
inclusion, 
the cokernel $H_2 \hq H_1$ (in Hopf algebras) 
identifies as 
$
 H_2 \otimes _{H_1} \rat.
$

If $\rat = H_0 \subset H_1 \subset H_2 \subset H_3 \subset \ldots \subset H$ 
is an increasing, exhaustive filtration 
of the connected, graded-commutative Hopf algebra $H$, the associated 
graded is the Hopf algebra
$
 \bigotimes_{i\geq 1}
 H_i\hq H_{i-1}.
$
In particular, taking $H_i$ to be the sub-Hopf algebra generated by elements of degree at most $i$, $H_i \hq H_{i-1}$ is the Hopf algebra $\gsym^* ((QH)^i)$ primitively-generated by $(QH)^i$. This gives the following result, related to the Borel structure theorem for Hopf algebras (see \cite{MM}, for example):

\begin{lem}
\label{lem:Hopf_filtration}
 Let $H$ be a graded-commutative, connected $\rat$-Hopf algebra. Then 
$H$ has a natural filtration 
 as Hopf algebras with associated graded 
 $
 \bigotimes_{i \geq 1} \gsym^* ((QH)^i).  
 $
 The underlying graded-commutative algebra of $H$ is non-canonically isomorphic to $ \bigotimes_{i \geq 1} \gsym^* 
((QH)^i)$.
\end{lem}

Lemma \ref{lem:Hopf_filtration} refines to treat the associated  functor $\Psi H \in \fcatk[\nat \times \gr]$, where $\nat$ corresponds to the grading. 

\begin{prop}
\label{prop:filter_Psi_H}
Let $H$ be a graded-commutative, connected
 $\rat$-Hopf algebra. Then  $\Psi H$ has natural 
filtration with associated graded 
$
 \bigotimes_{ n \geq 1} \gsym^* ((QH)^n \otimes \aQ). 
$ 

In particular, the  grading $N=0$ component, $(\Psi H)^0$, is isomorphic to the constant functor $\rat$ and,  in grading  $N\in \nat^*$, $(\Psi H)^N$ has filtration in $\fcatQ[\gr]$ with associated 
graded:
\[
 \bigoplus_{\underline{b}} \big( 
 \bigotimes_{n=1}^N
 \gsym^{b_n} (  \aQ \otimes (QH)^n ) 
 \big),
\]
where the sum is taken over sequences $\underline{b}= (b_1, \ldots , b_N)$ of natural numbers 
such that $\sum_n b_n n =N$. The functor $\bigotimes_{n=1}^N  \gsym^{b_n} (\aQ \otimes  (QH)^n)$ has polynomial degree $ \sum_n a_n \in [1, N]$. 
\end{prop}

\begin{proof}
By Lemma \ref{lem:Hopf_filtration}, after evaluation on any $\zed^{\star t}$, the 
result holds as filtered $\rat$-modules,  hence it suffices to establish naturality with respect to $\gr$.

By Proposition \ref{prop:Phi_Psi_(co)mult}, $\Psi H$ is a functor on $\gr$ to graded-commutative  $\rat$-algebras. In particular, there is a natural augmentation and hence one can form the indecomposables $Q (\Psi H)$, a functor on $\gr$ with values in $\rat$-vector spaces. 

The natural multiplicative structure allows the result to be proved by induction upon the  degree $N$,  as follows.
The case $N=0$ is immediate and, in degree one, $H^1 = (QH)^1$ and it is clear 
that the functor obtained is $\aQ \otimes H^1$. 

For the inductive step in degree $N$, one shows that $\big(Q (\Psi H)\big)^N$  is naturally isomorphic to $\aQ \otimes (QH)^N $ 
 by checking that the functor is additive. This provides the top filtration quotient 
$
 (\Psi H) ^N \twoheadrightarrow \aQ \otimes (QH)^N.
$
The kernel is a quotient of 
\[
 \bigoplus_{i=1}^{N-1} 
 (\Psi H) ^i \otimes (\Psi H)^{N-i}.
\]
Using the inductive hypothesis together with the multiplicative structure gives 
the required result.
\end{proof}

\begin{exam}
\label{exam:TW_thm1}
 Let $H$ be the connected, graded-commutative Hopf algebra $\Lambda (x) \otimes 
\rat [y]$, with $|x|= 1$ and 
 $|y|=2$, with $x$ primitive and reduced diagonal $\overline{\Delta} y = x 
\otimes x$. (Note that, since
 $x$ has odd degree, $H$ is not graded-cocommutative.) Thus $(QH)^1 = 
\kring = (QH)^2$ 
  and $(QH)^n =0$ for $n >2$. 
 
Hence, for $N\in \nat^*$,  the functor $(\Psi H)^N$ has filtration with associated graded 
 \[
  \bigoplus _{2a +b = N} (S^a \otimes \Lambda^b ) \circ \aQ.
 \]
Here $(S^a \otimes \Lambda^b ) \circ \aQ$ has polynomial degree $a+b$ and one can check that the  filtration coincides with the polynomial filtration of $(\Psi 
H)^N$ (cf. Section \ref{subsect:poly_filt}). This filtration is not split, as is exhibited by Example \ref{exam:Theorem_1_TW} below. 
\end{exam}

\subsection{The functors $\Psi T_\cog(V)$ and $\Psi \pcoalg$}

In this section we apply Proposition \ref{prop:filter_Psi_H} to the case $H = T_\cog (V)$, the commutative Hopf algebra given by the tensor coalgebra with the shuffle product, using the length grading and working {\em without} Koszul signs. The constructions are natural with respect to $V \in \ob \fmodq$.

\begin{rem}
There are analogues for the graded setting {\em with} Koszul signs, which corresponds to working with $T_\cog ^\dagger (V)$. The results in this case can be deduced from the ones presented here by using the functor $(-)^\dagger$.
\end{rem}

In the following, $\liemod (n)$ denotes the $n$th Lie module (see Section \ref{subsect:lie}); this is a $\sym_n$-module and can thus be considered as an $\fb$-module supported on $\mathbf{n}$;  $\lieschur (n) (-) \in \ob \fcatQ[\fmodq]$ denotes the associated Schur functor  (see Section \ref{subsect:schur_correspond})
 and $\lie (V)$ the free Lie algebra on $V$.

  \begin{lem}
 \label{lem:indec_Tprime}
  There is a  natural isomorphism with respect to $V \in \ob \fmodq$:
  \begin{eqnarray*}
   Q T_\cog (V) 
   &\cong& 
   \bigoplus _{n \geq 1} 
   \lieschur  (n)(V)  \ \cong \  \lie (V).
    \end{eqnarray*}
 \end{lem}

 \begin{proof}
Since $V$ is finite-dimensional, $T_\cog(V)$ is of finite type  with respect to the length grading. This allows us to work with the dual. 

The dual Hopf algebra  of $T_\cog (V)^\sharp$  is isomorphic  to the universal enveloping  algebra $U \lie (V^\sharp)$, which has primitives $P (U \lie (V^\sharp)) = \lie (V^\sharp)$. 

Hence $Q T_\cog (V) $ (which is dual to the primitives  $P T_\cog (V)^\sharp$) is isomorphic to 
$ \lie (V^\sharp)^\sharp$. This is isomorphic to $\lie (V)$ (using Proposition \ref{prop:self-duality}) which identifies as given in terms of the Lie modules.
\end{proof}
 
\begin{rem}
\label{rem:Psi_length_grading}
For $V \in \ob \fmodq$, $T_\cog  (V)\in \ob \hcom(\fmodq)$ so we have $\Psi T_\cog  (V) \in \ob \fcatQ[\gr]$. We denote by $\Psi T_\cog(-)$ the object of  $\fcatQ[ \gr \times \fmodq]$ given by $(\zed^{\star r}, V) \mapsto  \Psi T_\cog  (V)(\zed^{\star r})$. The functor $\Psi T_\cog(-)$ inherits a  $\nat$-grading  from the length grading of $T_\cog  (V)$, 
hence can be  considered as an object of $\fcatQ[\nat \times \gr \times \fmodq]$.
\end{rem}

Before stating the next result, we note that  a functor of the form $\alpha_d S_\lambda \boxtimes \schur_\mu(-)$ is simple in $\fcatQ[\gr \times \fmodq]$. This follows (by the Schur correspondence) from the fact that $S_\lambda \boxtimes S_\mu$ is a simple $\sym_{|\lambda|} \times \sym_{|\mu|}$-module.

\begin{prop}
\label{prop:psiTV_lie}
The functor $\Psi T_\cog(-) \in \ob \fcatQ[ \gr \times \fmodq]$ has a natural filtration 
 with associated graded 
 \[
 \bigotimes _{n \geq 1} 
  S^* (\aQ \boxtimes  \lieschur(n) (-) ) 
 \]
 where $\aQ \in \ob \fcatQ[\gr]$, $\lieschur (n)(-)  \in \ob \fcatQ[\fmodq]$, and $\boxtimes$ is the external tensor product. 
 
Moreover, this functor lies in the essential image of $\fcatQ[\ab \times \fmodq]$ in $\fcatQ[\gr \times \fmodq]$. The component of polynomial degree $d$ with respect to $\gr$ is given by 
\[
\bigotimes_{\underline{b}} \bigotimes _{n \geq 1} 
  S^{b_n} (\aQ \boxtimes  \lieschur(n) (-) ) 
\]
where $\underline{b}   = (b_n |n \in \nat^*)$ is a sequence of natural numbers with $\sum_n b_n=d$.

The component of polynomial degree $N$ with respect to $\fmodq$ and polynomial degree $d$ with respect to $\gr$  decomposes as a finite direct sum of simple functors of the form 
\[
\alpha_d S_\lambda \boxtimes \schur_\mu(-)
\] 
where $\lambda \vdash d$ and $\mu \vdash N$ are partitions and $\boxtimes$ denotes the exterior tensor product. 
\end{prop}

\begin{proof}
 The indecomposables of $T_\cog (V)$ are identified by Lemma \ref{lem:indec_Tprime}. The analogue without Koszul signs of Proposition 
\ref{prop:filter_Psi_H} gives the first statement. Moreover, it is clear that this functor lies in the essential image of $\fcatQ[\ab \times \fmodq]$.

In particular, this functor can be written as:
 \[
\bigoplus_{\underline{b}} \bigotimes_{n \geq 1} S ^{b_n} (\aQ\boxtimes \lieschur(n) (-) ),
\]
where the sum is over sequences $\underline{b}   = (b_n |n \in \nat^*)$ with finite support (i.e., $b_n =0$ for $n \gg 0$). 

Replacing the symmetric power functors by tensor product functors, since we are working over $\rat$, the term indexed by $\underline{b}$ is a direct summand of 
\[
\bigotimes _n
 (\aQ \boxtimes  \lieschur(n) (-) ) ^{\otimes b_n}
 \cong 
 \big( \aQ^{\otimes \sum_n b_n}\big) \boxtimes  \bigotimes _n \big( \lieschur(n) (-) ^{\otimes b_n} \big),
\]
where the isomorphism is obtained by reordering the tensor product. By inspection, this is semi-simple in $\fcatQ[\ab \times \fmodq]$ and hence in $\fcatQ[\gr \times \fmodq]$, given as a finite direct sum of simple objects. 

Now, $\aQ^{\otimes \sum_n b_n}$ has polynomial degree $\sum_n b_n$ with respect to  $\gr$ and, with respect to $\fmodq$,  $\bigotimes _n  \lieschur(n) (-) ^{\otimes b_n}$ has polynomial degree $\sum_n n b_n$. From this, one deduces the second statement. 

The terms of polynomial degree $N$ with respect to $\fmodq$ arise from the sequences $\underline{b}$ such that  $\sum_n n b_n=N$. Clearly there are only finitely many sequences $\underline{b}$ that satisfy this equality (in particular, $b_n=0$ for $n>N$).  The final statement follows.
\end{proof}

\begin{cor}
\label{cor:isotypical_cpt}
The associated graded to the functor $\Psi T_\cog(-)$ given in Proposition \ref{prop:psiTV_lie}
admits an isotypical decomposition with respect to the functoriality in $\fmodq$:
\begin{eqnarray}
\label{eqn:decompose_isotypical_F_mu}
\bigotimes _{n \geq 1} 
  S^* (\aQ \boxtimes  \lieschur(n) (-) ) 
  \cong 
  \bigoplus_{\mu} F^\mu \boxtimes \schur_\mu (-),
\end{eqnarray}
where $F^\mu \in \ob \fpiQ{\gr}$ is a finite, semisimple polynomial functor, in particular $F^\mu$ is in the essential image of $\circ \A : \fpiQ{\ab} \rightarrow \fpiQ {\gr}$.
\end{cor}

\begin{proof}
The existence of a decomposition $\bigoplus_{\mu} F^\mu \boxtimes \schur_\mu (-)$ with $F^\mu$ a semisimple functor in $\fcatQ[\gr]$ follows immediately from Proposition \ref{prop:psiTV_lie}. The proof that each $F^\mu$ has only finitely-many composition factors follows as in the proof of  Proposition \ref{prop:psiTV_lie}. Namely, one restricts to sequences $\underline{b}$ such that  $\sum_n n b_n=N$ (without also requiring $\sum_n b_n=d$ as in the Proposition). There are only finitely-many such sequences, from which one deduces the result.
\end{proof}

\begin{defn}
We refer to the functor $F^\mu$ appearing in Corollary \ref{cor:isotypical_cpt} as the $\schur_\mu (-)$ isotypical component of $\bigotimes _{n \geq 1} 
  S^* (\aQ \boxtimes  \lieschur(n) (-) ) $.
\end{defn}

Proposition \ref{prop:psiTV_lie} and Corollary \ref{cor:isotypical_cpt} imply the corresponding result for $\Psi \pcoalg$ via the Schur correspondence. In the following, for a partition $\mu \vdash m$, we consider the simple $\sym_{m}$-module $S_\mu$ as an $\fb$-module supported on $\mathbf{m}$.

\begin{thm}
\label{thm:decomp_Psi_pcoalg}
In the Grothendieck group of $\fcatQ[\gr \times \fb]$, 
\[
\big[\Psi \pcoalg\big ]
= 
\sum_{\mu}
[F^\mu \boxtimes S_\mu ],
\]
where $F^\mu$ is the $\schur_\mu (-)$ isotypical component of $\bigotimes _{n \geq 1} 
  S^* (\aQ \boxtimes  \lieschur(n) (-) ) $ determined by (\ref{eqn:decompose_isotypical_F_mu}).

Hence: 
\[
[\beta_{|\mu|} S_\mu] = [F^\mu]
\]
in the Grothendieck group of $\fcatQ[\gr]$.
\end{thm}

\begin{proof}
Using the Schur correspondence, the first statement follows directly  from Proposition \ref{prop:psiTV_lie} and Corollary \ref{cor:isotypical_cpt}.

For the second statement, Theorem \ref{thm:beta_description} identifies 
$
\boldbeta
\cong
 \Psi \pcoalg$ and Theorem \ref{thm:beta_d} gives $\beta_m S_\mu \cong \boldbeta \otimes_\fb S_\mu$. Hence $\beta_m S_\mu \cong \Psi \pcoalg \otimes _\fb S_\mu$. The result follows using that $S_\rho \otimes_{\sym_m} S_\mu$ is zero unless $\rho = \mu$, when it is $\rat$.
\end{proof}

Recall that, for $\mu \vdash d$, $\beta_d S_\mu$ is the injective envelope of the simple functor $\alpha_d S_\mu$ in $\fpiQ{\gr}$ (see Corollary \ref{cor:inj_envelope}).
 Theorem \ref{thm:decomp_Psi_pcoalg} thus allows the calculation of
 the degree $d-1$ part of the polynomial filtration of $\beta_d S_\mu$ (expressed in terms of the functor  $\qhat{d-1}$ of Notation \ref{nota:qhat}) as    $\qhat{d-1} (\beta_dS_\mu)$. This is explained explicitly below.  
 
Recall that $c^*_{*,*}$ denote the Littlewood-Richardson coefficients and that $\nu \preceq \rho$ means that the Young diagram of the partition $\nu$ is contained in that of $\rho$ (see Notation \ref{nota:preceq}). The Littlewood-Richardson coefficients that appear below satisfy the following, for $\nu \vdash d-2$ and $\rho \vdash d-1$:
\begin{enumerate}
\item 
$c^\rho_{\nu,1}$ is $1$ if $\nu \preceq \rho$ and $0$ otherwise (this corresponds to the Pieri rule);
\item 
$c^{\mu}_{\nu,(11)} =1 $  if  $\nu \preceq \mu$ and the Young diagram of the skew partition $\mu /\nu$ does not consist of two boxes in the same row and is $0$ otherwise.
\end{enumerate}

\begin{cor}
\label{cor:d-1_beta_S_mu}
For $\mu \vdash d$, 
\[
\dim 
\ext^1_{\fcatQ[\gr]}
(\alpha_{|\rho|} S_\rho, \alpha_d S_\mu)
= 
\left\{
\begin{array}{ll}
0& |\rho |\neq d-1\\
\sum_{\nu \vdash d-2 } c^\rho_{\nu, 1}c^{\mu}_{\nu,11}
&|\rho |= d -1.
\end{array}
\right.
\]
In particular, this vanishes for $d\leq 1$.

Equivalently,  there is an isomorphism:
\[
\qhat{d-1} \big( \beta_dS_\mu \big) 
\cong 
\bigoplus_{\substack{\nu \vdash d-2 \\ \nu \preceq \mu }}
\bigoplus_{\substack{\rho \vdash d-1 \\ \nu \preceq \rho}}
(\alpha_{d-1} S^\rho)^{\oplus c^{\mu}_{\nu,(11)}}.
\]
\end{cor}

\begin{proof}
We first establish the second statement. By Theorem  \ref{thm:decomp_Psi_pcoalg}, $\qhat{d-1} (\beta_dS_\mu)$ is the degree $d-1$ component of $F^\mu$, using that $[\beta_d S\mu] = [F^\mu]$.

Using the notation of Proposition \ref{prop:psiTV_lie}, the degree $d-1$ component of $F^\mu$ arises from the sequences $\underline{b}= (b_n | n \in \nat^*)$ such that $\sum_n b_n =d-1$ and $\sum_n n b_n = d$, so that $\sum_n (n-1) b_n=1$. It follows that $b_n=0$ for $n>2$, $b_2 =1$ and $b_1 = d-2$. In particular, there is no such solution if $d \leq 1$. 

In the case $d\geq 2$, it follows that  the degree $d-1$ component of $F^\mu$ is given by the $\schur_\mu (-)$-isotypical component of 
\[
S^{d-2} \left(\aQ \boxtimes \schur_{(1)} (-)\right) \otimes \left(\aQ \boxtimes \schur_{(11)} (-)\right),
\]
using that $\lieschur(1) (-)= \schur_{(1)} (-)$ and $\lieschur(2) (-)= \schur_{(11)} (-)$.

The  Cauchy identity (see Proposition \ref{prop:cauchy_identities}) gives:
\[
S^{d-2} (\aQ \boxtimes \schur_{(1)} (-))
\cong 
\bigoplus_{\nu \vdash d-2} \alpha_{d-2} S_\nu \boxtimes \schur_{\nu} (-).
\]

Hence, $\qhat{d-1} (\beta_dS_\mu)$ is the $\schur_\mu(-)$-isotypical component of 
\[
\big(\bigoplus_{\nu \vdash d-2} \alpha_{d-2} S_\nu \boxtimes \schur_{\nu} (-)\big)
\otimes (\aQ \boxtimes \schur_{(11)} (-))
= 
\bigoplus_{\nu \vdash d-2} (\alpha_{d-2} S_\nu \otimes \aQ) \boxtimes (\schur_{\nu} (-) \otimes \schur_{(11)} (-)).
\]
By the Pieri rule, $\alpha_{d-2} S_\nu \otimes \aQ$ is isomorphic to $\bigoplus _{\substack{\rho \vdash d-1 \\ \nu \preceq \rho}}
\alpha_{d-1} S_\rho$. The Littlewood-Richardson rule gives that $\schur_{\nu} (-) \otimes \schur_{(11)} = \bigoplus_{\gamma \vdash d} \schur_{\gamma}(-)^{\oplus c^\gamma_{\nu, 11}}$. 
Picking out the isotypical component corresponding to $\mu$ gives the result.

For the equivalent formulation, we use the results of Section \ref{sec:recoll};  in particular, we use that the socle filtration of $\beta_{d}S_\mu$ coincides with its polynomial filtration, by Proposition \ref{prop:socle_filtration_char_0}. This  implies that the beginning of the minimal injective resolution of $\alpha_{d} S_\mu$ is 
\[
0 
\rightarrow 
\alpha_d S_\mu 
\rightarrow 
\beta_{d} S_\mu \rightarrow 
\beta_{d-1} \cre_{d-1} \big( \qhat{d-1} (\beta_dS_\mu) \big),
\]
where the right hand map induces an isomorphism on applying $\qhat{d-1}$.

It follows  that
 $$\ext^1_{\fcatQ[\gr]}
(\alpha_{|\rho|} S_\rho, \alpha_d S_\mu) \cong \hom _{\fcatQ[\gr]} (\alpha_{|\rho|} S_\rho, \beta_{d-1} \cre_{d-1} ( \qhat{d-1} (\beta_dS_\mu ))).
$$
By Proposition \ref{prop:beta_socle}, the socle of $\beta_{d-1}\cre_{d-1} \big( \qhat{d-1}( \beta_dS_\mu )\big)$ is $\qhat{d-1} (\beta_dS_\mu)$, which is semi-simple with all composition factors of degree $d-1$, by construction. Hence, the right hand side  vanishes unless $|\rho|=d-1$; in the latter case, it counts the multiplicity of the composition factors indexed by $\rho$ in $\qhat{d-1}( \beta_dS_\mu)$.
\end{proof}  
  
\begin{rem}
\ 
\begin{enumerate}
\item 
The result can be paraphrased as follows: the multiplicity of $\alpha_{d-1} S_\rho$ in $\qhat{d-1} \big( \beta_dS_\mu \big) $ is the number of partitions $\nu\vdash d-2$ such that $\nu \preceq \rho$, $\nu \preceq \mu$ and $\mu / \nu$ has skew diagram with boxes in different rows.
\item 
The method used in the proof is developed in the following section, leading to the much more general Corollary \ref{cor:Grothendieck_ring_Psi_pcoalg}.
\item 
A particular case of \cite[Theorem 1]{V_ext} identifies $\ext^1 _{\fcatQ[\gr]} (\aQ^{\otimes d-1}, \aQ^{\otimes d})$ as a representation of $\sym_d \times \sym_{d-1}\op$ and  shows that  $\ext^1 _{\fcatQ[\gr]} (\aQ^{\otimes n}, \aQ^{\otimes d})$ vanishes for $n \neq d-1$. This   information can be used to give an alternative proof of Corollary \ref{cor:d-1_beta_S_mu}.
\end{enumerate}
\end{rem}
  
\begin{exam}
\label{exam:ext1}
To illustrate Corollary \ref{cor:d-1_beta_S_mu}, we give 
$$\dim 
\ext^1_{\fcatQ[\gr]}
(\alpha_{|\rho|} S_\rho, \alpha_d S_\mu)
$$ for the partitions $\mu \in \{ (d), (1^d), (d-1,1), (d-2,2), (d-2, 1,1), (d-k, 1^{k})\}$ and $\rho \vdash d-1$. For this, it is necessary to pay attention to the degenerate cases corresponding to small $d$ (and small $k$ in the final case).
\begin{enumerate}
\item 
For $d \geq 1$ and $\mu = (d)$,
$
\dim 
\ext^1_{\fcatQ[\gr]}
(\alpha_{d-1} S_\rho, \alpha_d S_{(d)}) = 0.
$
\item
For $d \geq 2$ and $\mu = (1^d)$,
\begin{eqnarray*}
\dim 
\ext^1_{\fcatQ[\gr]}
(\alpha_{d-1} S_\rho, \alpha_d S_{(1^d)}) &=& \left\{ 
\begin{array}{lllll}
1 
&
\rho \in \{ (1) \} &  d=2\\
1 
&
\rho \in \{ (1^{d-1}), (2,1^{d-3}) \} &  d\geq 3 \\
0 & \mbox{otherwise.}
\end{array}
\right. 
\end{eqnarray*}
\item 
For $d \geq 3$ and $\mu = (d-1,1)$, 
\begin{eqnarray*}
\dim 
\ext^1_{\fcatQ[\gr]}
(\alpha_{d-1} S_\rho, \alpha_d S_{(d-1,1)}) &=& \left\{ 
\begin{array}{ll}
1 
&
\rho \in \{ (d-1), (d-2,1) \} \\
0 & \mbox{otherwise.}
\end{array}
\right. 
\end{eqnarray*}
\item 
For $d \geq 4$ and $\mu = (d-2,2)$, $\dim 
\ext^1_{\fcatQ[\gr]}
(\alpha_{d-1} S_\rho, \alpha_d S_{(d-2,2)})$ is 
\begin{eqnarray*}
 \left\{ 
\begin{array}{llll}
1 & \rho \in \{ (2,1), (1,1,1) \}  & d=4\\
1 & \rho \in \{ (d-2,1), (d-3,2), (d-3,1,1) \}  & d\geq 5\\
0 & \mbox{otherwise.}
\end{array}
\right. 
\end{eqnarray*}
\item 
For $d\geq 4$ and $\mu = (d-2,1,1)$, $\dim 
\ext^1_{\fcatQ[\gr]}
(\alpha_{d-1} S_\rho, \alpha_d S_{(d-2,1,1)}) $ is
\begin{eqnarray*}
\left\{ 
\begin{array}{lll}
2 & \rho = (d-2,1) \\
1 & \rho \in \{ (3), (1,1,1) \} & d=4\\
1 & \rho \in \{ (d-1), (d-3,2), (d-3,1,1) \}  & d\geq 5\\
0 & \mbox{otherwise.}
\end{array}
\right. 
\end{eqnarray*}
\item
For $d-k\geq 2$, $k \geq 3$ and $\mu = (d-k, 1^k)$, $\dim 
\ext^1_{\fcatQ[\gr]}
(\alpha_{d-1} S_\rho, \alpha_d S_{(d-k,1^k)})$ is
\begin{eqnarray*}
 \left\{ 
\begin{array}{ll}
2 & \rho = (d-k,1^{k-1}) \\
1 & \rho \in \{ (d-k+1, 1^{k-2}), (d-k-1, 1^k), (d-k,2, 1^{k-3}), (d-k-1,2,1^{k-2}) \} \\
0 & \mbox{otherwise.}
\end{array}
\right. 
\end{eqnarray*}
(The degenerate case $k=2$ is similar.)
\end{enumerate}
To indicate how these calculations are carried out, consider the last case. By Corollary \ref{cor:d-1_beta_S_mu}, for $\rho \vdash d-1$ we have: 
\[
\dim 
\ext^1_{\fcatQ[\gr]}
(\alpha_{d-1} S_\rho, \alpha_d S_{(d-k,1^k)}) 
= 
\sum_{\nu \vdash d-2 } c^\rho_{\nu, 1}c^{(d-k,1^k)}_{\nu,11}
\]
By Example \ref{exam:CR_1n}, $c^{(d-k,1^k)}_{\nu,11} \neq 0$ iff $\nu \in \{ (d-k,1^{k-2}), (d-k-1,1^{k-1}) \}$ and in these cases $c^{(d-k,1^k)}_{\nu,11}=1$ so 
\[
\dim 
\ext^1_{\fcatQ[\gr]}
(\alpha_{d-1} S_\rho, \alpha_d S_{(d-k,1^k)}) 
= 
c^\rho_{(d-k,1^{k-2}), 1}+c^\rho_{(d-k-1,1^{k-1}), 1}
\]
By Example \ref{exam:CR_1n}, $c^\rho_{(d-k,1^{k-2}), 1}\neq 0$ for $\rho \in \{ (d-k+1,1^{k-2}), (d-k,2, 1^{k-3}), (d-k,1^{k-1}) \}$ and $c^\rho_{(d-k-1,1^{k-1}), 1}\neq 0$ for $\rho \in \{ (d-k,1^{k-1}), (d-k-1,2, 1^{k-2}), (d-k-1,1^{k})\}$ and in these cases the Littlewood-Richardson coefficients are equal to $1$. The result follows.
\end{exam}

At the other extreme, one can consider composition factors of degree one:

\begin{prop}
\label{prop:except_linear}
For $\mu \vdash d $, the multiplicity of  $\alpha_1 S_{(1)} = \aQ$ in $\beta_{d}S_\mu$ is equal to that of $S_\mu$ in $\liemod(d)$.
\end{prop}

\begin{proof}
By Theorem \ref{thm:decomp_Psi_pcoalg},  the only way to acquire a composition factor of $\aQ$ in $\beta_d S_\mu$ is from 
the term corresponding to the sequence $\underline{b}$ with  $b_d =1$ and all other terms zero. This gives $\aQ \boxtimes \lieschur (d)(-)$, using the notation of Proposition \ref{prop:psiTV_lie}. The result follows by splitting into isotypical components.  
\end{proof}

This shows that the functors $\beta_d S_\mu$ are highly non-trivial in that, apart from certain exceptional cases, the socle filtration (or, equivalently, the polynomial filtration according to Corollary \ref{exam:beta_M_socle_filt})  has the maximal possible length:

\begin{cor}
\label{cor:max_length_beta_d}
For $\mu \vdash d $, $\beta_{d}S_\mu$  has socle length $d$, unless $\mu$ is one of the following partitions: $(d)$ for $d\geq 2$; $(1^d)$ for $d \geq 3$; $(2^2)$ for $d=4$; $(2^3)$ for $d=8$.
\end{cor}  

\begin{proof}
The functor $\beta_d S_\mu$ has polynomial degree exactly $d$ and, by Corollary \ref{exam:beta_M_socle_filt}, its polynomial filtration coincides with the socle filtration (up to reindexing). If $\beta_d S_\mu$ contains $\aQ$ as a composition factor, its polynomial filtration must have length $d$, hence its socle length is $d$.

Using Proposition \ref{prop:except_linear}, the result then follows from Reutenauer's result (stated here as Theorem \ref{thm:Reutenauer_exceptional}), which determines the partitions $\mu \vdash d$ that do not occur as a composition factor of $\liemod (d)$.
\end{proof}

\subsection{Isotypical components}
The calculation of the isotypical components $F^\mu$ of $\bigotimes_{n\geq 1} S^* (\aQ \boxtimes  \lieschur(n) (-) ) $ given in Corollary \ref{cor:isotypical_cpt} is accessible by standard methods of representation theory. For this, one exploits the fact that the Lie modules $\liemod(n)$ have an explicit, combinatorial description (see Theorem \ref{thm:Reutenauer_multiplicities}).

\begin{prop}
 \label{prop:psi_tcoalg} 
For $N \in \nat^*$, the functor $\big(\Psi T_\cog  (-)\big)^N$  in $\fcatQ[\gr\times \fmodq]$ 
admits a finite filtration with 
 associated graded:
 \[
 \bigoplus_{\substack{\underline{b}\\ \sum n b_n =N}}
 \Big( 
 \bigotimes_{\substack{n=1}}^N
 \big(
\bigoplus_{\lambda(n) \vdash b_n} \alpha_{b_n} S_{\lambda(n)} 
 \boxtimes 
 \big(\schur_{\lambda(n)} \circ \lieschur(n) (-)\big)
 \big)
 \Big),
\]
where the sum is over sequences $\underline{b} = (b_n | n \in \nat^*)$ of natural numbers such that $\sum_n nb_n=N$.
 
Moreover, for a given indexing sequence $\underline{b}$,   
\begin{eqnarray*}
\bigotimes_{\substack{n=1 }}^N 
 \big(
\bigoplus_{\lambda(n) \vdash b_n} \alpha_{b_n} S_{\lambda(n)} 
 \boxtimes 
 \big(\schur_{\lambda(n)} \circ \lieschur(n) (-)\big)
 \big)
\cong 
\quad \quad
\\
\quad \quad 
\bigoplus_{\underline{\lambda}\vdash \underline{b}}
\Big(
\big(
\bigotimes_{\substack{n}} \alpha_{b_n} S_{\lambda(n)} 
\big)
\ 
\boxtimes 
\ 
\bigotimes_{\substack{n}} \big( \schur_{\lambda(n) } \circ \lieschur(n)(-) \big)
\Big),
\end{eqnarray*}
where the sum is over sequences of partitions $\underline{\lambda} = \{\lambda (n) \}$ such that $\lambda (n) \vdash b_n$.
\end{prop}

\begin{proof}
Proposition \ref{prop:psiTV_lie} implies that 
$(\Psi 
T_\cog(-))^N$ has a finite filtration with associated  graded:
 \[
 \bigoplus_{\substack{\underline{b}\\ \sum n b_n =N}}
 \big( 
 \bigotimes_{n=1}^N
 S^{b_n} (\aQ \boxtimes  \lieschur(n) (-) ) 
 \big).
\]

The first statement follows by using the Cauchy identity (see Proposition \ref{prop:cauchy_identities}):
for $b\in 
\nat$, this  gives the 
isomorphism of functors:
\[
 S^{b} ( \aQ \boxtimes \lieschur(n)(-) )
 \cong
 \bigoplus_{\lambda \vdash b} \big( \schur_\lambda \circ \aQ\big)
 \boxtimes 
 \big(\schur_\lambda \circ \lieschur(n) (-)\big). 
\]
To conclude, we use that, for $\lambda \vdash b$, the functor $\schur_\lambda \circ \aQ $ identifies with $\alpha_b S_\lambda$.

The final isomorphism follows by rearranging the tensor and the exterior tensor products.
\end{proof}

To illustrate Proposition \ref{prop:psi_tcoalg}, we consider the cases  $N \leq 4$. The case $N=1$ is trivial, and $N=2$ is straightforward; the cases $N=3$ and $N=4$ are treated below.

\begin{exam}
\label{exam:D=3}
Consider  $(\Psi T_\cog(-))^3  \in \ob \fcatQ[\gr \times \fmodq]$, which  has a  filtration with associated graded
 \[
  \bigoplus_{\substack{\underline{b}\\ \sum n b_n =3}}
 \Big( 
 \bigotimes_{\substack{n=1}}^3 
 \big(
\bigoplus_{\lambda(n) \vdash b_n}   \alpha_{b_n} S_\lambda
 \boxtimes 
 \big(\schur_{\lambda(n)} \circ \lieschur(n)(-) \big)
 \Big).
\]
There are three sequences $\underline{b}$ to consider, $(b_1, b_2, b_3)\in \{(3,0,0), (1,1,0), (0,0,1)\}$, where $b_n =0$ for $n >3$.
\begin{enumerate}
\item 
$\underline{b}= (3,0,0)$; the sum is indexed by partitions $\lambda \vdash 3$ for $n=1$. Since $\lieschur (1)(-)$ is the identity functor, this gives:
\[
\bigoplus_{\lambda \vdash 3}  \alpha_3 S_\lambda \boxtimes \schur_\lambda(-).  
\] 
\item 
$\underline{b}= (1,1,0)$; there is a unique term, corresponding to $\underline{\lambda}= ((1),(1))$.
 Since  $\liemod (1) = S_{(1)}$ and $\liemod(2)=S_{(11)}$ this gives:
\[
(\aQ \otimes \aQ) \boxtimes (\schur_{(1)} (-)\otimes \schur_{(11)}(-) ) \cong \alpha (S_{(2)} \oplus S_{(11)}) \boxtimes (\schur_{(111)}(-) \oplus \schur_{(21)}(-)),
\] 
using  $\aQ \otimes \aQ \cong \alpha_2 (S_{(2)} \oplus S_{(11)})$ and $\schur_{(1)} \otimes \schur_{(11)}(-) \cong \schur_{(111)}(-) \oplus \schur_{(21)}(-)$ (cf. Example \ref{exam:LR}).
\item 
$\underline{b}= (0,0,1)$; there is a single term with $\lambda = (1) \vdash b_3 =1$; since $\liemod (3)= S_{(21)}$ this gives: 
\[
\alpha_1 S_{(1)} \boxtimes \schur_{(21)}(-).
\] 
\end{enumerate}
These give the following isotypical components, $\bigoplus_\mu  F^\mu \boxtimes \schur_\mu (-)$, as in Corollary \ref{cor:isotypical_cpt}: 
\[
\begin{array}{|l|l|}
\hline
\schur_\mu (-) & F^\mu \\
\hline 
\hline
\schur_{(111)}(-) & 
\alpha_3 S_{(111)}  \oplus \alpha_2 (S_{(2)} \oplus S_{(11)} )\\
\hline
\schur_{(21)}(-)  & \alpha_3 S_{(21)} \oplus \alpha_2 (S_{(2)} \oplus S_{(11)}) \oplus  \alpha_1 S_{(1)}
\\
\hline
\schur_{(3)}(-)& \alpha S_{(3)} 
\\
\hline 
\end{array}
\]
\end{exam}

\begin{exam}
\label{exam:D=4}
Consider  $(\Psi T_\cog(-))^4  \in \ob \fcatQ[\gr \times \fmodq]$, which  has a  filtration with associated graded
 \[
  \bigoplus_{\substack{\underline{b}\\ \sum n b_n =4}}
 \Big( 
 \bigotimes_{\substack{n=1}}^4 
 \big(
\bigoplus_{\lambda(n) \vdash b_n}   \alpha_{b_n} S_\lambda
 \boxtimes 
 \big(\schur_{\lambda(n)} \circ \lieschur(n)(-) \big)
 \Big).
\]
There are five sequences $\underline{b}$ to consider:
$$
(b_1 , b_2, b_3, b_4)
\in 
\{(4,0,0,0), (2,1,0,0),(0,2,0,0), (1,0,1,0), (0,0,0,1)\},
$$ where $b_n = 0$ for $n >4$. We proceed as in the case $N=3$.
\begin{enumerate}
\item 
$\underline{b}= (4,0,0,0)$;  this gives:
\[
\bigoplus_{\lambda \vdash 4}  \alpha_4 S_\lambda \boxtimes \schur_\lambda(-) .
\] 
\item 
$\underline{b}= (2,1,0,0)$; there are two terms given by $\underline{\lambda}\in \{ ((11), (1)), ((2),(1)) \}$, giving: 
\[
\alpha_3 (S_{(2)} \tenfb S_{(1)}) \boxtimes (\schur_{(2)}(-) \otimes \schur_{(1,1)}(-) ) 
\ 
\oplus 
\ \alpha_2 (S_{(1,1)} \tenfb S_{(1)}) \boxtimes (\schur_{(1,1)}(-) \otimes \schur_{(1,1)}(-) ),
\] 
whence (using Example \ref{exam:LR}) the contribution:
\begin{eqnarray*}
&&\big(\alpha_3 (S_{(3)} \oplus S_{(21)})\boxtimes  (\schur_{(31)}(-) \oplus \schur_{(211)}(-))\big)
\ \oplus \ 
\\
&& \quad \quad 
\big(\alpha_3 (S_{(111)} \oplus S_{(21)})\boxtimes  (\schur_{(1111)}(-) \oplus \schur_{(211)}(-) \oplus \schur_{(22)}(-))\big).
\end{eqnarray*}
\item 
$\underline{b}= (0,2,0,0)$; again the sum is indexed over $\lambda \vdash b_2 = 2$, giving:
\[
\big( \alpha_2 S_{(2)} \boxtimes \schur_{(2)} \circ \schur_{(11)}(-) \big) 
\oplus 
\big( \alpha_2 S_{(11)} \boxtimes \schur_{(11)} \circ \schur_{(11)}(-) \big).
\]
Hence, using Lemma \ref{lem:plethysm}, this gives:
\[
\big( \alpha_2 S_{(2)} \boxtimes (\schur_{(1111)}(-) \oplus \schur_{(22)}(-))  \big) 
\oplus 
\big ( \alpha_2 S_{(11)} \boxtimes \schur_{(211)}(-) \big).
\]
\item 
$\underline{b}= (1,0,1,0)$; there is a unique term corresponding to $\lambda_1 \vdash b_1 =1$ and $\lambda_3 \vdash b_3=1$. This gives:
\[
(\aQ \otimes \aQ) \boxtimes (\schur_{(1)}(-) \otimes \schur_{(21)}(-) ) 
\ \cong \  
\alpha_2(S_{(2)} \oplus S_{(11)}) \boxtimes (\schur_{(31)}(-) \oplus \schur_{(211)}(-) \oplus \schur_{(22)}(-)), 
\]
using Example \ref{exam:LR} for the isomorphism.
\item 
$\underline{b}= (0,0,0,1)$; there is a unique term corresponding to $\lambda_4 \vdash 1$, giving  $\aQ \boxtimes  \lieschur (4)(-) $ and  hence (using Example \ref{exam:liemod}):
\[
\alpha_1 S_{(1)} \boxtimes (\schur_{(31)} (-) \oplus \schur_{(211)}(-)).
\]
\end{enumerate}
These give the following isotypical components:
\[
\begin{array}{|l|l|}
\hline
\schur_\mu (-) & F^\mu \\
\hline 
\hline
\schur_{(1111)}(-) & 
\alpha_4 S_{(1111)} 
\oplus 
\alpha_3 (S_{(111)} \oplus S_{(21)} )
\oplus 
\alpha_2 S_{(2)} \\
\hline
\schur_{(211)}(-) & 
\alpha_4 S_{(211)}
\oplus 
\alpha_3 (S_{(3)} \oplus S_{(21)}^{\oplus 2} \oplus  S_{(111)})
\oplus
\alpha_2 (S_{(11)}^{\oplus 2} \oplus S_{(2)})
\oplus 
\alpha_1 S_{(1)}\\
\hline
\schur_{(22)} (-) &
\alpha_4 S_{(22)} 
\oplus 
\alpha_3 (S_{(111)} \oplus S_{(21)} )
\oplus 
\alpha_2 (S_{(2)}^{\oplus 2} \oplus S_{(11)})
\\
\hline
\schur_{(31)}(-)
&
\alpha_4 S_{(31)}
\oplus 
\alpha_3 (S_{(3)} \oplus  S_{(21)}) 
\oplus 
\alpha_2 (S_{(2)} \oplus S_{(11)}) 
\oplus 
\alpha_1 S_{(1)}
\\
\hline
\schur_{(4)}(-) & 
\alpha_4 S_{(4)}
\\
\hline
\end{array}
\]
\end{exam}

\ 
Having illustrated the method, we proceed to the general statement, for which we introduce the following notation.

\begin{nota}
\label{nota:sequence_partition_CR_pleth}
For $\underline{\lambda}: = \{ \lambda (i)\  | \  i \in \mathscr{I} \subset \nat \}$ a finite sequence of partitions and partitions $\mu, \nu$, 
\begin{enumerate}
\item 
 extending the notation for Littlewood-Richardson coefficients (cf. Definition \ref{def:LR-coefficients}), let $c^\nu_{\underline{\lambda}}$ be the multiplicity of $S_\nu$ in $\bigodot_n S_{\lambda (n)}$; equivalently, in terms of Schur functors, this is the multiplicity of $\schur_\nu (-)$ in $\bigotimes_n \schur_{\lambda (n)} (-)$;
\item 
extending the notation for the plethysm coefficients (cf. Definition \ref{def:plethysm-coefficients}), let $p_{\mu, \liemod (n)}^\nu$ denote the multiplicity of $S_\nu$ in $S_\mu \circ \liemod (n)$, where $n \in \nat$, equivalently, in terms of Schur functors, this is the multiplicity of $\schur _\nu (-) $ in $\schur _\mu \circ \lieschur(n) (-)$.
\end{enumerate}
\end{nota}

Proposition \ref{prop:psi_tcoalg} has the following Corollary, which should be viewed as an explicit form of Theorem \ref{thm:decomp_Psi_pcoalg}:

\begin{cor}
\label{cor:Grothendieck_ring_Psi_pcoalg}
There is an equality in the Grothendieck group of $\fcatQ[\gr \times \fb]$:
\[
\big[\Psi \pcoalg\big ]
= 
\sum_{\rho,\nu}
\sum_{\underline{b}}
\sum_{\underline{\lambda}\vdash \underline{b}}
\sum_{\underline{\mu}}
\Big(
c^\rho _{\underline{\lambda}}
c^\nu _{\underline{\mu}}
\prod_{\substack{n \\  b_n\neq 0}} p^{\mu(n)}_{\lambda(n), \liemod(n)}
\Big)
\ [\alpha_{|\rho|} S_\rho \boxtimes S_\nu ]
\]
where $\underline{b} = (b_n | n \in \nat^*)$ is a sequence of natural numbers, $\rho, \nu$ are partitions, $\underline{\lambda}$ is as in Notation \ref{nota:sequence_partition_CR_pleth},  and $\underline{\mu}$ ranges over sequences of partitions $\{ \mu (n) | 1 \leq n \in \nat\}$ such that $|\mu (n)| = n |\lambda (n)|$ for all $n$.

Hence, for $\nu \vdash d>0$, in the Grothendieck group of $\fcatQ[\gr]$:
\[
\big[\beta_d S_\nu]
= 
\sum_{\rho}
\sum_{\underline{b}}
\sum_{\underline{\lambda}\vdash \underline{b}}
\sum_{\underline{\mu}}
\Big(
c^\rho _{\underline{\lambda}}
c^\nu _{\underline{\mu}}
\prod_{\substack{n \\  b_n\neq 0}} p^{\mu (n)}_{\lambda (n), \liemod(n)}
\Big)
\ [\alpha_{|\rho|} S_\rho].
\]
\end{cor}

\subsection{Three infinite families of examples}
\label{subsect:three_infinite_families}

This section illustrates Theorem \ref{thm:decomp_Psi_pcoalg} by considering the three infinite families of partitions: 
$(n)$, $(1^n)$ and $(n-1, 1)$, for $n \in \nat^*$ (with $n >1$ in the final case). The results can be proved directly using Corollary \ref{cor:Grothendieck_ring_Psi_pcoalg}. To try and make the argument more conceptual, we proceed 
 by using Proposition \ref{prop:psi_tcoalg} (which is a reformulation of Theorem \ref{thm:decomp_Psi_pcoalg}).
 
Proposition \ref{prop:psi_tcoalg} implies that, for our chosen partition $\mu$, we require to calculate the isotypical component of $\schur_\mu(-)$ in:
\begin{eqnarray}
\label{eqn:isotypical_cpt}
\bigoplus_{\underline{b}}
\bigoplus_{\underline{\lambda}\vdash \underline{b}}
\Big(
\big(
\bigotimes_{\substack{n }} \alpha_{b_n} S_{\lambda(n)} 
\big)
\ 
\boxtimes 
\ 
\bigotimes_{\substack{n}} \big( \schur_{\lambda(n) } \circ \lieschur(n)(-) \big)
\Big) 
\end{eqnarray}
where the sum is over sequences $\underline{b} = (b_n | n \in \nat^*)$ of natural numbers such that $\sum_n nb_n = |\mu|$. Thus the argument boils down  to determining for which sequences of partitions $\underline{\lambda}$, $\schur_\mu (-)$ occurs as a composition factor of 
\begin{eqnarray}
\label{eqn:lambda_seq_schur}
\bigotimes_{\substack{n}} \big( \schur_{\lambda(n) } \circ \lieschur(n)(-) \big).
\end{eqnarray}
 
 The key fact that makes the cases considered here accessible is the following property of the Lie representations, which follows from Theorem \ref{thm:Reutenauer_multiplicities} (and can also be proved directly, by elementary means).

\begin{lem}
\label{lem:lie_easy} 
For $n \in \nat$:
\begin{enumerate}
\item 
$S_{(n)}$ is a composition factor of $\liemod(n)$ if and only if $n=1$;
\item 
$S_{(1^n)}$ is a composition factor of $\liemod(n)$ if and only if $n\in \{1,2 \}$;
\item 
$S_{(n-1,1)}$ is a composition  factor of $\liemod(n)$ if and only if $n \geq 3$.
\end{enumerate}
Moreover, in each case, the simple occurs with multiplicity one.
\end{lem}

We first treat the family of partitions $(n)$, which exhibits exceptional behaviour:

\begin{prop}
\label{prop:isotypical_trivial}
For $n \in \nat$,  
\[
\beta_n S_{(n)} \cong \alpha_n S_{(n)};
\]
in particular $\alpha_n S_{(n)}$ is injective in $\fpiQ{\gr}$.
\end{prop}

\begin{proof}
Using the first statement of Lemma \ref{lem:lie_easy}, it is clear that $\schur_{(n)}(-)$ occurs in (\ref{eqn:lambda_seq_schur}) for a sequence of partitions $\underline{\lambda}$ if and only if $\lambda (1)=n$ and all other $\lambda(i)$ are zero. Hence the corresponding isotypical component in equation (\ref{eqn:isotypical_cpt}) is $\alpha_n S_{(n)}$, as required. The final statement follows since $\beta_n S_{(n)}$ is injective in $\fpiQ{\gr}$, by Theorem \ref{thm:injectivity_beta}.
\end{proof}

\begin{rem}
Proposition \ref{prop:isotypical_trivial}  can also be deduced from \cite[Theorem 4.2]{V_ext}, where it is shown  that,  $\ext^*_{\fcatk[\gr]}(\aQ^{\otimes m}, \alpha S_{(n)})$ is $\rat$ for $m=n$ and $*=0$ and is $0$ otherwise. 
\end{rem}

The case of the family of partitions $(1^n)$ is slightly more complicated:

\begin{prop}
\label{prop:isotypical_sgn}
For $n \in \nat^*$,  
in the Grothendieck group of $\fcatQ[\gr]$:
\begin{eqnarray*}
\big[\beta_n S_{(1^n)}]
&=& 
\sum_{\substack{a, b \in \nat \\a + 2b = n}}
[\alpha_a S_{(1^a)} \otimes \alpha_b S_{(b)}] 
\\
&=& \sum_{\substack{a, b \in \nat, a + 2b = n}}
([\alpha_{a+b} S_{(b,1^a)}] + [ \alpha_{a+b} S_{(b+1,1^{a-1})}])
\end{eqnarray*}
where,  $S_{(0,1^a)}$ and $S_{(a+1,1^{-1})}$ (corresponding to $b=0$ and $a=0$ respectively) are understood to be zero.
\end{prop}

\begin{proof}
The proof is analogous to that of Proposition \ref{prop:isotypical_trivial}, this time using the second statement of Lemma \ref{lem:lie_easy}. In this case, $\schur_{(1^n)} (-)$ occurs as a composition factor of (\ref{eqn:lambda_seq_schur}) if and only if $\lambda (1) = (1^a)$, $\lambda (2) = (b)$ with all other $\lambda (i)$ zero, so that $a+2b=n$; in this case $\schur_{(1^n)} (-)$ occurs with multiplicity one. 

The result then follows by identifying the corresponding isotypical component in (\ref{eqn:isotypical_cpt}).
\end{proof}

\begin{rem}
In Theorem \ref{thm:omega_omega1_sign} of  Section \ref{subsect:omega_beta_1^n} we give more information on the structure of $\beta_n S_{(1^n)}$.
\end{rem}

Finally, we treat the case of the family of partitions $(n-1, 1)$:

\begin{prop}
\label{prop:beta_(n-1)1}
For $2 \leq n \in \nat$,  in the Grothendieck group of $\fcatQ[\gr]$:
\begin{eqnarray*}
[\beta_n S_{((n-1)1)}] &=&  \sum_{j=2}^n \Big ( [\alpha_j S_{((j-1)1)} ] +   [\alpha_{j-1} S_{(j-1)}]\Big).
\end{eqnarray*}
\end{prop}

\begin{proof}
As in Propositions \ref{prop:isotypical_trivial} and \ref{prop:isotypical_sgn}, using Lemma \ref{lem:lie_easy} we first identify the $\underline{\lambda}$ for which $\schur_{(n-1,1)}(-)$ can occur. 

One checks that there are two possible cases:
\begin{enumerate}
\item 
$\lambda (1) = (n-1,1)$, with all other terms zero; 
\item 
$\lambda (1) = (a)$, $\lambda (b) = (1)$, with all other terms zero, were $a, b \in \nat$ with $b \geq 2$ and $a+b =n$. 
\end{enumerate}
In each case,  $\schur_{(n-1,1)}(-)$ occurs with multiplicity one. 

In the corresponding isotypical component, using (\ref{eqn:isotypical_cpt}), the first case contributes $\alpha_n S_{(n-1,1)}$ and the second $\alpha_2 S_{(a)} \otimes \alpha_1 S_{(1)}$. Using Pieri's rule gives the result. 
\end{proof}

In the following Proposition and later in the paper, the structure of a functor $F \in \ob\fcatk[\gr]$ is given using Loewy diagrams.  Such diagram is associated to the socle filtration (see Definition \ref{defn:socle_length}) and  is arranged so that the simple factors of the socle appear at its base and the $i$-th row gives the composition factors of $\soc_i(F)/\soc_{i-1}(F)$. The Loewy diagram also indicates how the successive layers of the socle filtration are joined together: a line indicates the \textit{existence} of a non-trivial extension between the corresponding composition factors.

By Corollary \ref{exam:beta_M_socle_filt}, up to reindexing, the socle filtration of $\beta_d M$ coincides with the polynomial filtration of $\beta_d M$. So, in the Loewy diagram, the composition factors of $\beta_d M$ having the same polynomial degree appear on the same horizontal level. This allows the polynomial filtration to be read off from the Loewy diagram, with polynomial degree {\em decreasing} as one moves {\em up} from the socle.  

\begin{rem}
\label{rem:examples_Loewy}
In most of the examples given in this paper, it follows from Corollary \ref{cor:d-1_beta_S_mu} and Example \ref{exam:ext1}, that a non-trivial extension indicated by a line in the Loewy diagram is actually \textit{unique} (up to a scalar). In this case, the Loewy diagram essentially gives the complete structure of the functor.
\end{rem}

The following Proposition gives the complete structure of the functor $\beta_n S_{((n-1)1)}$.
 
\begin{prop}
\label{prop:Loewy_S(n-1)1}
The structure of $\beta_n S_{((n-1)1)}$ is represented by the following Loewy diagram, which has the form of a comb:
\[
\xymatrix@R-1.5pc@C-1.5pc{
&&& \alpha_1 S_{(1)} 
\\
&&\cdots &&\alpha_2 S_{(11)} \ar@{-}[ul]
\\
& \alpha_{n-2} S_{(n-2)}&& \cdots \ar@{-}[ur]\ar@{-}[ul] \\
 \alpha_{n-1} S_{(n-1)} && \alpha_{n-1}  S_{((n-2)1)} \ar@{-}[ur]\ar@{-}[ul] \\
& \alpha_n S_{((n-1)1)}.  \ar@{-}[ur]\ar@{-}[ul]
}
\] 
Namely, each possible extension between the simple composition factors occurs and the extensions are unique up to non-zero scalar multiple.
\end{prop}

\begin{proof}
Proposition \ref{prop:beta_(n-1)1} gives the composition factors of $\beta_n S_{((n-1)1)}$. In particular, each factor is of multiplicity one. 
For the extensions, by Corollary \ref{cor:d-1_beta_S_mu} we can have non-trivial extensions only between composition factors of degree $d$ and of degree $d-1$. By Example  \ref{exam:ext1}, there is a unique extension, up to non-zero scalar multiple, between $\alpha_k S_{(k-1,1)}$ and $\alpha_{k-1} S_{(k-1)}$ and there is no extension between $\alpha_k S_{(k)}$ and $\alpha_{k-1} S_\rho$ for $\rho \vdash k-1$. (The latter also follows from the fact that  the functors $\alpha_k S_{(k)}$ are injective in $\fpiQ{\gr}$, by Proposition \ref{prop:isotypical_trivial}.)

Since the socle of $\beta_n S_{((n-1)1)}$  is the simple functor $\alpha_n S_{((n-1)1)}$ (by Proposition \ref{prop:beta_socle}), the structure of $\beta_n S_{((n-1)1)}$ is determined by the  Loewy diagram given in the statement.
\end{proof}

\section{Calculating $\omega \Psi $}
\label{sec:omega_psi}

The purpose of this section is to address the problem of calculating $\omega \Psi$. More particularly, we are interested in understanding $\omega \Psi \pcoalg$ by exploiting the results on calculating $\omega$ on exponential functors given in Section \ref{sec:omega_exponential}. 

In Section \ref{subsect:gen_theory_omega_Psi} we apply these results to study $\omega \beta_d S_\lambda$ for $\lambda \vdash d$ (see Proposition \ref{prop:coadbar_es_isotypical}). The key player there is the map  $\coadbar_\lambda$ derived from the coadjoint coaction. We use this to give a first approximation to  $\omega \beta_d S_\lambda$ in Corollary \ref{cor:bounds_omega}.

In Section \ref{subsect:coadbar_non_triv}, we provide the other fundamental ingredient, which gives a quantitative statement ensuring that $\coadbar_\lambda$ is non-trivial, except in the case $\lambda = (d)$ (see Theorem \ref{thm:coadbar}).

Some cohomological consequences of Theorem \ref{thm:coadbar} are presented in Section \ref{subsect:cohom_consequences_coadbar}. For instance, in Corollary \ref{cor:ext1_gr_fout}, we calculate $\ext^1$ in $\foutQ[\gr]$ between simple polynomial functors.

The section concludes by analysing the examples covered in Section \ref{sec:psi}.

\subsection{The general theory}
\label{subsect:gen_theory_omega_Psi}

This section addresses the problem of calculating $\omega \Psi \pcoalg$ and, hence, the calculation of the functors $\omega \beta_d S_\lambda$, for $\lambda \vdash d$. By the Schur correspondence, this can be achieved by calculating $\omega \Psi T_\cog (V)$ as a functor of $V \in \ob \fmodq$. This is expressed in terms of the coadjoint coaction; by Proposition \ref{prop:Omega_tensor_algebra}, there is an exact sequence in $\fcatQ[\gr]$:
\begin{eqnarray}
\label{eqn:coadbar_es_schur}
\\
\nonumber
0
\rightarrow 
\omega \Psi T_\cog (V)
\rightarrow 
\Psi T_{\mathrm{coalg}}(V) 
\stackrel{\coadbar}{\longrightarrow}
\Psi T_{\mathrm{coalg}}(V) \otimes V
\rightarrow 
\coker \big( \coadbar_{\Psi T_{\mathrm{coalg}}} \big)(V)
\rightarrow 
0
\end{eqnarray}
that is natural in $V$.

By the Schur correspondence, this corresponds to 
the  exact sequence in $\fcatQ[\gr \times \fb]$:
\begin{eqnarray}
\label{eqn:coadbar_es_fb}
\\
\nonumber
0
\rightarrow 
\omega \Psi \pcoalg
\rightarrow 
\Psi \pcoalg
\stackrel{\coadbar}{\longrightarrow} 
\Psi \pcoalg \tenfb P^\fb_\mathbf{1}
\rightarrow 
\coker \big (\coadbar_{\Psi \pcoalg} \big)
\rightarrow 
0.
\end{eqnarray}

It is useful to pass to isotypical components, which gives the following:

\begin{prop}
\label{prop:coadbar_es_isotypical}
For $\lambda \vdash d$, applying $-\otimes_\fb S_\lambda$ to the exact sequence (\ref{eqn:coadbar_es_fb}) yields the exact sequence:
\begin{eqnarray}
\label{eqn:omega_beta_d_lambda}
\\
\nonumber
0
\rightarrow 
\omega \beta_d S_\lambda 
\rightarrow 
\beta_d S_\lambda
\stackrel{\coadbar_\lambda}{\longrightarrow} 
\bigoplus_{\substack{\mu \preceq \lambda \\ |\mu | = d-1} }
\beta_{d-1} S_\mu
\rightarrow 
\coker \big (\coadbar_\lambda \big)
\rightarrow 
0, 
\end{eqnarray}
where $\coadbar_\lambda$ is shorthand for $\coadbar \otimes_\fb S_\lambda$. 
\end{prop}

All of the terms appearing in the exact sequence (\ref{eqn:omega_beta_d_lambda}) are finite functors in $\fpiQ{\gr}$. Hence we can reason using the Grothendieck group of $ \fpiQ{\gr}$ (which is equivalent to working in terms of the multiplicity of composition factors).

\begin{nota}
\label{nota:grothendieck_group}
For $F$ a finite functor in $\fpiQ{\gr}$, write $[F]$ for its image in the Grothendieck group. 
For $F_1, F_2$ two such finite functors:
\begin{enumerate}
\item 
write $[F_1] \leq [F_2]$ if, for each partition  $\mu$, the multiplicity of $\alpha_{|\mu|} S_\mu$ in $F_1$ is less than or equal to its multiplicity in $F_2$; 
\item 
write $[F_1]\cap [F_2]$ for the element of the Grothendieck group represented by the maximal semi-simple functor $G$ such that $[G] \leq [F_i]$ for $i \in \{1, 2\}$. 
\end{enumerate}
\end{nota}

Using this notation, we have the following equalities and bounds for the objects appearing in (\ref{eqn:omega_beta_d_lambda}):

\begin{cor}
\label{cor:bounds_omega}
For $\lambda \vdash d$, there are equalities in the Grothendieck group of $\fpiQ{\gr}$:
\begin{eqnarray*}
[\omega \beta_d S_\lambda] -  
[\beta_d S_\lambda]
&=&
[\bigoplus_{\substack{\mu \preceq \lambda \\ |\mu | = d-1} }
\beta_{d-1} S_\mu]
-[\coker \big (\coadbar_\lambda \big)]
\\ \ 
[\omega \beta_d S_\lambda] & =& 
[\beta_d S_\lambda] - [\mathrm{image \ } (\coadbar_\lambda)] 
\\ \ 
[\coker \big (\coadbar_\lambda \big)] &=&
[\bigoplus_{\substack{\mu \preceq \lambda \\ |\mu | = d-1} }
\beta_{d-1} S_\mu]
- [\mathrm{image \ } (\coadbar_\lambda)].
\end{eqnarray*}
Moreover, there are inequalities:
\begin{eqnarray*}
[\mathrm{image \ } (\coadbar_\lambda)] &\leq & [\beta_d S_\lambda] \cap [\bigoplus_{\substack{\mu \preceq \lambda \\ |\mu | = d-1} }
\beta_{d-1} S_\mu], 
\\ \ 
[\omega \beta_d S_\lambda] & \geq & 
[\beta_d S_\lambda] - \Big( [\beta_d S_\lambda] \cap [\bigoplus_{\substack{\mu \preceq \lambda \\ |\mu | = d-1} }
\beta_{d-1} S_\mu] \Big) 
\\ \ 
[\coker \big (\coadbar_\lambda \big)] &\geq &
[\bigoplus_{\substack{\mu \preceq \lambda \\ |\mu | = d-1} }
\beta_{d-1} S_\mu] - \Big( [\beta_d S_\lambda] \cap [\bigoplus_{\substack{\mu \preceq \lambda \\ |\mu | = d-1} }
\beta_{d-1} S_\mu] \Big) .
\end{eqnarray*}
\end{cor}

\begin{proof}
This is a straightforward consequence of Proposition \ref{prop:coadbar_es_isotypical} .
\end{proof}

\begin{rem}
\label{rem:omega_beta_versus_coker_coad}
The composition factors of $\beta_n S_\lambda$ and of $\bigoplus_{\substack{\mu \preceq \lambda \\ |\mu | = n-1} }\beta_{n-1} S_\mu$ can be calculated by using Theorem \ref{thm:decomp_Psi_pcoalg} (or the more explicit form, Corollary \ref{cor:Grothendieck_ring_Psi_pcoalg}). Corollary \ref{cor:bounds_omega} thus has the following consequence:
\begin{enumerate}
\item 
knowing $[\omega \beta_d S_\lambda]$ is equivalent to knowing $[\coker \big (\coadbar_\lambda \big)]$; 
\item 
the upper bound for $[\mathrm{image \ } (\coadbar_\lambda)]$ and the lower bounds for 
$[\omega \beta_d S_\lambda]$ and $[\coker \big (\coadbar_\lambda \big)]$ can be calculated explicitly.
\end{enumerate}
It is expected that these bounds are close to the true values, since $\coadbar$ is highly non-trivial.
\end{rem}

\subsection{Non-triviality of $\coadbar$ in lowest possible degree}
\label{subsect:coadbar_non_triv}

The purpose of this section is to establish a non-triviality result for $\coadbar$. More precisely, for $\lambda \vdash d$, we consider $\coadbar_\lambda$ in polynomial degree $d-1$ (i.e., after applying the functor $\qhat{d-1}$). This gives the natural transformation
\[
\qhat{d-1} \big( \beta_d S_\lambda \big)
\stackrel{\qhat{d-1} (\coadbar_\lambda)}{\longrightarrow}
\bigoplus_{\substack{\mu \preceq \lambda \\ |\mu | = d-1}}
\qhat{d-1} \big( \beta_{d-1} S_\mu \big)
 \cong 
\bigoplus_{\substack{\mu \preceq \lambda \\ |\mu | = d-1} }
\alpha_{d-1} S_\mu.
\]

The main result of the section is the following, which establishes the non-triviality of $\coadbar_\lambda$, apart from in the exceptional case $\lambda = (d)$.

\begin{thm}
\label{thm:coadbar}
For $\lambda \vdash d$:
\begin{enumerate}
\item 
the morphism 
$$
\beta_d S_\lambda
\stackrel{\coadbar_\lambda}{\longrightarrow} 
\bigoplus_{\substack{\mu \preceq \lambda \\ |\mu | = d-1} }
\beta_{d-1} S_\mu
$$
 is zero if and only if $\lambda = (d)$;
\item 
if $d \geq 2$ and $\lambda \neq (d)$, the morphism 
\[
{\qhat{d-1} (\coadbar_\lambda)} \ : \ 
\qhat{d-1} \big( \beta_d S_\lambda \big)
\longrightarrow 
\bigoplus_{\substack{\mu \preceq \lambda \\ |\mu | = d-1} }
\alpha_{d-1} S_\mu
\]
is surjective.
\end{enumerate}
\end{thm}

\begin{rem}
For $d\in \{ 0,1 \}$ it is easy to see that $\coadbar_\lambda$ is zero. Hence we concentrate on the case $d\geq 2$ below.
\end{rem}

By the Schur correspondence, the result can be proved by working with  $$\Psi T_{\mathrm{coalg}}(V) \stackrel{\coadbar}{\longrightarrow} \Psi T_{\mathrm{coalg}}(V) \otimes V.$$
 Focussing on the $S_\lambda$-isotypical component then corresponds to picking out the $\schur_\lambda (-)$-isotypical component; in particular, we are interested in terms of polynomial degree exactly $d$ with respect to $V \in \ob \fmodq$.

To understand the polynomial degree $(d-1)$ part (with respect to $\gr$), we use Proposition \ref{prop:psiTV_lie}, which describes the associated graded of $\Psi T_\cog  (-)$  as
 \[
 \bigotimes _{n \geq 1} 
  S^* (\aQ \boxtimes  \lieschur(n) (-) ). 
 \]
The component of polynomial degree $d$ with respect to $V$ and $d-1$ with respect to $\gr$ is:
\[
 S^{d-2} (\aQ \boxtimes \schur_{(1)} (-)) \otimes (\aQ \boxtimes  \schur_{(1^2)} (-) ),
\] 
as in the proof of Corollary \ref{cor:d-1_beta_S_mu}. 

A similar argument applies to $\Psi T_{\mathrm{coalg}}(V) \otimes  V$, with the corresponding component being $
S^{d-1} (\aQ \boxtimes \schur_{(1)} (-)) \otimes \schur_{(1)}(-).
$ 
Hence, we are interested in the component of the coadjoint coaction:
\begin{eqnarray}
\label{eqn:coadbar_n_n-1}
\\
\nonumber
S^{d-2} (\aQ \boxtimes \schur_{(1)} (-)) \otimes (\aQ \boxtimes  \schur_{(1^2)} (-) )
\stackrel{\coadbar}{\longrightarrow}  
S^{d-1} (\aQ \boxtimes \schur_{(1)} (-)) \otimes \schur_{(1)}(-).
\end{eqnarray}
Here $\schur_{(1)}(-)$ is the identity functor and $\schur_{(1^2)} (-)$ identifies as $\Lambda^2(-)$, the second exterior power.

This is made more transparent by applying the cross-effect functor $\cre_{d-1}$. In the following statement, the actions of the symmetric groups arise from the place permutations of the tensor factors:

\begin{lem}
\label{lem:coadbar_n_n-1}
For $2\leq d  \in \nat$, applying the functor $\cre_{d-1}$ to (\ref{eqn:coadbar_n_n-1})  yields the  $\sym_{d-1}$-equivariant natural transformation 
\begin{eqnarray}
\label{eqn:explicit_coadbar_n_n-1}
\big(V^{\otimes d-2} \otimes \Lambda^2 (V) \big) \uparrow_{\sym_{d-2}}^{\sym_{d-1}} 
\rightarrow 
(V^{\otimes d-1}) \otimes V 
\cong V^{\otimes d}
\end{eqnarray}
induced up from $V^{\otimes d-2} \otimes \Lambda^2 (V) \subset V^{\otimes d-1} \otimes V$ given by the inclusion $\Lambda^2 (V) \subset V^{\otimes 2}$ applied to the last tensor factors.
\end{lem}

\begin{proof} 
We consider the map obtained by  applying the functor $\cre_{d-1}$ to (\ref{eqn:coadbar_n_n-1}).
 Evaluating on $V$ (and adjusting the side on which the symmetric group acts appropriately),  the domain can be rewritten as:
\[
\aQ^{\otimes d-1} \otimes_{\sym_{d-2}} \big( V^{\otimes d-2} \otimes \Lambda^2 V\big) ,
\] 
where $\sym_{d-2}$ acts by place permutations on the first $d-2$ tensor factors of $\aQ^{\otimes d-1}$ and $V^{\otimes d-2} \otimes \Lambda^2 V$ respectively. Likewise, the codomain can be written as:
\[
\aQ^{\otimes d-1 } \otimes_{\sym_{d-1}} (V^{\otimes d-1}) \otimes V.
\] 
The passage to cross-effects on objects is then clear. 

The identification of the map follows using the definition of $\coadbar$. 
\end{proof}

\begin{rem}
\label{rem:coadbar_restrict_alternative}
It is useful to compare this with the behaviour of the coadjoint map 
\[
\coadbar : 
T_\cog (V) ^{\otimes r} \rightarrow T_{\cog} (V)^{\otimes r} \otimes V.
\]

The cross effect argument used above allows us to replace $T_\cog (V)$ with the augmentation ideal $\overline{T_\cog (V)}$ and reduce to working with $r=d-1$ and the terms of polynomial degree $d$ with respect to $V$. 

This reduces us to studying the restriction 
\[
\coadbar : 
\big(V ^{\otimes d-2} \otimes T^2(V)\big)\uparrow_{\sym_{d-2}} ^{\sym_{d-1}}  \rightarrow V^{\otimes d-1} \otimes V
\]
where the codomain is the degree $d$ part of $\overline{T_\cog (V)}^{\otimes d-1}$ with respect to $V$, writing $T^2(V)$ for $V^{\otimes 2}$. 

Consider the case $d=2$ for simplicity. In this case, the above restriction of $\coadbar$ identifies (by the definition of $\coadbar$) as the composite 
\[
T^2 (V)
\twoheadrightarrow 
\Lambda^2 (V) 
\hookrightarrow 
V^{\otimes 2},
\]
where the first map is the canonical surjection and the second the canonical inclusion.

It follows that the restriction of $\coadbar$ considered above factors across the map of Lemma \ref{lem:coadbar_n_n-1} via the surjection 
\[
\big(V ^{\otimes d-2} \otimes T^2(V)\big)\uparrow_{\sym_{d-2}} ^{\sym_{d-1}} 
\twoheadrightarrow 
\big(V ^{\otimes d-2} \otimes \Lambda^2(V)\big)\uparrow_{\sym_{d-2}} ^{\sym_{d-1}} 
\]
induced by $T^2 (V) \twoheadrightarrow \Lambda^2 (V)$.

This shows that understanding the cokernel of $\coadbar$ after passage to cross-effects as in Lemma \ref{lem:coadbar_n_n-1} is equivalent to studying the more elementary restriction discussed in this remark.
\end{rem}

The key calculational input is then the following basic result:

\begin{prop}
\label{prop:coker_coadbar_n_n-1}
For $2\leq d \in \nat$, the map (\ref{eqn:explicit_coadbar_n_n-1}) fits into the natural exact sequence:
\[
\big(V^{\otimes d-2} \otimes \Lambda^2 (V) \big) \uparrow_{\sym_{d-2}}^{\sym_{d-1}} 
\rightarrow 
 V^{\otimes d}
\rightarrow 
S^d (V) 
\rightarrow 0,
\]
where $S^d (-)$ is the $d$th symmetric power functor, that identifies with $\schur_{(d)} (-)$.
\end{prop}

\begin{proof}
For $d=2$, this corresponds to the usual (split) short exact sequence $0\rightarrow \Lambda^2(V) \rightarrow V^{\otimes 2} \rightarrow  S^2 (V) \rightarrow  0$.

For $d>2$, the term with $\Lambda^2 (V)$ in the $i$th tensor factor (for $1\leq i \leq d-1$), imposes the `commutativity' of terms in the $i$th and $d$th tensor factors of $V^{\otimes d}$ similarly to above. Since the transpositions $(i, d)$ generate the symmetric group $\sym_d$, the result follows.
\end{proof}

\begin{proof}[Proof of Theorem \ref{thm:coadbar}]
The case $\lambda = (d)$ can be analysed directly, since $\coadbar _{(d)}$ reduces to 
\[
\beta_d S_{(d)} \cong \alpha_d S_{(d)}
\rightarrow 
\beta_{d-1} S_{(d-1) } \cong \alpha_{d-1} S_{(d-1)}
\] 
(if $d=0$ the codomain is understood to be zero), using Proposition \ref{prop:isotypical_trivial} for the identifications.  This map is zero, since $\alpha_d S_{(d)}$ and $\alpha_{d-1} S_{(d-1)}$ are non-isomorphic simple functors.

Now consider the case $\lambda\neq (d)$ (so that $d \geq 2$). Using the identification provided by Lemma \ref{lem:coadbar_n_n-1}, Proposition \ref{prop:coker_coadbar_n_n-1} implies that the $\schur_\lambda (-)$-isotypical component of ${\qhat{d-1} (\coadbar_\lambda)}$ is surjective.
 In particular, since $d \geq 2$ (in particular $d\geq 1$), it is non-trivial, since the set $\{ \mu \preceq \lambda \ | \ |\mu| = d-1 \}$ is non-empty. 
\end{proof}

\subsection{Consequences of Theorem \ref{thm:coadbar}}
\label{subsect:cohom_consequences_coadbar}

The non-triviality of $\coadbar$ as established in Theorem \ref{thm:coadbar} has some important consequences, as presented in this section.

\begin{cor}
\label{cor:beta_not_fout}
For $\lambda \vdash d$,  $\beta_{d} S_\lambda$ is an object of $\foutQ[\gr]$ if and only if $\lambda = (d)$.

For $\lambda = (d)$, one has the identifications:
\begin{eqnarray*}
\omega \beta_d S_{(d)} & \cong & \alpha_d S_{(d)} \\
\coker (\coadbar_{(d)}) & \cong &  \alpha_{d-1} S_{(d-1)},
\end{eqnarray*}
where $S_{(d-1)}$ is taken to be zero for $d=0$. In particular, $\alpha_d S_{(d)}$ is injective in $\fpoutgrQ[<\infty]$.
\end{cor}

\begin{proof}
By construction, $\beta_d S_\lambda$ belongs to $\fpoutgrQ[<\infty]$ if and only if $\omega \beta_d S_\lambda = \beta_d S_\lambda$. The latter condition is equivalent to the triviality of $\coadbar_\lambda$, by Proposition \ref{prop:coadbar_es_isotypical}, whence the first statement follows from Theorem \ref{thm:coadbar}. The analysis for $\lambda =(d)$ is straightforward, using Proposition \ref{prop:inj_cogenerators_Out_poly} for the injectivity statement.
\end{proof}

Theorem \ref{thm:coadbar} also gives information on the polynomial degree of $\coker (\coadbar_\lambda)$:

\begin{cor}
\label{cor:poly_deg_coker_coadbar}
For $2 \leq d \in \nat$ and $\lambda \vdash d$,  $\omega \beta_d S_\lambda$ has polynomial degree exactly $d$.

If $\lambda = (d)$, then $\coker (\coadbar_\lambda) $ has polynomial degree exactly $d-1$. If $\lambda \neq (d)$, then 
\[
\coker (\coadbar_\lambda) \in \f_{d-2} (\gr; \rat). 
\]
\end{cor}

\begin{proof}
If $\lambda \neq (d)$, Theorem \ref{thm:coadbar} implies that the cokernel of $\coadbar_\lambda$ is a quotient of 
\[
\bigoplus_{\substack{\mu \preceq \lambda \\ |\mu | = d-1} }
\big(\beta_{d-1} S_\mu
/ \alpha_{d-1} S_\mu\big)
\]
and each $\beta_{d-1} S_\mu
/ \alpha_{d-1} S_\mu$ has polynomial degree at most $d-2$. 
\end{proof} 

\begin{rem}
Theorem \ref{thm:omega_omega1_sign} below shows that $\coker (\coadbar _{(1^d)}) $ has polynomial degree exactly $d-2$ if $d>2$.
\end{rem}

Combined with Corollary \ref{cor:d-1_beta_S_mu}, Theorem \ref{thm:coadbar} implies the following description of $\ext^1_{\foutQ[\gr]}$ between simple functors:

\begin{cor}
\label{cor:ext1_gr_fout}
For $\lambda \vdash d$, where $d>0$, $\ext^1_{\foutQ[\gr]}
(\alpha_{|\rho|} S_\rho, \alpha_d  S_\lambda)=0$ if $|\rho | \neq d-1$. 

If $\rho \vdash d-1$ and $\lambda \neq (d)$:
\[
\dim 
\ext^1_{\foutQ[\gr]}
(\alpha_{d-1} S_\rho, \alpha_d  S_\lambda)
= 
\left\{
\begin{array}{ll}
\sum_{\nu \vdash d-2 } c^\rho_{\nu, 1}c^{\lambda}_{\nu,11} 
&\rho \not \preceq \lambda 
\\ 
\big(\sum_{\nu \vdash d-2 } c^\rho_{\nu, 1}c^{\lambda}_{\nu,11} \big) - 1 
& \rho \preceq \lambda. 
\end{array}
\right.
\]
\end{cor}

\begin{proof}
The vanishing if $|\rho | \neq d-1$ is immediate from Corollary \ref{cor:d-1_beta_S_mu}. 

If $\lambda \neq (d)$, Theorem \ref{thm:coadbar} determines $\qhat{d-1} (\omega \beta_d S_\lambda)$ as follows. Since since $\qhat{d-1}$ is exact (see Proposition \ref{prop:unimodularity_char_0}),  there is a short exact 
sequence 
\[
0
\rightarrow 
\qhat{d-1} (\omega \beta_d S_\lambda)
\rightarrow 
\qhat{d-1} (\beta_d S_\lambda)
\rightarrow 
\bigoplus_{\substack{\rho \preceq \lambda \\ |\rho | = d-1} }
\alpha_{d-1} S_\rho
\rightarrow 
0.
\]
Corollary \ref{cor:d-1_beta_S_mu} determines $\qhat{d-1} (\beta_d S_\lambda)$, from which one deduces $\qhat{d-1} (\omega \beta_d S_\lambda)$ and hence the  $\ext^1_{\foutQ[\gr]}$ by standard methods (see the proof of Corollary \ref{cor:d-1_beta_S_mu}). Explicitly, the cokernel removes one composition factor of $\alpha_{d-1} S_\rho$ for each $\rho \preceq \lambda$. 
\end{proof}

\begin{rem}
Note that, for any $\rho \preceq \lambda$ with $|\rho |= d-1$, if $\lambda \neq (d)$ there exists $\nu \preceq \rho$ with $|\nu |= d-2$ such that the  Young diagram of the skew partition $\lambda / \nu$ does not have two boxes in the same row. Thus the dimensions given in the statement are always non-negative.
\end{rem}

\begin{exam}
\label{exam:b_1a_Ext_Fout}
Consider a positive integer $a$ and the partitions $\rho:= (1^{a-1}) \preceq (1^a)=: \lambda$, taking $d= a$. 
Then Corollary \ref{cor:ext1_gr_fout} shows that 
\[
\ext^1_{\foutQ[\gr]}
(\alpha_{d-1} S_{(1^{a-1})} , \alpha_d  S_{(1^a))})=0,
\]
since there is a single $\nu \vdash d-2$ such that $\nu \preceq \rho$, namely $(1^{a-2})$.
\end{exam}

\begin{prop}
\label{prop:fout_injective_simples}
For $\lambda \vdash d$, the canonical inclusion $\alpha_d S_\lambda \hookrightarrow \omega \beta_d S_{\lambda}$ is an isomorphism if and only if either $\lambda = (d)$ or, for $d>1$, $\lambda = (d-1,1)$.

Equivalently, $\alpha_d S_\lambda$ is injective in $\fpoutgrQ[<\infty]$ if and only if either $\lambda = (d)$ or, for $d>1$, $\lambda = (d-1,1)$.
\end{prop}

\begin{proof}
That the two conditions are equivalent is clear. 

Corollary \ref{cor:beta_not_fout} implies that $\alpha_d S_{(d)}$ is injective in $\fpoutgrQ[<\infty]$. One can analyse $\alpha_d S_{(d-1,1)}$ similarly (for $d>1$). The only partition $\nu \vdash d-2$ such that $c_{\nu, 11}^{(d-1,1)}$ is non-zero is $(d-2)$ (and the Littlewood-Richardson coefficient is one). The partitions $\rho$ for which $c_{(d-2 ), (1)}^{\rho}$ is not zero are $(d-1)$ and $(d-2,1)$, with both coefficients one, by the Pieri rule. In particular, both satisfy $\rho \preceq (d-1,1)$. The $\ext^1_{\foutQ[\gr]}$-calculation in Corollary  \ref{cor:ext1_gr_fout} therefore gives the vanishing that ensures that $\alpha_d S_{(d-1,1)}$ is injective in $\fpoutgrQ[<\infty]$.

It remains to show that, if $\lambda \not \in \{ (d), (d-1,1) \}$, there exists $\rho \vdash d-1$ such that $\ext^1_{\foutQ[\gr]}
(\alpha_{d-1} S_\rho, \alpha_d  S_\lambda) \neq 0$. By Corollary  \ref{cor:ext1_gr_fout}, it suffices to show that there exists $\rho \not \preceq \lambda$ and $\nu \vdash d-2$ with $c^\rho_{\nu, 1}c^{\lambda}_{\nu,11} \neq 0$. 

Suppose that $\nu \vdash d-2$ such that $\nu \preceq \lambda$ and $c_{\nu,11}^\lambda \neq 0$, so that the skew diagram $\lambda/ \nu$ has two boxes not in the same row. Consider the following two conditions: 
\begin{enumerate}
\item 
If $\nu_1 = \lambda_1$ (i.e., the first row of $\nu$ coincides with that of $\lambda$), one can take $\rho$ such that $\rho_1= \lambda_1 +1$ and $\rho_i = \nu_i$ for $i>1$. Thus $\nu \preceq \rho \not \preceq \lambda$, as desired. 
\item 
The conjugate situation, arguing with the first column, namely $\nu_1 ^\dagger = \lambda_1^\dagger$. One constructs $\rho$ so that $\rho^\dagger_1 = \lambda^\dagger_1 +1$ and $\rho^\dagger_i = \nu^\dagger_i$ for $i>1$. Again $\nu \preceq \rho \not \preceq \lambda$, as desired. 
\end{enumerate}
Thus, if either of these conditions holds, one can construct a suitable $\rho$, since it is clear that $c^\rho_{\nu, 1}c^{\lambda}_{\nu,11} \neq 0$, by construction.

To conclude, one observes that the only partitions $\lambda$ for which no partition $\nu \vdash d-2$ exists with $c_{\nu, 11}^\lambda \neq 0$ and one of the above conditions  satisfied are the partitions $ (d)$ and $((d-1),1)$.
\end{proof}

This implies:

\begin{cor}
\label{cor:fpout_semisimple_dleq2}
The category  $\f_d^{\mathrm{Out}} (\gr; \rat)$ is semi-simple if and only if $d \leq 2$. 

In particular, any object of $\f_2^{\mathrm{Out}} (\gr; \rat)$ is injective in $\f_{<\infty}^{\mathrm{Out}} (\gr; \rat)$ and hence in $\f_d^{\mathrm{Out}} (\gr; \rat)$ for any $d \geq 2$.
\end{cor}

\begin{proof}
The first statement follows immediately from Proposition \ref{prop:fout_injective_simples}. 

For $d=2$, this implies that any object of $\f_2^{\mathrm{Out}} (\gr; \rat)$ is injective in that category. Since the functor $q_2^\gr : \f_{<\infty} (\gr; \rat) \rightarrow \f_2 (\gr; \rat)$ is exact (by Proposition \ref{prop:unimodularity_char_0}) and restricts to $q_2^\gr : \f_{<\infty}^{\mathrm{Out}} (\gr; \rat) \rightarrow \f_2^{\mathrm{Out}} (\gr; \rat)$, it follows by a standard argument that any such object is also injective in $\f_{<\infty}^{\mathrm{Out}} (\gr; \rat)$. This, in turn, implies injectivity in $\f_d^{\mathrm{Out}}(\gr; \rat)$, for any $d\geq 2$. 
\end{proof}

\begin{exam}
The simple objects of $\f_3^{\mathrm{Out}} (\gr; \rat)$  of polynomial degree exactly $3$  are 
$$\alpha_3 S_{(3)}, 
\quad \alpha_3 S_{(21)}, \quad \alpha_3 S_{(111)}.
$$ 
The first two are injective in $\f_3^{\mathrm{Out}} (\gr; \rat)$, whereas the last is not. The only partition $\rho \vdash 2$ for which $
\ext^1 _{\foutk[\gr]} (\alpha_2 S_{\rho}, \alpha_3 S_{(111)} ) $ is non-zero is $\rho = (2)$ and there is (up to isomorphism) a unique non-trivial extension in this case. 
The corresponding functor is the `smallest' object of $\fpoutgrQ[<\infty]$ not lying in the essential image of $\fcatQ[\ab]$.  

This is a conceptual,  functorial analogue of \cite[Theorem 1]{TW}. Turchin and Willwacher observed that the corresponding representation of $\out (\zed^{\star 3})$ (obtained from our viewpoint by evaluation of the functor) has dimension $7$ and that it  does not factor through $GL_3 (\zed)$. This is the smallest dimension in which such a representation can exist.
\end{exam}

As in Theorem \ref{thm:fcatk_ab_semisimple}, we consider the global (injective) dimension:

\begin{cor}
\label{cor:gl_dim_fpoutgr}
For $1<d \in \nat$, the category $\f_d^{\mathrm{Out}} (\gr; \rat)$ has global dimension at most $d-2$, with equality in the cases $d\in\{ 2,3\}$.

 In particular, for $d \geq 2$, 
 \[
 \mathrm{gl.dim} \f_d^{\mathrm{Out}} (\gr; \rat)< \mathrm{gl.dim} \fpolyQ{d}{\gr} = d-1.
 \]
and  $\mathrm{gl.dim} \f_d^{\mathrm{Out}} (\gr; \rat)>0$ for $d >2$.
\end{cor}

\begin{proof}
The case $d=2$ is clear, hence we suppose that $d\geq 3$. The general bound is proved as in \cite{DPV} for Theorem \ref{thm:fcatk_ab_semisimple}, taking into account that $\f_2^{\mathrm{Out}} (\gr; \rat)$ is semi-simple. 

Namely,  for any object $F$ of $\f_d^{\mathrm{Out}} (\gr; \rat)$, consider the construction of a minimal injective resolution $I^\bullet$ in $\f_d^{\mathrm{Out}} (\gr; \rat)$. This will have  the property that the term $I^t$ in cohomological degree $t$ has polynomial degree $d-t$. 

Consider the construction of $I^{d-2}$: this is defined to be  the injective envelope of the cokernel of $I^{d-4}\rightarrow I^{d-3}$ ($F \hookrightarrow I^0$ if $d=3$). The cokernel is an outer functor of polynomial degree $2$, hence is injective in $\f_d ^{\mathrm{Out}}(\gr; \rat)$, by Corollary \ref{cor:fpout_semisimple_dleq2}. It follows that the construction of the minimal injective resolution stops here. 

This gives the first statement, the equality for $d=3$ following since $\f_3^{\mathrm{Out}} (\gr; \rat)$ is not semisimple. Likewise for $d>2$, this gives $\mathrm{gl.dim}  \f_d^{\mathrm{Out}} (\gr; \rat)>0$ for $d>2$.

Clearly one has $\mathrm{gl.dim} \f_d^{\mathrm{Out}} (\gr; \rat)< \mathrm{gl.dim} \fpolyQ{d}{\gr} = d-1 $, where the equality is given by \cite{DPV} (stated as Theorem \ref{thm:fcatk_ab_semisimple} here).
\end{proof}

\begin{rem}
To resume, one has the inclusions of full sub-categories of polynomial functors:
\[
\fpiQ{\ab}
\subsetneq
\f_{<\infty}^{\mathrm{Out}} (\gr; \rat)
\subsetneq
\fpiQ{\gr}.
\]
The category $\fpiQ{\ab}$ is semi-simple, by Theorem \ref{thm:fcatk_ab_semisimple}. Corollary \ref{cor:ext1_gr_fout} together with Corollary \ref{cor:beta_not_fout} imply that the two inclusions are highly non-trivial. It is expected that the global dimension of $\f_{<\infty}^{\mathrm{Out}} (\gr; \rat)$ is unbounded.
\end{rem}

\subsection{The case $\lambda = ((n-1),1)$}
\label{subsect:(n1)_omega}

In this section we describe the exact sequence  of Proposition \ref{prop:coadbar_es_isotypical} in the case $\lambda = ((n-1),1)$, for $n>1$.

\begin{prop}
\label{prop:omega_beta_Sn1}
For 
$1<n \in \nat$, the exact sequence of Proposition \ref{prop:coadbar_es_isotypical} identifies as:
 \[
0 
\rightarrow 
\alpha_n S_{((n-1),1)}
\rightarrow 
 \beta_{n} S_{((n-1),1)}
\stackrel{\coadbar_{((n-1),1)}}{\longrightarrow}
\
\beta_{n-1} S_{((n-2),1)} 
\oplus 
\alpha_{n-1} S_{(n-1)}
\rightarrow 
0 
\]
($S_{((n-2)1)} $ is understood to be zero for $n=2$). In particular, $\coker (\coadbar_{((n-1),1)} ) =0$.
\end{prop}

\begin{proof}
Theorem \ref{thm:coadbar} implies that the components of $\coadbar_{((n-1),1)}$ to $\beta_{n-1} S_{((n-2),1)} $ and to $
\beta_{n-1} S_{(n-1)}$ are both non-trivial and the kernel of $\coadbar_{((n-1),1)}$ is $\alpha_n S_{(n-1),1}$, by Proposition \ref{prop:fout_injective_simples}.

It remains to show the surjectivity of $\coadbar_{((n-1),1)}$; this follows from   Proposition \ref{prop:Loewy_S(n-1)1}.  (In fact, it suffices to check this at the level of composition factors.)
\end{proof}

\begin{rem}
\label{rem-12.22}
Extending Proposition \ref{prop:Loewy_S(n-1)1}, Proposition \ref{prop:omega_beta_Sn1} can be illustrated by the 
 Loewy diagram:
\[
\xymatrix@R-1.5pc@C-1.5pc{
&&& \alpha_1 S_{(1)} 
\\
&&\cdots &&\alpha_2 S_{(11)} \ar@{-}[ul]
\\
& \alpha_{n-2} S_{(n-2)}&& \cdots \ar@{-}[ur]\ar@{-}[ul] \\
 \alpha_{n-1} S_{(n-1)} && \alpha_{n-1}  S_{((n-2)1)} \ar@{-}[ur]\ar@{-}[ul] \\
& \boxed{ \alpha_n S_{((n-1)1)}}.  \ar@{-}[ur]\ar@{-}[ul]
}
\] 
The boxed term corresponds to $\omega \beta_n S_{((n-1),1)}$, which identifies with the socle.
\end{rem}

\subsection{The case $\lambda =(1^n)$}
\label{subsect:omega_beta_1^n}
We now revisit the structure of $\beta_n S_{(1^n)}$ (cf. Proposition \ref{prop:isotypical_sgn}), identifying $\omega \beta_n S_{(1^n)}$ and $\coker (\coadbar_{(1^n)})$.

\begin{thm}
\label{thm:omega_omega1_sign}
For $2 \leq  n \in \nat$,  
 $
\mathrm{image\  }(\coadbar_{(1^n)})
\cong 
\omega \beta_{n-1} S_{(1^{n-1} )}$,
 so that there is a short exact sequence:
\[
0
\rightarrow 
\omega \beta_n S_{(1^n)}
\rightarrow 
\beta_n S_{(1^n)} 
\rightarrow 
\omega \beta_{n-1} S_{(1^{n-1} )}
\rightarrow 
0
\]
induced by $\coadbar_{(1^n)}$.

Hence, for $n \geq 3$, $\coker (\coadbar_{(1^n)})
\cong 
\omega \beta_{n-2} S_{(1^{n-2} )}$, so that the exact sequence of Proposition \ref{prop:coadbar_es_isotypical} has the form:
\[
0
\rightarrow 
\omega \beta_n S_{(1^n)} 
\rightarrow 
\beta_n S_{(1^n)} 
\stackrel{\coadbar_{(1^n)}}{\longrightarrow} 
\beta_{n-1} S_{(1^{n-1})}
\rightarrow 
\omega \beta_{n-2} S_{(1^{n-2} )}
\rightarrow 
0.
\]

In the Grothendieck group of $\fcatQ[\gr]$:
\[
[\omega \beta_n S_{(1^n)}]
= 
\sum_{\substack{a+2b=n+1\\b>0}}  [\alpha_{a+b} S_{(b1^a)}],
\]
$\omega \beta_n S_{(1^n)}$ is uniserial and has socle length $[\frac{n+1}{2}]$.
\end{thm}

\begin{rem}
\label{rem:complex_beta_sign}
Theorem \ref{thm:omega_omega1_sign} implies that the maps $\coadbar$ induce an 
exact complex in $\fcatQ[\gr]$:
\begin{eqnarray}
\label{eqn:cx_beta_sign}
\ldots 
\rightarrow 
\beta_n S_{(1^n)}
\rightarrow 
\beta_{n-1} S_{(1^{n-1})}
\rightarrow 
\beta_{n-2} S_{(1^{n-2})}
\rightarrow 
\ldots 
\rightarrow 
\beta_1 S_{(1)} = \alpha_1 S_{(1)}.
\end{eqnarray}
\end{rem}

\begin{proof}[Proof of Theorem \ref{thm:omega_omega1_sign}]
One way to prove Theorem \ref{thm:omega_omega1_sign} is by exploiting the relationship with higher Hochschild homology. This allows  an argument of Turchin and Willwacher  \cite{TW} to be used, as sketched below in Example \ref{exam:Theorem_1_TW}.
Here we sketch a direct approach, taking as input Proposition \ref{prop:isotypical_sgn}.

First we explain why (\ref{eqn:cx_beta_sign}) is a complex. The image of the composite 
$\beta_n S_{(1^n)}
\rightarrow 
\beta_{n-1} S_{(1^{n-1})}
\rightarrow 
\beta_{n-2} S_{(1^{n-2})}$ (assuming $n \geq 3$) is a subfunctor of $\beta_{n-2} S_{(1^{n-2})}$, which has socle $\alpha_{n-2} S_{(1^{n-2})}$, by Proposition \ref{prop:beta_socle}. Hence the image must contain a composition factor $\alpha_{n-2} S_{(1^{n-2})}$ if it is non-zero.
 However, by Proposition \ref{prop:isotypical_sgn}, $\beta_n S_{(1^n)}$ does not have a composition factor $\alpha_{n-2} S_{(1^{n-2})}$. It follows that the image is zero, as required. 

Theorem \ref{thm:omega_omega1_sign} is now proved by showing that the complex (\ref{eqn:cx_beta_sign})  is exact. This requires some information on the differential $\beta_n S_{(1^n)}
\rightarrow 
\beta_{n-1} S_{(1^{n-1})}$ ($n>1$). It suffices to study the associated graded of the polynomial filtration and use the identification given by  Proposition \ref{prop:isotypical_sgn}:
$$
\big[\beta_n S_{(1^n)}]
= 
\sum_{\substack{a, b \in \nat \\a + 2b = n}}
[\Lambda^a \circ \aQ  \otimes S^b \circ \aQ], 
$$
where we have identified 
$\alpha_a S_{(1^a)} \otimes \alpha_b S_{(b)} \cong \Lambda^a \circ \aQ  \otimes S^b \circ \aQ$.

On the associated graded of the polynomial filtration, the differential $\beta_n S_{(1^n)}
\rightarrow 
\beta_{n-1} S_{(1^{n-1})}$ has components of the form
\begin{eqnarray}
\label{eqn:ddR}
\Lambda^a \circ \aQ  \otimes S^b \circ \aQ
\rightarrow 
\Lambda^{a+1} \circ \aQ  \otimes S^{b-1} \circ \aQ
\end{eqnarray}
(understood to be zero if $b=0$), where $a+2b =n$. Here (\ref{eqn:ddR}) corresponds to the polynomial degree $a+b$ slice of the differential. 

Now, inspection of the composition factors shows that $\hom_{\fcatQ[\gr]} (\Lambda^a \circ \aQ  \otimes S^b \circ \aQ, 
\Lambda^{a+1} \circ \aQ  \otimes S^{b-1} \circ \aQ) =\rat$ if $b>0$, with generator provided by the dual de Rham differential. 
Moreover, one can check readily  that the  exactness of the complex is equivalent to the assertion that each of these maps (for $b>0$) is non-zero. (More precisely, if these maps are non-zero,  the layers of the polynomial filtration of the complex identify with dual de Rham complexes, which are exact.)

Hence, to prove the result, it suffices to show that the morphism $\coadbar$ induces a non-trivial map (\ref{eqn:ddR}), for $b>0$. This is achieved using the explicit definition of $\coadbar$ and unravelling the above identifications; indeed, this analysis shows why $\coadbar$ induces the 
dual de Rham differential on the associated graded of the polynomial filtration. 

That $\omega \beta_n S_{(1^n)}$ is uniserial follows from the fact that it has simple socle $\alpha_n S_{(1^n)}$ and at most one composition factor in each polynomial degree. Since $\ext^1_{\fcatQ[\gr]}$ between simple polynomial functors vanishes if the difference in polynomial degrees is not equal to $1$ (see Corollary \ref{cor:d-1_beta_S_mu} for the precise statement), one deduces uniseriality.
\end{proof}

\begin{rem}
The structure of  $ \omega \beta_n S_{(1^n)}$ is represented by the Loewy diagram:
\[
\xymatrix@R-1.5pc@C-1.5pc{
\alpha_{p+r} S_{(p 1^r)}
\ar@{-}[d]
\\
\alpha_{p+r+1} S_{((p-1) 1^{2+r})}
\ar@{-}[d]
\\
{\cdots}
\\
\alpha_{n-2} S_{(31^{n-5})}
\ar@{-}[u]
\\
\alpha_{n-1} S_{(21^{n-3})}
\ar@{-}[u]
\\
\alpha_n S_{(1^n)}
\ar@{-}[u]
}
\]
where $p$ and $r$ correspond to reduction modulo $2$: $n+1=2p+r$ with $r \in \{0,1\}$. 

More explicitly, if $n=2p$ is even, the top composition factor is $\alpha_{p+1}S_{(p1)}$; if $n=2p-1$ is odd, then the top composition factor is $\alpha_{p}S_{(p)}$.
\end{rem}

\begin{exam}
\label{exam:Loewy_diagrams_S1n}
To illustrate Theorem \ref{thm:omega_omega1_sign}, we consider the Loewy diagrams for $\omega \beta_n S_{(1^n)} \subset \beta_n S_{(1^n)}$ for $ n \leq 5$. In the Loewy diagrams,  the composition factors in  $\omega \beta_n S_{(1^n)}$ are boxed.

The composition factors of  $\beta_n S_{(1^n)}$  are given by Proposition \ref{prop:isotypical_sgn} and those of $\omega \beta_n S_{(1^n)}$ by Theorem \ref{thm:omega_omega1_sign}.  By Proposition \ref{prop:beta_socle}, the socle of $ \beta_n S_{(1^n)}$ is the simple functor $\alpha_n  S_{(1^n)}$. By Theorem \ref{thm:omega_omega1_sign}, the quotient of $\beta_n S_{(1^n)}$ by $\omega \beta_n S_{(1^n)}$  identifies with $\omega \beta_{n-1} S_{(1^{n-1})}$ if $n>1$, which appears as the boxed terms in the preceding diagram. In each case, $\omega \beta_n S_{(1^n)}$ is a uniserial functor with socle $\alpha_n S_{(1^n)}$.  

For the extensions, by Corollary \ref{cor:d-1_beta_S_mu} we can have non-trivial extensions only between composition factors of degree $d$ and of degree $d-1$. By Example  \ref{exam:ext1}, there is a unique extension, up to non-zero scalar multiple, between $\alpha_k S_{(1^k)}$ and $\alpha_{k-1} S_{(1^{k-1})}$, between $\alpha_k S_{(1^k)}$ and $\alpha_{k-1} S_{(2,1^{k-2})}$ and between $\alpha_4 S_{(211)}$ and $\alpha_{3} S_{(3)}$. The lines indicate these unique (up to non-zero scalar multiple) extensions.

\begin{enumerate}
\item 
$\beta_ 1 S_{(1)} =\omega \beta_1 S_{(1)}   = \alpha_1 S_{(1)}$:
\[
\xymatrix@R-1.5pc@C-1.5pc{
\boxed{\alpha_1 S_{(1)}}\ .
}
\]
\item 
$\beta_2 S_{(11)}$:
\[
\xymatrix@R-1.5pc@C-1.5pc{
&\alpha_1 S_{(1)} \\
\boxed{\alpha_2 S_{(11)}}\ .\ar@{-}[ur]
}
\]
\item 
$\beta_3 S_{(111)}$:
\[
\xymatrix@R-1.5pc@C-1.5pc{
\boxed{\alpha_2 S_{(2)}} && \alpha_2 S_{(11)} \\
&\boxed{\alpha_3 S_{(111)}}\ . \ar@{-}[ur]\ar@{-}[ul]
}
\]
\item 
$\beta_4 S_{(1^4)}$:
\[
\xymatrix@R-1.5pc@C-1.5pc{
&\alpha_2 S_{(2)}\\
\boxed{\alpha_3 S_{(21)}}
\ar@{--}[ur]
 && \alpha_3 S_{(111)} \ar@{-}[ul] \\
&\boxed{\alpha_4 S_{(1111)}} \  .\ar@{-}[ur]\ar@{-}[ul]
}
\]
The dotted line indicates a non-trivial extension that is believed to exist, due to a `symmetry' between the behaviour of the partitions $(2,1)$ and $(1^3)$ that results from the relation between the associated central primitive idempotents in $\rat [\sym_3]$.
\item 
$\beta_5 S_{(1^5)}$:
\[
\xymatrix@R-1.5pc@C-1.5pc{
\boxed{\alpha_3 S_{(3)}}
&&
\alpha_3 S_{(21)}
\\
&
\boxed{\alpha_4 S_{(211)}}
\ar@{-}[ul]
\ar@{.}[ur]
&
&
\alpha_4 S_{(1111)} 
\ar@{-}[ul]
\\
&
&
\boxed{\alpha_5 S_{(11111)}}\ .
\ar@{-}[ul]
\ar@{-}[ur]
}
\]
The dotted line indicates a non-trivial extension that is expected to exist. 
\end{enumerate}
\end{exam}

\subsection{Examples for $|\nu|=4$}

The structure of the functors 
\[
\omega \beta_4 S_\nu \subset \beta_4 S_\nu
\]
for $|\nu|=4$ has already been determined for $\nu \in \{ (4), (1111), (31) \}$, since these fit into the families of partitions $(n)$, $(1^n)$ and $((n-1),1)$ respectively. It remains to consider the cases $\nu \in \{ (22), (211) \}$.

\begin{exam}
\label{exam:S22}
For $\nu =(22)$, the composition factors are  given by Example \ref{exam:D=4}. 
The map $\coadbar_{(22)}$ in Proposition \ref{prop:coadbar_es_isotypical} 
 reduces to the unique (up to non-zero scalar) non-trivial map
\[
\beta_4 S_{(22)}
\rightarrow 
\beta_3 S_{(21)}. 
\]
The structure of $\beta_3 S_{(21)}$ is known by Proposition \ref{prop:Loewy_S(n-1)1}.  This leads to the following Loewy diagram:
\[
\xymatrix@R-1.5pc@C-1.5pc{
\boxed{\alpha_2 S_{(2)}} &&
\alpha_2 S_{(11)} &&
\alpha_2 S_{(2)}
\\
&
\boxed{\alpha_3 S_{(111)}}
 \ar@{-}[ul]
\ar@{--}[ur]
&&
\alpha_3 S_{(21)}
\ar@{-}[ul]
\ar@{-}[ur]
\\
&&
\boxed{\alpha_4 S_{(22)}}\ ,
\ar@{-}[ul]
\ar@{-}[ur]
}
\]
in which the boxed terms correspond to $\omega \beta_4 S_{(22)}$. As for the case of $\beta_4 S_{(1^4)}$, the dashed arrow represents a non-trivial extension, that is believed to exist, for the same reason.

In particular, $\omega \beta_4 S_{(22)}$ is uniserial, with composition factors $\alpha_4 S_{(22)}$,  
$\alpha_3 S_{(111)}$ and $\alpha_2 S_{(2)}$. The cokernel of $\coadbar_{(22)}$ is $\alpha_1 S_{(1)}$. 
\end{exam}

\begin{exam}
\label{exam:S211}
For $\nu =(211)$, the composition factors of $\beta_4 S_{(211)}$ are again given in Example \ref{exam:D=4}. 
The map $\coadbar_{(211)}$ is a morphism
\[
\beta_4 S_{(211)}
\rightarrow 
\beta_3 S_{(21)}
\oplus 
\beta_3 S_{(111)}
\]
in which both components are non-trivial. 

The functors $\omega \beta_4 S_{(211)} $ and $\coker (\coadbar_{(211)})$ can be determined 
 using our understanding of $\foutQ[\gr]$, in particular the fact that $\alpha_3 S_{(3)}$ and $\alpha_3 S_{(21)}$ are injective in $\fpoutgrQ[<\infty]$
 
It follows that $\omega \beta_4 S_{(211)} \subset \soc_2 \beta_4 S_{(211)} $ where $\soc_2 \beta_4 S_{(211)}$ has Loewy diagram:
\[
\xymatrix@R-1.5pc@C-1.5pc{
\boxed{\alpha_3 S_{(3)}} & \boxed{\alpha_3 S_{(21)}} & 
\alpha_3 S_{(21)}&
\alpha_3 S_{(111)}
\\
& \boxed{\alpha_4 S_{(211)}}  \ .
\ar@{-}[ul]
\ar@{-}[u]
\ar@{-}[ur]
\ar@{-}[urr]
}
\]
Moreover, $\coker (\coadbar_{(211)}) \cong \alpha_2 S_{(2)}$. In particular, this gives another example where the polynomial degree upper bound given by Corollary \ref{cor:poly_deg_coker_coadbar} is exact. 

However, this information does not uniquely determine the structure of $\beta_4 S_{(211)}$ (which is why we restricted to $\soc_2$ above).
\end{exam}

\part{Higher Hochschild homology}
\label{part:HHH}

\section{Higher Hochschild homology}
\label{sec:hhh}

The purpose of this section is to review the theory of higher Hochschild homology. Except for our functorial framework presented in Section \ref{subsect:hhh_gr}, the material is well-known and we make no claim to originality.

Higher Hochschild (co)homology is defined here in the pointed context, taking coefficients in $\Gamma$-modules (i.e., $\fcatk[\Gamma]$), where $\Gamma$ is the category of finite pointed sets\footnote{Warning: here we adopt the convention used by Pirashvili \cite{Phh} rather than that of  Segal.}.  To release the basepoint, 
one takes coefficients arising from $\fin$-modules, where $\fin$ is the category of finite sets. 
This is covered in Section \ref{subsect:defn_hhh}, following Pirashvili \cite{Phh}.

Section \ref{subsect:loday_A} recalls the Loday constructions, which provide the coefficients that are used  in this paper. In particular, we treat the multiplicative properties that are key in the applications. 

Section \ref{subsect:shuffle} takes up the thread by recalling the shuffle products in higher Hochschild homology, reviewing basic theory that will be exploited in the proofs of the main results. 

Finally, in Section \ref{subsect:hhh_gr}, we introduce the functors on $\gr$ that are derived from higher Hochschild homology. This corresponds to restricting along the nerve (or classifying  space functor) $B : \gr \rightarrow \Delta\op \setspt$.  When the coefficients are independent of the basepoint, this yields {\em outer functors}, i.e. functors in $\fout{\gr}{\kring}$.

That higher Hochschild homology for coefficients arising from $\fcatk[\fin]$ gives rise to representations of the outer automorphism groups $\out(\fr)$ as above was observed and exploited by Turchin and Willwacher in \cite{TW}. The result presented here, Proposition \ref{prop:HH_group_case}, insists upon the full functorial behaviour with respect to $\gr$, namely that one obtains objects of $\foutk[\gr]$.

\subsection{Defining Higher Hochschild homology}
\label{subsect:defn_hhh}
 Unless stated otherwise, $\kring$ is an arbitrary unital, commutative ring. 

\begin{nota}
\label{nota:fin_Famma}
\nomenclature{$\sets$}{category of sets\nomrefpage}
\nomenclature{$\fin$}{category of finite sets\nomrefpage}
\nomenclature{$\sets_*$}{category of pointed sets\nomrefpage}
\nomenclature{$\Gamma$}{category of finite pointed sets\nomrefpage}
\nomenclature{$\theta$}{forgetful functor $\Gamma \rightarrow \fin$\nomrefpage}
Denote by 
\begin{enumerate}
\item 
$\fin \subset \sets$ the full subcategory of  finite  sets;
\item 
$\Gamma \subset \sets_*$ the full subcategory of  finite pointed 
sets; 
\item 
$(-)_+ : \fin \rightarrow \Gamma$ the left adjoint to the 
functor $\theta : \Gamma \rightarrow \fin$ that forgets the basepoint.
\end{enumerate}
\end{nota}

The categories $\fin$ and $\Gamma$ have small 
skeleta with objects $\n:= \{ 1, \ldots , n \}\in \ob \fin$ (by convention, 
$\mathbf{0}=\emptyset$) and $\n_+ \in \ob 
\Gamma$, for $n \in \nat$.

By Proposition \ref{prop:precomp_adjunction}, one has the following  (which introduces the {\em ad hoc} notation $\theta^+$):
 
\begin{prop}
\label{prop:adjoints_theta}
\nomenclature{$\theta^*$}{restriction along $\theta$\nomrefpage}
\nomenclature{$\theta^+$}{restriction along $(-)_+$\nomrefpage}
 The adjunction $(-)_+ : \fin\rightleftarrows \Gamma : \theta$ 
induces an adjunction
 \[
  \theta^* :  \fcatk[\fin ] \rightleftarrows \fcatk[\Gamma] : \theta^+.
\]
The functors $\theta^*$ and $\theta^+$ are 
 exact and symmetric monoidal.
 \end{prop}

Left Kan extension gives $\fcatk[\Gamma] \rightarrow \fcatk[\sets_*]$ and 
right Kan extension gives $\fcatk[\Gamma\op] \rightarrow \fcatk[\sets_*\op]$.
Hence, composition defines  functors:
\begin{eqnarray*}
 \ssetpt \times \fcatk[\Gamma] &\rightarrow& \fcatk[\Delta \op] \\
 (\ssetpt) \op \times \fcatk[\Gamma \op] & \rightarrow & \fcatk [\Delta],
\end{eqnarray*}
where $\Delta$ is the category of ordinals and $\ssetpt$ is the category of pointed simplicial sets. Here,  $\fcatk[\Delta \op]$ and 
$\fcatk [\Delta]$ are respectively the categories of simplicial and cosimplicial $\kring$-modules.  In the following, as usual,  $\pi_*$ denotes the homotopy of a simplicial $\kring$-module and $\pi^*$ the cohomotopy of a cosimplicial $\kring$-module.

\begin{defn}
\label{defn:higher_HH}
\nomenclature{$HH_*(X; L)$}{higher Hochschild homology\nomrefpage}
\cite{Phh}
 For $X$ a pointed simplicial set,
\begin{enumerate}
\item 
the higher Hochschild homology of $X$ with coefficients  $L \in \ob \fcatk[\Gamma]$ 
is $HH_*(X; L):= \pi_* (L(X))$;
\item 
the higher Hochschild cohomology of $X$ with coefficients $R \in \ob \fcatk[\Gamma\op]$ 
is $HH^*(X; R):= \pi^* (R(X))$.
\end{enumerate}
\end{defn}

\begin{rem}
The {\em unpointed} version of higher Hochschild homology is given by taking coefficients in 
a $\Gamma$-module of the form $\theta^* L'$, for $L' \in \ob \fcatk[\fin]$. Likewise for higher Hochschild cohomology.
\end{rem}

If $\kring$ is a field, the duality functors $D_\Gamma$ and $D_{\Gamma \op}$  relate higher Hochschild homology and cohomology (for clarity, $(-)^\sharp$ is used to denote vector space duality):

\begin{prop}
\label{prop:duality_HHH}
Suppose that $\kring$ is a field. The following diagrams commute up to natural isomorphism:
\[
 \xymatrix{
  (\ssetpt)\op \times \fcatk[\Gamma] \op
  \ar[r]
  \ar[d]_{1_{(\ssetpt)\op} \times D_\Gamma}
  &
   \fcatk[\Delta \op]\op
   \ar[d]^{D_{\Delta\op}}   
   \\
    (\ssetpt)\op \times \fcatk[\Gamma \op ]
    \ar[r]
    &
    \fcatk[\Delta] 
 }
\]
and 
\[
 \xymatrix{
  \ssetpt \times \fcatk[\Gamma \op]\op
  \ar[r]
  \ar[d]_{1_{\ssetpt} \times D_{\Gamma \op}}
  &
   \fcatk[\Delta]\op
   \ar[d]^{D_{\Delta}}   
   \\
    \ssetpt \times \fcatk[\Gamma ]
    \ar[r]
    &
    \fcatk[\Delta \op]. 
 }
\]
In particular, for $X \in \ob \ssetpt$ and $L \in \ob \fcatk[\Gamma]$, 
$R \in \ob \fcatk[\Gamma \op]$, there are natural
isomorphisms:
\begin{eqnarray*}
HH^*(X; D_\Gamma L) &\cong& HH_* (X ; L)  ^\sharp \\
HH_* (X; D_{\Gamma \op} R ) & \cong & HH^* (X ;R ) 
^\sharp.  
\end{eqnarray*}
Thus, if $L$ takes finite-dimensional values and $X$ values in $\Gamma\subset  \setspt$,  $ HH_* (X; L ) $ is a graded $\kring$-vector space of finite type and 
there is a natural isomorphism:
\[
 HH_* (X; L )  \cong  HH^* (X ; D_\Gamma L ) ^\sharp.
\]
\end{prop}

\begin{proof}
 The first duality statements follow from the fact that post- and pre- compositions commute. 
 The remaining statements follow, since $\kring$ is  a field, 
 with the final statement providing the necessary finiteness hypotheses. 
\end{proof}

\begin{rem}
\label{rem:duality_HHH}
Proposition \ref{prop:duality_HHH} means that, in our applications, it will be
 sufficient to restrict to  considering higher Hochschild homology: results for cohomology are then recovered by duality.
\end{rem}

One of the key properties of higher Hochschild homology is its {\em homotopy invariance}, which is a Corollary of \cite[Theorems 2.4]{Phh} (which supposes that $\kring$ is a field).

\begin{thm}
\label{thm:hhh_homotopy_inv}
Let $\kring$ be a field.
\begin{enumerate}
\item 
For $L \in \ob \fcatk[\Gamma]$, higher Hochschild homology $HH_* (-; L)$ factors naturally across the 
homotopy category $\mathscr{H}_\bullet$ of pointed simplicial sets. 
\item 
For $L' \in \ob \fcatk[\fin]$, higher Hochschild homology $HH_* (-; \theta^* L')$ factors naturally across the homotopy category $\mathscr{H}$ of simplicial sets, forgetting the basepoint. 
\end{enumerate}
\end{thm}

\subsection{The Loday constructions}
\label{subsect:loday_A}

The Loday constructions (or Loday functors) give an important source of $\Gamma$-modules. For this section, $\kring$ is taken to be a general unital, commutative ring; all tensor products are taken over $\kring$ unless indicated otherwise.

\begin{defn}
\cite[Section 1.7]{Phh}
\label{defn:loday_construct}
\nomenclature{$\loday (A,M)$}{Loday construction\nomrefpage}
\nomenclature{$\loday'(A)$}{unpointed Loday construction\nomrefpage}
 For $A$ a commutative, unital $\kring$-algebra and $M$ a right $A$-module, let 
 \begin{enumerate}
  \item 
  $\loday (A,M) \in \ob \fcatk[\Gamma]$ be the left $\Gamma$-module 
   $
   \n_+ \mapsto M \otimes A^{\otimes n};
  $
\item 
$\loday'(A) \in \ob \fcatk[\fin]$ be the left $\fin$-module 
$
 \n \mapsto A^{\otimes n}.
$ 
 \end{enumerate}
\end{defn}
 
\begin{rem}
\label{rem:hhh_Hochschild}
By considering the chain complex associated to the standard simplicial model $\Delta^1 / \partial \Delta^1$ for the circle $S^1$, one sees that  $HH_* (S^1; \loday (A,M))$ is the classical Hochschild homology $HH_* (A; M)$ \cite{L}. 
\end{rem}

The   Loday functors given in Definition \ref{defn:loday_construct} are related by the adjunction of Proposition \ref{prop:adjoints_theta}:

\begin{lem}
\label{lem:theta_T_T'}
\ 
Let $A$ be a commutative, unital $\kring$-algebra. 
\begin{enumerate}
 \item 
  There is a natural isomorphism $\theta^* \loday' (A) \cong \loday (A,A)$ in 
$\fcatk[\Gamma]$.
\item 
If $A$ is augmented, there is a natural isomorphism   $\theta^+ \loday (A, 
\kring) \cong \loday' (A) $ in $\fcatk[\fin]$, 
hence 
$
 \loday (A,A) \cong \theta^* \theta^+ \loday (A, \kring),
$
 where $\kring$ is considered as an $A$-module via the augmentation.
\end{enumerate}
\end{lem}

\begin{proof}
By definition, for $n \in \nat$,  $(\theta^* \loday'(A)) (\mathbf{n}_+) = \loday'(A) (\mathbf{n}_+)$, forgetting that $+$ is the basepoint, 
hence identifies as $A \otimes A^{\otimes n}$, where the first factor corresponds to $+$. This identifies in turn with  $\loday (A, A) (\mathbf{n}_+)$ and it is straightforward to check that the $\Gamma$-module structures correspond.

Similarly, $(\theta^+ \loday (A, \kring))(\mathbf{n}) = \loday (A, \kring ) (\mathbf{n}_+)$, hence identifies with $A^{\otimes n} = \loday' (A)$. Again, one checks that the $\fin$-module structures correspond. The final statement follows by combining the previous two ones. 
\end{proof}

We will exploit the multiplicative structure of the Loday construction given by the following Proposition (compare Remark \ref{rem:loday_mult_revisit} below). This uses the fact that, for any unital, commutative $\kring$-algebra $A$ and $n \in \nat$, the $n$-fold tensor product is a unital commutative $\kring$-algebra for the component-wise multiplication.

\begin{prop}
\label{prop:loday_mult}
For a unital, commutative $\kring$-algebra $A$, the $\Gamma$-module $\loday(A, A)$ takes values in unital  commutative $A$-algebras. Hence, forgetting structure, $\loday (A,A)$ defines a $\Gamma$-module with values in $A$-modules. 

If $A$ is augmented, the $\Gamma$-module $\loday(A, \kring)$ takes values in unital  commutative $\kring$-algebras such that there is an isomorphism of $\Gamma$-modules with values in $\kring$-algebras:
\[
\kring \otimes _A \loday (A,A) \cong \loday (A, \kring).
\]
\end{prop} 

\begin{proof}
By definition, $\loday (A,A) (\mathbf{n}_+)$ is equal to $ A \otimes A^{\otimes n}$, considered as an algebra as above. Since $A$ is commutative, it is clear that the $\Gamma$-module structure morphisms act by morphisms of algebras. Moreover, this defines a $\Gamma$-module in $A$-algebras by using multiplication on the distinguished tensor factor corresponding to the basepoint. 

If $A$ is augmented, the multiplicative structure on $\loday (A, \kring)$ is defined similarly, using $ \loday (A,\kring) (\mathbf{n}_+) = \kring  \otimes A^{\otimes n} \cong A^{\otimes n} $. This clearly identifies with the base change $\kring \otimes _A \loday (A,A)$,  using the $A$-algebra structure on $\loday (A,A)$ introduced in the first part of the Proposition.
\end{proof}

\begin{lem}
 \label{lem:loday_AA_augmented}
 For $A$ an augmented, unital, commutative $\kring$-algebra with augmentation ideal $\overline{A}$, the projection  $\loday (A,A) \twoheadrightarrow \loday (A, \kring) 
\cong \kring \otimes _A \loday (A, A)$ induced by the augmentation fits into a short exact sequence of $\Gamma$-modules with values in $A$-modules:
  \[
   0
   \rightarrow
  \loday (A ,\overline{A})
  \rightarrow 
  \loday (A, A) 
  \rightarrow 
  \loday (A , \kring) 
  \rightarrow 
  0,
  \]
  where $\loday (A , \kring) $ is considered as an $A$-module via the augmentation $A \rightarrow \kring$.
\end{lem}

\begin{proof}
 The short exact sequence is obtained by applying $\loday (A, -) $ to the exact sequence of $A$-modules 
$
 0
 \rightarrow 
 \overline{A}
 \rightarrow 
 A
 \rightarrow 
 \kring 
 \rightarrow 
 0
$, which splits as a sequence of $\kring$-modules. That one obtains an exact sequence of $A$-modules as stated is clear.
 \end{proof}

The following straightforward result is the origin of the exponential property (in the sense of Section \ref{sec:expo})  which arises when considering higher Hochschild homology.

\begin{prop}
\label{prop:expo_loday_AA}
 For $X, Y \in \ob \Gamma$ and $A$ a commutative, unital $\kring$-algebra, there 
is a natural isomorphism 
 \[
  \loday (A, A)(X \vee Y) 
  \cong 
   \loday (A, A)(X) \otimes_{A} 
  \loday (A, A)(Y)
 \]
 that makes $\loday (A,A)$ symmetric monoidal with respect to $\vee$ and 
$\otimes_{A}$.
 
 If $A$ is augmented, then base change using $\kring \otimes _A - $ induces the 
natural isomorphism:
  \[
  \loday (A, \kring)(X \vee Y) 
  \cong 
   \loday (A, \kring )(X) \otimes_\kring
  \loday (A, \kring )(Y)
 \]
 that makes $\loday (A, \kring)$ symmetric monoidal with respect to $\vee$ and 
$\otimes_\kring$.
\end{prop}

\begin{rem}
\label{rem:loday_mult_revisit}
Proposition \ref{prop:expo_loday_AA} gives an alternative viewpoint on the multiplicative structures of Proposition \ref{prop:loday_mult}. For instance, using the above isomorphism, the fold map $X \vee X \rightarrow X$ induces
\[
   \loday (A, A)(X) \otimes_{A} 
  \loday (A, A)(X)
   \cong 
\loday (A, A)(X \vee X) 
\rightarrow
\loday (A,A) (X)
\]
that is natural with respect to $X$. This induces  the product on $\loday (A,A)(X)$. 
\end{rem}

\begin{rem}
\label{rem:generalize_loday}
The Loday constructions can be formed in any symmetric monoidal category. For example, the above  results have analogues in the graded setting (with or without Koszul signs).
\end{rem}

\subsection{The shuffle product in higher Hochschild homology}
\label{subsect:shuffle} 

The classical Hochschild homology $HH_* (A,A)$ of a commutative algebra $A$ is graded commutative for the shuffle product (see \cite[Section 4.2]{L} for example). 
 Underlying this construction is the shuffle map. We will present this using the normalized chain complex $NC$ associated to a simplicial abelian group $C$ together with the shuffle map \cite[Chapter 8]{MR0349792}, following the presentation in  \cite[Section 2]{MR1997322}. For two simplicial abelian groups $C_1, C_2$, the shuffle map induces 
\[
\nabla :
N (C_1)  \otimes N ( C_2) \rightarrow N (C_1 \otimes C_2)
\]
that is lax monoidal and symmetric. (Here the tensor product in the domain is of chain complexes and, in the codomain, of simplicial abelian groups, with the diagonal structure.) The shuffle map is a quasi-isomorphism.

The shuffle product generalizes to higher Hochschild homology; this is well-known, the following statement is included simply to introduce the background. 

\begin{prop}
\label{prop:shuffle_product}
For $A$ a unital, commutative $\kring$-algebra and $X$ a pointed simplicial set, 
\begin{enumerate}
\item 
$HH_* (X; \loday (A,A)) $ is a graded commutative $A$-algebra for the shuffle product; 
\item 
if $A$ is augmented, then $HH_* (X; \loday (A,\kring)) $ is a graded commutative $\kring$-algebra for the shuffle product and the morphism induced by the base change map $\loday(A,A) \twoheadrightarrow \kring \otimes_A \loday(A,A) \cong \loday (A,\kring)$ induces a morphism of graded commutative $A$-algebras;
\[
HH_* (X; \loday (A,A))
\rightarrow 
HH_* (X; \loday (A,\kring)),
\]
where the codomain is considered as an $A$-algebra via the augmentation.
\end{enumerate}
\end{prop}

\begin{proof}
For the first statement, by Proposition \ref{prop:loday_mult}, $\loday (A,A) (X)$ is a simplicial commutative $A$-algebra; passing to homotopy, this gives a graded-commutative $A$-algebra. Explicitly, $HH_* (X; \loday (A,A))$  
 is the homology of the chain complex $N (\loday (A,A)(X))$ and the product is given by the composite:
 \[
 \resizebox{\hsize}{!}{%
 $
 N (\loday (A,A)(X) )\otimes N (\loday (A,A)(X))
 \stackrel{\nabla}{\rightarrow} 
 N \big(\loday (A,A)(X) \otimes \loday (A,A)(X)\big) 
 \rightarrow 
 N (\loday (A,A)(X) ),
$ %
}
 \]
 where the second map is induced by the simplicial algebra structure of $\loday (A,A)(X)$. This makes $N (\loday (A,A)(X))$ into a commutative differential graded algebra.

The proof for the second statement is similar and the naturality statement follows from considering the morphism of simplicial commutative $A$-algebras $\loday (A,A)(X) \rightarrow \loday (A,\kring ) (X)$. 
\end{proof}

\begin{exam}
\label{exam:norm_ch_cx_structure}
To illustrate the above, we consider the  classical Hochschild homology for a commutative $\kring$-algebra $A$, as in \cite[Section 4.2]{L}, but using the general presentation. This corresponds to taking  `higher' Hochschild homology for   $S^1$; we take the simplicial set $\Delta^1 /\partial \Delta^1$ as  our model for the circle. The  simplicial set $\Delta^1 /\partial \Delta^1$ has $|(\Delta^1 /\partial \Delta^1)_n|=n+1$; there is one non-degenerate $0$-simplex and one non-degenerate $1$-simplex; all others are degenerate. 

The underlying simplicial $\kring$-module of $\loday (A,A) (\Delta^1 /\partial \Delta^1)$ identifies in degree $n$ as $A \otimes A^{\otimes n}$; the $i$th face map $d_i$ is the $\kring$-linear map 
\[
A \otimes (A^{\otimes n}) \rightarrow A \otimes (A^{\otimes n-1})
\]
that multiplies the $(i+1)$st and $(i+2)$st tensor factors for $i<n$; for $i=n$, it multiplies the first and last tensor factors, with image the first tensor factor. The associated (unnormalized) chain complex is the cyclic bar complex. 

The associated normalized chain complex $N (\loday (A,A) (\Delta^1 /\partial \Delta^1))$ is $A \otimes \overline{A}^{\otimes n}$ in degree $n$, with differential induced by $\sum_{i=0}^n (-1)^i d_i$. This is a commutative differential graded algebra for the inner shuffle product map.

Now suppose that $A$ is augmented, so that we may consider  $\loday (A,\kring)$. The underlying simplicial $\kring$-module of $\loday (A,\kring) (\Delta^1 /\partial \Delta^1)$ identifies  in degree $n$ as $\kring \otimes A^{\otimes n} \cong A^{\otimes n}$; the $i$th face map $d_i$, for $0\leq i \leq n$, is the $\kring$-linear map 
\[
A^{\otimes n} \rightarrow A^{\otimes n-1}
\]
that multiplies the $i$th and $(i+1)$st tensor factors for $0<i<n$; for $i \in \{0, n\}$ it is induced by the augmentation $A \rightarrow \kring$ applied to the first and last factors respectively. 

The associated normalized chain complex $N (\loday (A,\kring) (\Delta^1 /\partial \Delta^1))$ is $ \overline{A}^{\otimes n}$ in degree $n$, with differential induced by $\sum_{i=1}^{n-1} (-1)^i d_i$. This is a commutative differential graded algebra with respect to the corresponding inner shuffle product map. 

The comparison map is induced by the morphism of commutative differential graded algebras:
\[
A \otimes \overline{A} ^{\otimes *}
\rightarrow 
\overline{A}^{\otimes *}
\]
defined by applying the augmentation to the first tensor factor.
\end{exam}

The shuffle map is key in  showing that $HH_* (-; \loday (A, \kring))$ is monoidal with respect to $\vee$: 

\begin{prop}
\label{prop:shuffle_vee}
For  a unital, augmented, commutative $\kring$-algebra $A$ and pointed simplicial sets $X, Y \in \Delta \op \setspt$, the shuffle map induces:
\[
\resizebox{\hsize}{!}{%
 $
N (\loday (A, \kring ) (X)) \otimes _\kring N(\loday (A, \kring) (Y) ) 
\rightarrow 
N \big(\loday (A, \kring ) (X) \otimes _\kring \loday (A, \kring) (Y)\big)
\cong 
N (\loday (A, \kring ) (X \vee Y) )
$%
}
\]
that is symmetric and compatible with the monoidal structures of $\otimes _\kring$ on chain complexes and $\vee$ on pointed simplicial sets. This is a quasi-isomorphism. 

In particular, if $\kring$ is a field, in homology this induces an isomorphism:
\[
HH_* (X; \loday (A, \kring) ) 
\otimes_\kring 
HH_* (Y; \loday (A, \kring) )
\cong 
HH_*(X \vee Y;  \loday (A, \kring) )
\]
that is symmetric monoidal, using the Koszul-signed symmetry on graded $\kring$-vector spaces.
\end{prop}

\begin{rem}
This result is at the origin of the fact that higher Hochschild homology yields an exponential functor in Theorem \ref{thm:hhh_Loday_AV_k}.
\end{rem}

\subsection{Restricting along the classifying space functor}
\label{subsect:hhh_gr}

The simplicial nerve  gives 
$
  B : \gr \rightarrow \sset. 
  $ 
  By composition with the higher Hochschild homology functor, one obtains  the functor 
  $$\zed^{\star r} \mapsto HH_* (B \zed^{\star r} ; L),$$ for any $L \in \ob \fcatk[\Gamma]$.

\begin{prop}
\label{prop:HH_group_case}
Higher Hochschild homology induces 
 $$HH_* (B(-); -) : \fcatk[\Gamma] \rightarrow  \fcatk[\nat \times \gr].
 $$
  If $\kring$ is a field, this induces  
\[
HH_* (B(-); \theta^* (-)) :\fcatk[\fin]  \rightarrow   \foutk[\nat \times \gr].
\]

If $L \in \ob \fcatk[\Gamma]$  takes finite-dimensional values, $HH_n (B (-) ; L)$ is finite-dimensional, for all $  n \in \nat$. In particular, as functors in $\fcatk[\nat \times \gr \op]$,  higher Hochschild cohomology satisfies:
\[
HH^* (B (-); D_\Gamma L) 
\cong 
D_\gr HH_* (B(-); L).
\]
\end{prop}

\begin{proof}
The first statement is clear. 

For coefficients arising from $\fcatk[\fin]$,  
we require to show that the conjugation action is trivial. This follows from Theorem \ref{thm:hhh_homotopy_inv} (where the hypothesis that $\kring$ is a field is required) since, for $G$ a discrete group  
and $g \in G$ an element, the adjoint action  
 $
  B \ad (g) : BG 
  \rightarrow 
  BG
 $ is (unpointed) homotopic to the identity (see  \cite[Theorem II.1.9]{AM}, for example).

For $\fr \in \ob \gr$, the classifying space is equivalent to, $\bigvee_r S^1$, the wedge of $r$ circles 
in $\mathscr{H}_\bullet$ and the latter is represented by a finite simplicial set. It follows that 
$HH_n (B (-) ; L)$ takes finite-dimensional values.  The final statement then follows from the duality result Proposition \ref{prop:duality_HHH}.
\end{proof}

\section{The Loday construction on square-zero extensions}
\label{sec:loday}

This section investigates the coefficients for higher Hochschild homology that we will exploit,  namely the $\Gamma$-modules given by  the Loday constructions $\loday(A_V; \rat)$ and $\loday (A_V, A_V)$ where $A_V$ is the square-zero extension $A_V := \rat \oplus V$, considered as a functor of  $V\in \ob \fmodq$.

These coefficients have precursors under the Schur correspondence of Section \ref{sec:fb}, exploiting the naturality in $\fmodq$, which gives the dictionary (see Proposition \ref{prop:loday_AV} for the precise statement):
\begin{eqnarray*}
\inj ^\Gamma & \leftrightarrow & \loday (A_V; \rat) \\
\theta^* \inj ^\fin & \leftrightarrow & \loday (A_V, A_V) 
\end{eqnarray*}
for $\inj^\Gamma$ the $\Gamma$-module introduced in Definition \ref{defn:inj_Gamma_inj_Fin} and $\theta^*\inj^\fin$ the $\Gamma$-module derived from the $\fin$-module $\inj^\fin$, also introduced in Definition \ref{defn:inj_Gamma_inj_Fin}.

The significance of the $\Gamma$-module $\inj^\Gamma$ is that it encodes the simple $\Gamma$-modules, by Proposition \ref{prop:properties_injgamma_otimes_Sigma}. The $\Gamma$-module $\theta^* \inj ^\fin$ plays a similar role when working without basepoints. This explains the interest of the Loday constructions $\loday(A_V; \rat)$ and $\loday (A_V, A_V)$ considered as functors of $V \in \ob \fmodq$.

Section \ref{subsect:further_inj} outlines the fundamental structure on the $\Gamma$-modules $\inj^\Gamma$ and $\theta^* \inj^\fin$ that is analogous to that on the Loday functors presented in Section \ref{subsect:loday_A}. For convenience, so as to be able to exploit the Schur correspondence, we work over $\rat$; the theory works in much greater generality.

\subsection{$\Gamma$ and $\fin$-modules from injections}
\label{subsect:inj}

We introduce the fundamental objects of this section:

 \begin{defn}
 \label{defn:inj_Gamma_inj_Fin}
 \nomenclature{$\inj^\Gamma$}{coefficient $\Gamma$-module from injections of finite pointed sets\nomrefpage}
 \nomenclature{$\inj^\fin$}{coefficient $\Gamma$-module from injections of finite sets\nomrefpage}
 \ 
 \begin{enumerate}
 \item 
 Let $\inj^\Gamma \in \ob \fcatQ[\fb\op \times \Gamma]$  be the functor 
$$
\inj ^\Gamma  : \ 
(\mathbf{m}, X)\mapsto \rat[\mathrm{inj}_{\Gamma} (\mathbf{m}_+, X)],
$$ 
where the right hand side is the quotient of $ \rat [\hom_{\Gamma} (\mathbf{m}_+,X ) 
]$ by the subspace generated by the non-injective maps.
\item 
Let $\inj^\fin  \in \ob \fcatQ[\fb\op \times \fin]$ be the functor 
$$
\inj ^\fin  : \ 
(\mathbf{m}, U) 
\mapsto  
\rat [\mathrm{inj}_\fin (\mathbf{m}, U)],
$$
where the right hand side is the quotient of $ \rat [\hom_{\fin} (\mathbf{m},U ) 
]$ by the subspace generated by the non-injective maps.
\end{enumerate}
 \end{defn}

These are related by a counterpart of the isomorphism  $\theta^+ \loday (A, 
\kring) \cong \loday' (A) $ given in  Lemma \ref{lem:theta_T_T'}:

\begin{lem}
\label{lem:inj^Fin}
There is an isomorphism $\inj^\fin \cong \theta^+ \inj^\Gamma$  in $\ob \fcatQ[\fb \op \times \fin]$.
\end{lem} 
 
\begin{proof}
By definition, $\theta^+ \inj^\Gamma$ sends  $(\mathbf{m}, U)$, for a finite set $U$, to $\inj^\Gamma (\mathbf{m}, U_+) = \rat[\mathrm{inj}_{\Gamma} (\mathbf{m}_+, U_+)]$. The latter is isomorphic to $\rat [\mathrm{inj}_\fin (\mathbf{m}, U)]$, using the injectivity criterion and the fact that the basepoint is preserved. Moreover, this identification is natural with respect to $\fb\op \times \fin$. 
\end{proof}
 
 The $\Gamma$-module $\inj^\Gamma$  encodes the simple $\Gamma$-modules, by the following result. To state this, we use the tensor product $\otimes_\fb$ (see Section \ref{subsect:tensor_over_C}), which  yields the functor 
$
\inj^\Gamma \otimes_\fb - : \fcatQ[\fb] \rightarrow \fcatQ[\Gamma]$.

\begin{prop}
\label{prop:properties_injgamma_otimes_Sigma}
\ 
\begin{enumerate}
\item 
The functor $\inj^\Gamma \otimes_\fb - : \fcatQ[\fb] \rightarrow \fcatQ[\Gamma]$ is exact and the image of $\rat [\sym_m]$, considered as a $\fb$-module supported on $\mathbf{m}$ for $m\in \nat$, is the $\Gamma$-module $\inj^\Gamma (\mathbf{m}, -) = \rat [\inj_\Gamma (\mathbf{m}_+, -) ]$.
\item 
The functor $\inj^\Gamma \otimes_\fb - $ induces a bijection between the isomorphism classes of simple objects of the categories $\fcatQ[\fb]$ and $\fcatQ[\Gamma]$. Explicitly, 
 $$
\{ \inj^\Gamma (\mathbf{m}, - ) \otimes_{\sym_m}S_\lambda \ | \ m \in \nat, \lambda \vdash m\}
$$
 is a set of isomorphism class representatives of the  simple $\Gamma$-modules, where $S_\lambda$ is the simple $\sym_m$-representation indexed by $\lambda$. 
\item 
For $m\in\nat$, $\inj^\Gamma (\mathbf{m}, - )$ has direct sum decomposition into simple objects in  $\fcatQ[\Gamma]$:
\[
\inj^\Gamma (\mathbf{m}, - ) 
\cong 
\bigoplus_{\lambda \vdash m} 
\Big(\inj^\Gamma (\mathbf{m}, - ) \otimes_{\sym_m}S_\lambda \Big) ^{\oplus \dim S_\lambda}.
\]
\end{enumerate}
\end{prop}

\begin{proof}
The first statement is immediate in characteristic zero.

The second statement is best understood by using  Pirashvili's Dold-Kan theorem \cite{Pdk}: this gives an equivalence of categories 
\[
\fcatQ[\fs] \stackrel{\simeq}{\rightarrow} \fcatQ[\Gamma],
\]
where $\fs$ is the category of finite sets and surjections. There is an inclusion $\fb \subset \fs$ and extension by zero gives the exact functor $\fcatQ[\fb]\rightarrow \fcatQ[\fs]$; for example, any $\sym_m$-module can be considered as an $\fs$-module supported on $\mathbf{m}$. In particular, this functor induces a bijection between the respective sets of isomorphism classes of simple objects. 

Using the explicit form of Pirashvili's equivalence, the composite $\fcatQ[\fb]\rightarrow \fcatQ[\fs]\stackrel{\simeq}{\rightarrow} \fcatQ[\Gamma]$ can be seen to identify (up to natural isomorphism) with the functor $\inj^\Gamma \otimes_\fb - $.  The second statement thus follows
 from the relation between simples given by $\fcatQ[\fb]\rightarrow \fcatQ[\fs]$.

The third statement follows using the isomorphism of $\sym_m$-modules $\rat [\sym_m] \cong \bigoplus _{\lambda \vdash m} S_\lambda ^{\oplus \dim S_\lambda}$.
\end{proof}

\subsection{The associated Schur functors}

In this section, we apply the  Schur functor construction to  $\inj^\Gamma$ and $\theta^*\inj^\fin$ (adjusting variance using $\fb \cong \fb\op$, where appropriate). Recall from Section \ref{subsect:schur_correspond} that the Schur functor construction is 
$$\talg \otimes_\fb - : \f(\fb; \rat) \rightarrow \f(\fmodq; \rat),$$
the main properties of which   are given in Proposition \ref{prop:schur_functor}.

\begin{nota}
 \label{nota:AV}
For $V \in \ob \fmodq$,  denote by $A_V:= \rat \oplus 
V$,  the square-zero extension, augmented $\rat$-algebra, with 
augmentation ideal $V$. 
\end{nota}

The association $V \mapsto A_V$ defines a functor from $\fmodq$ to augmented $\rat$-algebras, hence the functors $V \mapsto \loday (A_V, \rat)$ and $V \mapsto \loday (A_V, A_V)$ define objects in $\fcatQ[\fmodq \times \Gamma]$, denoted $\loday (A_{(-)}, \rat)$ and $\loday (A_{(-)}, A_{(-)})$ respectively. The following statement gives the dictionary of the beginning of this section:

\begin{prop}
 \label{prop:loday_AV}
There are isomorphisms in $\fcatQ[\fmodq \times \Gamma]$:
\begin{eqnarray*}
 \loday (A_{(-)}, \rat)& \cong & \talg \otimes_{\fb} \inj ^\Gamma
  \\
  \loday (A_{(-)}, A_{(-)}) &\cong& \talg  \otimes_{\fb} \theta^* \theta^+ \inj ^\Gamma \cong \talg  \otimes_{\fb}  \theta^* \inj^\fin. 
\end{eqnarray*}
\end{prop}

\begin{proof}
The first isomorphism is most easily proved by using Pirashvili's Dold-Kan theorem, which is applied to the case of a Loday functor in \cite[Section 1.10]{Phh}. The case of a square-zero extension is especially simple; the associated $\Gamma$-module identifies as stated. 

The second isomorphism can be established similarly. It also follows from  the first  by applying Lemma \ref{lem:theta_T_T'}, which provides the natural isomorphism  $\loday (A_V, A_V) \cong \theta^* \theta^+ 
\loday (A_V, \kring)$.
\end{proof}

We note the following consequence of Lemma \ref{lem:loday_AA_augmented}:

\begin{lem}
\label{lem:ses_AV_loday}
There is a short exact sequence of $\Gamma$-modules:
  \begin{eqnarray*}
   \label{eqn:loday_ses}
  0
   \rightarrow
  \loday (A_V ,\rat) \otimes V 
  \rightarrow 
  \loday (A_V, A_V) 
  \rightarrow 
  \loday (A_V , \rat) 
  \rightarrow 
  0,
  \end{eqnarray*} 
that is natural with respect to $V \in \ob \fmodq$.

This is naturally a short exact sequence of $A_V$-modules, where $\loday (A_V, \rat)$ is considered as an $A_V$-module by restriction along the augmentation $A_V \twoheadrightarrow \rat$.
\end{lem}

\begin{proof}
This follows from Lemma \ref{lem:loday_AA_augmented} by taking $A=A_V$ and using the identification $\loday (A_V, \overline{A_V}) \cong \loday (A_V ,\rat) \otimes V$ of $\Gamma$-modules, naturally in $V$. 
\end{proof}

\begin{rem}
\label{rem:loday_splitting}
The Schur functor constructions appearing in Proposition \ref{prop:loday_AV} encode the splittings into homogeneous components (as functors of $V$):
\begin{eqnarray*}
 \loday (A_V, \rat) 
 &\cong &
 \bigoplus_{d \geq 0} 
 \loday (A_V, \rat)^{(d)}
\\
 \loday (A_V, A_V) 
 &\cong & 
 \bigoplus_{d \geq 0} 
 \loday (A_V, A_V)^{(d)}
 \\
 \loday (A_V, \rat)  \otimes V
 &\cong &
 \bigoplus_{d \geq 1} 
 (\loday (A_V, \rat) \otimes V)^{(d)}
\cong 
 \big( \bigoplus_{d \geq 1} 
 \loday (A_V, \rat)^{(d-1)}\big) \otimes V.
 \end{eqnarray*}
 Explicitly,  for $d \in \nat$: 
\begin{eqnarray*}
\loday (A_{(-)}, \rat)^{(d)}& \cong &\talg^d (-) \otimes_{\sym_d} \inj ^\Gamma (\mathbf{d}_+, -)
\\
\loday (A_{(-)}, A_{(-)})^{(d)} &\cong &   \talg^d (-) \otimes_{\sym_d} \theta^* \theta^+ \inj ^\Gamma (\mathbf{d}_+, -)
\cong \talg^d (-) \otimes_{\sym_d} \theta^* \inj ^\fin (\mathbf{d}, -).
\end{eqnarray*}
\end{rem}

\subsection{Further structure}
\label{subsect:further_inj}
In this section, we exhibit the structure on $\inj^\Gamma$  and $\theta^*\inj ^\fin$ that corresponds to that on $\loday (A_{(-)}, \rat)$ and $\loday (A_{(-)}, A_{(-)})$ presented in Section \ref{subsect:loday_A}. The fact that these correspond via the Schur correspondence can be deduced using the properties given in Proposition \ref{prop:schur_functor}.

We consider $\inj^\Gamma$ (respectively $\inj ^\fin$) as taking values in $\fcatQ[\fb]$ and we exploit the convolution product $\tenfb$. Throughout, we work with the symmetric monoidal structure $(\fcatQ[\fb], \tenfb, \unit , \tau)$ on $\fb$-modules.

The following result is the counterpart of Proposition \ref{prop:expo_loday_AA}:

\begin{prop}
\label{prop:wedge_tenfb}
\
\begin{enumerate}
\item
For $X, Y \in \ob \Gamma$, there is a natural isomorphism in $\fcatQ[\fb]$:
\[
\inj^\Gamma ( X \vee Y)
 \cong 
\inj^{\Gamma} (X)
\tenfb
\inj^{\Gamma} ( Y)  
 \]
making $\inj^\Gamma$ symmetric monoidal with respect to $(\Gamma, 
\vee)$ and $(\fcatQ[\fb], \tenfb)$.
\item 
For $U, V \in \ob \fin$, there is a natural isomorphism in $\fcatQ[\fb]$:
\[
\inj^\fin( U  \amalg V)
 \cong 
\inj^{\fin} (U)
\tenfb
\inj^{\fin} (V)  
 \]
so that  $\inj^\fin$ is symmetric monoidal with respect to $(\fin, 
\amalg)$.
\end{enumerate}
\end{prop}

\begin{proof}
The proof of the first statement is analogous to that of 
Proposition \ref{prop:proj_smod_tenfb}. Namely, an injection $f : \mathbf{n}_+ \hookrightarrow X \vee Y$ of pointed sets 
is equivalent to the pair of restrictions $f' : \mathbf{n'}_+ \hookrightarrow X$ and $f'' : \mathbf{n''}_+ 
\hookrightarrow Y$, where $\mathbf{n'}_+ = f^{-1} (X)$ and $\mathbf{n''}_+ = f^{-1} (Y)$, so that 
 $\mathbf{n} = \mathbf{n'} \amalg 
\mathbf{n''}$. This is encoded in $\tenfb$ and is compatible with the respective symmetric monoidal structures. 

The second statement follows from the first by applying the functor $\theta^+$, using the identification of Lemma \ref{lem:inj^Fin}. 
 \end{proof}

As in Remark \ref{rem:loday_mult_revisit}, this induces multiplicative structures, leading to the following counterpart of Proposition \ref{prop:loday_mult}:

\begin{cor}
\label{cor:cmon_inj_Gamma_Fin}
\ 
\begin{enumerate}
\item 
The functor $\inj ^\Gamma$ takes values in unital, commutative monoids in $\fcatQ[\fb]$.
\item 
The functor $\inj^\fin$ takes values in unital, commutative  monoids in $\fcatQ[\fb]$.
\end{enumerate}

Moreover,  $\theta^* \inj^\fin \cong \theta^* \theta^+  \inj ^\Gamma$ takes values in unital, commutative monoids  and 
the adjunction counit $\theta^* \theta^+  \inj ^\Gamma \rightarrow \inj^\Gamma$ is a morphism of commutative monoids. 
 Evaluated on $X\in\ob \Gamma$ and $\mathbf{m} \in \ob \fb$, this map identifies as the surjection 
\[
\rat [\mathrm{inj}_\fin (\mathbf{m}, X)]
\rightarrow 
\rat [\mathrm{inj}_\Gamma (\mathbf{m}_+, X)]
\]
induced by extending a map by  sending $+$ to the basepoint of $X$.
\end{cor}

We now present the analogue of the natural $A$-module structure on $\loday (A,A)$ 
 given by restricting the above multiplicative structure. 
 
In the following statement, we consider $\unit \oplus P^\fb_{\mathbf{1}}$ as a unital commutative monoid in $\fcatQ[\fb]$ with respect to the convolution product, using the square-zero extension structure. (In fact, this is the {\em unique} unital commutative monoid structure on $\unit \oplus P^\fb_{\mathbf{1}}$.)

\begin{rem}
The Schur functor associated to $\unit \oplus P^\fb_{\mathbf{1}}$  is $V \mapsto \rat \oplus V$, considered as a square zero extension.
\end{rem}

\begin{lem}
\label{lem:square_zero_module}
\ 
\begin{enumerate}
\item
The unital commutative monoid structure of  $\inj^\fin (-, \mathbf{1})$ given by Corollary \ref{cor:cmon_inj_Gamma_Fin}, where $\mathbf{1}$ is considered as an object of $\fin$, identifies with the square zero extension $\unit \oplus P^\fb_{\mathbf{1}}$ in $\fcatQ[\fb]$.
\item 
There is a natural inclusion on commutative monoids $\unit \oplus P^\fb_{\mathbf{1}} \hookrightarrow \theta^* \theta^+  \inj ^\Gamma $ in $\f (\Gamma; \f(\fb; \rat))$,  where $\unit \oplus P^\fb_{\mathbf{1}}$ is considered as constant with respect to $\Gamma$. This is induced by the inclusion of the basepoint.
\end{enumerate}
In particular,  $ \theta^* \theta^+  \inj ^\Gamma$ (which is isomorphic to $\theta^* \inj ^\fin$) takes values in  $(\unit \oplus P^\fb_{\mathbf{1}})$-modules in $\f(\fb; \rat)$.
\end{lem}

The $(\unit \oplus P^\fb_{\mathbf{1}})$-module structure on  $\theta^* \theta^+  \inj ^\Gamma$ restricts to 
$(\theta^* \theta^+ \inj^\Gamma) \tenfb P^\fb _\mathbf{1}  \rightarrow \theta^* \theta^+ \inj^\Gamma$. One checks that this factors across the surjection $\theta^* \theta^+ \inj^\Gamma \tenfb P^\fb _\mathbf{1} \twoheadrightarrow \inj^\Gamma  \tenfb P^\fb _\mathbf{1} $ (induced by the adjunction unit) to give 
\begin{eqnarray}
\label{eqn:kernel_adj_counit}
\inj^\Gamma  \tenfb P^\fb _\mathbf{1}
\rightarrow \theta^* \theta^+ \inj^\Gamma
\cong \theta^* \inj^\fin.
\end{eqnarray}

For $X \in \ob \Gamma$ and $\mathbf{m}\in \ob \fb$, by the definition of $\tenfb$,  this identifies with 
\[
\bigoplus_{\substack{S \subset \mathbf{m} \\|S|= m-1}} \rat [\mathrm{inj}_\Gamma (S_+, X) ] 
\rightarrow 
\rat[\mathrm{inj}_\fin (\mathbf{m}, X)]
\]
that sends a generator $[i :S_+ \hookrightarrow X]$ to the corresponding generator $[i : \mathbf{m} \hookrightarrow X]$,  using the obvious identification $S_+ \cong \mathbf{m}$.

The following statement gives the analogue of the short exact sequence of Lemma \ref{lem:ses_AV_loday}.

\begin{prop}
\label{prop:inj_Gamma_X+_decomp}
The  adjunction counit $\theta^* \inj^\fin \cong \theta^* \theta^+ \inj^\Gamma
\rightarrow \inj^\Gamma $ and the map (\ref{eqn:kernel_adj_counit}) fit into 
 a short exact sequence in $(\unit \oplus P^\fb_{\mathbf{1}})$-modules in $\f(\Gamma; \f(\fb; \rat))$:
\begin{eqnarray}
\label{eqn:ses_inj}
0
\rightarrow 
\inj^\Gamma \tenfb P^\fb _\mathbf{1} 
\rightarrow 
\theta^* \inj^\fin 
\rightarrow 
\inj^\Gamma 
\rightarrow 
0,
\end{eqnarray}
where $\inj^\Gamma$ is considered as a module via restriction along the augmentation $\unit \oplus P^\fb_{\mathbf{1}} \twoheadrightarrow \unit$.
\end{prop}

\begin{proof}
The adjunction counit is clearly surjective and a morphism of $(\unit \oplus P^\fb_{\mathbf{1}})$-modules.

The injectivity of (\ref{eqn:kernel_adj_counit}) is clear from the explicit description given above which also shows that it maps to the kernel of the adjunction counit.  Hence, it suffices to show exactness in the middle.

It suffices to show that, when evaluated on $X\in\ob \Gamma$ and $\mathbf{m} \in \ob \fb$, the kernel of the adjunction counit is the image of the map (\ref{eqn:kernel_adj_counit}). The adjunction counit is described explicitly in Corollary \ref{cor:cmon_inj_Gamma_Fin};  the kernel is  the subspace generated by maps $\mathbf{m} \hookrightarrow X$ that contain the basepoint of $X$ in their image; this is the image of (\ref{eqn:kernel_adj_counit}), by the explicit description of this map given above.
\end{proof}

\section{Higher Hochschild homology with coefficients $\inj^\Gamma$}
\label{sec:hhh_triv}

In this section we treat higher Hochschild homology with coefficients in $\loday(A_V; \rat)$, where $A_V$ is the square-zero extension $\rat \oplus V$,  restricted to a functor on $\gr$:
$$
\zed^{\star r} \mapsto 
HH_* (B \zed^{\star r} ; \loday ( A_V, \rat))
.
$$
This is considered as a functor of $V \in \ob \fmodq$ and takes values in commutative graded $\rat$-algebras, using the shuffle product, by Proposition \ref{prop:shuffle_product}. 

The main result of this section is Theorem \ref{thm:hhh_Loday_AV_k}, which establishes that this is an exponential functor; explicitly:
\[
HH_* (B (-); \loday (A_{V} , \rat) ) 
\cong 
\Psi (T_\mathrm{coalg}(sV)).
\]

Using the Schur correspondence, since we are working over $\rat$,  it is equivalent to study the functor 
$$
\zed^{\star r} \mapsto 
HH_* (B \zed^{\star r} ; \inj^\Gamma),
$$
where the coefficients $\inj^\Gamma \in \fcatQ[\fb\op \times \Gamma]$ are as in Section \ref{sec:loday}. By construction,  this functor takes  values in graded $\fb$-modules.

The main result of the section is equivalent to Theorem \ref{thm:hhh_inj_Gamma}, which gives the  isomorphism:
$$
HH_* (B (-) ; \inj ^\Gamma)\cong
\Psi ({\pcoalg} ^\dagger).
$$ 

\subsection{Coefficients $\loday (A_V, \rat)$}

Recall that  $T_{\mathrm{coalg}}(V)$ denotes the tensor coalgebra Hopf algebra (with deconcatenation coproduct and shuffle product). Koszul signs are introduced using the $(-)^\dagger$ construction (see Notation \ref{nota:dagger_schur}). For the purposes of this section, it is convenient to treat $(-)^\dagger$  via 
\[
V \mapsto T_{\mathrm{coalg}}(sV),
\]
mapping to graded $\rat$-vector spaces, using $sV$ placed in homological degree one (cf. Remark \ref{rem:heuristic_dagger}).

This is a cocommutative Hopf algebra in graded vector spaces, with deconcatenation coproduct and shuffle product with the appropriate Koszul signs. 

The exponential functor $\Psi$ is applied in this context by working in graded $\rat$-vector spaces with symmetry involving Koszul signs.

\begin{thm}
\label{thm:hhh_Loday_AV_k} 
There is a natural isomorphism of functors with values in graded-commutative $\rat$-algebras:
\[
HH_* (B (-); \loday (A_{V} , \rat) ) 
\cong 
\Psi (T_\mathrm{coalg}(sV)),
\]
where $HH_* (B (-); \loday (A_{V} , \rat) ) $ is equipped with the shuffle product. This is natural with respect to $V \in \ob \fmodq$.
\end{thm}

\begin{proof}
Proposition \ref{prop:shuffle_vee} implies that $HH_* (B (-); \loday (A_{V} , \rat) ) $ is an exponential functor on $\gr$ with values in graded $\rat$-vector spaces. Hence, by Theorem \ref{thm:expo_Hopf_general}, it suffices to identify the Hopf algebra $HH_* (B \zed; \loday (A_{V} , \rat) ) = HH_* (S^1; \loday (A_{V} , \rat) )$.

The identification of the underlying  augmented graded-commutative algebra is standard (cf.  \cite[Section 4]{L}).
 This can be seen by using the structure reviewed in Example \ref{exam:norm_ch_cx_structure}. Namely, the normalized chain complex $N (\loday (A, \rat) (\Delta^1/ \partial \Delta^1))$ is isomorphic to $V^{\otimes n}$ in degree $n$ and the product is the shuffle product (with Koszul signs). This is the underlying  augmented graded  commutative algebra structure of $T_{\mathrm{coalg}}(sV)$. 

For the coproduct one can proceed as in \cite[Section 3]{AR} (this is written for  higher Hochschild homology with coefficients in $\loday (A, A)$ for $A$ a commutative ring -- the arguments transpose {\em mutatis mutandis} to the current situation). This identifies the coproduct as the deconcatenation coproduct, as required.
\end{proof}

\begin{rem}
\label{rem:purity_schur}
The functor $V \mapsto HH_* (B (-); \loday (A_{V} , \rat) ) $ can be considered as  a functor on $\gr$ with values in $\f(\nat \times \fmodq; \rat)$. This  takes values in {\em pure} objects of $\f(\nat \times \fmodq; \rat)$, defined as follows: 

An object $F$ of $\f(\nat \times \fmodq; \rat)$ is {\em pure} if, for each $n \in \nat$,  the $n$th graded component (corresponding to functoriality with respect to $\nat$)  is  a homogeneous polynomial functor of degree $n$ (as a functor from $\fmodq$ to $\rat$-vector spaces), in the terminology of Definition \ref{defn:homog_poly_functor}.

The significance of {\em purity} is that the homological grading is determined by the polynomial degree of the underlying functor. In the next section  this will be carried over across the Schur correspondence (see Definition \ref{defn:pure} and Proposition \ref{prop:pure_symmetric_monoidal}), where the notion is essential so as to recover the homological grading.
\end{rem}

\subsection{Coefficients $\inj^\Gamma$}
\label{subsect:coeff_inj_gamma}

The coefficients $\inj ^\Gamma$ were introduced in Definition \ref{defn:inj_Gamma_inj_Fin}; adjusting the variance using $\fb\cong \fb\op$, $\inj^\Gamma$ can be considered as a $\Gamma$-module with values in $\f(\fb; \rat)$. By Proposition  \ref{prop:loday_AV}, passing to the associated Schur functor gives:
\[
 \loday (A_{(-)}, \rat) \cong \talg \otimes_{\fb} \inj ^\Gamma. 
\]

We deduce:

\begin{prop}
\label{prop:HH_hom_int_Gamma}
For $X \in \ob \Delta \op \setspt$, there is a natural isomorphism 
\[
HH_* (X; \loday (A_{(-)}, \rat)) \cong 
\talg \otimes _{\fb} HH_* (X; \inj^\Gamma).
\]
\end{prop}

\begin{proof}
This follows since the Schur functor construction $\talg \otimes_\fb -$ is exact, hence commutes with the passage to homology.
\end{proof}

By the Schur correspondence this implies that, to understand $HH_* (X; \loday (A_{V}, \rat))$ naturally with respect to $V$, it is equivalent to understanding $HH_* (X; \inj^\Gamma)$  taking  values in $\fb$-modules. This leads  to the consideration of homologically graded $\fb$-modules; these are the object of $\f(\nat \times \fb; \rat)$.

The following is the analogue of purity in the Schur functor context (cf. Remark \ref{rem:purity_schur}):

\begin{defn}
\label{defn:pure}
A graded $\fb$-module $F \in \ob \fcatQ[\nat \times \fb]$ is pure if $F (d, \mathbf{n})=0$ whenever $d\neq n$. 
Write $\f^{\mathrm{pure}}(\nat \times \fb; \rat)$ for the full subcategory of pure functors.
\end{defn}

The convolution product $\tenfb$ on $\fb$-modules passes to $\nat$-graded $\fb$-modules, using the tensor product of graded vector spaces. The symmetry $\tau$ has a graded analogue $\tau_\nat$, defined using the symmetry of graded $\rat$-vector spaces with Koszul signs, thus defining a symmetric monoidal structure $(\f(\nat \times \fb; \rat), \tenfb, \unit, \tau_\nat)$. 

\begin{rem}
The notation $\tau_\nat$ was chosen so as to indicate that the signs arise from the underlying symmetric monoidal structure of graded vector spaces (with Koszul signs).  This choice affects the symmetric monoidal behaviour exhibited in Proposition \ref{prop:pure_symmetric_monoidal}.
\end{rem}

It is clear that, if $F_1, F_2 \in \ob \f^{\mathrm{pure}}(\nat \times \fb; \rat)$, then so is $F_1 \tenfb F_2$, hence $\tenfb$ restricts to a symmetric monoidal structure on $\f^{\mathrm{pure}}(\nat \times \fb; \rat)$. In fact, one has the following:

\begin{prop}
\label{prop:pure_symmetric_monoidal}
The functor $\fcatQ[\fb] \rightarrow \f^{\mathrm{pure}}(\nat \times \fb; \rat)$ that sends $G \in \ob \smod$ to the pure graded functor with $G (n, \mathbf{n}) = G(\mathbf{n})$  for $n \in \nat$, is an isomorphism of symmetric monoidal categories with respect to  $(\fcatQ[\fb], \tenfb, \unit, \sigma)$ and  $(\f^{\mathrm{pure}}(\nat \times \fb; \rat), \tenfb ,\unit, \tau_\nat)$.

Hence, the composite:
\[
\fcatQ[\fb]
\stackrel{(-)^\dagger}{\rightarrow}
\fcatQ[\fb] \rightarrow \f^{\mathrm{pure}}(\nat \times \fb; \rat)
\]
is an isomorphism of symmetric monoidal categories, where the domain is equipped with the symmetric monoidal structure $(\fcatQ[\fb], \tenfb, \unit, \tau)$.
\end{prop}

Using this, together with the Schur functor correspondence, Theorem \ref{thm:hhh_Loday_AV_k} implies:

\begin{thm}
\label{thm:hhh_inj_Gamma}
There is a natural isomorphism of $\nat$-graded functors with values in $\fb$-modules:
\[
HH_* (B (-) ; \inj ^\Gamma)
\cong 
\Psi ({\pcoalg} ^\dagger),
\]
where the right hand side is considered as a pure functor via Proposition \ref{prop:pure_symmetric_monoidal}.
\end{thm}

\begin{proof}
This follows by applying the Schur functor construction, which is symmetric monoidal, by Proposition \ref{prop:schur_functor}. This property allows the naturality result, Proposition \ref{prop:naturality_Phi_Psi}, for the exponential functor construction $\Psi$ to be applied.

Using the identifications of Section \ref{sec:pcoalg}, in the ungraded setting, it follows that the Schur functor associated to $\Psi \pcoalg$ is isomorphic to $\Psi T_\mathrm{coalg}(-)$, since the Schur functor associated to $\pcoalg$ identifies as $T_\mathrm{coalg}(-)$ as a commutative Hopf algebra.

The passage to the graded setting is encoded in the functor $(-)^\dagger$ and its compatibility with the Schur functor construction.  
\end{proof}

\begin{rem}
Our viewpoint is that $  HH_* (B (-) ; \inj ^\Gamma)$ is the most natural object to consider in this context, rather than the associated Schur functor $V  \mapsto HH_* (B (-); \loday (A_{V} , \rat) )$.  
\end{rem}

\section{Higher Hochschild homology with coefficients $\theta^* \inj^\fin$}
\label{sec:hhh_calc}

In this section, we free up the base point, by considering higher Hochschild homology 
$$HH_* (B (-);\loday (A_V, A_V))$$ 
with the coefficients in $\loday (A_V, A_V)$ replacing $\loday (A_V , \rat)$ (as considered in Section \ref{sec:hhh_triv}). Once again, it is essential that these structures be  considered functorially  with respect to $V \in \ob \fmodq$.  By Proposition \ref{prop:HH_group_case}, the functor $HH_* (B (-);\loday (A_V, A_V))$ on $\gr$  takes values in $\foutQ[\nat \times \gr]$, where $\nat$ corresponds to the homological  grading.  

The functor $HH_* (B (-);\loday (A_V, A_V))$ is described in Theorem \ref{thm:HH_lodayAVAV_car0}, which also treats the multiplicative structure. 
 By the Schur correspondence, this is equivalent to studying 
$$HH_* (B (-); \theta^* \inj^\fin)$$
 in place of $HH_* (B (-);  \inj^\Gamma)$ (as considered in Section \ref{sec:hhh_triv}).
The description of $HH_* (B (-); \theta^* \inj^\fin)$ is then deduced in Theorem \ref{thm:HH_theta_inj_Fin}.

In Section \ref{subsect:isotypical_components}, we consider the decomposition of $HH_* (B (-); \theta^* \inj^\fin)$ into isotypical components, using Theorem \ref{thm:HH_theta_inj_Fin}.
The main result, Theorem \ref{thm:beta_hhh}, establishes the close relationship between higher Hochschild homology with coefficients $\theta^* \inj^\fin$ and the structure of $\fpiQ {\gr}$, notably of the injective cogenerators of this category. Explicitly,  
for $n \in \nat^*$ and $\lambda \vdash n$, there  are isomorphisms in $\foutQ[\gr]$:
\[
HH_* \big(B (-);\theta^* \inj^\fin) \otimes_\fb S_\lambda
\cong 
\left\{
\begin{array}{ll}
\omega \beta_n S_{\lambda^\dagger} & *=n \\
\coker (\coadbar_{\lambda^\dagger}) & *= n-1\\
0  & \mbox{otherwise.}
\end{array}
\right.
\]
As explained in Remark \ref{rem:TW_question}, this gives a conceptual, theoretical answer to a question of Turchin and Willwacher, although it does not yield the full calculation that they sought.

Finally, in Section \ref{subsect:DT}, we show how understanding the structure of the polynomial functors on $\gr$ that arise can be used to deduce information on the behaviour of higher Hochschild homology. Namely, using Proposition \ref{prop:abelianization_dual_numbers}, we give a new proof of the result of Dundas and Tenti \cite{DT} showing that higher Hochschild homology is not a stable invariant (see Remark \ref{rem:Dundas_Tenti}).

\subsection{Coefficients  $\loday (A_V, A_V)$}

The analysis of higher Hochschild homology with coefficients  $\loday (A_V, A_V)$ is based on the short exact sequence of $\Gamma$-modules given by Lemma \ref{lem:ses_AV_loday}:
\begin{eqnarray}
\label{eqn:ses_loday}
0
\rightarrow 
V \otimes \loday (A_V, \rat) 
\rightarrow 
\loday (A_V, A_V) 
\rightarrow 
\loday (A_V, \rat) 
\rightarrow 
0
\end{eqnarray}
that is natural with respect to $V \in \ob \fmodq$. 

We also consider the multiplicative structures; Proposition \ref{prop:loday_mult} gives that $\loday(A_V, A_V)$ and $\loday (A_V, \rat)$ take values in unital commutative $\rat$-algebras and the projection $\loday (A_V, A_V) 
\rightarrow 
\loday (A_V, \rat) $ respects these structures. 

Evaluated on a fixed finite pointed set, the multiplicative structure of $\loday (A_V, A_V)$ is given in terms of that of $\loday(A_V, \rat)$ by:

\begin{lem}
\label{lem:loday_square_zero}
For $Y \in \ob \Gamma$,  (\ref{eqn:ses_loday}) 
 exhibits $\loday (A_V,A_V) (Y)$ as the square zero extension of the commutative algebra $\loday (A_V, \rat) (Y)$ by the module $V \otimes \loday (A_V, \rat) (Y)$.  This is natural with respect to $V \in \ob \fmodq$.
\end{lem}

\begin{proof}
This is immediate from the definition of the respective algebra structures, using that $A_V = \rat \oplus V$ is a square zero extension of $\rat$.
\end{proof}

\begin{rem}
\label{rem:not_split}
Lemma \ref{lem:loday_square_zero} provides the  splitting 
\[
\loday (A_V, A_V) (Y) \cong \loday (A_V, \rat) (Y) \oplus \big( V \otimes \loday (A_V, \rat) (Y) \big).
\]
This splitting is natural with  respect to $V \in \ob \fmodq$ but  is {\em not} natural with respect to $Y \in \ob \Gamma$.
\end{rem}

From this we deduce the following general result:

\begin{prop}
\label{prop:HH_loday_AV_AV}
For $X \in \ob \Delta\op \setspt$, the shuffle product gives natural graded-commutative $\rat$-algebra structures on $HH_*(X; \loday (A_V, A_V))$ and $HH_*(X; \loday(A_V, \rat))$.

The short exact sequence (\ref{eqn:ses_loday}) induces a long exact sequence 
\[
 \resizebox{\hsize}{!}{%
 $
\ldots 
\rightarrow 
V \otimes HH_*(X; \loday(A_V, \rat))
\rightarrow
HH_*(X; \loday (A_V, A_V))
\rightarrow 
HH_*(X; \loday(A_V, \rat))
\rightarrow 
\ldots
$%
}
\]
such that 
\begin{enumerate}
\item the map 
$HH_*(X; \loday (A_V, A_V))
\rightarrow 
HH_*(X; \loday(A_V, \rat))$ is a morphism of graded-commutative algebras;
\item  
the image of $V \otimes HH_*(X; \loday(A_V, \rat))
\rightarrow
HH_*(X; \loday (A_V, A_V))$ is an ideal with trivial multiplication.
\end{enumerate}

These structures are natural with respect to $V\in \ob \fmodq$.
\end{prop}

\begin{proof}
The multiplicative structures of $\loday (A_V, A_V)$ and $\loday (A_V, \rat)$ induce the graded-commutative algebra structure on higher Hochschild homology via the shuffle product (see Proposition \ref{prop:shuffle_product}).

The long exact sequence in homology is induced by the short exact sequence of coefficients, using the obvious isomorphism $HH_*(X; V \otimes  \loday(A_V, \rat))\cong
V \otimes HH_*(X; \loday(A_V, \rat))$. The multiplicative properties follow from Lemma \ref{lem:loday_square_zero} and Proposition \ref{prop:shuffle_product}.
\end{proof}

In general, the connecting morphism is non-trivial, as will be seen below. More precisely, we specialize and consider the functor on $\gr$ given by exploiting the classifying space functor $\zed^{\star r} \mapsto B \zed^{\star r} $. This gives the long exact sequence of functors on $\gr$:
\begin{eqnarray}
\label{eqn:les_gr}
&&
\\
\nonumber
&& \resizebox{.95\hsize}{!}{%
$
\rightarrow 
V \otimes HH_*(B(-); \loday(A_V, \rat))
\rightarrow
HH_*(B(-); \loday (A_V, A_V))
\rightarrow 
HH_*(B(-); \loday(A_V, \rat))
\stackrel{\mathfrak{d}}{\rightarrow } 
,
$
} 
\end{eqnarray}
where the connecting morphism is denoted $\mathfrak{d}$ and the sequence is natural with respect to $V \in \ob \fmodq$.

Theorem \ref{thm:hhh_Loday_AV_k} gives the isomorphism of functors on $\gr$:
\[
HH_* (B (-); \loday (A_{V} , \rat) ) 
\cong 
\Psi (T_\mathrm{coalg}(sV)), 
\]
where the homological grading is encoded by the homological suspension $sV$ (this corresponds to {\em purity})  and $\Psi$ is defined using the symmetry on graded $\rat$-vector spaces with Koszul signs. Again, this is natural with respect to $V \in \ob \fmodq$.

Using this isomorphism and taking into account the degree shift, the connecting morphism of the long exact sequence has the form:
\[
\mathfrak{d} : 
\Psi (T_\mathrm{coalg}(sV)) \rightarrow sV \otimes \Psi ( T_\mathrm{coalg}(sV)).
\]
The key input to understanding $HH_*(B(-); \loday (A_V, A_V))$ is the following:

\begin{prop}
\label{prop:connecting}
The connecting morphism $\mathfrak{d}$ of the long exact sequence (\ref{eqn:les_gr}) identifies as
\[
\coadbar : \Psi (T_\mathrm{coalg}(sV)) \rightarrow sV \otimes \Psi ( T_\mathrm{coalg}(sV))
\]
\end{prop}

\begin{proof}
First consider the map obtained by  evaluating on $\zed \in \ob \gr$ (this is equivalent to considering higher Hochschild homology of $S^1$). The connecting morphism is thus of the form:
\[
T_\mathrm{coalg}(sV) \rightarrow sV \otimes   T_\mathrm{coalg}(sV).
\]
This is calculated by using the differential on the normalized chain complex 
$$N (\loday (A_V, A_V)(\Delta^1/\partial \Delta^1)).$$ As in Example \ref{exam:norm_ch_cx_structure}, in degree $n$ this is isomorphism to 
\[
(\rat \oplus V)  \otimes (sV)^{\otimes n}. 
\]
The differential is induced by $d_0 + (-1)^n d_n$; in this case it is only non-trivial on the  summand $\rat   \otimes (sV)^{\otimes n} \cong (sV)^{\otimes n}$, since the multiplication on $V$ is trivial. From this one deduces that the connecting map in this case  is the graded coadjoint coaction $\coadbar$. 

This extends to the case of $\bigvee_{i=1}^r \Delta^1/\partial \Delta^1$ by exploiting the shuffle maps.  This uses the analogue of Proposition \ref{prop:shuffle_vee}, with $\loday (A_V, A_V)$ in place of $\loday (A,\kring)$, which provides the quasi-isomorphism: 
\[
N (\loday (A_V, A_V ) (X)) \otimes _{A_V} N(\loday (A_V, A_V) (Y) ) 
\rightarrow 
N (\loday (A_V, A_V ) (X \vee Y) ).
\]

This implies that the normalized chain complex $N (\loday (A_V, A_V)(\bigvee_{i=1}^r \Delta^1/\partial \Delta^1))$ is quasi-isomorphic to 
\[
(\rat \oplus V) \otimes \big( T_{\mathrm{coalg}}(sV))^{\otimes r},
\]
with differential that is identified as above. Proceeding as in the case $r=1$, one identifies the connecting map as required.
\end{proof}

This leads to the main result of this section:

\begin{thm}
\label{thm:HH_lodayAVAV_car0}
For  $V \in \ob \fmodq$, there is 
a natural isomorphism in $\fcatQ[\nat \times \gr]$:
\begin{eqnarray}
\label{eqn:HH_loday_AVAV}
\\
\nonumber
HH_* \big(B (-); \loday (A_V, A_V))
\cong 
\omega \Psi (T_{\mathrm{coalg}}(sV) )
\oplus 
\coker( \coadbar_{T_{\mathrm{coalg}}(sV)})[-1],
\end{eqnarray}
where $[-1]$ denotes the shift in homological degree. 

Moreover, 
\begin{enumerate}
\item 
$\omega \Psi (T_{\mathrm{coalg}}(sV) )$ is a graded-commutative algebra;
\item 
$\coker( \coadbar_{T_{\mathrm{coalg}}(sV)})$ is a graded $\omega \Psi (T_{\mathrm{coalg}}(sV) )$-module. 
\end{enumerate}
With respect to these structures, the isomorphism (\ref{eqn:HH_loday_AVAV}) is an isomorphism of graded-commutative algebras, where the right hand side is considered as a square-zero extension.

These identifications are natural with respect to $V$.
\end{thm}

\begin{proof}
Combining the long exact sequence (\ref{eqn:les_gr}) with the identification of the connecting morphism $\mathfrak{d}$ given in Proposition \ref{prop:connecting}, one obtains the short exact sequence in $\fcatQ[\nat \times \gr]$:
\[
0
\rightarrow 
\coker( \coadbar_{T_{\mathrm{coalg}}(sV)})[-1]
\rightarrow 
HH_* \big(B (-); \loday (A_V, A_V))
\rightarrow 
\omega \Psi (T_{\mathrm{coalg}}(sV) )
\rightarrow 
0
\]
that is natural with respect to $V$, using the construction of $\omega$ (see Proposition \ref{prop:Omega_tensor_algebra}, which is stated for the ungraded case).

In homological degree $d$, considered as a functor of $V$, $\coker( \coadbar_{T_{\mathrm{coalg}}(sV)})[-1]$ is homogeneous polynomial of degree $d+1$ and $\omega \Psi (T_{\mathrm{coalg}}(sV) )$ is homogeneous polynomial of degree $d$. The fact that the category of polynomial functors on $\fmodq$ splits as the direct sum of its homogeneous components (cf. Remark \ref{rem:homog_poly_fmodq}) implies that the short exact sequence splits as a functor on $\nat \times \gr$. This gives the first statement.

The multiplicative properties then follow from the general considerations of Proposition \ref{prop:HH_loday_AV_AV}.
\end{proof}

\subsection{Coefficients $\theta^* \inj^\fin$}
 There is a counterpart for $HH_* \big(B (-);\theta^* \inj^\fin)$ for Theorem \ref{thm:HH_lodayAVAV_car0} 
analogous to Theorem \ref{thm:hhh_inj_Gamma} for the coefficients $\inj^\Gamma$. This can be deduced using the Schur correspondence.

Recall from Lemma \ref{lem:inj^Fin} that $\inj^\fin$ is isomorphic to $\theta^+ \inj ^\Gamma$ and hence $\theta^* \inj^\fin$ is isomorphic to $\theta^* \theta^+ \inj ^\Gamma$. Proposition \ref{prop:loday_AV} identifies the associated Schur functor:
\[
  \loday (A_{(-)}, A_{(-)}) \cong \talg  \otimes_{\fb} \theta^* \theta^+ \inj ^\Gamma \cong \talg  \otimes_{\fb} \theta^* \inj^\fin. 
  \]

Under the Schur correspondence, calculating $  HH_* \big(B (-); \loday (A_V, A_V))$ naturally with respect to $V \in \ob \fmodq$ is equivalent to studying 
$
HH_* \big(B (-);\theta^* \inj^\fin).
$ 
Theorem \ref{thm:HH_lodayAVAV_car0} thus implies:

\begin{thm}
\label{thm:HH_theta_inj_Fin}
There is a natural isomorphism in $\f(\nat \times \gr \times \fb; \rat)$:
\[
HH_* \big(B (-);\theta^* \inj^\fin)
\cong 
\omega \Psi (\pcoalg^\dagger) 
\oplus 
\coker (\coadbar_{\pcoalg^\dagger} ) [-1]
\]
where $\omega \Psi (\pcoalg^\dagger)$ and $\coker (\coadbar_{\pcoalg^\dagger} )$ are considered as pure graded functors.
\end{thm}

\begin{rem}
One can also consider the multiplicative structures. However, since $HH_* \big(B (-);\theta^* \inj^\fin)$ is not a pure graded functor, this cannot be encoded entirely in terms of the symmetry $\tau_\nat$ as defined in Section \ref{subsect:coeff_inj_gamma}. 

However, one does have:
\begin{enumerate}
\item 
$\omega \Psi (\pcoalg^\dagger)$ has the structure of a unital commutative monoid (with respect to the symmetry $\tau_\nat$);
\item 
$\coker (\coadbar_{\pcoalg^\dagger} )$ that of a $\omega \Psi (\pcoalg^\dagger)$-module.
\end{enumerate}
The isomorphism of Theorem \ref{thm:HH_theta_inj_Fin} is compatible with these structures if the shift $[-1]$ in homological degree is taken into account. 
\end{rem}

\begin{rem}
Theorem \ref{thm:HH_theta_inj_Fin} can also be proved directly, mirroring the proof of Theorem \ref{thm:HH_lodayAVAV_car0}. For this, the natural short exact sequence (\ref{eqn:ses_loday}) is replaced by the short exact sequence (\ref{eqn:ses_inj}) of Proposition \ref{prop:inj_Gamma_X+_decomp}:
$$
0
\rightarrow 
\inj^\Gamma \tenfb P^\fb _\mathbf{1} 
\rightarrow 
\theta^* \theta^+ \inj^\Gamma
\rightarrow 
\inj^\Gamma 
\rightarrow 
0.
$$

The analysis of the associated long exact sequence then leads to the exact sequence (considered in $\f (\gr \times \fb; \rat)$)  by neglecting the homological grading:
\begin{eqnarray} 
\label{eqn:complex_Psi_pcoalg_coadbar}
\\
\nonumber
0
\rightarrow 
\omega \Psi ({\pcoalg}^\dagger)
\rightarrow 
\Psi ({\pcoalg}^\dagger) 
\stackrel{\coadbar}{\rightarrow} 
\Psi ({\pcoalg}^\dagger) \tenfb P^\fb_\mathbf{1}
\rightarrow 
\coker (\coadbar_{\pcoalg^\dagger} ) 
\rightarrow 
0.
\end{eqnarray}
\end{rem}

\subsection{Isotypical components}
\label{subsect:isotypical_components}

In this section, we consider $HH_* \big(B (-);\inj^\Gamma)$ and $HH_* \big(B (-);\theta^* \inj^\fin)$ as  objects of $\f (\nat \times \gr; \f(\fb\op; \rat))$. This allows us to exploit the functor $\otimes _\fb$ to form isotypical components.

Explicitly, for $\lambda\vdash n$ and the associated simple $\rat [\sym_n]$-module, $S_\lambda$, considered as an $\fb$-module supported on $\mathbf{n} \in \ob \fb$, applying the functor $- \otimes _\fb S_\lambda$ gives the respective isotypical components indexed by $\lambda$:
\begin{eqnarray*}
&& HH_* \big(B (-);\inj^\Gamma)\otimes_\fb S_\lambda
\\
&& HH_* \big(B (-);\theta^* \inj^\fin) \otimes_\fb S_\lambda.
\end{eqnarray*}

There are the corresponding decompositions into isotypical components: 
\begin{eqnarray*}
HH_* \big(B (-);\inj^\Gamma )(\mathbf{n})
&\cong &
\bigoplus_{\lambda\vdash n}
\big(
HH_* \big(B (-);\inj^\Gamma ) \otimes_\fb S_\lambda
\big)^{\oplus \dim S_\lambda}
\\
HH_* \big(B (-);\theta^* \inj^\fin)(\mathbf{n})
&\cong &
\bigoplus_{\lambda\vdash n}
\big(
HH_* \big(B (-);\theta^* \inj^\fin) \otimes_\fb S_\lambda
\big)^{\oplus \dim S_\lambda},
\end{eqnarray*}
where $HH_* \big(B (-);\inj^\Gamma )(\mathbf{n})$ and $HH_* \big(B (-);\theta^* \inj^\fin)(\mathbf{n})$ correspond to evaluating on $\mathbf{n} \in \ob \fb$.

Theorems \ref{thm:hhh_inj_Gamma} and \ref{thm:HH_theta_inj_Fin} have the following Corollary:

\begin{cor}
\label{cor:isotypical_components}
For $n \in \nat^*$ and $\lambda \vdash n$, there is an isomorphism in $\f(\gr; \rat)$:
\[
HH_* \big(B (-);\inj^\Gamma ) \otimes_\fb S_\lambda
\cong 
\left\{
\begin{array}{ll}
\big( \Psi ({\pcoalg}^\dagger) \big)\otimes _\fb S_\lambda & *=n \\
0  & \mbox{otherwise.}
\end{array}
\right.
\]

There are isomorphisms in $\foutQ[\gr]$:
\[
HH_* \big(B (-);\theta^* \inj^\fin) \otimes_\fb S_\lambda
\cong 
\left\{
\begin{array}{ll}
\big(\omega \Psi ({\pcoalg}^\dagger) \big)\otimes _\fb S_\lambda & *=n \\
\big(\coker (\coadbar_{\pcoalg^\dagger} ) \big)\otimes _\fb S_\lambda & *= n-1\\
0  & \mbox{otherwise.}
\end{array}
\right.
\]
\end{cor}

This result can be reformulated using the fundamental functors $\beta_d$ and $\omega \beta_d$ ($d\in \nat$) of the theory of polynomial functors on $\gr$. In the statement, $\preceq$ is the partial order on partitions of Notation \ref{nota:preceq};  for a partition $\lambda$, $\lambda^\dagger$ denotes the conjugate partition.

\begin{prop}
\label{prop:Psi_to_beta}
For $n \in \nat^*$, there are isomorphisms  in $\fpiQ{\gr}$:
\begin{eqnarray*}
\big( \Psi ({\pcoalg}^\dagger) \big)\otimes _\fb S_\lambda 
& \cong &
\beta_n S_{\lambda^\dagger} 
\\
\big( \Psi ({\pcoalg}^\dagger) \tenfb P^\fb_\mathbf{1} \big)\otimes _\fb S_\lambda 
& \cong &
\bigoplus_{\substack{\mu \preceq \lambda \\ |\mu|= n-1}} \beta_{n-1} S_{\mu^\dagger} 
\\
\big(\omega \Psi ({\pcoalg}^\dagger) \big)\otimes _\fb S_\lambda 
& \cong &
\omega \beta_n S_{\lambda^\dagger}. 
\end{eqnarray*}

Hence $\big(\coker (\coadbar_{\pcoalg^\dagger} ) \big)\otimes _\fb S_\lambda $ is the cokernel of the map:
\[
\beta_n S_{\lambda^\dagger} 
\rightarrow 
\bigoplus_{\substack{\mu \preceq \lambda \\ |\mu|= n-1}} \beta_{n-1} S_{\mu^\dagger} 
\]
corresponding to $\coadbar$. This identifies with $\coadbar_{\lambda^\dagger}$, using the notation introduced in Proposition \ref{prop:coadbar_es_isotypical}.
\end{prop}

\begin{proof}
By Proposition \ref{prop:omega_pcoalg_dagger}, $\Psi ({\pcoalg}^\dagger)$ is isomorphic to $\big(\Psi \pcoalg\big)^\dagger$, where the first functor $\Psi$ is calculated with respect to the signed symmetry $\sigma$ and the second with respect to $\tau$. 

From the construction of the functor $(-)^\dagger$ on $\fb$-modules, one deduces the isomorphisms:
\[
\big(\Psi \pcoalg\big)^\dagger \otimes_\fb S_\lambda 
\cong 
\Psi \pcoalg \otimes_\fb (S_\lambda)^\dagger
\cong 
\Psi \pcoalg \otimes_\fb S_{\lambda^\dagger}.
\]
Theorem \ref{thm:beta_d} provides the isomorphism $\Psi \pcoalg \otimes_\fb S_{\lambda^\dagger} \cong \beta_n S_{\lambda^\dagger}$. This gives the first statement. 

The second statement follows from the first by using the general behaviour of the functor $-\tenfb P_{\mathbf{1}}^\fb$ given in Example \ref{exam:tenfb_P1}. This gives the isomorphism:
$$
\big( \Psi ({\pcoalg}^\dagger) \tenfb P^\fb_\mathbf{1} \big)\otimes _\fb S_\lambda 
\cong 
\big( \Psi ({\pcoalg}^\dagger) \big)\otimes_\fb (S_\lambda)\downarrow^{\sym_n}_{\sym_{n-1}},
$$
where $(S_\lambda)\downarrow^{\sym_n}_{\sym_{n-1}}$ is the restriction of $S_\lambda$ as a $\sym_{n-1}$-module. By the Pieri rule, this is isomorphic to $\bigoplus_{\substack{\mu \preceq \lambda \\ |\mu|= n-1}}  S_{\mu}$. The result then follows from the first statement.  

The third statement follows similarly by using Corollary \ref{cor:omega_boldbeta_dagger}.

The final statement follows by applying the exact functor $\otimes _\fb S_\lambda$ to the  exact sequence (\ref{eqn:complex_Psi_pcoalg_coadbar}) and using the previous identifications; the identification with $\coadbar_{\lambda^\dagger}$ is straightforward, noting that $\mu \preceq \lambda$ if and only if $\mu^\dagger \preceq \lambda^\dagger$.
\end{proof}

Putting together Corollary \ref{cor:isotypical_components} with Proposition \ref{prop:Psi_to_beta} one obtains the main result of this section:

\begin{thm}
\label{thm:beta_hhh}
For $n \in \nat^*$ and $\lambda \vdash n$, there  are isomorphisms in $\fpoutgrQ[<\infty]$:
\[
HH_* \big(B (-);\theta^* \inj^\fin) \otimes_\fb S_\lambda
\cong 
\left\{
\begin{array}{ll}
\omega \beta_n S_{\lambda^\dagger} & *=n \\
\coker (\coadbar_{\lambda^\dagger}) & *= n-1\\
0  & \mbox{otherwise.}
\end{array}
\right.
\]
\end{thm}

\begin{rem}
\label{rem:TW_question}
Theorem \ref{thm:beta_hhh} gives our interpretation of the splitting given by Turchin and Willwacher in \cite{TW} into isotypical components, as we explain below.
 Turchin and Willwacher work with higher Hochschild {\em cohomology} (dealt with below by using vector space duality) and give their splitting by working at the level of the associated Schur functors. 

This leads to the modules $U^I_\lambda$ and $U^{II}_\lambda$ introduced in \cite[Section 2.5]{TW}. With our viewpoint, these arise from {\em functors} on $\gr\op$ (the variance resulting from their usage of cohomology).

The corresponding dual functors on $\gr$ identify as
\begin{eqnarray*}
D_{\gr\op} U^I_\lambda &=&  \omega \beta_n S_{\lambda^\dagger}  \\
D_{\gr\op}U^{II}_\lambda &=& \coker (\coadbar_{\lambda^\dagger})
\end{eqnarray*}

Remark \ref{rem:omega_beta_versus_coker_coad} implies that  knowledge of the composition factors of $\omega \beta_n S_{\lambda^\dagger}$ determines those of $\coker (\coadbar_{\lambda^\dagger})$.
In particular,  the Turchin and Willwacher problem is equivalent to calculating the composition factors of $\omega \beta_n S_{\lambda^\dagger}$, for each partition $\lambda$. 

We stress that  the functor $\omega \beta_n S_{\lambda^\dagger}$ has a conceptual interpretation in $\fpoutgrQ[<\infty]$ as the injective envelope of the simple $\alpha_n S_{\lambda^\dagger}$, by Theorem \ref{thm:injectivity_beta} in conjunction with Proposition \ref{prop:beta_socle}.
\end{rem}

\subsection{Splitting as Schur functors} 
 
The splitting of $HH_* \big(B (-);\theta^* \inj^\fin)$ into isotypical components yields that of $V \mapsto HH_* (\big( B(-); \loday (A_V, A_V))$ as a Schur functor:

\begin{prop}
\label{prop:isotypical_schur}
There is an isomorphism in $\f (\nat \times \gr ; \f (\fmodq ; \rat))$:
\[
HH_* (\big( B(-); \loday (A_{(-)}, A_{(-)}))
\cong 
\bigoplus _{n\in \nat} \bigoplus_{\lambda \vdash n}
\Big(
HH_* \big(B (-);\theta^* \inj^\fin) \otimes_\fb S_\lambda 
\Big) \otimes \schur_\lambda (-),
\]
where $\schur_\lambda (-)$ is the Schur functor associated to $S_\lambda$.
\end{prop}

\begin{exam}
\label{exam:Theorem_1_TW}
\cite[Theorem 1]{TW} treats the case $V= \rat$, so that $A_V$ is the ring of dual numbers $\rat[\epsilon]$ (ungraded). The only partitions $\lambda$ for which $\schur_\lambda (\rat) \neq 0$ are the partitions $(n)$, for $n \in \nat$. Hence, Proposition \ref{prop:isotypical_schur} gives the splitting:
\begin{eqnarray*}
 HH_* \big(B (-); \loday (\rat[\epsilon],\rat[\epsilon]))
 &\cong &
 \Big( \bigoplus_{n \in \nat} HH_* (B (-); \theta^* \inj^\fin )\otimes_\fb S_{(n)}\Big) \otimes S^{n}(\rat)
 \\
 &\cong & \bigoplus_{n \in \nat} HH_* (B (-); \theta^* \inj^\fin )\otimes_\fb S_{(n)},
\end{eqnarray*}
where we have used that the $n$th symmetric power $S^n (\rat)$ evaluated on $\rat$ is isomorphic to $\rat$.

Corollary \ref{cor:isotypical_components} and Proposition \ref{prop:Psi_to_beta} can be applied to  $HH_* (B (-); \theta^* \inj^\fin )\otimes_\fb S_{(n)}$. In degree $n$  this identifies as $\omega \beta_n S_{(1^n)}$, using that the conjugate partition of $(n)$ is $(1^n)$. In degree $n-1$ it is the cokernel of the map:
\[
\coadbar_{(1^n)} : 
 \beta_n S_{(1^n)}
 \rightarrow 
 \beta_{n-1} S_{(1^{n-1})}. 
\]
This can be analysed as in Theorem \ref{thm:omega_omega1_sign}, for example.

An alternative approach is that adopted by Turchin and Willwacher.  Theorem \ref{thm:HH_lodayAVAV_car0} gives the isomorphism in $\fcatQ[\nat \times \gr]$:
$$
HH_* \big(B (-); \loday (A_\rat, A_\rat)
\cong 
\omega \Psi (T_{\mathrm{coalg}}^\dagger(\rat) )
\oplus 
\coker( \coadbar_{T_{\mathrm{coalg}}^\dagger(\rat)})[-1],
$$
and the Hopf algebra $T_{\mathrm{coalg}}^\dagger(\rat)$ identifies as the free graded-commutative algebra $\gsym (x, y  )$, where $|x|=1$ and $|y|=2$; here $x$ is primitive and the reduced diagonal of $y$ is $\overline{\Delta} y = x \otimes x$. 

Turchin and Willwacher observed in \cite{TW} that  the connecting morphism $\mathfrak{d}$ (corresponding to $\coadbar$) identifies with 
$
\Psi \gsym (x, y  )
\rightarrow 
\Psi \gsym (x, y  ),
$
 the derivation induced by the de Rham differential $d$ given by $x \mapsto 0$, $y \mapsto x$.

In particular acyclicity of $(\gsym (x, y  ), d)$ relates $\ker d $ and $\coker d$, with an appropriate degree shift. As in the proof of Theorem \ref{thm:omega_omega1_sign}, this determines the full structure.
 \end{exam} 

\subsection{Exploiting the functoriality with respect to $\gr$}
\label{subsect:DT}

To illustrate the interest of these structural results, we consider the relationship with the  higher Hochschild homology of the classifying space of free abelian groups.  

For $r \in \nat$, consider the natural surjection 
$ 
\zed^{\star r}\rightarrow \zed^r \cong \A (\zed^{\star r})$   
,  viewed as a natural transformation of functors from $\gr$ to groups. Composing with the functor  on groups $HH_* (B(-);\loday (\rat [\epsilon], \rat [\epsilon]) )$ gives the natural transformation 
\[
HH_* (B \zed^{\star r}; \loday (\rat [\epsilon], \rat [\epsilon])  ) 
\rightarrow 
HH_* (B \A ( \zed^{\star r});\loday (\rat [\epsilon], \rat [\epsilon]) ),
\]
where the codomain lies  in the essential image of $\fcatQ[\ab]$ in  $\fcatQ[\gr]$.

\begin{prop}
\label{prop:abelianization_dual_numbers}
For $n \in \nat^*$, the kernel of the natural transformation
\[
HH_n (B \zed^{\star r}; \loday (\rat [\epsilon], \rat [\epsilon]))
\rightarrow 
HH_n (B \A (\zed^{\star r}); \loday (\rat [\epsilon], \rat [\epsilon]))
\]
contains the uniserial functor $\soc_{[\frac{n+1}{2}] -1} \omega \beta_n S_{(1^n)}$. 
\end{prop}

\begin{proof}
The functor $\zed^{\star r} \mapsto HH_* (B \zed^{\star r}; \loday (\rat [\epsilon], \rat [\epsilon]))$ is given by  Example \ref{exam:Theorem_1_TW} in conjunction with Theorem \ref{thm:omega_omega1_sign}. Explicitly, by Theorem \ref{thm:beta_hhh}, this gives in degree $n \geq 1$:
\[
HH_n (B \zed^{\star r}; \loday (\rat [\epsilon], \rat [\epsilon]))
\cong 
\omega \beta_n S_{(1^n)} \oplus \omega \beta_{n-1} S_{(1^{n-1})},
\]
where the term $\omega \beta_{n-1} S_{(1^{n-1})}$ is the contribution from $\coadbar_{(1^{n+1})}$. By Theorem  \ref{thm:omega_omega1_sign},  the functor $\omega \beta_n S_{(1^n)}$  is uniserial, with socle length $[\frac{n+1}{2}]$. 

Now, by construction, the functor $\zed^{\star r} \mapsto HH_n (B \A(\zed^{\star r}); \loday (\rat [\epsilon], \rat [\epsilon]))$ lies in the essential image of $\fpiQ{\ab} $ in $\fpiQ{\gr}$. 
In particular, this functor is semi-simple.

The composite map 
$$\omega \beta_n S_{(1^n)} \subset HH_n (B \zed^{\star r}; \loday (\rat [\epsilon], \rat [\epsilon])) \rightarrow HH_n (B \A(\zed^{\star r}); \loday (\rat [\epsilon], \rat [\epsilon]))$$
 is a morphism from the uniserial functor $\omega \beta_n S_{(1^n)}$ of socle length $[\frac{n+1}{2}]$ to a semi-simple object.
The socle filtration of a finite uniserial object identifies (up to reindexing) with its radical filtration.  
 It follows that the kernel of the composite map  contains $\soc_{[\frac{n+1}{2}] -1} \omega \beta_n S_{(1^n)}$, by the defining properties of the radical filtration. 
\end{proof}

\begin{rem}
\label{rem:Dundas_Tenti}
Proposition \ref{prop:abelianization_dual_numbers} illustrates how dramatically higher Hochschild homology 
$$X \mapsto HH_* (X; \loday (\rat [\epsilon], \rat [\epsilon]))$$
 fails to be a stable invariant, i.e., does not depend only upon the stable homotopy type of $X$. (This was  shown by Dundas and Tenti \cite{DT}, by a different argument.)

If higher Hochschild homology were a stable invariant, then
\[
HH_* (B \zed^{\star r}; \loday (\rat [\epsilon], \rat [\epsilon])) 
\rightarrow 
HH_* (B \zed^r; \loday (\rat [\epsilon], \rat [\epsilon])),
\]
would be a monomorphism, since stably $B \zed^r$ contains $B \zed^{\star r}$ as a wedge summand (cf. \cite{DT}). Proposition \ref{prop:abelianization_dual_numbers} shows that this is not the case in homological degree $n \geq 3$.
\end{rem}

\section{Pirashvili's Hodge filtration}
\label{sect:hodge_filt}

In this section,  we explain how Pirashvili's Hodge filtration \cite{Phh} is related to the functorial picture  considered here.

The Hodge filtration is that associated to Pirashvili's spectral sequence (recalled in Section \ref{subsect:Hodge_filt_ss} below) which calculates the higher Hochschild homology $HH_* (X; L)$ of a pointed simplicial space $X$  with coefficients in a $\Gamma$-module $L$, starting from information given by certain $\tor$-groups for $\Gamma$-modules.

Sections \ref{subsect:Hodge_filt_ss} and  \ref{subsect:Hodge_filt_B} do not use the results of Sections \ref{sec:hhh_triv} and \ref{sec:hhh_calc}. Then Section \ref{subsect:calculate_Gamma_cohomology} applies the results of Sections \ref{subsect:Hodge_filt_ss} and  \ref{subsect:Hodge_filt_B} in conjunction with Theorem \ref{thm:hhh_inj_Gamma} to calculate some fundamental $\ext$-groups in $\Gamma$-modules.

In Section \ref{subsect:Hodge_filt_B}, 
 we restrict to functors on $\gr$ by considering
\[
HH_* (B(-); L).
\]
Theorem \ref{thm:Hodge_filt_groups} gives a  functorial description of the Hodge filtration of this functor and relates this to the canonical polynomial filtration with respect to $\gr$ (cf. Section \ref{sec:recoll}). In particular, this result shows that, after passing to the associated graded of the Hodge filtration,  higher Hochschild homology is given in terms of $\Gamma$-module functor homology, as explained in Remark \ref{rem:HH_versus_Gamma_module_homology}.

Theorem \ref{thm:Hodge_filt_groups} is used in Section \ref{subsect:calculate_Gamma_cohomology} to calculate 
\[
\ext^*_{\fcatQ[\Gamma]} (\inj^\Gamma, {\inj^\Gamma})
\]
where $\inj^\Gamma$ is the $\Gamma$-module with values in $\fb$-modules introduced in Section \ref{sec:loday}.

The result is stated in Theorem \ref{thm:Koszul_duality} and is a shadow of the Koszul duality that underlies this story.


\subsection{The Hodge filtration}
\label{subsect:Hodge_filt_ss}

For $X$ a pointed simplicial set and $L\in \ob \fcatQ[\Gamma]$, Pirashvili \cite[Theorem 2.4]{Phh} constructed the Hodge filtration spectral sequence 
\begin{eqnarray}
\label{eqn:hodge_ss_E2}
E_{pq}^2 
= 
\tor^\Gamma _p (\mathcal{T}_q H_* (X; \rat)  , L) 
\Rightarrow 
HH_* (X;L). 
\end{eqnarray}

Here: 
\nomenclature{$\mathcal{T}$}{dual Loday construction\nomrefpage}
\begin{enumerate}
\item
the functors $\tor^\Gamma$ are the derived functors of the  tensor product $\otimes_\Gamma$ (cf. Section \ref{subsect:tensor_over_C});
\item 
$H_*(X; \rat)$ is considered as a coaugmented coalgebra, as usual;
\item 
$\mathcal{T}_q H_* (X; \rat)$ is the degree $q$ part of Pirashvili's dual Loday construction $\mathcal{T}H_* (X; \rat)$ on $H_* (X; \rat)$. (For $C$ a  graded cocommutative, coaugmented $\rat$-coalgebra,  $\mathcal{T} C $ is the $\Gamma\op$-module  $\mathbf{n}_+ \mapsto C^{\otimes n}$, which is `dual' to the Loday construction $\loday (A, \rat)$ for augmented algebras of Section \ref{subsect:loday_A} -- see \cite[Section 1.7]{Phh}.) 
\end{enumerate}

The Hodge filtration of $ HH_* (X;L)$ is the filtration associated to Pirashvili's spectral sequence (\ref{eqn:hodge_ss_E2}).

\begin{rem}
\ 
\begin{enumerate}
\item 
The Hodge filtration and the associated spectral sequence are natural with respect to both $X$ and $L$.
\item 
In \cite{TW}, in the case $X= B \zed^{\star r}$, Turchin and Willwacher also work with a `Hodge filtration' (in cohomology); the reader is warned that their filtration is not the dual  of Pirashvili's Hodge filtration considered here.
\end{enumerate}
\end{rem}

We shall apply vector space duality to various functor categories, corresponding to post-composition with $V \mapsto V^\sharp$. To simplify notation, in this section duality will be denoted simply by $D$ (as opposed to Notation \ref{nota:duality}). 
 For instance, duality relates $\Gamma$-modules and $\Gamma\op$-modules:
\begin{eqnarray*}
 D&:& \fcatQ[\Gamma]\op 
 \rightarrow 
 \fcatQ[\Gamma\op].
\\
 D&:& \fcatQ[\Gamma\op]\op 
 \rightarrow 
 \fcatQ[\Gamma].
\end{eqnarray*}
Duality also relates $\tor^\Gamma$  to  $\ext^*_{\fcatQ[\Gamma]}$, under the appropriate finiteness hypotheses. The following is standard (recalling that an object of an abelian category is finite if it has a finite composition series):

\begin{prop}
\label{prop:ext_tor}
For $F, G$ finite objects of $\fcatQ[\Gamma]$,  there is a natural isomorphism 
\[
\tor^\Gamma_* (D F,  G) 
^\sharp 
\cong 
\ext^* _{\fcatQ[\Gamma]} (G, F).
\]
\end{prop}

This gives the following reinterpretation of the $E^2$-page of Pirashvili's spectral sequence:

\begin{cor}
\label{cor:tor_to_ext}
If $L$ is a finite $\Gamma$-module, $X$ is connected and $H_*(X; \rat)$ has finite type as a graded $\rat$-vector space, the $E^2$-page of Pirashvili's spectral sequence is 
\[
E^2_{pq}
=
\ext^p _{\fcatQ[\Gamma]} (L,  \loday^q (H^* (X; \rat), \rat))^\sharp.
\]
Here, the Loday construction  $\loday (-, \rat)$ is applied in the graded context, with Koszul signs and  $\loday^q (H^* (X; \rat), \rat)$ is the component of cohomological degree $q$.
\end{cor}

\begin{proof}
The hypothesis on $H_*(X; \rat)$ ensures that $\mathcal{T}_q H_* (X; \rat)$ is a finite $\Gamma$-module and is dual to $\loday^q (H^* (X; \rat), \rat)$. The result thus follows from Proposition \ref{prop:ext_tor}.
\end{proof}

\subsection{Specializing to classifying spaces of free groups}
\label{subsect:Hodge_filt_B}
We take $X$ to be the classifying space $B \zed^{\star r}$, considered as a functor of the free group $\zed^{\star r} $. The following is clear:

\begin{lem}
\label{lem:cohomology_BFr}
The functor from $\gr\op$ to graded commutative, augmented algebras,  
$ 
\fr  \mapsto H^* (B \fr  ; \rat),
$
is isomorphic to the functor $\rat \oplus s \aQ^\sharp $, considered as the square-zero extension algebra of $\rat$ by the functorial $\rat$-module $\aQ^\sharp$ concentrated  in cohomological degree one.
\end{lem}

This allows the identification of $\loday (H^*(B(-); \rat), \rat)$:

\begin{lem}
\label{lem:loday_H^*B}
The functor $\fr \mapsto \loday (H^*(B \fr; \rat), \rat)$, considered as an object of $\fcatQ[\nat \times \gr\op \times \Gamma]$, is isomorphic to 
\[
\talg (s \aQ^\sharp) \otimes _\fb \inj^\Gamma.
\]

In particular, in cohomological degree $q\in \nat$, the associated object of $\fcatQ[\gr\op \times \Gamma]$ is 
\[
(\mathrm{sgn}_q \otimes(\aQ^\sharp)^{\otimes q}) \otimes_{\sym_q}  \inj ^\Gamma (\mathbf{q}, -)
\cong 
(\aQ^\sharp)^{\otimes q} \otimes_{\sym_q} (\mathrm{sgn}_q \otimes \inj ^\Gamma (\mathbf{q}, -)),
\]
in which $\sym_q$ acts diagonally on the terms $\mathrm{sgn}_q \otimes -$.
\end{lem}

\begin{proof}
The first statement is a variant of Proposition \ref{prop:loday_AV}, passing to the graded context.  The second statement simply identifies the degree $q$ component, using that  the cohomological suspension introduces the  Koszul signs corresponding to the signature representation $\mathrm{sgn}_q$. 
\end{proof}

\begin{nota}
\label{nota:ing_Gamma_dagger}
Write  ${\inj^\Gamma}^\dagger$ for the $\fb \times \Gamma$-module $ 
(\mathbf{q}, Y) \mapsto \mathrm{sgn}_q \otimes \inj ^\Gamma (\mathbf{q}, Y).
$ 
\end{nota}

\begin{rem}
\ 
\begin{enumerate}
\item
 ${\inj^\Gamma}^\dagger$ encodes the functor $\loday (H^*(B(-); \rat), \rat)$, by 
Lemma \ref{lem:loday_H^*B}.
\item 
Lemma \ref{lem:loday_H^*B} dualizes to give a description of $\mathcal{T} H_* (B(-); \rat)$,  expressed in terms of the dual, $D {\inj^\Gamma}^\dagger$.
\end{enumerate}
\end{rem}

For any $\Gamma$-module, the Hodge filtration for $HH_* (B(-); L)$ is an increasing filtration (corresponding to filtering by $p \leq t$ in the spectral sequence, for $t \in \nat$ indexing the term of the filtration). In a fixed homological degree $d$, for current purposes it is useful to index this by $q$, which gives the {\em decreasing} filtration  
of functors on $\gr$:
\[
 \resizebox{\hsize}{!}{%
 $
\ldots \subset  
 \filt^{q+1} HH_d (B(-); L)
\subset  
 \filt^q HH_d (B(-); L)
 \subset  
 \ldots 
 \subset
 \filt^0 HH_d (B(-); L)
 =HH_d(B(-); L)
$
}
\]
such that $\filt^{d+1} HH_d (B(-); L)=0$.

The filtration quotients of this filtration are described by the following, in which $(\aQ^{\otimes q})^\dagger$ denotes $(\mathrm{sgn}_q \otimes (\aQ^{\otimes q}))$ with diagonal action of $\sym_q$, and $\alpha_q$ is the functor appearing in Theorem \ref{thm:recollement_poly_d}:

\begin{thm}
\label{thm:Hodge_filt_groups}
For $L \in \ob \fcatQ[\Gamma]$, the Hodge filtration of $
 HH_d (B(-); L)$ has  filtration quotients in homological degree $d$: 
\begin{eqnarray*}
 \filt^q HH_d (B(-); L) / \filt^{q+1} HH_{d} (B(-); L) 
 &\cong &
 (\aQ^{\otimes q})^\dagger  \otimes_{\sym_q} \tor^\Gamma _{d-q} ( D \inj ^\Gamma (\mathbf{q}, -) , L)
 \\
 &\cong &
 \alpha_q \Big( 
 \tor^\Gamma _{d-q} ( D {\inj ^\Gamma}^\dagger (\mathbf{q}, -) , L)
 \Big).
\end{eqnarray*}
In particular, 
\begin{enumerate}
\item 
the functor $HH_d (B(-); L)$ has polynomial degree $d$; 
\item 
the Hodge filtration of $ HH_d (B(-); L) \in \ob \fcatQ[\gr]$ coincides with the polynomial filtration (up to indexing).
\end{enumerate}
\end{thm}

\begin{proof}
We use Pirashvili's Hodge filtration spectral sequence \cite[Theorem 2.4]{Phh}. This has  $E^2$-page 
\[
E^2_{pq}=\tor^\Gamma _p (\mathcal{T}_q H_* (B(-); \rat)  , L). 
\]
Using the dual to Lemma \ref{lem:loday_H^*B}, this is isomorphic to 
\[
(\aQ^{\otimes q})^\dagger  \otimes_{\sym_q} \tor^\Gamma _{d-q} ( D \inj ^\Gamma (\mathbf{q}, -),L)
\]
where $d=p+q$ and we have used that $(\aQ^{\otimes q})^\dagger  \otimes_{\sym_q} -$ is exact, hence commutes with the formation of $\tor^\Gamma$. This term can be rewritten using the functor $\alpha_q$ as:
\[
\alpha_q \Big( 
 \tor^\Gamma _{d-q} ( D {\inj ^\Gamma}^\dagger (\mathbf{q}, -) , L)
 \Big),
 \]
where the twist by $\mathrm{sgn}_q$ has been incorporated as ${\inj^\Gamma}^\dagger$.

This implies that the $q$th row in the spectral sequence consists of functors on $\gr$ that are polynomial, homogeneous of degree exactly $q$ (more precisely, they lie in the image of the functor $\alpha_q$ of Section \ref{sec:recoll}).
 
Since  the spectral sequence is natural with respect to $\gr$, the differentials are morphisms in $\fcatQ[\gr]$. We deduce that all differentials are zero, since there are no non-trivial morphisms between an object of the image of $\alpha_q$  and an object of the image of $\alpha_{q'}$ if $q\neq q'$ (this is a consequence of the equality $\hom_{\fcatQ[\gr]}(\aQ^{\otimes q}, \aQ ^{\otimes q'}) =0$, if $q\neq q'$, as exploited in Section \ref{sec:recoll}). It follows that the spectral sequence degenerates at the $E^2$-page, giving the stated identification of the associated graded of the Hodge filtration.

This shows that the Hodge filtration of $HH_d (B(-); L) $ is a finite length, decreasing filtration, with filtration quotient $\filt^q/\filt^{q+1}$ a  polynomial functor of degree $q$ that lies in the image of $\alpha_q$. In particular,  $HH_d (B(-); L) $ is polynomial of degree at most $d$ and Proposition \ref{prop:canon_poly_filt} implies that this filtration coincides with the polynomial filtration (up to indexing).
\end{proof}

\begin{rem}
\ 
\begin{enumerate}
\item
In \cite[Theorem 2.6]{Phh} Pirashvili showed that the Hodge filtration spectral sequence for spheres degenerates, by using an {\em ad hoc} argument. In our setting, we can exploit the naturality with respect to $\gr$ to prove degeneracy. 
\item 
The Hodge filtration of Theorem \ref{thm:Hodge_filt_groups} does  not split functorially. This can be seen by considering the structure of $HH_* (B(-); \inj^\Gamma )$, using 
Theorem \ref{thm:hhh_inj_Gamma}  and the identifications in Section \ref{subsect:isotypical_components}; for an explicit example, consider the isotypical component of $\lambda = (1^n)$ for $n >1$.
\end{enumerate}
\end{rem}

\begin{rem}
\label{rem:HH_versus_Gamma_module_homology}
Theorem \ref{thm:Hodge_filt_groups} shows that the associated graded of the  Hodge filtration
is determined by the $\nat$-graded $\fb$-module: 
 \[
\tor^\Gamma_* (D \inj^\Gamma, L) .
\]
These homology groups generalize the usual definition of functor homology for $\Gamma$-modules.

If $L$ is a finite $\Gamma$-module, the groups can be expressed in terms of $\ext^*_{\fcatQ[\Gamma]}$ using Proposition \ref{prop:ext_tor} as 
\[
\ext^*_{\fcatQ[\Gamma]} (L, \inj^\Gamma)^\sharp.
\]
\end{rem}
 
\subsection{Calculating (co)homology of $\Gamma$-modules}
\label{subsect:calculate_Gamma_cohomology}

In this section we indicate how to use Theorem \ref{thm:Hodge_filt_groups} to calculate certain $\ext$-groups in $\Gamma$-modules, exploiting our analysis of higher Hochschild homology with coefficients  $L= \inj^\Gamma$ (which takes values in $\fb$-modules). 

Theorem \ref{thm:hhh_inj_Gamma} identifies:
 \begin{eqnarray}
 \label{eqn:HH_inj_Gamma}
 HH_* (B (-); \inj^\Gamma) \cong \Psi (\pcoalg^\dagger)
 \end{eqnarray}
 as functors in $\f(\nat \times \gr \times \fb; \rat)$, where $\nat$ corresponds to the homological grading. 
 
Theorem \ref{thm:Hodge_filt_groups} gives that the associated graded of the polynomial filtration of  $\Psi (\pcoalg^\dagger)$ determines the family (for $m,n\in \nat$) of $\sym_m\op \times \sym_n$-representations:
\[
\ext^* _{\fcatQ[\Gamma]}
(\inj ^\Gamma (\mathbf{n}, -),\inj^\Gamma (\mathbf{m},-)),
\] 
where the $\sym_m$- and $\sym_n$-actions come from the $\fb$-module structure of $\inj^\Gamma$. The rest of this section is devoted to outlining how this is done; the result is stated in Theorem \ref{thm:Koszul_duality}.

We first extract the dagger by using the isomorphism given by Proposition \ref{prop:omega_pcoalg_dagger}:
\[
\Psi (\pcoalg^\dagger)
 \cong 
\big( \Psi \pcoalg \big)^\dagger,
\]
which allows us to focus upon $\Psi \pcoalg$.

Theorem \ref{thm:beta_description} provides the isomorphism 
\begin{eqnarray}
\label{eqn:boldbeta}
\boldbeta
\cong
 \Psi \pcoalg,
 \end{eqnarray}
  where $\boldbeta \in \ob \fcatQ[\gr \times \fb]$ such that, for $d \in \nat$, $\boldbeta (-, \mathbf{d})$ is the functor $\beta_d \rat [\sym_d]$ on $\gr$. In particular, this is an injective in the category $\fpiQ{\gr}$ of polynomial functors on $\gr$, by the results of Section \ref{sec:beta}.
 
 The functor represented by $\boldbeta (-, \mathbf{d})$  is described succinctly below by using the functor  $\fbcr : \f(\gr;\rat) \rightarrow \f (\fb ; \rat)$ introduced in Definition \ref{defn:fbcr}. (Recall that, for $F$ a polynomial functor on $\rat$, $\fbcr F$ encodes the associated graded of the polynomial filtration of $F$.) 
 
 \begin{prop}
 \label{prop:fbcr_boldbeta}
 For $F \in \fpiQ{\gr}$ and $d\in \nat$, there is a natural isomorphism of $\sym_d$-modules
 \[
\big( \fbcr F (\mathbf{d}) \big) ^\sharp \cong \hom_{\fcatQ[\gr]} (F, \boldbeta (-, \mathbf{d})).
 \]
 \end{prop}
 
\begin{proof}
This follows from the adjunctions for the functors $\cre_d$, $\qhat{d}$, $\beta_d$, using the methods of section \ref{subsect:poly_filt_revisit}.
\end{proof} 
 
This is the main input into:

\begin{thm}
\label{thm:Koszul_duality}
For  $m,n \in \nat^*$, there is a natural isomorphism of $\sym_m \times \sym_n\op$-modules:
\begin{eqnarray*}
&&\ext^* _{\fcatQ[\Gamma]}
(\inj ^\Gamma (\mathbf{n}, -),\inj^\Gamma (\mathbf{m},-))
\\
&&
\quad 
\cong
\left\{
\begin{array}{ll}
 \mathrm{sgn}_m \otimes \hom_{\fcatQ[\gr]} ( \boldbeta (-, \mathbf{n}), \boldbeta (-, \mathbf{m}) )
\otimes \mathrm{sgn}_n
 & *= n-m \\
 0 & \mbox{otherwise.}
\end{array}
\right.
\end{eqnarray*}
\end{thm}

\begin{proof}
Theorem \ref{thm:Hodge_filt_groups}  leads to the isomorphism
\[
\ext^*_{\fcatQ[\Gamma]} (\inj^\Gamma, {\inj^\Gamma}^\dagger)
\cong 
\big(\fbcr HH_*(B(-); \inj^\Gamma)\big)^\sharp. 
\]
The right hand side is calculated using Proposition \ref{prop:fbcr_boldbeta}, applied by using the isomorphisms (\ref{eqn:HH_inj_Gamma}) and (\ref{eqn:boldbeta}). The details are left to the reader.
\end{proof}

\begin{rem}
\ 
\begin{enumerate}
\item 
This theorem has a conceptual interpretation in terms of Koszul duality. This is treated (working with functors on $\gr\op$) in \cite{2021arXiv211001934P} by using operadic Koszul duality.
\item 
The modules $\hom_{\fcatQ[\gr]} ( \boldbeta (-, \mathbf{n}), \boldbeta (-, \mathbf{m}) )$ can be calculated using standard methods of representation theory in terms of the family of representations $\mathrm{Lie} (t)$, $t\in \nat$, appearing in the Lie operad. (This is encoded in the theory developed in \cite{2021arXiv211001934P}.)
\end{enumerate}
\end{rem}

\appendix
\part{Appendices}
\label{part:app}

\section{Background on functor categories}
\label{sec:background}

This appendix reviews the basic theory of functor categories, with the aim of making the paper reasonably self-contained. 

\subsection{Functor categories}
\label{subsect:functor}

 For $\calc$ a small category and $\cala$ an abelian category, let 
$\fcat{\calc}{\cala}$ denote the category  of functors from $\calc$ to $\cala$. This is sometimes termed the category of $\calc$-modules (with values in $\cala$); it is an abelian category.

Evaluation on an object $C$ of $\calc$ gives the exact functor 
 $
 \eval_C : \fcat{\calc}{\cala} \rightarrow \cala.
$ 
If $(\cala, \otimes, \unit, \tau)$ is a symmetric monoidal structure, then $\fcat{\calc}{\cala}$ 
inherits a symmetric monoidal structure  with tensor product defined objectwise in $\cala$. 

\nomenclature{$\fcatk$}{functors from $\calc$ to $\kring\dash\modules$\nomrefpage}
\nomenclature{$P^\calc_C$}{standard projective associated to object $C$ of $\calc$\nomrefpage}
The shorthand $\fcatk$ is used for  $\fcat{\calc}{\kring\dash\modules}$. The category is 
equipped with tensor product $\otimes _\kring$ and has sufficiently many 
projectives and  injectives. For example, the standard projective $P^{\calc}_C$ is $\kring [\hom _\calc (C, -)]$; Yoneda's lemma gives that $P^{\calc}_C$ corepresents evaluation on $C$. 

Duality is a useful tool, especially when $\kring$ is a field:

\begin{nota}
\label{nota:duality}
\nomenclature{$D_\calc$}{duality functor for $\fcatk$\nomrefpage}
For $\kring$ a field, denote by 
$
  D_{\calc} : \fcatk \op \rightarrow \fcatk[\calc\op]
 $
 the functor given by $(D_\calc F) (X):= F(X)^\sharp$, where $\sharp$ denotes vector space duality.
 \end{nota}

When $\kring$ is a field,  the full subcategory of functors taking finite-dimensional values is denoted  $\fcatkfd \subset \fcatk$.

\begin{prop}
\label{prop:duality}
(Cf. \cite[Proposition C.2.7]{D}.)
 If $\kring$ is a field, then the duality functors $D_{\calc}$ and $D_{\calc 
\op}$ are adjoint,
 \[
  D_{\calc} : \fcatk \op \rightleftarrows \fcatk[\calc \op] : D_{\calc \op} \op,
 \]
and the adjunction restricts to an equivalence of categories 
$$\fcatkfd \op 
\cong 
\fcatkfd[\calc \op].$$
 In particular, for $F \in \ob \fcatkfd$, the adjunction 
unit induces an isomorphism $F \cong D_{\calc \op} D_\calc F$.
\end{prop}

An adjunction between small categories induces an adjunction between the 
associated functor categories:

\begin{prop}
\label{prop:precomp_adjunction}
 Let $\cala$ be an abelian category and $L : \calc \rightleftarrows \cald : R$ 
be an adjunction between 
 small categories. Precomposition induces an adjunction
 \[
  R^* : \fcat{\calc}{\cala} \rightleftarrows \fcat{\cald}{\cala} : L^* 
 \]
of exact functors between abelian categories. If $\cala$ is tensor abelian, then the functors $R^*$ and $L^*$ are monoidal. 
\end{prop}

\begin{rem}
The unit $\mathrm{Id}_\calc \rightarrow R L $   of the adjunction of Proposition \ref{prop:precomp_adjunction} induces, by precomposition, $  \mathrm{Id}_{\fcat{\calc}{\cala}}\rightarrow (RL)^* = L^* R^*$. This is the unit of the induced adjunction at the level of the functor categories. In particular, this explains why $L^*$ is the {\em right} adjoint. 

Moreover, developing this leads to a proof of the Proposition, using the formulation of an adjunction in terms of the unit and counit natural transformations.
\end{rem}

\subsection{Polynomial functors}
\label{sec:poly}

Throughout this section, suppose:

\begin{hyp}
\label{hyp:C_monoidal}
 The category $\calc$ is pointed by $*$ and is equipped with monoidal 
structure $\vee$ for which $*$ is the unit
\end{hyp}

The Eilenberg-Mac Lane notion of polynomial functor \cite{EM} generalizes to 
the setting of $\fcat{\calc}{\cala}$, where $\cala$ is abelian, as in 
\cite[Section 3]{HPV}.

\begin{nota}
\label{nota:poly_functor}
 For $(\calc, \vee, *)$ as above, let $\f_d(\calc ; \cala) \subset 
\fcat{\calc}{\cala}$ denote the full 
subcategory of polynomial functors of degree at most $d \in \nat$. 
\end{nota}

\begin{prop}(Cf. \cite{HPV}.)
\label{prop:poly_postcompose_additive}
Suppose that $\calc$ satisfies Hypothesis \ref{hyp:C_monoidal}. For  
$\alpha : \cala \rightarrow \calb$ an additive 
functor between abelian categories and $d \in \nat$, the functor 
$\fcat{\calc}{\cala}\rightarrow 
\fcat{\calc}{\calb}$ induced by   composition with $\alpha$ restricts to 
  $
   \f_d (\calc ; \cala)
   \rightarrow 
   \f_d (\calc ; \calb).
  $
\end{prop}

For the remainder of the section, suppose that $\cala = \kring\dash\modules$.

\begin{nota}
\label{nota:p_q_poly} For $d \in \nat$, 
let $q_d^\calc : \fcatk[\calc] \rightarrow  \fpoly{d}{\calc}$ denote the left 
adjoint  to the inclusion $\fpoly{d}{\calc} \hookrightarrow \fcatk[\calc]$ and 
$p_d^\calc : \fcatk[\calc] \rightarrow  \fpoly{d}{\calc}$  the right adjoint. 
(Where no confusion can result, these will be denoted simply by $q_d$ and 
$p_d$.)
\end{nota}

The following is standard:

\begin{prop}
\label{prop:poly_standard}
Suppose that $\calc$ satisfies Hypothesis \ref{hyp:C_monoidal} and let $d\in \nat$. 
\begin{enumerate}
 \item 
 The category $\fpoly{d}{\calc}$ is a thick subcategory of $\fcatk$ and is 
stable under  limits and colimits.
\item 
 The category $\fpoly{d}{\calc}$ has enough projectives and 
enough injectives.
\item 
The tensor product of $\fcatk$ restricts to 
\[
 \otimes : 
\fpoly{d}{\calc} 
\times 
\fpoly{e}{\calc}
\rightarrow 
\fpoly{d+e}{\calc}.
\]
\item 
There are  canonical inclusions $\fpoly{d}{\calc} \subset
\fpoly{d+1}{\calc} \subset \fcatk$  which induce natural transformations:
$
 p^\calc_d \hookrightarrow p^\calc_{d+1} \hookrightarrow 1_{\fcatk}$
and $1_{\fcatk} \twoheadrightarrow q^\calc_{d+1} \twoheadrightarrow q^\calc_d$.
 \end{enumerate}
\end{prop}

\begin{rem}
The category $(\calc \op , \vee, *)$ satisfies Hypothesis \ref{hyp:C_monoidal}, 
so that polynomial functors 
are defined for contravariant functors;  the covariant and contravariant notions of polynomial functor are related 
\cite[Section 3]{HPV}. 
\end{rem}

\begin{prop} (Cf. \cite{HPV}.)
\label{prop:polynomial_duality}
For $\calc$  satisfying Hypothesis \ref{hyp:C_monoidal} and $\kring$ a field, the duality adjunction restricts to 
 $
 D_\calc : \fpoly{d}{\calc} \op \rightleftarrows \fpoly{d}{\calc \op} : 
D_{\calc 
\op}.
$ 
\end{prop}

\subsection{Bifunctors}
\label{subsect:bifunctors}

Let $\calc$ and $\cald$ be small categories; $\fcatk[\calc \times \cald]$ 
is considered as the category of bifunctors on $\calc$ and $\cald$. 

\begin{prop}
\label{prop:bifunctor_exponential}
 For $\calc$ and $\cald$ small categories, there are equivalences of tensor 
abelian categories 
 $
  \fcatk[\calc \times \cald] 
  \cong 
  \fcat{\calc}{\fcatk[\cald]}
  \cong 
  \fcat{\cald}{\fcatk[\calc]}.
 $
\end{prop}

\begin{lem}
\label{lem:bifunctor_eval}
For $\calc$ and $\cald$ small categories, evaluation on $D \in \ob \cald$ 
induces an exact, monoidal functor $
  \eval_D : \fcatk[\calc \times \cald] \rightarrow \fcatk[\calc]$, 
$ F(-,-) \mapsto F(- , D)$.
\end{lem}

\begin{nota}
\label{nota:d[C]}
Suppose  that $\calc$ satisfies Hypothesis \ref{hyp:C_monoidal}.  For $d \in \nat$, let
$\f_{d[\calc]} (\calc \times \cald; \kring)$ denote the full subcategory of 
$\fcatk[\calc\times \cald]$ that corresponds to $\f_d (\calc ; \fcatk[\cald])$ 
under the equivalence of 
Proposition \ref{prop:bifunctor_exponential}.
\end{nota}

\begin{prop}
\label{prop:poly_bifunctor_equivalent}
 Let $\calc$ and $\cald$ be small categories such that $\calc$ satisfies Hypothesis \ref{hyp:C_monoidal}   and let $d \in \nat$.  
 \begin{enumerate}
  \item 
  For $D \in \ob \cald$, the evaluation functor $\eval_D$ restricts to 
  \[
   \eval_D : \fpoly{d[\calc]}{\calc \times \cald} \rightarrow \fpoly{d}{\calc}.
  \]
\item 
A bifunctor $F \in \ob \fcatk[\calc \times \cald]$ lies in 
$\fpoly{d[\calc]}{\calc \times \cald}$ if and only 
if $\eval_D F$ lies in $\fpoly{d}{\calc}$ for all $D \in \ob \cald$.
  \end{enumerate}
\end{prop}

\subsection{Tensor products over $\calc$}
\label{subsect:tensor_over_C}

The notion of external tensor product is familiar:

\begin{nota}
\label{nota:boxtimes}
\nomenclature{$\otimes_\calc$}{tensor product over category $\calc$\nomrefpage}
\nomenclature{$\boxtimes$}{exterior tensor product\nomrefpage}
For $\calc$, $\cald$ small categories and $\kring$ a commutative ring, let $\boxtimes_\kring$ be the external tensor product:
$\boxtimes_\kring :    \fcatk[\calc ] \times \fcatk[\cald] \rightarrow 
\fcatk[\calc  \times \cald]$, $(F, G) \mapsto F (-) \otimes G(-)$.
 Where no confusion can result, this is written simply as $\boxtimes$.
\end{nota}

\begin{nota}
\label{nota:otimes_calc}
For $\calc$ a small category, let 
 $
 \otimes_\calc : \fcatk[\calc \op] \times \fcatk[\calc] \rightarrow \kring 
\dash \modules
 $
 denote the biadditive functor given by composing the external tensor product $\boxtimes$
  with the coend $\fcatk[\calc 
\op \times \calc] \rightarrow \kring 
\dash \modules$.
\end{nota}

\begin{rem}
For $F \in \ob \fcatk[\calc \op]$ and $G \in \ob \fcatk[\calc]$, where $\calc$ is a small category, $F \otimes_\calc G$ is the coequalizer of the diagram:
\[
\bigoplus _{f \in \hom_\calc (X,Y)}
F(Y) \otimes G(X)
\rightrightarrows 
\bigoplus _{Z \in \ob \calc} 
F(Z) \otimes G(Z),
\]
where the arrows are determined by $F(X) \otimes G(X) \stackrel{F(f) \otimes \id} \leftarrow F(Y) \otimes G(X) \stackrel{\id \otimes G(f)}{\rightarrow} F(Y) \otimes G(Y)$.  This corresponds to a generalization to `rings with many objects' of the usual tensor product of right and left modules over a unital, associative ring.
\end{rem}

The naturality of $\otimes_\calc$ gives:

\begin{lem}
 \label{lem:bifunctor}
For small categories $\calc, \cald$, the functor $\otimes_\calc$ 
induces a biadditive functor:
$
\otimes _\calc  :  \fcatk[\calc \op \times \cald] \times \fcatk[\calc]  
\rightarrow \fcatk[\cald].
$
\end{lem}

\section{Representations of the symmetric groups in characteristic zero}
\label{app:rep_sym_car0}

This appendix provides a brief survey of the representation theory of the symmetric groups over $\rat$ that is required here; the reader can consult standard references for this material, including \cite{F,FH}. 

Recall the following standard definition:

\begin{defn}
\label{def:partitions}
\nomenclature{$\lambda \vdash n$}{partition of $n$\nomrefpage}
A partition $\lambda$ is a weakly decreasing sequence of natural numbers $\lambda_i$, $i \in \nat^*$ such that $\lambda_i =0 $ for $i \gg 0$. The size $|\lambda|$ of $\lambda$ is $\sum_i \lambda_i$ and, if $|\lambda |= n$, then $\lambda$ is said to be a partition of $n$, denoted $\lambda \vdash n$. 
The length $l(\lambda)$ of $\lambda$ is the number of non-zero entries and a partition of length $l$ 
 is written simply as $(\lambda_1, \ldots , \lambda_l)$. 
 
 The transpose partition $\lambda^\dagger$ is defined by $\lambda^\dagger_i = \sharp \{ j | \lambda_j \geq i \}$. (Geometrically this corresponds to reflecting the Young diagram  representing $\lambda$ in its diagonal.)
\end{defn}

\begin{nota}
\label{nota:preceq}
\nomenclature{$\preceq$}{partial order on the set of partitions\nomrefpage}
For partitions $\lambda$, $\mu$ write $\mu \preceq \lambda$ if $\mu_i \leq \lambda_i$ for all terms of the partition (unspecified terms being taken as zero); equivalently, the Young diagram of $\mu$ is contained within that of $\lambda$.
\end{nota}

The simple  $\rat [\sym_n]$-modules (for $n \in \nat$) are indexed by the partitions $\lambda$ of $n$, the simple associated to $\lambda \vdash n$ is written $S_\lambda$. 

\begin{exam}
\label{exam:example_representations}
To fix conventions, for $0< n \in \nat$:
\begin{enumerate}
\item 
$S_{(n)}$ is the trivial representation $\rat$; 
\item 
$S_{(1^n)}$ is the sign representation $\mathrm{sgn}$; 
\item 
$S_{(n-1, 1)}$ is the standard representation of dimension $n-1$, given by the quotient of the permutation representation $\rat[\mathbf{n}]$ by the trivial representation (given by the trace). 
\end{enumerate}
Here and elsewhere, where no confusion can result, commas are omitted from the notation for partitions.
\end{exam}

\begin{rem}
\label{rem:dim_irreducibles}
The dimension of $S_\lambda$ can be calculated combinatorially from the partition $\lambda$. 
For instance, it is the number of {\em standard tableaux} of type $\lambda$. This dimension can also be calculated by the {\em hook-length formula}.
\end{rem}

\begin{prop}
\label{prop:self-duality}
Finite-dimensional representations of $\rat[\sym_n]$ are self-dual under vector space duality $(-)^*$; in particular, $(S_\lambda) ^*\cong S_\lambda$. 
\end{prop}

\begin{nota}
\label{nota:dagger_sym}
For $M$ a representation of $\rat[\sym_n]$, let $M^\dagger$ denote the representation $ M \otimes \mathrm{sgn}_n$ (with diagonal action), where $\mathrm{sgn}_n$ is the signature representation.
\end{nota}

\begin{exam}
If $\lambda \vdash n$, then $(S_{\lambda})^\dagger \cong S_{\lambda^\dagger}$, so that the two usages of the notation $^\dagger$ are compatible.
\end{exam}

\begin{nota}
\label{nota:Schur_functors}
\nomenclature{$S_\lambda$}{simple $\sym_n$-module indexed by $\lambda \vdash n$\nomrefpage}
\nomenclature{$\schur_\lambda$}{Schur functor associated to $S_\lambda$\nomrefpage}
For $n \in \nat$ and $S_\lambda$ the simple $\rat[\sym_n]$-module induced by the partition $\lambda$ of $n$, let  $\schur_\lambda$ denote the Schur functor $V \mapsto S_\lambda \otimes_{\sym_n} V^{\otimes n}$. (See Section \ref{sec:fb} for Schur functors.)
\end{nota}

\begin{exam}
\label{exam:schur_symm_ext}
For $n \in \nat$, 
\begin{enumerate}
\item 
$\schur_{(n)}$ is the $n$th symmetric power functor, $V \mapsto S^n (V)$; 
\item 
$\schur_{(1^n)}$ is the $n$-th exterior power functor, $V \mapsto \Lambda^n (V)$.
\end{enumerate}
\end{exam}

The Pieri rule is important for understanding the induction of representations of symmetric groups:

\begin{prop}
\label{prop:pieri}
For $S_\lambda$ a simple representation of $\rat[\sym_n]$,
$
S_\lambda \uparrow_{\sym_n}^{\sym_{n+1}} 
\cong
\bigoplus 
S_\mu, 
$
where the sum is taken over partitions $\mu$ such that $|\mu| = 1 + |\lambda|$ and $\lambda \subset \mu$.
\end{prop}

\begin{exam}
\label{exam:pieri}
Let  $\lambda$ be a partition with $|\lambda|= n$. 
\begin{enumerate}
\item 
$S_{(1^{n+1})}$ is a factor of $S_\lambda \uparrow_{\sym_n}^{\sym_{n+1}}$ if and only if $\lambda= (1^n)$, in which case $S_{(1^n)}\uparrow_{\sym_n}^{\sym_{n+1}} = S_{(1^{n+1})} \oplus S_{(2 1^{n-1})}$; 
\item 
$S_{(n+1)}$ is a factor of $S_\lambda \uparrow_{\sym_n}^{\sym_{n+1}}$ if and only if $\lambda= (n)$, in which case $S_{(n)}\uparrow_{\sym_n}^{\sym_{n+1}} = S_{(n+1)} \oplus S_{(n1)}$.
\end{enumerate}
\end{exam}

If $M$ is a representation of $\sym_m$ and $N$ of $\sym_n$, then $M \tenfb N$  is the representation of $\sym_{m+n}$ given by
\[
(M \boxtimes N)\uparrow_{\sym_m \times \sym_n}^{\sym_{m+n}}. 
\]
For instance, for $\lambda \vdash m$ and $\mu \vdash n$, the Schur functor associate to $S_\lambda \tenfb S_\mu$ is $\schur_\lambda \otimes \schur_\mu$.

Since the category $\rat[\sym_{m+n}]\dash\modules$ is semi-simple, this decomposes as a direct sum of simple modules.

\begin{defn}
\label{def:LR-coefficients}
\nomenclature{$c_{\lambda , \mu}^\nu$}{Littlewood-Richardson coefficient\nomrefpage}
For partitions $\lambda , \mu, \nu$, the Littlewood-Richardson coefficient $c_{\lambda , \mu}^\nu$ is defined by the isomorphism
$
S_\lambda \tenfb S_\mu \cong \bigoplus_\nu S_\nu ^{\oplus c_{\lambda, \mu}^\nu}.
$ 
\end{defn}

\begin{rem}
From the definition, it is clear that $c^{\nu}_{\lambda, \mu}=0$ if $|\nu|\neq |\lambda|+ |\mu|$ and $c_{\lambda, \mu}^\nu =  c_{\mu , \lambda}^\nu$, 
$ c_{\lambda^\dagger, \mu^\dagger}^{\nu^\dagger} =  c_{\lambda, \mu}^\nu .$
\end{rem}

More importantly, the Littlewood-Richardson coefficients can be calculated combinatorially.

\begin{thm}
\label{thm:LR-coeff}
For partitions $\lambda, \mu, \nu$, the Littlewood–Richardson coefficient $c_{\lambda,\mu}^\nu$ is zero unless $\lambda \subset \nu$. If $\lambda \subset \nu$, $c_{\lambda,\mu}^\nu$ is the number of semi-standard skew tableaux of shape $\nu \backslash \lambda$ and content $\mu$ (that is $i$ appears $\mu_i$ times) which satisfy the following lattice word condition: the sequence obtained by concatenating its reversed rows is a lattice word (for each initial segment of the sequence and for each $i$, $i$ occurs at least as many times as $i+1$). 
\end{thm}

\begin{exam}
\label{exam:CR_1n}
The calculation of $c^\nu_{\lambda,\mu}$ is especially simple when $\mu = (1^n)$. In this case, the lattice word condition implies :
\begin{enumerate}
\item 
$c^\nu_{\lambda,1^n} =0$ if $\exists i$ such that $\nu _i - \lambda_i >1$ (that is, if the skew tableau $\nu \backslash \lambda$ has more than one box in any row); 
\item 
in the remaining case, there is a unique semi-standard skew-tableau $\nu \backslash \lambda$ with content $1^n$ satisfying the lattice word condition (with increasing entries). 
\end{enumerate}
Thus, for example, $c^\nu _{\lambda, 1^n}=0$ if $n > l (\nu)$.
\end{exam}

\begin{exam}
\label{exam:LR}
For convenience, the following Littlewood-Richardson decompositions are given:
\begin{eqnarray*}
S_{(1)} \tenfb S_{(1)} &=& S_{(2)} \oplus S_{(11)} \\
S_{(2)} \tenfb S_{(1)} &=& S_{(3)} \oplus S_{(21)}\\
S_{(11)} \tenfb S_{(1)} &=& S_{(111)} \oplus S_{(21)}\\
S_{(3)} \tenfb S_{(1)} &=& S_{(4)} \oplus S_{(31)} \\
S_{(21)} \tenfb S_{(1)} &=& S_{(31)} \oplus S_{(22)} \oplus S_{(211)} \\
S_{(111)} \tenfb S_{(1)} & =& S_{(1111)} \oplus S_{(211)} \\
S_{(2)} \tenfb S_{(2)} &=& S_{(4)} \oplus S_{(31)} \oplus S_{(22)}\\
S_{(2)}\tenfb S_{(11)} &=& S_{(31)} \oplus S_{(211)} \\
S_{(11)}\tenfb S_{(11)} &=& S_{(1111)} \oplus S_{(211)} \oplus S_{(22)}
\end{eqnarray*}
\end{exam}

Similarly, the plethysm operation for representations of symmetric groups are described by coefficients:

\begin{defn}
\label{def:plethysm-coefficients}
\nomenclature{$p_{\lambda , \mu}^\nu$}{plethysm coefficient\nomrefpage}
For partitions $\lambda , \mu, \nu$, the plethysm coefficient $p_{\lambda , \mu}^\nu$ is defined by the isomorphism
$
S_\lambda \circ S_\mu \cong \bigoplus_\nu S_\nu ^{\oplus p_{\lambda, \mu}^\nu}, 
$ 
where $S_\lambda \circ S_\mu $ is the representation associated to the composite functor $\schur_\lambda \circ \schur_\mu$. 
(These coefficients are zero unless $|\nu|= |\lambda|\ |\mu|$.)
\end{defn}

\begin{rem}
Unlike the Littlewood-Richardson coefficients, there is no known general combinatorial description of the plethysm coefficients.
\end{rem}

For reference we record the following:

\begin{lem}
\label{lem:plethysm}
There are isomorphisms:
\begin{eqnarray*}
S_{(11)} \circ S_{(11)} &\cong & S_{(211)} \\
S_{(2)} \circ S_{(11)} & \cong & S_{(1111)} \oplus S_{(22)} .
\end{eqnarray*}
\end{lem}

\begin{proof}
This can be given an elementary proof. Since $S_{(1)} \tenfb S_{(1)} \cong S_{(11)} \oplus S_{(2)}$, 
it follows easily that 
\[
\big(S_{(11)} \circ S_{(11)}\big) 
\oplus 
\big( S_{(2)} \circ S_{(11)}\big)
\cong 
S_{(11)} \tenfb S_{(11)}
\]
and the latter identifies with $ S_{(1111)} \oplus S_{(211)} \oplus S_{(22)}$ (see Example \ref{exam:LR}). It remains to check how these factors are distributed. 

The product of exterior powers induces $S_{(11)} \tenfb S_{(11)} \rightarrow S_{(1111)}$  which is symmetric, which shows that $S_{(1111)}$ lies in $S_{(2)} \circ S_{(11)}$. Moreover, evaluating the associated Schur functors on $\rat^{\oplus 2}$ shows that $S_{(22)} $ must also lie in $S_{(2)} \circ S_{(11)}$. 
\end{proof}

\subsection{The Cauchy identities}
\label{subsect:cauchy}

The Cauchy identities in characteristic zero give a fundamental tool for  considering symmetric (respectively exterior) powers applied to tensor products.

Use Notation \ref{nota:Schur_functors} for the Schur functors associated to simple representations.

\begin{prop}
\label{prop:cauchy_identities}
For $0< n \in \nat$, there are natural isomorphisms for $V, W \in \ob \fmod$:
 \begin{eqnarray*}
  S^n (V \otimes W) &\cong & \bigoplus_{\lambda \vdash n} \schur_\lambda (V) 
\otimes \schur_\lambda (W) \\
  \Lambda^n (V \otimes W) &\cong & \bigoplus_{\lambda \vdash n} \schur_\lambda 
(V) \otimes \schur_{\lambda^\dagger} (W)
 \end{eqnarray*}
where the sum is taken  over all partitions of $n$.
\end{prop}

\begin{proof}
This result is contained in \cite[(6.2.8)]{SS}.

Since the functorial nature is essential here, a proof is sketched. Namely, the result concerns bi-Schur functors and therefore reduces to a result about bimodules over the symmetric group $\sym_n$. The heart of the proof is the structure of $\rat [\sym_n]$ as a bimodule over itself. The decomposition into isotypical components say for the left action gives a decomposition as bimodules. Using the fact that the simple representations of $\sym_n$ are self-dual, there is a bimodule decomposition 
$
\rat [\sym_n] \cong \bigoplus _{\lambda \vdash n} S_\lambda \boxtimes S_\lambda.
$  
The Cauchy identities follow by using the fact that $S^n (V \otimes W)$ is the bi-Schur functor associated to the bimodule 
$ 
\rat \otimes_{\sym_n} (\rat [\sym_n] \otimes \rat [\sym_n]), 
$ 
where the left action is diagonal upon the tensor product. The result follows by the Schur Lemma. 
\end{proof}

\subsection{The Lie modules}
\label{subsect:lie}

\begin{nota}
\label{nota:lie_n}
\nomenclature{$\liemod(n)$}{$n$th Lie module\nomrefpage}
Let  $\liemod (n) \in \ob \rat [\sym_n]\dash\modules$, $n \in \nat$, denote the $n$th Lie module; these modules assemble to define $\liemod(-) \in \ob \f(\fb;\rat) $, the $\fb$-module underlying the Lie operad (cf. \cite[Section 1.3.3]{LV}, \cite{Reu}).
\end{nota}

\begin{nota}
\label{nota:Schur_Lie}
\nomenclature{$\lieschur(n)$}{$n$th Lie (Schur) functor\nomrefpage}
Denote the associated Schur functors of degree $n \in \nat$ by:
  \begin{enumerate}
   \item 
   $ \lieschur(n) (V) := \liemod (n)\otimes_{\sym_n} 
  V^{\otimes n}$; 
  \item 
  $ \lieschur(n)^{\dagger} (V) := \liemod (n)^{\dagger}\otimes_{\sym_n} 
  V^{\otimes n}$.
  \end{enumerate}
 \end{nota}

\begin{rem}
\label{free-Lie}
By definition,  the free Lie algebra $\lie(V)$ on a $\rat$-vector space  $V$ is:
\[
\lie (V) 
\cong 
\bigoplus_{n \in \nat} \liemod(n) \otimes _{\sym_n} V^{\otimes n}.
\]
\end{rem}

\begin{exam}
\label{exam:liemod}
From \cite[Table 8.1, page 208]{Reu}, one has the description of the first $\liemod (n)$:
\begin{eqnarray*}
\liemod(1) &=& S_{(1)}\\
\liemod(2) &=&S_{(11)}\\
\liemod(3) & =& S_{(21)}\\
\liemod(4) &=& S_{(31)}\oplus S_{(211)}\\
\liemod(5) &=& S_{(41)} \oplus S_{(32)}\oplus S_{(311)} \oplus S_{(221)} \oplus S_{(2111)}\\
\liemod(6) &=& S_{(51)} \oplus S_{(42)} \oplus S_{(411)}^{\oplus 2} \oplus S_{(33)} \oplus  S_{(321)}^{\oplus 3} \oplus S_{(3111)} \oplus S_{(2211)}^{\oplus 2} \oplus S_{(21111)}.
\end{eqnarray*}
\end{exam}

The following underlines both the complexity of the modules $\liemod (n)$ and the special role played by the trivial and sign representations.

\begin{thm}
\label{thm:Reutenauer_exceptional}
\cite[Theorem 8.12]{Reu}
For $n \in \nat^*$ and $\lambda \vdash n$, $S_\lambda$ is a composition factor of $\liemod (n)$ if and only if $\lambda$ is not one of the following:
$(s)$ for $s\geq 2$; $(1^t)$ for $t \geq 3$; $(2^2)$; $(2^3)$.
\end{thm}

More precisely, one has the following determination of $\liemod(n)$:

\begin{thm}
\label{thm:Reutenauer_multiplicities}
\cite[Corollary 8.10]{Reu} 
Let $(i,n ) =1$ for $i,n \in \nat$. For $\lambda \vdash n$, the multiplicity of $S_\lambda$ as a composition factor of  $\liemod(n)$ is the number of standard tableaux of shape $\lambda $ and of major index congruent to $i$ mod $n$. 
\end{thm}

\begin{rem}
For a standard tableau $T$ of shape $\lambda\vdash n$, the descent set $D(T)$  of $T$ is the number of $i \in \{ 1, \ldots , n-1 \}$ such that $i+1$ appears in a lower row than $i$ in $T$. The major index of $T$ is $\mathrm{maj} (T) := \sum_{i \in D(T)} i$. 
\end{rem}

\newpage

\bibliographystyle{smfalpha}
\bibliography{gr_out.bib}

\newpage

\printnomenclature[2.5cm]
\vfill 
\end{document}